\documentclass[11pt]{article}

\usepackage[T1]{fontenc}
\usepackage[utf8]{inputenc}
\usepackage{textcomp}
\usepackage{lmodern}
\usepackage{microtype}

\usepackage{amsmath}
\usepackage{amssymb}
\usepackage{amsthm}
\usepackage{amscd}
\usepackage{mathtools}
\usepackage[inline]{enumitem} 
\usepackage{array} 
\usepackage{multirow} 
\usepackage{bm}
\usepackage{bbm}

\usepackage{tikz} 
\usetikzlibrary{cd} 

\usepackage[in]{fullpage} 

\usepackage[%
  pdftex,
  plainpages=false,
  pdfpagelabels,
  pagebackref,
  colorlinks,
  citecolor=black,
  linkcolor=black,
  urlcolor=black,
  filecolor=black,
  bookmarksopen=false
]{hyperref} 
\usepackage[all]{hypcap} 

\newtheoremstyle{thmstyle}
  {\medskipamount}
  {\smallskipamount}
  {\slshape}
  {0pt}
  {\bfseries}
  {.}
  { }
  {\thmname{#1}\thmnumber{ #2}{\normalfont\thmnote{ (#3)}}}
  
\newtheoremstyle{plainstyle}
  {\medskipamount}
  {\smallskipamount}
  {\rmfamily}
  {0pt}
  {\bfseries}
  {.}
  { }
  {\thmname{#1}\thmnumber{ #2}{\normalfont\thmnote{ (#3)}}}

\theoremstyle{thmstyle}
\newtheorem{theorem}{Theorem}[section]
\newtheorem{lemma}[theorem]{Lemma}
\newtheorem{proposition}[theorem]{Proposition}
\newtheorem{corollary}[theorem]{Corollary}

\newtheorem{claim}{Claim}[theorem]

\theoremstyle{plainstyle}
\newtheorem{definition}[theorem]{Definition}

\newtheorem{remark}[theorem]{Remark}

\newtheorem{example}[theorem]{Example}

\newenvironment{proofof}[1]{\begin{proof}[Proof of #1.]}{\end{proof}}

\setlist[enumerate]{label={\roman*.}, ref={(\roman*)}}

\numberwithin{equation}{section} 

\DeclareMathOperator{\qr}{qr}
\DeclareMathOperator{\qrT}{qrT}
\DeclareMathOperator{\twist}{twist}
\DeclareMathOperator{\semitwist}{semitwist}
\DeclareMathOperator{\bp}{bp}
\DeclareMathOperator{\disc}{disc}
\DeclareMathOperator{\cycle}{cycle}
\DeclareMathOperator{\odd}{odd}
\DeclareMathOperator{\weak}{weak}
\DeclareMathOperator{\dsct}{dsct}
\DeclareMathOperator{\ovlp}{ovlp}
\DeclareMathOperator{\amlg}{amlg}
\DeclareMathOperator{\dom}{dom}
\DeclareMathOperator{\im}{im}
\DeclareMathOperator{\id}{id}
\newcommand{\kqrO}[1][k]{#1\operatorname{-qrO}}
\DeclareMathOperator{\rk}{rk}
\def\Tind{T_{\operatorname{ind}}}
\def\tind{t_{\operatorname{ind}}}
\DeclareMathOperator{\Aut}{Aut}
\newcommand{\HomT}[1][T]{\operatorname{Hom}^+(\cA[#1],\RR)}

\def\property{\texttt}
\def\Independence{\property{Independence}}
\def\UCouple{\property{UCouple}}
\def\UInduce{\property{UInduce}}
\def\Disc{\property{Disc}}

\newcommand{\TLinOrder}{T_{\operatorname{LinOrder}}}
\newcommand{\TkO}[1][k]{T_{#1\operatorname{-Orientation}}}

\def\cat{\textsc}
\def\Set{\cat{Set}}
\def\Inj{\cat{Inj}}
\def\Int{\cat{Int}}

\def\down{\mathord{\downarrow}}
\def\up{\mathord{\uparrow}}
\def\pto{\rightharpoonup}

\let\epsilon\varepsilon
\newcommand{\df}{\stackrel{\text{def}}{=}}
\def\rn{\bm}
\def\rest{\mathord{\vert}}
\def\comp{\mathbin{\circ}}
\def\symdiff{\mathbin{\triangle}}
\def\disjcup{\mathbin{\stackrel{\cdot}{\cup}}}
\def\place{\mathord{-}}
\newcommand{\floor}[1]{\lfloor #1\rfloor}

\newcommand{\given}[1][]{\mathbin{#1\vert}}
\newcommand{\psilin}{\psi_{\operatorname{lin}}}

\def\EE{\mathbb{E}}
\def\NN{\mathbb{N}}
\def\PP{\mathbb{P}}
\def\QQ{\mathbb{Q}}
\def\RR{\mathbb{R}}

\def\pp{\ensuremath{\mathbbm{p}}}
\def\qq{\ensuremath{\mathbbm{q}}}

\def\One{\mathbbm{1}}

\def\cA{{\mathcal{A}}}
\def\cB{{\mathcal{B}}}
\def\cC{{\mathcal{C}}}
\def\cD{{\mathcal{D}}}
\def\cE{{\mathcal{E}}}
\def\cG{{\mathcal{G}}}
\def\cH{{\mathcal{H}}}
\def\cK{{\mathcal{K}}}
\def\cL{{\mathcal{L}}}
\def\cM{{\mathcal{M}}}
\def\cN{{\mathcal{N}}}
\def\cO{{\mathcal{O}}}
\def\cP{{\mathcal{P}}}
\def\cR{{\mathcal{R}}}
\def\cS{{\mathcal{S}}}
\def\cX{{\mathcal{X}}}

\def\Caratheodory{Carath\'{e}odory}

\title{
  On the equivalence of quasirandomness\\
  and exchangeable representations\\
  independent from lower-order variables%
}
\author{%
  Leonardo N.~Coregliano \and%
  Henry P.~Towsner%
}
\date{\today}

\begin{document}

\maketitle

\begin{abstract}
  It is often convenient to represent a process for randomly generating a graph as a graphon. (More precisely, these give \emph{vertex exchangeable} processes---those processes in which each vertex is treated the same way.) Other structures can be treated by generalizations like hypergraphons, permutatons, and, for a very general class, theons. These representations are not unique: different representations can lead to the same probability distribution on graphs. This naturally leads to questions (going back at least to Hoover's proof of the Aldous--Hoover Theorem on the existence of such representations) that ask when quasirandomness properties on the distribution guarantee the existence of particularly simple representations.

  We extend the usual theon representation by adding an additional datum of a random permutation to each tuple, which we call a $\ast$-representation. We show that if a process satisfies the \emph{unique coupling} property $\UCouple[\ell]$, which says roughly that all $\ell$-tuples of vertices ``look the same'', then the process is $\ast$-$\ell$-independent: there is a $\ast$-representation that does not make use of any random information about $\ell$-tuples (including tuples of length $<\ell$). Simple examples show that the use of $\ast$-representations is necessary.

  This resolves a question of Coregliano and Razborov, since it easily follows that $\UCouple[\ell]$ implies $\Independence[\ell']$ (the existence of an $\ell'$-independent ordinary representation) for $\ell'<\ell$. 
\end{abstract}

\section{Introduction}

Consider a rule for randomly generating a graph on a set of vertices $V$:
\begin{itemize}
\item for each vertex $v\in V$ we pick a random value $\rn{\omega}_v\in[0,1]$ uniformly and independently,
\item for each pair of vertices $\{v,w\}\in\binom{V}{2}$ we pick a random value
  $\rn{\omega}_{v,w}\in[0,1]$, also uniformly and independently,
\item we have a measurable set $\cN_E\subseteq[0,1]^3$, and we place an edge\footnote{Our
  notation here foreshadows the general situation, where we might be considering structures with
  multiple edge relations simultaneously. In that setting, we will need a different set
  $\cN_E$ for each edge relation $E$.} between the pair $\{v,w\}$ if
  $(\rn{\omega}_v,\rn{\omega}_w,\rn{\omega}_{v,w})\in\cN_E$.
\end{itemize}
Such a construction is sometimes called a \emph{$2$-hypergraphon}\footnote{A \emph{graphon} is a
  closely related object: specifically, given $\cN_E$, we may define a graphon
  $f_E\colon [0,1]^2\to[0,1]$ by $f_E(\omega_v,\omega_w)\df\lambda(\{\omega_{v,w}\mid
  (\omega_v,\omega_w,\omega_{v,w})\in\cN_E\})$, where $\lambda$ is the Lebesgue
  measure.}. These arise, for instance, as a way to represent the limit of a suitably convergent
sequence of graphs, and they have been studied under using several essentially equivalent
formalisms---exchangeable relations~\cite{Aus08,DJ08}, graph limits and graphons and their
generalizations~\cite{LS06,Lov12,ES12,CR20}, flag algebras~\cite{Raz07} and
ultraproducts~\cite{ES12,NO12,AC14,NO20}. The Aldous--Hoover Theorem~\cite{Hoo79,Ald81,Ald85} (see
also~\cite[Chapter~7]{Kal05}) says that \emph{any} method for randomly generating a graph that
treats the vertices symmetrically (that is, which is \emph{exchangeable}) can be represented as a
distribution over $2$-hypergraphons.

\begin{example}[The Quasirandom Graph]\label{ex:qr}
  A fundamental example is the case where $\cN^{\qr}_E\df\{(\omega_v,\omega_w,\omega_{v,w})\mid
  \omega_{v,w}<1/2\}$. This is essentially the process where we flip a coin for each pair $\{v,w\}$
  and place an edge if the coin comes up heads.
\end{example}

The representation given by a $2$-hypergraphon is not unique: there are other $2$-hypergraphons that
are essentially the same. For instance, any measure-preserving permutation of $[0,1]$ gives us
essentially the same rule---$\cN^{\qr'}_E\df\{(\omega_v,\omega_w,\omega_{v,w})\mid
\omega_{v,w}<1/4\text{ or }\omega_{v,w}\geq 3/4\}$ is not meaningfully different from
$\cN^{\qr}_E$.

The more fundamental object is the probability distribution that $\cN$ induces on a set of
vertices. That is, if we fix a finite vertex set $V$ and let $\cK_V$ be the set of graphs on
$V$ (with vertices labeled by the elements of $V$), $\cN$ induces a probability
distribution $\mu^\cN_V$ on $\cK_V$ where the probability of
$K\in\cK_V$ is precisely the probability that generating a graph using $\cN$ by
the method described above produces the graph $K$. (This idea also makes sense, with slightly more
care about measurability, when $V$ is infinite.) For example, when $V$ is finite,
$\mu^{\cN^{\qr}}_V$ is the uniform distribution on (labeled) graphs on $V$.

We can then say that $\cN$ and $\cN'$ are equivalent if $\mu^\cN_V$
and $\mu^{\cN'}_V$ are the same for all finite sets $V$.

The example $\cN^{\qr}$ has a distinctive property: the variables $\omega_v$ and $\omega_w$
are not actually used, that is, $\cN^{\qr}_E=[0,1]^2\times \cH^{\qr}_E$ for a set
$\cH^{\qr}_E$. We call this property \emph{$1$-independence}, because the $2$-hypergraphon is
independent of the unary variables $\omega_v,\omega_w$.

Complicating matters, a $1$-independent $2$-hypergraphon may be equivalent to one that is not
$1$-independent.
\begin{example}\label{ex:twist}
  Consider the rule where $\cN^{\twist}_E\df\{(\omega_v,\omega_w,\omega_{v,w})\mid
  (\omega_v+\omega_{w}+\omega_{v,w})\bmod 1<1/2\}$.
\end{example}
It is not difficult to check that $\mu^{\cN^{\twist}}_V$ is the same as $\mu^{\cN^{\qr}}_V$ for
every finite set $V$.

To get a property of the underlying object, instead of just the representation, we can say that
$\cN$ has $\Independence[1]$ if it is equivalent to some $1$-independent set. For graphs, there is
not much to say: any $2$-hypergraphon with $\Independence[1]$ is equivalent to
$\{(\omega_v,\omega_w,\omega_{v,w})\mid \omega_{v,w}<p\}$ for some $p\in[0,1]$. However we will see
that the situation is more complicated when we consider directed graphs and hypergraphs.

The main topic of this paper is understanding how the quasirandomness properties of the measures
$(\mu^\cN_V)_V$ relate to whether $\mu^\cN_V$ has a $1$-independent---or, more generally,
$\ell$-independent---representation. This question was already raised in Hoover's original
paper~\cite{Hoo79}, where he asks when $\cN$ has a representation omitting some
variables. Nevertheless, our result here is the first general result showing how to take a suitably
quasirandom $(\mu^\cN_V)_V$ and produce a representation omitting some of the variables.

To explain our motivation for doing this, we need to consider what happens when we try generalize
the notion of quasirandomness beyond graphs. For the purposes of the introduction, it will suffice
to consider $3$-hypergraphs. We will let $H$, rather than $E$, be the name of our hypergraph
relation. A $3$-hypergraphon is given by a measurable set $\cN_H\subseteq[0,1]^7$. The seven here is
the number of non-empty subsets of a set of three elements\footnote{In our general notation below,
  where we will consider relations on $d$-tuples, we will write $\cN_H\subseteq\cE_{k(H)}$. Here
  $k(H)$ is the arity of $H$, in this case $3$, and $\cE_{k(H)}$ will be defined to be
  $[0,1]^{r(k(H))}$, where $r(k(H))$ is the set of non-empty subsets of
  $[k(H)]\df\{1,\ldots,k(H)\}$.}. To generate a $3$-hypergraph on $V$, we select random variables
$(\rn{\omega}_v)_{v\in V}$, $(\rn{\omega}_{v,w})_{\{v,w\}\in\binom{V}{2}}$, and
$(\rn{\omega}_{v,w,x})_{\{v,w,x\}\in\binom{V}{3}}$ uniformly and independently in $[0,1]$---that is,
we choose random variables for each vertex, each unordered pair of vertices, and each unordered
triple of vertices---and we place a $3$-edge on the triple $\{v,w,x\}$ if
\begin{align*}
  (\rn{\omega}_v,\rn{\omega}_w,\rn{\omega}_x,
  \rn{\omega}_{v,w},\rn{\omega}_{v,x},\rn{\omega}_{w,x},
  \rn{\omega}_{v,w,x})
  & \in\cN_H.
\end{align*}
We abbreviate the long tuple by $(\rn{\omega}_s)_{s\in r(\{v,w,x\})}$ to emphasize that we are indexing
the variables by non-empty subsets of the index set $\{v,w,x\}$, where we write $r(V)$ for the set
of non-empty subsets of $V$.

There are a number of ways we might generalize the properties of $\cN^{\qr}$ to a
$3$-hypergraph~\cite{LM15b,Tow17,ACHPS18}. Let us focus just on those ways that $\cN_H$ might be
quasirandom relative to the unary data $\{\omega_v,\omega_w,\omega_x\}$.

Before stating these, it is helpful to generalize the $\mu^\cN_V$ notation to allow the possibility
that some of our random variables have already been chosen. Given a finite set $V$ and a set
$S\subseteq r(V)$, let us write $\mu^\cN_V[(\omega_s)_{s\in S}]$ for the distribution on $\cK_V$
given by taking the given values $(\omega_s)_{s\in S}$, randomly choosing all those $\rn{\omega}_s$
with $s\in r(V)\setminus S$ and placing $(v,w,x)\in H$ exactly when $((\omega_s)_{s\in
  S},(\rn{\omega}_s)_{s\in r(V)\setminus S})\in\cN_H$.

Then saying that $\cN_H$ is quasirandom could mean:
\begin{enumerate}[label={(\arabic*)}, ref={(\arabic*)}]
\item $\mu^\cN_{\{v,w,x\}}[(\omega_v,\omega_w,\omega_x)]$ is independent of the choice of
  $\omega_v,\omega_w,\omega_x$ (except on a set of measure $0$). This property is called $\Disc[1]$,
  standing for ``discrepancy'', and has been extensively
  studied~\cite{CGW89,KRS02,KNRS10,LM15a,LM15b,LM17,Tow17,ACHPS18}.
\item For every finite set $V$, $\mu^\cN_V[(\omega_v)_{v\in V}]$ is independent of the choice of
  $(\omega_v)_{v\in V}$ (except on a set of measure $0$). This property was introduced
  in~\cite{CR23} under the name $\UCouple[1]$. (This stands for ``unique coupling'', a formulation
  of this property introduced in that paper and proven equivalent.)
\item There is a representation equivalent to $\cN$ of the form $[0,1]^3\times\cM_H$. This
  property is called $\Independence[1]$, and is also introduced in~\cite{CR23}.
\end{enumerate}
The first two properties can be shown to be representation independent. It is not hard to see that
these are ordered from weakest to strongest. For $2$-hypergraphons, these are all
equivalent\footnote{This follows as an easy exercise from the equivalence $P_1(t)\iff P_1(4)\iff
  P_4$ of the seminal paper on graph quasirandomness~\cite{CGW89}.}. For $3$-hypergraphs, however, these
notions are all distinct.

To see that the $\UCouple[1]$ is stronger than $\Disc[1]$, consider the following example, slightly
modified from one in~\cite[Theorem~3.15]{CR23}.
\begin{example}
  Let $\cN^{\disc}_H$ be the set of
  $(\omega_v,\omega_w,\omega_x,\omega_{v,w},\omega_{v,x},\omega_{w,x},\omega_{v,w,x})$ such that
  either:
  \begin{itemize}
  \item $\min\{\omega_v,\omega_w,\omega_x\}<1/2$ and $\omega_{v,w,x}<1/2$, or
  \item $\min\{\omega_v,\omega_w,\omega_x\}\geq 1/2$ and an odd number of the values
    $\omega_{v,w},\omega_{v,x},\omega_{w,x}$ are $<1/2$.
  \end{itemize}
\end{example}
Observe that $\mu^{\cN^{\disc}}_{\{v,w,x\}}[(\omega_v,\omega_w,\omega_x)]$ always places measure
$1/2$ on the case where the $3$-edge appears. However on a set of four vertices
$V=\{v_1,v_2,v_3,v_4\}$, the distribution is no longer independent of the $\omega_v$. For instance,
if $\omega_v<1/2$ for all $v\in V$ then $\mu^{\cN^{\disc}}_V[(\omega_v)_{v\in V}]$ is the uniform
distribution on $\cK_V$, while if $\omega_v\geq 1/2$ for all $v\in V$ then
$\mu^{\cN^{\disc}}_V[(\omega_v)_{v\in V}]$ assigns probability $0$ to any $3$-hypergraph on $V$ with
an odd number of edges.

To see that $\Independence[1]$ is stronger than $\UCouple[1]$, it helps to first consider the random
tournament.
\begin{example}\label{ex:qrT}
  $\cN^{\qrT}_{\rightarrow}\subseteq[0,1]^3$ is the set of triples
  $(\omega_v,\omega_w,\omega_{v,w})$ such that exactly one of $\omega_v<\omega_w$ and
  $\omega_{v,w}<1/2$ is true.
\end{example}
We write $\rightarrow$ rather than $E$ for the name of our relation because it is asymmetric:
$\mu^{\cN^{\qrT}}_V$ concentrates on structures where, for any pair $\{v,w\}$ of distinct vertices,
exactly one of the edges $(v,w)$ or $(w,v)$ is present---that is, a tournament. In fact, it is the
uniform distribution on tournaments, and further, $\mu^{\cN^{\qrT}}_V((\omega_v)_{v\in V})$ is still
the uniform distribution, no matter what the values $(\omega_v)_{v\in V}$ are, so this has
$\UCouple[1]$.

On the other hand, any representation of the form $[0,1]^2\times\cH$ must concentrate on symmetric
structures, so $\cN^{\qrT}$ cannot have $\Independence[1]$. This can be extended to an example on
$3$-hypergraphs that similarly has $\UCouple[1]$ but not $\Independence[1]$, as is done
in~\cite[Theorem~3.8]{CR23}: for instance, take $\cN^{\cycle}_H$ to be the $3$-hypergraphon
consisting of triples on which $\cN^{\qrT}_{\rightarrow}$ would induce a cycle.

The first author and Razborov~\cite{CR23} introduced $\UCouple[1]$ and $\Independence[1]$, as well
as higher-arity generalizations, and settled most questions about which combinations are possible,
leaving only the question of whether $\UCouple[\ell]$ implies $\Independence[\ell']$ when
$\ell'<\ell$. Our original motivation was settling this implication: given $\cN_H$
satisfying $\UCouple[\ell]$, we show how to construct an equivalent representation satisfying
$\Independence[\ell']$---that is, a representation that does not use any of the variables $\omega_s$
with $\lvert s\rvert\leq\ell'$.

Our approach begins by looking at the proof that the quasirandom tournament does not have
$\Independence[1]$. That argument turned on a quirk of the representation: the binary data in a
$2$-hypergraphon is intrinsically symmetric, forcing us to use the unary data to encode the
asymmetry of the quasirandom tournament. However, like the quasirandom graph, the quasirandom
tournament seems like an essentially binary object, the representation just has no way to express
that.

Facing a similar issue, Crane and the second author proposed~\cite{CT18} a modified representation
that adds in an antisymmetric datum at each level. For binary relations, our modified representation
takes $\cN_E$ to be a subset of $[0,1]^3\times\cO_{\{v,w\}}$, where $\cO_{\{v,w\}}$ is just the set
of orders of $\{v,w\}$ with the uniform distribution. For example, in this modified representation,
we can take $\cN^{\qrT}_{\rightarrow}$ to be $\{(\omega_v,\omega_w,\omega_{v,w},<_{v,w})\mid v
<_{v,w} w\}$. We should view $(\omega_v,\omega_w)$ as the unary data and $(\omega_{v,w},<_{v,w})$ as
the binary data, so this gives a representation of the quasirandom tournament that is independent of
the unary data, but only using our modified representation. We will call this a
$\ast$-$1$-independent representation. (We could, in fact, add the requirement that $\cN_E$ satisfy
a suitable symmetry property in the $\omega_s$ variables, essentially requiring that \emph{all} the
asymmetry be expressed by the $<_{v,w}$ variables, but since we do not need this restriction, we do
not impose it here.)

The bulk of this paper is devoted to showing, as part of Theorem~\ref{thm:UCouple}, that
$\UCouple[\ell]$ is \emph{equivalent} to the existence of a $\ast$-$\ell$-independent
representation. This shows that the quasirandom tournament is the essential difference between
$\UCouple[1]$ and $\Independence[1]$, and similarly a ``quasirandom orientation'' describes the gap between
$\UCouple[\ell]$ and $\Independence[\ell]$. We also show, in Corollary~\ref{cor:astindep->indep},
that we can easily convert a $\ast$-$\ell$-independent representation into an $\ell'$-independent
representation for any $\ell'<\ell$ (because we can use any lower arity data, not just the unary
data, to simulate the asymmetry).

Our construction begins with the following observation. When we have a representation $\cN_H$ and
choose a value $\omega_v$ for some vertex $v$, we determine, in some probabilistic sense, how $v$
will behave in the broader structure---for instance $\omega_v$ determines the degree of $v$ (with
high probability, up to a small error). However different values of $\omega_v$ might determine the
same behavior for $v$; for instance, in $\cN^{\qr}$, $\omega_v$ is entirely ignored---there is
only one ``type'' of point. In the twisted representation, the same thing is true: even though the
value $\omega_v$ appears in the definition of $\cN^{\twist}_E$, it does not ultimately affect the
behavior of $v$: every vertex ``looks the same'', no matter what $\omega_v$ was. On the other hand,
there are exactly two ``types'' of pairs---the types with an edge and the types without.

We construct a new representation in which we randomly choose, for each tuple, the type that tuple
should have. To say this slightly more formally, the shape of our construction will be as follows:
\begin{itemize}
\item for each $\ell$, we will define a space $\Lambda_\ell$ of types of $\ell$-tuples,
\item for each finite set $A$, we will define a function
  $\pp^\cN_A\colon[0,1]^{r(\ell)}\to\Lambda_{\lvert A\rvert}$
  mapping the random data $(\omega_s)_{s\in r(A)}$ to the type in $\Lambda_{\lvert A\rvert}$ that
  $\cN$ determines on the data $(\omega_s)_{s\in r(A)}$.
\end{itemize}
Under these definitions, $\cN$ determines a probability measure on $\Lambda_1$, $\cN$ together with
a pair $(p,p')\in \Lambda_1^2$ determines a probability measure on $\Lambda_2$, and so on. We can
identify these with $[0,1]$ by choosing measure-preserving functions $u\colon[0,1]\to\Lambda_1$,
$u_{p,p'}\colon[0,1]\to\Lambda_2$, and so on. Then our new representation is given by taking
$p_v=u(\omega_v)$, $p_{v,w}=u_{u(\omega_v),u(\omega_w)}(\omega_{v,w})$, and similarly at higher
arity. We need the data from $<_{u,v}$ to break symmetry: if $p_v,p_w$ are the same then
$\omega_{v,w}$ can determine something about the unordered set $\{v,w\}$, but if it specifies
something asymmetric---say, that exactly one of the two edges $(v,w)$ and $(w,v)$ is present---then
our representation needs the asymmetric datum $<_{v,w}$ to choose among the ways to resolve the
asymmetry.

It would be trivial to just take $\Lambda_\ell=[0,1]^{r(\ell)}$ with $\pp^\cN_\ell$ as the
identity. However, since we want to extract a $\ast$-$\ell$-independent representation from this, we
want to show that $\UCouple[\ell]$ implies that there is only one type of $\ell$-tuple. This will
force us to choose a definition in which the type does not contain unused information---points which
behave the same way really should get the same type.

On the other hand, our notion of type has to contain enough information to reconstruct
$\mu^\cN_V$. Certainly this means that $\pp^\cN_{k(H)}(\cN_H)$ has to be disjoint from
$\pp^\cN_{k(H)}([0,1]^{r(k(H))}\setminus\cN_H)$---that is, a $k(H)$-tuple that is an edge needs to
have a different type from one that does not have an edge. Slightly more generally, points that lead
to different distributions need to have different types: that is, if $\mu^\cN_V[(\omega_s)_{s\in
    r(A)}]\neq \mu^\cN_V[(\omega'_s)_{s\in r(A)}]$ for some finite set $V$ containing $A$, then we
need to have $\pp^\cN_A((\omega_s)_{s\in r(A)})\neq \pp^\cN_A((\omega'_s)_{s\in r(A)})$.

This gives our first real candidate for the definition of a type, which we will call a \emph{weak
  type}: we could take $\pp^{\weak,\cN}_A((\omega_s)_{s\in r(A)})$ to be $(\mu^\cN_{A\cup
  B}[(\omega_s)_{s\in r(A)}])_{B\text{ a finite set}}$. That is, two tuples have the same weak type
if, whenever we have a finite set containing that tuple, the distributions of structures on that set
are the same\footnote{The reader familiarized with flag algebras should note that this is almost a
  flag algebra homomorphism extension. More specifically, if $\phi$ is the flag algebra homomorphism
  corresponding to $\cN$, then the flag algebra extension $\rn{\phi^{\sigma}}$ corresponds to the
  \emph{distribution of the weak type} $\pp^{\weak,\cN}_A((\rn{\omega}_s)_{s\in r(A)})$ when
  $(\rn{\omega}_s)_{s\in r(A)}$ is picked uniformly at random conditioned to
  $\mu^{\cN_H}_A[(\rn{\omega}_s)_{s\in r(A)}]$ being concentrated on $\sigma$.}.

A standard example shows that this is insufficient.
\begin{example}[Complete Bipartite Graph]
  Let $\cN^{\bp}_E$ be the set of $(\omega_v,\omega_w,\omega_{v,w})$ such that exactly one of
  $\omega_v$ and $\omega_w$ is $<1/2$.
\end{example}
The random structure generated by $\mu^{\cN^{\bp}}_V$ is a complete bipartite graph, with each vertex
having equal chance to belong to each of the parts. Since the parts are unlabeled, there is only one
possible value of $\mu^{\cN^{\bp}}_{\{v\}\cup B}[\omega_v]$. However there is no way to
replicate $\cN^{\bp}_E$ with only one type of point: our construction should have two types
of points, corresponding to the two parts of the bipartition.

Formally, what goes wrong is a failure of \emph{amalgamation}. If we know the types of $v$ and $w$,
they should determine a distribution on the possible types of the pair $\{v,w\}$. That means that if
we have two different pairs, $\omega_v,\omega_w$ and $\omega'_v,\omega'_w$ that lead to the same
pair of types---that is, $\pp^{\weak,\cN}_{\{v\}}(\omega_v)=\pp^{\weak,\cN}_{\{v\}}(\omega'_v)$ and
$\pp^{\weak,\cN}_{\{w\}}(\omega_w)=\pp^{\weak,\cN}_{\{w\}}(\omega'_w)$---then we need the distribution on types of pairs
extending $\omega_v,\omega_w$ to be the same as the distribution on types of pairs extending
$\omega'_v,\omega'_w$.

With weak types for $\cN^{\bp}_E$, this fails: we have
$\pp^{\weak,\cN^{\bp}}_{\{v\}}(1/4)=\pp^{\weak,\cN^{\bp}}_{\{v\}}(3/4)$, but the pair $(1/4,1/4)$ needs to
place measure $1$ on the $2$-type that is not an edge, while the pair $(1/4,3/4)$ needs to place
measure $1$ on the $2$-type that is an edge.

Our types need more information than weak types. To do this, we want to replace the distribution
$\mu^\cN_{A\cup B}[(\omega_s)_{s\in r(A)}]$ with some pointwise information. We define the
\emph{dissociated type} $\pp^{\dsct,\cN}_A((\omega_s)_{s\in r(A)})$ to be the family of
functions
\begin{align*}
  (\omega_s)_{s\in r(B)}\mapsto \mu^\cN_{A\cup B}[(\omega_s)_{s\in r(A)\cup r(B)}],
\end{align*}
where $B$ ranges over finite sets. That is, in order to have the same type, two sets of values
$\{\omega_s\}_{s\in r(A)}$, $\{\omega'_s\}_{s\in r(A)}$ not only need to lead to the same
distribution on $\cK_{A\cup B}$, they need to lead to the same distribution for each fixed
set of data on $B$. (We call these dissociated types because we look at data on $A$ and $B$
separately, but not at any data on sets overlapping $A$ and $B$.)\footnote{There is a very close
  analogy to the situation in model theory: types over a set are not well-behaved. Our solution is
  analogous to considering only types over \emph{models}: allowing data from additional sets $B$
  amounts to letting us name additional fixed elements which we use, for example, to distinguish the
  equivalence classes. In the model theoretic setting, a great deal of work has gone into replacing
  types over models with notions of strong types, which add only enough information to resolve these
  obstacles. Hrushovski has asked whether there is a way to transfer the model theoretic
  approach~\cite{hrushovskitalk}.}

This will turn out to be the correct definition for $1$-types, and we can prove that amalgamation
holds going from $1$-types to $2$-types: if
$\pp^{\dsct,\cN}_{\{v\}}(\omega_v)=\pp^{\dsct,\cN}_{\{v\}}(\omega'_v)$
and
$\pp^{\dsct,\cN}_{\{w\}}(\omega_w)=\pp^{\dsct,\cN}_{\{w\}}(\omega'_w)$
then $\pp^{\dsct,\cN_H}_{\{v,w\}}(\omega_v,\omega_w,\rn{\omega}_{v,w})$ and
$\pp^{\dsct,\cN_H}_{\{v,w\}}(\omega'_v,\omega'_w,\rn{\omega}_{v,w})$ have the same
distribution when $\rn{\omega}_{v,w}$ is picked uniformly at random in $[0,1]$.

Unfortunately, the next step of amalgamation fails, as the following example illustrates.
\begin{example}
  Let $\cN^{\odd}_H$ consist of those $(\omega_s)_{s\in r(3)}$ such that an odd number
  of the values $\omega_{1,2},\omega_{1,3},\omega_{2,3}$ are $<1/2$.
\end{example}
This has a unique dissociated $2$-type\footnote{This representation is $1$-independent, so we will
  ignore the unary data. We need to show that
  $\pp^{\dsct,\cN^{\odd}}_{\{v,w\}}(\omega_{v,w})=\pp^{\dsct,\cN^{\odd}}_{\{v,w\}}(\omega'_{v,w})$
  for any $\omega_{v,w},\omega'_{v,w}$. Obviously the interesting case is when, without loss of
  generality, $\omega_{v,w}<1/2$ while $\omega'_{v,w}\geq 1/2$. In this case, observe that, for any
  finite set $B$ and data $\{\omega_s\}_{s\in r(\{v,w\}\cup B)}$, replacing $\omega_{v,w}$ with
  $\omega'_{v,w}$ and also replacing each $\omega_s$ with $s\cap\{v,w\} = \{v\}$ with $1/2-\omega_s$
  will generate the same structure. This is a measure-preserving transformation on
  $\{\omega_s\}_{s\in r(\{v,w\}\cup B)\setminus(r(\{v,w\})\cup r(B))}$, so in particular shows that
  $\mu^\cN_{\{v,w\}\cup B}[(\omega_{v,w},(\omega_s)_{s\in r(B)})]=\mu^\cN_{\{v,w\}\cup
    B}[(\omega'_{v,w},(\omega_s)_{s\in r(B)})]$.}. However, the $3$-type is completely determined by
the information at the $2$-ary level: for instance, the values $(1/4,1/4,3/4)$ and $(1/4,3/4,3/4)$
for $(\omega_{1,2},\omega_{1,3},\omega_{2,3})$ lead to different distributions on $3$-types.

We conclude that dissociated types still do not contain enough information when the arity is greater
than $1$. Examining the proof that $1$-types can amalgamate to $2$-types suggested the following
strengthening, which we call \emph{overlapping} types. $\pp^{\ovlp,\cN}_A((\omega_s)_{s\in r(A)})$
to be the family of functions
\begin{align*}
  (\omega_s)_{s\in r(A\cup B), A\not\subseteq s\not\subseteq A}
  \mapsto
  \mu^\cN_{A\cup B}[(\omega_s)_{s\in r(A\cup B), A\not\subseteq A},\omega_A],
\end{align*}
where $B$ ranges over finite sets. That is, we now consider more pointwise information, namely all
the data which overlaps some \emph{but not all} of $A$. (Including any $\omega_s$ with $A\subseteq
s$ would certainly be too much information, because it would amount to keeping all the data of the
original representation regarding $A$, leading to failure of uniqueness of $\ell$-types for
$\UCouple[\ell]$.) This gives the same definition of a $1$-type, but adds in more information at
higher arity.

This definition works for $\cN^{\odd}$---there are now two overlapping $2$-types, as there should
be, depending on whether $\omega_{v,w}$ is $<1/2$ or $\geq 1/2$. We also get some progress towards
amalgamation: overlapping types amalgamate to dissociated types, in any arity. That is, if we have
two $k$-tuples $(\omega_s)_{s\in r(k)\setminus\{[k]\}}$ missing only the top variable and such that,
for each $(k-1)$-tuple, the overlapping types are the same, then the distribution on dissociated
$k$-types obtained by choosing $\rn{\omega}_{[k]}$ uniformly at random in $[0,1]$ is the same.

However the overlapping type has too much information. For example, $\cN^{\twist}$ has many
overlapping $2$-types; indeed, if $(\omega_v,\omega_w,\omega_{v,w})$ and
$(\omega'_v,\omega'_w,\omega'_{v,w})$ have the same overlapping $2$-type in $\cN^{\twist}$ then not
only does $\omega_v+\omega_w+\omega_{v,w}<1/2$ if and only if
$\omega'_v+\omega'_w+\omega'_{v,w}<1/2$, but $\omega_v=\omega'_v$ and $\omega_w=\omega'_w$ as
well\footnote{This is because the overlapping $2$-type of $(\omega_v,\omega_w)$ knows for each
  $(\omega_u,\omega_{v,u},\omega_{w,u})$ whether $(\omega_v,\omega_u,\omega_{v,u})$ and
  $(\omega_w,\omega_u,\omega_{w,u})$ are edges, so it can subtract out $\omega_u+\omega_{v,u}$ and
  $\omega_u+\omega_{w,u}$ modulo $1$ to determine $\omega_v$ and $\omega_w$, respectively.}: the
overlapping $2$-type has completely learned the pointwise $1$-ary data, which is information about
the individual points $v$ and $w$ that wasn't present in the overlapping $1$-types. To see why this
is a failure of amalgamation, note that since $\cN^{\twist}$ has a unique overlapping $1$-type
(recall that overlapping and dissociated $1$-types are the same) and since the overlapping $2$-type
knows full pointwise $1$-ary information, changing $1$-ary point data changes the overlapping
$2$-type distribution without changing the overlapping $1$-types.

One of our guiding principles is that the right definition ought to properly stratify information:
all unary information should show up in the $1$-types, all binary information in the $2$-types, and
so on. But this example shows that overlapping $2$-types can contain new unary information that
wasn't already present in the $1$-types.

\medskip

Our final notion, which we call an \emph{amalgamating type}, is in between a dissociated type and an
overlapping type. Essentially, we want it to be a dissociated type together with all the information
from the overlapping type which isn't really lower arity data. Informally, it is essentially a pair
consisting of the dissociated type together with a function taking any $(\omega'_s)_{s\in
  r(A)\setminus\{A\}}$ consistent with the dissociated type to an overlapping type consistent with
$(\omega'_s)_{s\in r(A)\setminus\{A\}}$.

That is, the amalgamating type does not know what the overlapping type is, because that would be too
much information. But it knows everything about the overlapping type that isn't too much
information, because once you tell it the extra information, it can recover the overlapping type.

In the case of $2$-types, this works roughly as follows. Given $\omega_v,\omega_w,\omega_{v,w}$, we
will choose a dissociated $2$-type, which includes a choice of dissociated $1$-types. We then also
want to remember as much as we can about the overlapping $2$-type while still forgetting the
specific values $\omega_v,\omega_w$. Therefore, we will remember a recipe that says ``for any
$\omega'_v,\omega'_w$ consistent with my $1$-types, I know which overlapping $2$-type I would have
had if my unary data had been $\omega'_v,\omega'_w$''.

It is not immediately obvious that such a function makes sense, but its existence almost falls out
of the fact that we have amalgamation. Consider any $\omega'_v,\omega'_w$ with the same dissociated
$1$-types as $\omega_v,\omega_w$. Amalgamation tells us that the distribution of dissociated
$2$-types is the same. That is, there is a measure-preserving function\footnote{This is not quite
  correct, but for simplicity, we will assume this in the introduction only.}
$f_{\omega_v,\omega_w,\omega'_v,\omega'_w}\colon[0,1]\to[0,1]$ so that
\begin{align*}
  \pp^{\dsct,\cN}_{\{v,w\}}(\omega_v,\omega_w,\omega_{v,w})
  =
  \pp^{\dsct,\cN}_{\{v,w\}}(
  \omega'_v,\omega'_w, f_{\omega_v,\omega_w,\omega'_v,\omega'_w}(\omega_{v,w})
  ).
\end{align*}
Therefore we would like to take the amalgamating $2$-type of $(\omega_v,\omega_w,\omega_{v,w})$ to
be the function
\begin{align*}
  (\omega'_v,\omega'_w)
  \mapsto
  \pp^{\ovlp,\cN}_{\{v,w\}}(
  \omega'_v,\omega'_w,f_{\omega_v,\omega_w,\omega'_v,\omega'_w}(\omega_{v,w})
  )
\end{align*}
defined on those $\omega'_v,\omega'_w$ with the correct dissociated $1$-types.

There are some additional measure-theoretic formalities to make this work correctly. For example, we
need the functions $f_{\omega_v,\omega_w,\omega'_v,\omega'_w}$ to satisfy a cocycle condition
$f_{\omega'_v,\omega'_w,\omega''_v,\omega''_w}\circ
f_{\omega_v,\omega_w,\omega'_v,\omega'_w}=f_{\omega_v,\omega_w,\omega''_v,\omega''_w}$. Therefore it
is more natural to identify some target space $\Lambda$, obtain almost everywhere bijective
functions $f_{\omega_v,\omega_w}\colon[0,1]\to\Lambda$, and then set
$f_{\omega_v,\omega_w,\omega'_v,\omega'_w}=f^{-1}_{\omega'_v,\omega'_w}\circ
f_{\omega_v,\omega_w}$. Some work is needed to identify the target space in a way that depends
measurably on $\omega_v,\omega_w$.

Nevertheless, after solving the measure-theoretic issues, this definition suffices: amalgamating
$(k-1)$-types do determine the distribution of amalgamating $k$-types
(Lemma~\ref{lem:amlg}\ref{lem:amlg:strong}) and $\UCouple[\ell]$ implies that there is only one
amalgamating $\ell$-type (Remark~\ref{rmk:amlguniqueness}). (Actually, we are not quite able to
prove the second: we need the additional assumption that the representation is
$(\ell-1)$-independent; however this weaker result is sufficient because we can iterate the
construction.)

\subsection{Organization}

In Section~\ref{sec:basic}, we setup basic notation and definitions and formally state our main
theorems formally.

Section~\ref{sec:dreal} is devoted to the concept of $d$-realizations, which allows us to transform
a representation that uses order variables into one that does not use order variables at the expense
of at most one level of independence (Corollary~\ref{cor:astindep->indep}).

Section~\ref{sec:uniqueness} is devoted to extending the characterization of when two
representations are equivalent to cover order variables (Theorem~\ref{thm:dTU}). As a consequence,
we will derive that if $\cN$ is $\ast$-$\ell$-independent, then it must satisfy $\UCouple[\ell]$
(Proposition~\ref{prop:astellindep->UCouple}).

In Section~\ref{sec:types}, we begin our journey to prove the converse direction, by formalizing the
notions of dissociated and overlapping types and proving fundamental properties about them,
culminating in the proof of uniqueness of overlapping $\ell$-types in $(\ell-1)$-independent
representations (Proposition~\ref{prop:typeuniqueness}); this uniqueness will be inherited by
amalgamating $\ell$-types later.

In Section~\ref{sec:top}, we address measurability within the space of dissociated/overlapping types
by constructing a Polish topology over these spaces, so that equipping the space with the Borel
$\sigma$-algebra turns it into a standard Borel space (Proposition~\ref{prop:toptypes}). This allows
us to make sense of a ``random type''.

Section~\ref{sec:dsctovlp} is devoted to proving the aforementioned ``defective amalgamation''
(Proposition~\ref{prop:weakamalgdsctovlp}), that is, that overlapping types amalgamate to
dissociated types; this will serve as a building block for the final construction of amalgamating
types in Section~\ref{sec:amlg}.

However, before that, Section~\ref{sec:meas} proves a technical measure-theoretic lemma
(Lemma~\ref{lem:keymeas}) that will provide all the measurability and measure-preserving properties
required for the formalization of amalgamating types.

In Section~\ref{sec:amlg}, we finally define amalgamating types (Definition~\ref{def:amlg}) and
prove that every $\UCouple[\ell]$ object has a $\ast$-$\ell$-independent representation
(Proposition~\ref{prop:UCouple->1ellindep}).

In Section~\ref{sec:sim}, we show that every $\UCouple[\ell]$ object is the reduct of an independent
coupling of an $\Independence[\ell]$ object with a quasirandom $(\ell+1)$-orientation
(Proposition~\ref{prop:1ellindep->kqrOcoupling}). This is the formalization of the statement that
``the gap between $\UCouple[\ell]$ and $\Independence[\ell]$ is a quasirandom orientation''.

We conclude the paper with some final remarks and open problems in Section~\ref{sec:conc}.

More details on the location of the proofs of our main theorems can be found in
Table~\ref{tab:prooflocations} near the end of Section~\ref{sec:basic}.

\section{Notation, basic definitions and main theorems}
\label{sec:basic}

In this section, we establish basic notation and provide all basic definitions required to formally
state our main results.

We denote the set of non-negative integers by $\NN$ and the set of positive integers by $\NN_+$. For
$k\in\NN$, we let $[k]\df\{1,\ldots,k\}$. The set of injective functions from $A$ to $B$ is denoted
$(B)_A$ and we use the shorthand $(B)_k\df (B)_{[k]}$ when $k\in\NN$. The symmetric group on a set
$B$ is denoted $S_B$ and we also use the shorthand $S_k\df S_{[k]}$ for $k\in\NN$ (we also identify
$S_k=([k])_k$). For sets $A\subseteq B$, we let $\iota_{A,B}\colon A\to B$ be the inclusion map.

For an injection $\alpha\colon A\to B$, we let $\overline{\alpha}\colon A\to\im(\alpha)$ be the
bijection obtained by restricting the codomain of $\alpha$ to its image $\im(\alpha)$. As usual, if
$U\subseteq A$, we denote the restriction of $\alpha$ to $U$ by $\alpha\rest_U\colon U\to B$, but we
will also use the notation $\alpha\down_U\colon U\to\alpha(U)$ for the bijection obtained from
$\alpha$ by restricting its domain to $U$ and its codomain to $\alpha(U)$ (i.e.,
$\alpha\down_U\df\overline{\alpha\rest_U}$). Furthermore, if $A$ and $B$ are both disjoint from a
set $V$, we let $\alpha\up_V\colon A\cup V\to B\cup V$ be the extension of $\alpha$ that acts
identically on $V$.

For $\ell\in\NN$ and a set $V$, we let $\binom{V}{\ell}\df\{A\subseteq V \mid \lvert A\rvert=\ell\}$
be the set of subsets of $V$ of size $\ell$, we let $\binom{V}{>\ell}\df\{A\subseteq V \mid \lvert
A\rvert>\ell\}$, we let $r(V,\ell)\df\bigcup_{1\leq k\leq\ell}\binom{V}{k}$ and we let
$r(V)\df\bigcup_{\ell\in\NN_+}\binom{V}{\ell}$ be the set of \emph{finite} non-empty subsets of $V$.

By Borel space, we mean a standard Borel space, that is, a measurable space $(X,\cB)$, where $X$ can
be equipped with a Polish space structure such that $\cB$ is the Borel $\sigma$-algebra of the
corresponding topology; however, we will typically omit $\cB$ from the notation. The set of
probability measures on a Borel space $X$ is denoted by $\cP(X)$. By standard probability space, we
mean a probability space whose underlying measurable space is a (standard) Borel space. (Finite or)
countable products of Borel spaces will always be equipped with the product (Borel)
$\sigma$-algebra. Random variables will always be typeset in bold font. We will denote the Lebesgue
measure by $\lambda$ (with a small abuse of notation, we will use $\lambda$ for the Lebesgue measure
in any number of dimensions and even for a countable product of one-dimensional Lebesgue measures).

For a relational language $\cL$ and a predicate symbol $P\in\cL$, we denote the \emph{arity} of $P$
by $k(P)\in\NN_+$. A structure $M$ in $\cL$ is \emph{canonical} if $P^M\subseteq (V(M))_{k(P)}$ for
every $P\in\cL$, that is, if all its predicates can only hold in injective tuples. A universal
theory is \emph{canonical} if all of its models are canonical. Since every universal theory is
isomorphic (via a quantifier-free definition) to a canonical universal theory
(see~\cite[Theorem~2.3]{CR20}, where quantifier-free definitions are called open interpretations),
\emph{all of our structures will be assumed to be canonical unless explicitly stated otherwise}. We
denote the set of (canonical) structures in $\cL$ over a set $V$ by $\cK_V[\cL]$ and for a canonical
theory $T$, the set of models of $T$ over a set $V$ is denoted $\cK_V[T]$; we will omit $\cL$ and
$T$ from the notation when these are clear from the context. The \emph{pure canonical theory} in
$\cL$ is the universal theory $T_\cL$ whose models are all canonical structures in $\cL$ (i.e.,
$\cK_V[T_\cL]=\cK_V[\cL]$). When $V$ is (finite or) countable, we equip $\cK_V$ with the
$\sigma$-algebra generated by cylinder sets, that is, sets of the form
\begin{align}\label{eq:cylinderset}
  C(U,K) & \df \{M\in\cK_V \mid M\rest_U = K\}
\end{align}
for some finite set $U\subseteq V$ and some $N\in\cK_U$; this makes $\cK_V$ a Borel space.

\subsection{Contra-variance and equivariance}
\label{subsec:contequiv}

Throughout this article, we will define several classes of objects $X_A$ indexed by countable sets
$A$ and most of the time, whenever $\alpha\colon A\to B$ is an injection, there will a
contra-variantly defined map $\alpha^*\colon X_B\to X_A$ in the sense that $\id_A^*$ is the identity
map on $X_A$ and $(\beta\comp\alpha)^* = \alpha^*\comp\beta^*$ whenever $\alpha\colon A\to B$ and
$\beta\colon B\to C$ are injective.

Formally, what is going on here is that there is a subcategory $\cX$ of $\Set$ whose objects are
precisely the $X_A$ and there is a contra-variant functor $F_{\cX}\colon\Inj\to\cX$ given by
$F_{\cX}(\alpha)\df\alpha^*$, where $\Inj$ is the subcategory of $\Set$ whose objects are the
countable sets and whose morphisms are the injective functions. However, since what will be
important is the functor $F_{\cX}$ rather than the precise morphisms of $\cX$, we will abbreviate
this abstract non-sense by simply saying that every $\alpha\colon A\to B$ \emph{contra-variantly}
defines the map $\alpha^*\colon X_B\to X_A$. Technically, we should add $\cX$ to the notation
$\alpha^*$, but since $\cX$ will be indicated by the domain and co-domain of the function, we opt
for the simpler notation.

A few important examples are the following:
\begin{enumerate}[label={\Roman*.}, ref={(\Roman*)}]
\item\label{it:alpha*ZB} If $Z$ is a set, then every injection $\alpha\colon A\to B$ between countable sets
  contra-variantly induces a map $\alpha^*\colon Z^B\to Z^A$ given by
  \begin{align*}
    \alpha^*(z)_t & \df z_{\alpha(t)} \qquad (z\in Z^B, t\in A).
  \end{align*}
  Recalling that $\overline{\alpha}\colon A\to\im(\alpha)$ is the restriction of the codomain of
  $\alpha$ to its image and that $\iota_{\im(\alpha),B}\colon\im(\alpha)\to B$ is the inclusion map,
  note that $\alpha^*=\overline{\alpha}^*\comp\iota_{\im(\alpha),B}^*$ due to
  $\alpha=\iota_{\im(\alpha),B}\comp\overline{\alpha}$ and contra-variance.
\item\label{it:alpha*Zc} If $Z$ is a set and for each countable set $A$, we have a collection
  $c(A)\subseteq 2^A$ of subsets of $A$ such that $\alpha(c(A))\subseteq c(B)$ whenever
  $\alpha\colon A\to B$ is an injection, then there is a contra-variantly defined $\alpha^*\colon
  Z^{c(B)}\to Z^{c(A)}$ given by
  \begin{align*}
    \alpha^*(z)_C & \df z_{\alpha(C)} \qquad (z\in Z^{c(B)}, C\in c(A)).
  \end{align*}

  Examples of $c(A)$ with the required property include $r(A)$, $r(A,\ell)$ for some fixed $\ell$,
  $\binom{A}{>\ell}$ for some fixed $\ell$, etc.
\item\label{it:alpha*mu} If $Z$ is a Borel space, then every injection $\alpha\colon A\to B$ between
  countable sets contra-variantly induces a map $\alpha^*\colon\cP(Z^B)\to\cP(Z^A)$, where for
  $\mu\in\cP(Z^B)$, $\alpha^*(\mu)$ is the pushforward measure under the map $\alpha^*\colon Z^B\to
  Z^A$ of item~\ref{it:alpha*ZB} above:
  \begin{align*}
    \alpha^*(\mu)(U) & \df \mu((\alpha^*)^{-1}(U)) = \mu(\{z\in Z^B \mid \alpha^*(z)\in U\}).
  \end{align*}
  (Typically, the pushforward construction is co-variant, but since this is the pushforward of the
  contra-variant $\alpha^*$ of item~\ref{it:alpha*ZB}, the composed functor is contra-variant.)
\item\label{it:alpha*M} For a fixed finite relational language $\cL$, every injection $\alpha\colon
  A\to B$ between countable sets contra-variantly induces a map $\alpha^*\colon\cK_B\to\cK_A$ given
  by
  \begin{align*}
    P^{\alpha^*(M)} & \df \{\beta\in A^{k(P)} \mid \alpha\comp\beta\in P^M\}
    \qquad (P\in\cL),
  \end{align*}
  where $k(P)\in\NN_+$ is the arity of $P\in\cL$, that is, $\alpha^*(M)$ is the unique structure in
  $\cK_A$ such that $\alpha$ is an embedding of $M$ in $\alpha^*(M)$.
\item Building on items~\ref{it:alpha*mu} and~\ref{it:alpha*M}, for a finite relational language
  $\cL$, an injection $\alpha\colon A\to B$ between countable sets contra-variantly induces a map
  $\alpha^*\colon\cP(\cK_B)\to\cP(\cK_A)$ defined by letting $\alpha^*(\mu)$ be the distribution of
  $\alpha^*(\rn{M})$ for the $\alpha^*\colon\cK_B\to\cK_A$ of item~\ref{it:alpha*M}, where $\rn{M}$
  is sampled in $\cK_B$ according to $\mu$. Note that in this language, an exchangeable distribution
  is simply a $\mu\in\cP(\cK_A)$ such that $\sigma^*(\mu)=\mu$ for every $\sigma\in S_A$.
\item In a similar fashion to item~\ref{it:alpha*Zc}, suppose we already have contra-variantly
  defined $\alpha^*\colon X_B\to X_A$ for every injection $\alpha\colon A\to B$ and suppose for each
  countable set $A$, we have a collection $c(A)\subseteq 2^A$ of subsets of $A$ such that
  $\alpha(c(A))\subseteq c(B)$ whenever $\alpha\colon A\to B$ is an injection.

  Then we have a further contra-variantly defined
  \begin{align*}
    \alpha^*\colon\prod_{C\in c(B)} X_C \to \prod_{C\in c(A)} X_C
  \end{align*}
  given by
  \begin{align*}
    \alpha^*(x)_C & \df \alpha\down_C^*(x_C)
    \qquad \left(x\in \prod_{D\in c(B)} X_D, C\in c(A)\right).
  \end{align*}
  (Recall that $\alpha\down_C\colon C\to\alpha(C)$ is the bijective restriction of $\alpha$ to $U$.)
\end{enumerate}

We will also perform several constructions $G_A\colon X_A\to X'_A$, where $X_A$ is an object of
$\cX$ and $X'_A$ is an object of $\cX'$ and most of the time, these constructions will be
\emph{equivariant} in the sense that the diagram
\begin{equation*}
  \begin{tikzcd}
    X_B
    \arrow[r, "G_B"]
    \arrow[d, "\alpha^*"']
    &
    X'_B
    \arrow[d, "\alpha^*"]
    \\
    X_A
    \arrow[r, "G_A"]
    &
    X'_A
  \end{tikzcd}
\end{equation*}
is commutative for every injection $\alpha\colon A\to B$. Formally, what is going on here is that
the collection of maps $(G_A)_A$ is a natural transformation between the contra-variant functors
$F_{\cX}\to F_{\cX'}$, but we will abbreviate this abstract non-sense by simply saying that the
construction $G_A$ is \emph{equivariant}.

One example of equivariance already appeared in the introduction: for a $k$-hypergraphon
$\cN\subseteq[0,1]^{r(k)}$ and a countable set $A$, we implicitly defined a function $M_A^\cN\colon
[0,1]^{r(A)}\to\cK_A$ that constructed a $k$-hypergraph $M_A^\cN(x)$ with vertex set $A$ from
$x\in[0,1]^{r(A)}$ and $\cN$ by declaring $e\in\binom{A}{k}$ to be an edge exactly when
$\alpha^*(x)\in\cN$ for $\alpha^*\colon [0,1]^{r(A)}\to[0,1]^{r(k)}$ given by item~\ref{it:alpha*Zc}
for any enumeration $\alpha\in A^k$ of $e$. It is straightforward to check that the construction
$M_A^\cN$ is equivariant, that is, the diagram
\begin{equation*}
  \begin{tikzcd}
    {[0,1]}^{r(B)}
    \arrow[r, "M_B^\cN"]
    \arrow[d, "\alpha^*"']
    &
    \cK_B
    \arrow[d, "\alpha^*"]
    \\
    {[0,1]}^{r(A)}
    \arrow[r, "M_A^\cN"]
    &
    \cK_A
  \end{tikzcd}
\end{equation*}
is commutative.

\subsection{$d$-peons, $d$-Euclidean structures and $d$-theons}

For a (finite or) countable set $V$, we let $\cO_V$ be the set of linear orders on $V$ and if $V$ is
finite, we equip $\cO_V$ with the uniform probability measure, which we denote by $\nu_V$. Note that
every injection $\alpha\colon U\to V$ contra-variantly defines $\alpha^*\colon \cO_V\to \cO_U$ by
\begin{align*}
  u_1 \mathbin{\alpha^*(\lhd)} u_2 & \iff  \alpha(u_1) \lhd \alpha(u_2)
  \qquad (u_1,u_2\in U).
\end{align*}

Given further an atomless probability space $\Omega=(X,\mu)$, a countable set $V$ and
$d,\ell\in\NN$, we let
\begin{align*}
  \cE_V^{(d)}(\Omega) & \df \prod_{A\in r(V)} (X\times\cO_A^d),
  &
  \cE_{V,\ell}^{(d)}(\Omega) & \df \prod_{A\in r(V,\ell)} (X\times\cO_A^d),
\end{align*}
be equipped with the product (Borel) $\sigma$-algebra and with the product measures
\begin{align*}
  \bigotimes_{A\in r(V)}\left(\mu\otimes\bigotimes_{i=1}^d \nu_A\right),
  & &
  \bigotimes_{A\in r(V,\ell)}\left(\mu\otimes\bigotimes_{i=1}^d \nu_A\right),
\end{align*}
respectively, which we denote both by $\mu^{(d)}$ by abuse of notation. We will omit $\Omega$ from
the notation when it is clear from context and we will omit $(d)$ from the notation when $d=0$ (the
same convention will be used in terms defined later that use $\Omega$ and $d$). We will also use the
shorthands $\cO_k$, $\cE_k^{(d)}$ and $\cE_{k,\ell}^{(d)}$ when $V=[k]$ for $k\in\NN$. Finally, we
will naturally identify $\cE_V^{(d)}$ and $\cE_{V,\ell}^{(d)}$ with the spaces
\begin{align*}
  \left(\prod_{A\in r(V)} X\right)\times\prod_{A\in r(V)}\cO_A^d,
  & &
  \left(\prod_{A\in r(V,\ell)} X\right)\times\prod_{A\in r(V,\ell)}\cO_A^d,
\end{align*}
respectively.

Note that every injection $\alpha\colon U\to V$ contra-variantly defines
$\alpha^*\colon\cE_V^{(d)}\to\cE_U^{(d)}$ by
\begin{align*}
  \alpha^*(x,{\lhd}) & \df (y,{\ll})
  \qquad \left(x\in\cE_V, {\lhd}\in\prod_{A\in r(V)}\cO_A^d\right),
\end{align*}
where
\begin{align*}
  y_A
  & \df
  \alpha^*(x)_A = x_{\alpha(A)}
  \qquad (A\in r(V)),
  \\
  {\ll^i_A}
  & \df
  \alpha\down_A^*({\lhd^i_{\im(\alpha)}})
  \qquad (i\in[d], A\in r(V)).
\end{align*}
An analogous formula yields a contra-variantly defined
$\alpha^*\colon\cE_{V,\ell}^{(d)}\to\cE_{U,\ell}^{(d)}$.

For $d_1,d_2\in\NN$ and two atomless standard probability spaces $\Omega_1$ and $\Omega_2$, we will
naturally identify the spaces $\cE_V^{(d_1+d_2)}(\Omega_1\times\Omega_2)$ and
$\cE_V^{(d_1)}(\Omega_1)\times\cE_V^{(d_2)}(\Omega_2)$ via the natural bijection that maps
$x\in\cE_V^{(d_1+d_2)}(\Omega_1\times\Omega_2)$ to
$(y,z)\in\cE_V^{(d_1)}(\Omega_1)\times\cE_V^{(d_2)}(\Omega_2)$, where for every $A\in r(V)$, $y_A$ is
the projection of $x_A$ to $\Omega_1$ along with the first $d_1$ orders and $z_A$ is the projection
of $x_A$ to $\Omega_2$ along with the last $d_2$ orders.

The remainder of the section is devoted to generalizing the notions of theons of~\cite{CR20} and the
associated quasirandomness notions of~\cite{CR23} to include the extra order variables
of~\cite{CT18}.

Let $\Omega=(X,\mu)$ be an atomless probability space, let $d,\ell\in\NN$, let $\cL$ be a relational
language and $P\in\cL$.

A \emph{(Borel) $d$-$P$-on} over $\Omega$ is a measurable subset $\cN$ of
$\cE_{k(P)}^{(d)}(\Omega)$. We say that $\cN$ is \emph{$\ell$-independent} if
$\cN=\cE_{k(P),\ell}^{(d)}\times\cH$ for some $\cH\subseteq\prod_{A\in\binom{[k(P)]}{>\ell}}
(X\times\cO_A^d)$. Dually, the \emph{rank} of $\cN$ (denoted $\rk(\cN)$) is the minimum $\ell\in\NN$
such that $\cN=\cH\times\prod_{A\in\binom{[k(P)]}{>\ell}}(X\times\cO_A^d)$ for some
$\cH\subseteq\cE_{k(P),\ell}^{(d)}$.

A \emph{(Borel) $d$-Euclidean structure} in $\cL$ over $\Omega$ is a function $\cN$ that maps each
$P\in\cL$ to a $d$-$P$-on $\cN_P\subseteq\cE_{k(P)}^{(d)}(\Omega)$. We say that $\cN$ is
\emph{$\ell$-independent} if all of its $d$-peons $\cN_P$ are $\ell$-independent. Dually, the
\emph{rank} of $\cN$ is $\rk(\cN)\df\max\{\rk(\cN_P)\mid P\in\cL\}\cup\{0\}$.

For a $d$-Euclidean structure in $\cL$ over $\Omega$ and a countable set $V$, we let
$M^\cN_V\colon\cE_V^{(d)}\to\cK_V$ be the (Borel) measurable function given by
\begin{align*}
  P^{M^\cN_V(x)} & \df \{\alpha\in (V)_{k(P)} \mid \alpha^*(x)\in\cN_P\}
  \qquad (x\in\cE_V^{(d)}, P\in\cL),
\end{align*}
that is, the predicate $P$ holds on an injective tuple $\alpha$ in $M^\cN_V(x)$ exactly when the
point $\alpha^*(x)\in\cE_{k(P)}^{(d)}$ is an element of $\cN_P$. It is straightforward to check that
$M^\cN_V$ is equivariant in the sense of Section~\ref{subsec:contequiv}.

We also let $\mu^\cN_V\df (M^\cN_V)_*(\mu^{(d)})$ be the pushforward probability measure of
$\mu^{(d)}$ by $M^\cN_V$, that is, if $\rn{x}\sim\mu^{(d)}$ is picked in $\cE_V^{(d)}$ according
to $\mu^{(d)}$, then $\mu^\cN_V$ is the distribution of $M^\cN_V(\rn{x})$. The fact that the
distribution of $\rn{x}$ is $S_V$-invariant and $M^\cN_V$ is equivariant implies that the
distribution $\mu^\cN_V$ is $S_V$-invariant (i.e., \emph{exchangeable}). Obviously $\mu^\cN_V$ is
also \emph{local}, that is, for $A,B\subseteq V$ disjoint, $M^\cN_V(\rn{x})\rest_A$ is independent
from $M^\cN_V(\rn{x})\rest_B$.

For a finite (canonical) $\cL$-structure $K$, we let
\begin{align*}
  \Tind(K,\cN) & \df (M^\cN_{V(K)})^{-1}(K) \df \{x\in\cE_{V(K)}^{(d)} \mid M^\cN_{V(K)}(x)=K\},
  \\
  \tind(K,\cN) & \df \mu^{(d)}(\Tind(K,\cN)) = \mu^\cN_{V(K)}(\{K\}),
  \\
  \phi_\cN(K)
  & \df
  \frac{\lvert K\rvert!}{\lvert\Aut(K)\rvert}\tind(K,\cN)
  =
  \mu^\cN_{V(K)}(\{K'\in\cK_{V(K)} \mid K'\cong K\}).
\end{align*}
Exchangeability implies that $\tind(K,\cN)$ and $\phi_\cN(K)$ depends only on the isomorphism-type
of $K$.

For a canonical theory $T$, a \emph{(Borel) $d$-$T$-on} is a $d$-Euclidean structure $\cN$ such that
$\tind(K,\cN)=0$ whenever $K$ is \emph{not} a model of $T$. Let $\cN_1$ and $\cN_2$ be a
$d_1$-Euclidean structure over $\Omega_1$ and a $d_2$-Euclidean structure over $\Omega_2$,
respectively, both in $\cL$. We say that $\cN_1$ is \emph{equivalent} to $\cN_2$ if
$\mu^{\cN_1}_V=\mu^{\cN_2}_V$ for every countable set $V$ (this is easily summarized as
$\phi_{\cN_1}=\phi_{\cN_2}$).

The set of functions
\begin{align*}
  \HomT
  & \df
  \{\phi_\cN \mid \cN\text{ is a $0$-$T$-on}\land d\in\NN\}
\end{align*}
was completely characterized in the theory of flag algebras~\cite{Raz07} (see also~\cite{CR20} for
the connection to theons). Since the precise characterization will not be important for us, the
unfamiliarized reader can think of this as a set of representatives of each equivalence class of
$0$-$T$-ons. Furthermore, since elements $\HomT$ encode limits of convergent sequences of finite
models of $T$, we will refer to them as \emph{limits} of $T$ or as limits in $\cL$ when we do not
want to specify $T$. As one might expect, every $d$-$T$-on $\cN$ is equivalent to some $0$-$T$-on
$\cN'$, but it will be important to produce $\cN'$ from $\cN$ preserving as much independence as
possible.

We say that two $d$-Euclidean structures $\cN^1$ and $\cN^2$ in the same language $\cL$ over the
same space $\Omega=(X,\mu)$ are \emph{almost everywhere (a.e.) equal} if
$\mu^{(d)}(\cN^1_P\symdiff\cN^2_P)=0$ for every $P\in\cL$; in particular, this implies that they
represent the same limit: $\phi_{\cN^1}=\phi_{\cN^2}$.

We say that a limit $\phi$ is $\ell$-independent if there exists a $0$-Euclidean structure $\cN$
such that $\phi_\cN=\phi$; this property is also denoted $\Independence[\ell]$. We say that a limit
$\phi$ is $\ast$-$\ell$-independent if there exist $d\in\NN$ and a $d$-Euclidean structure $\cN$
such that $\phi_\cN=\phi$. These are not equivalent: our main result Theorem~\ref{thm:UCouple} says
that $\ast$-$\ell$-independence is equivalent to the property $\UCouple[\ell]$ defined below, which
is strictly weaker than $\Independence[\ell]$ (see~\cite{CR23}).

Dually, the \emph{rank} of a limit $\phi$ is the minimum rank of a $0$-$T$-on $\cN$ such that
$\phi=\phi_\cN$. We could define the $\ast$-rank by allowing $d$-$T$-ons, but as we will see in
Corollary~\ref{cor:rk}, such a definition would be equivalent to the usual rank.

Recall that a \emph{quantifier-free definition} (or \emph{open interpretation}) from a canonical
theory $T_1$ in $\cL_1$ to a canonical theory $T_2$ in $\cL_2$ is a function $I$ that maps predicate
symbols $P\in\cL_1$ to quantifier-free formulas $I(P)(x_1,\ldots,x_{k(P)})$ in $\cL_2$ so that
axioms of $T_1$ are mapped to theorems of $T_2$ when we declare $I$ to commute with logical
connectives; such a quantifier-free definition is denoted as $I\colon T_1\leadsto T_2$. For a set
$V$, a quantifier-free definition naturally defines a map $I^*_V\colon\cK_V[T_2]\to\cK_V[T_1]$ given
by\footnote{The $\ast$ in $I^*_V$ is to denote a different kind of contra-variance. Namely, this is
  a contra-variant functor from the category $\Int$ of quantifier-free definitions to $\Set$,
  see~\cite{CR20}. The next two defined maps, which will be denoted by $I^*$ also have similar
  contra-variances as functors from $\Int$.}
\begin{align*}
  P^{I^*_V(M)} & \df I(P)(M) \df \{\alpha\in (V)_{k(P)} \mid M\vDash I(P)(\alpha)\}
  \qquad (P\in\cL),
\end{align*}
that is, the predicate $P\in\cL_1$ holds on a tuple $\alpha$ in $I^*_V(M)$ exactly when the formula
$I(P)$ holds on $\alpha$ in $M$. It is straightforward to check that $I^*_V$ is equivariant in the
sense of Section~\ref{subsec:contequiv}. Similarly, $I\colon T_1\leadsto T_2$ also naturally defines
a map $I^*\colon\HomT[T_2]\to\HomT[T_1]$ given by
\begin{align*}
  I^*(\phi)(M) & \df \sum\{\phi(N)\mid N\in\cM_{V(M)}[T_2]\land I^*_{V(M)}(N)\cong M\},
\end{align*}
for every finite $\cL_1$-structure $M$, where $\cM_{V(M)}[T_2]$ is the set of models of $T_2$ over
$V(M)$ \emph{up to isomorphism} (in~\cite{Raz07,CR20} $I^*(\phi)$ is denoted as $\phi^I$). From a
probabilistic view, we have $I^*(\phi_\cN) = \phi_\cH$ exactly when for every countable set $V$, if
$\rn{M}$ is sampled at random according to $\mu^\cN_V$, then $I^*_V(\rn{M})\sim\mu^\cH_V$.

In turn, there is also a naturally induced map over $d$-Euclidean structures $\cN$ over
$\Omega=(X,\mu)$ defined as follows. First, for a quantifier-free formula $F(x_1,\ldots,x_n)$, we
define the \emph{truth set} $T(F,\cN)\subseteq\cE_n^{(d)}(\Omega)$ as follows:
\begin{itemize}
\item if $F$ is $P(x_{i_1},\ldots,x_{i_k})$ and $i_1,\ldots,i_k$ are not pairwise distinct or $F$ is
  $x_i = x_j$ for $i\neq j$, then $T(F,\cN)\df\varnothing$.
\item $T(x_i=x_i)\df\cE_n^{(d)}(\Omega)$.
\item if $F$ is $P(x_{i_1},\ldots,x_{i_k})$ and $i_1,\ldots,i_k$ are pairwise distinct, then
  $T(F,\cN)\df (i^*)^{-1}(\cN_P)$, where we view $i$ as an injection $[k]\to[n]$.
\item $T(\place,\cN)$ commutes with propositional connectives (e.g., we have $T(F_1\lor F_2,\cN)\df
  T(F_1,\cN)\cup T(F_2,\cN)$ and $T(\neg F,\cN)\df\cE_n^{(d)}(\Omega)\setminus T(F,\cN)$).
\end{itemize}
Then for a $d$-Euclidean structure $\cN$ over $\Omega=(X,\mu)$, we define the $d$-Euclidean structure
$I^*(\cN)$ over $\Omega=(X,\mu)$ by
\begin{align}\label{eq:I*cN}
  I^*(\cN)_P & \df T(I(P),\cN) \qquad (P\in\cL_1).
\end{align}
(In~\cite[Remark~6]{CR20}, this was defined only for the $d=0$ case and used the notation $I(\cN)$
instead.) These maps are connected to each other in the sense that $\phi_{I^*(\cN)} = I^*(\phi_\cN)$
(this is easily checked by noting that if $\rn{M}$ is sampled at random according to $\mu^\cN_V$,
then $I^*_V(\rn{M})\sim\mu^{I^*(\cN)}_V$).

A particularly important case of quantifier-free definitions are
\emph{reducts}\footnote{In~\cite{CR20,CR23}, these were called structure-erasing interpretations.},
i.e., quantifier-free definitions of the form $I\colon T_{\cL'}\leadsto T_\cL$ in which
$\cL'\subseteq\cL$ and $I$ acts identically on predicate symbols of $\cL'$. In this case,
$I^*_V\colon\cK_V[\cL]\to\cK_V[\cL']$ is simply the reduct map $M\rest_{\cL'}\df I^*_V(M)$ and we
also abbreviate $\phi\rest_{\cL'}\df I^*(\phi)$. It is straightforward to check that if $\cN$ is a
$d$-Euclidean structure in $\cL$, then $\phi_\cN\rest_{\cL'} = \phi_{\cN\rest_{\cL'}}$, where
$\cN\rest_{\cL'}$ is the $d$-Euclidean structure in $\cL'$ obtained from $\cN$ by simply restricting
it to $\cL'$ (which is a.e.\ equal to $I^*(\cN)$).

Given two theories $T_1$ and $T_2$ in $\cL_1$ and $\cL_2$, respectively, the \emph{disjoint union
  theory} is the theory $T_1\cup T_2$ in the disjoint union language $\cL_1\disjcup\cL_2$ that
contains both axioms of $T_1$ and $T_2$ (i.e., its models are models of $T_1$ and $T_2$ on the same
vertex set). With a small abuse of notation, for $i\in[2]$, we also call the quantifier-free
definitions of the form $I\colon T_i\leadsto T_1\cup T_2$ that act identically on predicate symbols
of $\cL_i$ \emph{reducts} and use the notation $\phi\rest_{T_i}$ for $\phi\rest_{\cL_i}$ when
$\phi\in\HomT[T_1\cup T_2]$.

A \emph{coupling} of two limits $\phi_1$ and $\phi_2$ in $\cL_1$ and $\cL_2$, respectively, is a
limit $\psi$ in the disjoint union language $\cL_1\disjcup\cL_2$ such that
$\psi\rest_{\cL_1}=\phi_1$ and $\psi\rest_{\cL_2}=\phi_2$. For example, if $\cN^1$ and $\cN^2$ are
$d$-Euclidean structures in $\cL_1$ and $\cL_2$, respectively, over the same $\Omega$ with
$\phi_{\cN^1}=\phi_1$ and $\phi_{\cN^2}=\phi_2$, then the (disjoint) union function
$\cN^1\disjcup\cN^2$ is a $d$-Euclidean structure in $\cL_1\disjcup\cL_2$ over $\Omega$ such that
$\phi_{\cN^1\disjcup\cN^2}$ is a coupling of $\phi_1$ and $\phi_2$.

The \emph{independent coupling} of two limits $\phi_1$ and $\phi_2$ in $\cL_1$ and $\cL_2$,
respectively is the limit $\phi_1\otimes\phi_2$ in $\cL_1\disjcup\cL_2$ given by
\begin{align*}
  (\phi_1\otimes\phi_2)(M)
  & \df
  \frac{
    \lvert\Aut(M\rest_{\cL_1})\rvert\cdot\lvert\Aut(M\rest_{\cL_2})\rvert
  }{
    \lvert\Aut(M)\rvert\cdot\lvert M\rvert!
  }\cdot
  \phi_1(M\rest_{\cL_1})\cdot\phi_2(M\rest_{\cL_2}).
\end{align*}
A more geometric description of the independent coupling is as follows: let $\cN^1$ be a
$d_1$-Euclidean structure in $\cL_1$ over $\Omega_1$ and $\cN^2$ be a
$d_2$-Euclidean structure in $\cL_2$ over $\Omega_2$. The \emph{independent coupling} of $\cN^1$
and $\cN^2$ is the $(d_1+d_2)$-Euclidean structure in $\cL_1\disjcup\cL_2$ over
$\Omega_1\times\Omega_2$ given by
\begin{align*}
  (\cN^1\otimes\cN^2)_P
  & \df
  \{(x^1,x^2)\in\cE_{k(P)}^{(d_1)}(\Omega_1)\times\cE_{k(P)}^{(d_2)}(\Omega_2) \mid
  x^i\in\cN^i_P\}
  \qquad (P\in\cL_i, i\in[2]).
\end{align*}
It is straightforward to check that for every finite $\cL_1\disjcup\cL_2$-structure $M$, we have
\begin{align*}
  \tind(M,\cN^1\otimes\cN^2) & = \tind(M\rest_{\cL_1},\cN^1)\cdot\tind(M\rest_{\cL_2},\cN^2),
\end{align*}
from which one gets $\phi_{\cN^1\otimes\cN^2}=\phi_{\cN^1}\otimes\phi_{\cN^2}$.

A more probabilistic view of couplings and independent couplings is as follows: $\phi_\cH$ is a
coupling of $\phi_{\cN^1}$ and $\phi_{\cN^2}$ if for every countable set $V$, if $\rn{N}$ is
sampled according to $\mu^\cH_V$, then $\rn{N}\rest_{\cL_1}\sim\mu^{\cN^1}_V$ and
$\rn{N}\rest_{\cL_2}\sim\mu^{\cN^2}_V$; if further $\rn{N}\rest_{\cL_1}\sim\mu^{\cN^1}_V$ and
$\rn{N}\rest_{\cL_2}\sim\mu^{\cN^2}_V$ are independent, then $\phi_\cH$ is the independent coupling
of $\phi_{\cN^1}$ and $\phi_{\cN^2}$.

Two limits $\phi_1$ and $\phi_2$ are \emph{uniquely coupleable} if there exists exactly one
coupling of $\phi_1$ and $\phi_2$ (which must be the independent coupling
$\phi_1\otimes\phi_2$).

We say that a limit $\phi$ is \emph{uniquely $\ell$-coupleable} if for every limit $\psi$ with
$\rk(\psi)\leq\ell$, $\phi$ and $\psi$ are uniquely coupleable; this property is also denoted
$\UCouple[\ell]$. It was shown in~\cite[Theorem~3.10]{CR23} that $\phi\in\UCouple[\ell]$ is
equivalent to the property that for some (equivalently, every) $0$-Euclidean structure $\cN$ over
$\Omega=(X,\mu)$ with $\phi_\cN=\phi$, if $V$ is a countable set and $\rn{x}$ is picked at random in
$\cE_V$ according to $\mu$, then $M^\cN_V(\rn{x})$ is independent (as a random variable) from
$(\rn{x}_A\mid A\in r(V,\ell))$. As we will see in Proposition~\ref{prop:astellindep->UCouple}, the
same characterization holds if we let $\cN$ be a $d$-Euclidean structure with $\phi_\cN=\phi$ for
$d\in\NN$. From this characterization, it is obvious that
$\Independence[\ell]\implies\UCouple[\ell]$ (see~\cite[Theorem~3.2]{CR23})

Furthermore, it was shown in~\cite[Theorem~3.3]{CR23} that both $\Independence[\ell]$ and
$\UCouple[\ell]$ are natural properties in the sense that if $\phi\in\HomT$ satisfies the property
and $I\colon T'\leadsto T$ is a quantifier-free definition, then $I^*(\phi)$ also satisfies the
property. Moreover, these properties are also preserved under independent
couplings~\cite[Theorem~3.4]{CR23} in the sense that if $\phi_1$ and $\phi_2$ both have the property
then $\phi_1\otimes\phi_2$ also has the property.

We end this section with a high-arity generalization of the quasirandom tournament of
Example~\ref{ex:qrT}.

\begin{definition}\label{def:kqrO}
  For $k\in\NN_+$, the \emph{theory of $k$-orientations} is the theory $\TkO$ in a
  language $\cL$ containing a single predicate symbol $P$, whose arity is $k(P)\df k$, and with
  axioms
  \begin{gather*}
    \forall x_1,\ldots,x_k, (P(x_1,\ldots,x_k)\to\neg P(x_{\sigma(1)},\ldots,x_{\sigma(k)}))
    \qquad (\sigma\in S_k\setminus\{\id_k\}),
    \\
    \forall x_1,\ldots,x_k,
    \left(\bigwedge_{i\neq j} x_i\neq x_j\to\bigvee_{\sigma\in S_k} P(x_{\sigma(1),\ldots,\sigma(k)})\right),
  \end{gather*}
  that is, for each $k$-set $A$ of vertices, the predicate $P$ holds in exactly one
  of the (injective) $k$-tuples $\alpha$ with $\im(\alpha)=A$.

  The \emph{quasirandom $k$-orientation} is the limit $\phi^{\kqrO}\in\HomT[\TkO]$ is the limit
  encoded by the ($(k-1)$-independent) $1$-$\TkO$-on $\cN^{\kqrO}$ over any space $\Omega=(X,\mu)$
  given by
  \begin{align}\label{eq:kqrO}
    \cN^{\kqrO}_P
    & \df
    \{(x,{\lhd})\in\cE_k^{(1)}(\Omega) \mid {\lhd}_{[k]} = {\leq}\},
  \end{align}
  where ${\leq}\in\cO_k$ is the usual order on $[k]$. Note that $\mu^\cN_V$ consists of the
  distribution in which each $k$-set is given a uniformly at random orientation independently.

  Alternatively, we will see in Remark~\ref{rmk:altkqrO} that
  $\phi^{\kqrO}=\phi_{\cH^{\kqrO}}$ for the $0$-$\TkO$-on $\cH^{\kqrO}$ over
  $\Omega=([0,1),\lambda)$ given by
  \begin{align}
      \cH^{\kqrO}_P
      & \df
      \{x\in\cE_k^{(0)}(\Omega) \mid \sigma_x = \gamma_{\floor{x_{[k]}\cdot k!}+1}\},
      \label{eq:altkqrO}
  \end{align}
  where $\sigma_x\in S_k$ is defined when $x$ has distinct coordinates as the unique permutation
  such that
  \begin{align*}
    x_{[k]\setminus\{\sigma^{-1}_x(1)\}} < x_{[k]\setminus\{\sigma^{-1}_x(2)\}}
    < \cdots < x_{[k]\setminus\{\sigma^{-1}_x(k)\}},
  \end{align*}
  and defined arbitrarily (but measurably) when $x$ has repeated coordinates and
  $\gamma_1,\ldots,\gamma_{k!}$ is some fixed enumeration of $S_k$.

  This latter representation shows that $\phi^{\kqrO}$ satisfies $\Independence[k-2]$.
\end{definition}

It is straightforward to check that the quasirandom $k$-orientation $\phi^{\kqrO}$ can be obtained
as a quantifier-free definition of the $(\Theta,p)$-quasirandom homomorphisms in $S_k$-action
theories defined in~\cite[Definition~2.9]{CR23}, which in particular implies that it satisfies
$\UCouple[k-1]$ (by~\cite[Proposition~9.1]{CR23}). However, we will see that
$\phi^{\kqrO}\in\UCouple[k-1]$ follows from the more conceptual fact that it is represented by the
$(k-1)$-independent $1$-$\TkO$-on $\cN^{\kqrO}$ (see Corollary~\ref{cor:kqrO}).

\subsection{Main theorems}
\label{subsec:mainthms}

Our main theorem provides some more characterizations of $\UCouple[\ell]$ complementing the ones
of~\cite[Theorem~3.10]{CR23}.

\begin{theorem}\label{thm:UCouple}
  The following are equivalent for a limit $\phi\in\HomT$ and $\ell\in\NN$.
  \begin{enumerate}
  \item\label{thm:UCouple:UCouple} $\phi\in\UCouple[\ell]$.
  \item\label{thm:UCouple:astellindep} $\phi$ is $\ast$-$\ell$-independent.
  \item\label{thm:UCouple:1ellindep} There exists an $\ell$-independent $1$-Euclidean structure
    $\cN$ such that $\phi=\phi_\cN$.
  \item\label{thm:UCouple:kqrOcoupling} There exist a canonical theory $T'$, an $\ell$-independent
    limit $\psi\in\HomT[T']$ and a quantifier-free definition $I\colon T\leadsto
    T'\cup\TkO[(\ell+1)]$ such that $\phi = I^*(\psi\otimes\phi^{\kqrO[(\ell+1)]})$.
  \end{enumerate}
\end{theorem}
The implication~\ref{thm:UCouple:1ellindep}$\implies$\ref{thm:UCouple:astellindep} follows trivially
from definitions.

As it was shown in~\cite[Theorems~3.6 and~3.8]{CR23} (see also Example~\ref{ex:qrT}),
$\Independence[\ell]$ (i.e., existence of an $\ell$-independent $0$-Euclidean structure $\cN$ with
$\phi=\phi_\cN$) is strictly stronger than $\UCouple[\ell]$ (see Example~\ref{ex:qrT}). The theorem
above says that not only is $\UCouple[\ell]$ (item~\ref{thm:UCouple:UCouple}) equivalent to
$\ast$-$\ell$-independence (item~\ref{thm:UCouple:astellindep}), but even adding extra order
variables does not increase the strength of quasirandomness: existence of an $\ell$-independent
$1$-Euclidean structure representation (item~\ref{thm:UCouple:1ellindep}) is still equivalent to
$\UCouple[\ell]$. Finally, the fact that item~\ref{thm:UCouple:kqrOcoupling} is also equivalent to
$\UCouple[\ell]$ can be seen as a classification of all $\UCouple[\ell]$ limits as those that
essentially consist of an $\Independence[\ell]$ part independently coupled with a quasirandom
$(\ell+1)$-orientation; in other words, the only difference between $\UCouple[\ell]$ and
$\Independence[\ell]$ is the potential presence of a hidden quasirandom $(\ell+1)$-orientation.

As a consequence of Theorem~\ref{thm:UCouple} above, we settle the final missing implications
between the quasirandomness properties of~\cite{CR23}:
\begin{theorem}\label{thm:Indep}
  For every $\ell,\ell'\in\NN$ with $\ell' < \ell$, we have
  $\UCouple[\ell]\implies\Independence[\ell']$.
\end{theorem}

Since by~\cite[Theorems~3.3 and~3.4]{CR23} independent couplings and quantifier-free definitions
preserve both $\Independence[\ell]$ and $\UCouple[\ell]$, one can easily see how
Theorem~\ref{thm:UCouple}\ref{thm:UCouple:UCouple}$\implies$\ref{thm:UCouple:kqrOcoupling} implies
Theorem~\ref{thm:Indep} as the representation $\phi^{\kqrO[(\ell+1)]}=\phi_{\cH^{\kqrO[(\ell+1)]}}$ shows
$\phi^{\kqrO[(\ell+1)]}\in\Independence[\ell-1]$. However, as it turns out, it will be easier to prove that
every $\ast$-$\ell$-independent limit satisfies $\Independence[\ell-1]$
(Corollary~\ref{cor:astindep->indep}), (so Theorem~\ref{thm:Indep} will follow from the implication
Theorem~\ref{thm:UCouple}\ref{thm:UCouple:UCouple}$\implies$\ref{thm:UCouple:astellindep} instead,
see Corollary~\ref{cor:Indep}).

Table~\ref{tab:prooflocations} contains the location of the proofs of the implications of
Theorem~\ref{thm:UCouple} and the proof of Theorem~\ref{thm:Indep}.

\begin{table}[htb]
  \centering
  \begin{tabular}{ll@{\hspace{5em}}l}
    \multicolumn{2}{l}{Result} & Proof location
    \\
    \hline
    \hline
    \multirow{5}{*}{Theorem~\ref{thm:UCouple}}
    &
    \ref{thm:UCouple:1ellindep}$\implies$\ref{thm:UCouple:astellindep}
    &
    Trivial from definitions
    \\
    &
    \ref{thm:UCouple:astellindep}$\implies$\ref{thm:UCouple:UCouple}
    &
    Proposition~\ref{prop:astellindep->UCouple}
    \\
    &
    \ref{thm:UCouple:kqrOcoupling}$\implies$\ref{thm:UCouple:UCouple}
    &
    Corollary~\ref{cor:kqrOcoupling->UCouple}
    \\
    &
    \ref{thm:UCouple:UCouple}$\implies$\ref{thm:UCouple:1ellindep}
    &
    Proposition~\ref{prop:UCouple->1ellindep}
    \\
    &
    \ref{thm:UCouple:1ellindep}$\implies$\ref{thm:UCouple:kqrOcoupling}
    &
    Proposition~\ref{prop:1ellindep->kqrOcoupling}
    \\
    \hline
    Theorem~\ref{thm:Indep}
    & &
    Corollaries~\ref{cor:astindep->indep} and~\ref{cor:Indep}
  \end{tabular}
  
  \caption{Proof locations for main theorems.}
  \label{tab:prooflocations}  
\end{table}

\section{$d$-Realizations}
\label{sec:dreal}

This section is dedicated to proving in Corollary~\ref{cor:astindep->indep} that every
$\ast$-$\ell$-independent limit is $(\ell-1)$-independent, which in particular means that
Theorem~\ref{thm:Indep} is a direct corollary of Theorem~\ref{thm:UCouple} (this will be put
together in Corollary~\ref{cor:Indep}). We start by recalling a few concepts and basic results from
exchangeability theory.

\begin{definition}
  Let $\Omega_1=(X_1,\mu_1)$ and $\Omega_2=(X_2,\mu_2)$ be atomless probability spaces, let
  $d_1,d_2,n\in\NN$ and let $f\colon\cE_n^{(d_1)}(\Omega_1)\to\Omega_2\times\cO_n^{d_2}$ be a function.

  We say that $f$ is \emph{symmetric} if it is a.e.\ $S_n$-invariant in the first component and
  a.e.\ $S_n$-equivariant in the order components in the sense that
  \begin{align*}
    f(\sigma^*(x)) & \df (f_0(x), \sigma^*(f_1(x)), \sigma^*(f_2(x)), \ldots, \sigma^*(f_{d_2}(x)))
  \end{align*}
  for $\mu_1^{(d_1)}$-almost every $x\in\cE_n^{(d_1)}(\Omega_1)$ and every $\sigma\in S_n$, where
  $f_0\colon\cE_n^{(d_1)}(\Omega_1)\to\Omega_2$ and $f_i\colon\cE_n^{(d_1)}(\Omega_1)\to\cO_n$
  ($i\in[d_2]$) denote the components of $f$.

  We say that $f$ is \emph{measure-preserving on highest order argument} (h.o.a.) if it is
  measurable and for almost every $x\in\cE_{n,n-1}^{(d_1)}$, the function
  \begin{align*}
    f(x,\place)\colon\Omega_1\times\cO_n^{d_1}\to\Omega_2\times\cO_n^{d_2}
  \end{align*}
  is measure-preserving (with respect to $\mu_1^{(d_1)}$ and $\mu_2^{(d_2)}$).

  Given a family $f = (f_i)_{i\in\NN_+}$ of symmetric functions
  $f_i\colon\cE_i^{(d_1)}(\Omega_1)\to\Omega_2\times\cO_i^{d_2}$ and a countable set $V$, we define
  the function $\widehat{f}_V\colon\cE_V^{(d_1)}(\Omega_1)\to\cE_V^{(d_2)}(\Omega_2)$ by
  \begin{align*}
    \widehat{f}_V(x)_A
    & \df
    (\overline{\alpha}_A^{-1})^*(f_{\lvert A\rvert}(\alpha_A^*(x)))
    \qquad (A\in r(V)),
  \end{align*}
  where $\alpha_A\colon [\lvert A\rvert]\to V$ is any injection with $\im(\alpha_A)=A$ (recall that
  $\overline{\alpha}_A\colon[\lvert A\rvert]\to A$ is the restriction of the codomain of
  $\alpha_A$). Since the $f_i$ are symmetric, this a.e.\ does not depend on the particular choice of
  $\alpha_A$.

  Given further a $d_2$-Euclidean structure $\cN$ in $\cL$ over $\Omega_2$, we define the
  $d_1$-Euclidean structure $f^*(\cN)$ in $\cL$ over $\Omega_1$ by
  \begin{align*}
    f^*(\cN)_P
    & \df
    \widehat{f}_{k(P)}^{-1}(\cN_P)
    \qquad (P\in\cL).
  \end{align*}
\end{definition}

The following lemma is an extension of a standard fact from exchangeability theory to account for
order variables.

\begin{lemma}\label{lem:widehat}
  Let $\Omega_1=(X_1,\mu_1)$ and $\Omega_2=(X_2,\mu_2)$ be atomless probability spaces, let
  $d_1,d_2,n\in\NN$ and let $f = (f_i)_{i\in\NN_+}$ of symmetric functions
  $f_i\colon\cE_i^{(d_1)}(\Omega_1)\to\Omega_2\times\cO_i^{d_2}$. Then the following hold.
  \begin{enumerate}
  \item\label{lem:widehat:measpres} If each $f_i$ is measure-preserving on h.o.a.\ and $V$ is a
    countable set, then $\widehat{f}_V$ is measure-preserving.
  \item\label{lem:widehat:equivariance} $\widehat{f}_V$ is a.e.\ equivariant, that is, if
    $\beta\colon U\to V$ is an injection between countable sets, then $\beta^*\comp\widehat{f}_V =
    \widehat{f}_U\comp\beta^*$ holds $\mu_1^{(d_1)}$-almost everywhere.
  \end{enumerate}
\end{lemma}

\begin{proof}
  Item~\ref{lem:widehat:measpres} is equivalent to proving that if $\rn{x}$ is picked at random in
  $\cE_V^{(d_1)}(\Omega_1)$ according to $\mu_1^{(d_1)}$, then $\rn{y}\df\widehat{f}_V(\rn{x})$ is
  distributed according to the measure $\mu_2^{(d_2)}$ of $\cE_V^{(d_2)}(\Omega_2)$. The definition
  of $\widehat{f}_V$ implies that each coordinate $\rn{y}_A$ is measurable with respect to the
  $\sigma$-algebra generated by $(\rn{x}_B \mid B\in r(A))$ and since all $f_i$ are
  measure-preserving on h.o.a., the conditional distribution of the coordinate $\rn{y}_A$ given
  $(\rn{x}_B \mid B\subsetneq A)$ is $\mu_2^{(d)}$, hence a simple induction and the fact that
  the coordinates of $\rn{x}$ are mutually independent yields item~\ref{lem:widehat:measpres}.

  \medskip

  For item~\ref{lem:widehat:equivariance}, note that for $\mu_1^{(d_1)}$-almost every
  $x\in\cE_V^{(d_1)}(\Omega_1)$ and every $A\in r(U)$, we have
  \begin{align*}
    & \!\!\!\!\!\!
    (\beta^*\comp\widehat{f}_V)(x)_A
    \\
    & =
    \Biggl(f_{\lvert\beta(A)\rvert,0}(\alpha_{\beta(A)}^*(x)),
    \\
    & \hphantom{{}={}}
    \beta\down_A^*\biggl((\overline{\alpha}_{\beta(A)}^{-1})^*
    \Bigl(f_{\lvert\beta(A)\rvert,1}\bigl(\alpha_{\beta(A)}^*(x)\bigr)\Bigr)\biggr),
    \ldots,
    \beta\down_A^*\biggl((\overline{\alpha}_{\beta(A)}^{-1})^*
    \Bigl(f_{\lvert\beta(A)\rvert,d_2}\bigl(\alpha_{\beta(A)}^*(x)\bigr)\Bigr)\biggr)
    \Biggr),
  \end{align*}
  where $f_{i,0}\colon\cE_i^{(d_1)}(\Omega_1)\to\Omega_2$ and
  $f_{i,j}\colon\cE_i^{(d_1)}(\Omega_1)\to\cO_i$ ($j\in[d_2]$) denote the components of $f_i$
  ($i\in\NN_+$).

  Note that the composition $\beta\down_A^{-1}\comp\overline{\alpha}_{\beta(A)}$ is a bijection
  $[\lvert A\rvert]\to A$. Let $\gamma_A\colon[\lvert A\rvert]\to U$ be obtained from
  $\beta\down_A^{-1}\comp\overline{\alpha}_{\beta(A)}$ by changing its codomain to $U$ (i.e., we
  have $\gamma_A\df\iota_{A,V}\comp\beta\down_A^{-1}\comp\overline{\alpha}_{\beta(A)}$). By
  construction, $\gamma_A$ is an injection with $\im(\gamma_A)=A$ and $\alpha_{\beta(A)} =
  \beta\comp\gamma_A$, so we get
  \begin{align*}
    f_{\lvert\beta(A)\rvert,0}(\alpha_{\beta(A)}^*(x))
    & =
    f_{\lvert A\rvert,0}(\gamma_A^*(\beta^*(x))).
  \end{align*}

  On the other hand, note that
  $\overline{\gamma}_A=\beta\down_A^{-1}\comp\overline{\alpha}_{\beta(A)}$, hence
  $\overline{\gamma}_A^{-1} = \overline{\alpha}_{\beta(A)}^{-1}\comp\beta\down_A$, so for every
  $j\in[d_2]$, we get
  \begin{align*}
    \beta\down_A^*\bigl((\overline{\alpha}_{\beta(A)}^{-1})^*(f_{\lvert\beta(A)\rvert,j}(\alpha_{\beta(A)}^*(x)))\bigr)
    & =
    (\overline{\gamma}_A^{-1})^*(f_{\lvert A\rvert,j}(\gamma_A^*(\beta^*(x))).
  \end{align*}

  Since $\im(\gamma_A)=A$, we conclude that
  \begin{align*}
    (\beta^*\comp\widehat{f}_V)(x)_A
    & =
    (\widehat{f}_U\comp\beta^*)(x)_A
  \end{align*}
  for $\mu_1^{(d_1)}$-almost every $x\in\cE_V^{(d_1)}(\Omega_1)$ and every $A\in r(U)$.
\end{proof}

In order to prove Theorem~\ref{thm:Indep}, we need to find a way of simulating the order variables
of a $d$-Euclidean structure by the ordinary variables of a $0$-Euclidean structures in a way that
preserves as much independence as possible. This will be done by $d$-realizations defined below,
which are heavily inspired by the different representations $\cN^{\kqrO}$ and $\cH^{\kqrO}$ of the
quasirandom $k$-orientation (see Definition~\ref{def:kqrO}).

\begin{definition}\label{def:dreal}
  Let $\Omega_1=(X_1,\mu_1)$ and $\Omega_2=(X_2,\mu_2)$ be atomless probability spaces and let
  $d\in\NN$.

  A \emph{$d$-realization} from $\Omega_1$ to $\Omega_2$ is a family $f = (f_i)_{i\in\NN_+}$ of
  symmetric measure-preserving on h.o.a.\ functions
  $f_i\colon\cE_i(\Omega_1)\to\Omega_2\times\cO_i^d$ satisfying the following properties:
  \begin{enumerate}
  \item\label{def:dreal:indep} For every $i\in\NN_+$, if the coordinate functions of $f_i$ are
    $f_{i,0}\colon\cE_i(\Omega_1)\to\Omega_2$ and $f_{i,j}\colon\cE_i(\Omega_1)\to\cO_i$
    ($j\in[d]$), then, except for a zero $\mu_1$-measure set, $f_{i,0}$ depends only on the top
    coordinate (indexed by $[i]$) and for every $j\in[d]$, except for a zero $\mu_1$-measure set,
    $f_{i,j}$ depends only on the coordinates indexed by sets in $\binom{[i]}{>i-2}$.
  \item\label{def:dreal:almostinverse} There exists a family $g = (g_i)_{i\in\NN_+}$ of symmetric
    measure-preserving on h.o.a.\ functions
    $g_i\colon\cE_i^{(d)}(\Omega_2\times\Omega_2)\to\Omega_1$ such that
    \begin{align}\label{eq:dreal:almostinverse}
      \widehat{f}_i(\widehat{g}_i(x,y,{\lhd})) & = (x,{\lhd})
    \end{align}
    for every $i\in\NN_+$ and $(\mu_2\otimes\mu_2)^{(d)}$-almost every
    $(x,y,{\lhd})\in\cE_i^{(d)}(\Omega_2\times\Omega_2)$.
  \end{enumerate}

  Any family of functions $g$ as in item~\ref{def:dreal:almostinverse} above is called an
  \emph{almost inverse family} of $f$.
\end{definition}

\begin{remark}
  Since both the $f_i$ and $g_i$ are symmetric, by
  Lemma~\ref{lem:widehat}\ref{lem:widehat:equivariance}, equation~\eqref{eq:dreal:almostinverse}
  implies
  \begin{align*}
    \widehat{f}_V(\widehat{g}_V(x,y,{\lhd})) & = (x,{\lhd})
  \end{align*}
  for every countable $V$ and $(\mu_2\otimes\mu_2)^{(d)}$-almost every
  $(x,y,{\lhd})\in\cE_V^{(d)}(\Omega_2)$.
\end{remark}

Our first order of business is to show that $d$-realizations exist.

\begin{lemma}\label{lem:dreal}
  For all $d\in\NN$ and all atomless standard probability spaces $\Omega_1=(X_1,\mu_1)$ and
  $\Omega_2=(X_2,\mu_2)$, there exists a $d$-realization from $\Omega_1$ to $\Omega_2$.
\end{lemma}

\begin{proof}
  First we claim that it is enough to show the result for only one fixed pair of spaces
  $(\Omega_1,\Omega_2)$. Indeed, if $f = (f_i)_{i\in\NN_+}$ is a $d$-realization from $\Lambda_1$ to
  $\Lambda_2$, $g = (g_i)_{i\in\NN_+}$ is an almost inverse family of $f$ and $\Omega_1$ and
  $\Omega_2$ are standard probability spaces, then we can let $F_1\colon\Omega_1\to\Lambda_1$ and
  $F_2\colon\Omega_2\to\Lambda_2$ be measure-isomorphisms modulo $0$ and for every $i\in\NN_+$,
  define $f'_i\colon\cE_i(\Omega_1)\to\Omega_2\times\cO_i^d$ and
  $g'_i\colon\cE_i^{(d)}(\Omega_2\times\Omega_2)\to\Omega_1$ by
  \begin{align*}
    f'_i & \df (F_2^{-1}\otimes\id_{\cO_i^d})\comp f_i\comp \bigotimes_{A\in r(i)} F_1,\\
    g'_i & \df F_1^{-1}\comp g_i\comp\bigotimes_{A\in r(i)}(F_2\otimes F_2\otimes\id_{\cO_A^d}).
  \end{align*}
  (Note that the above only gives an almost everywhere definition as $F_1^{-1}$ and $F_2^{-1}$ are
  only defined in a full measure set, but we can simply extend the definition arbitrarily but
  measurably.)

  It is straightforward to check that $f'=(f'_i)_{i\in\NN_+}$ is a $d$-realization from $\Omega_1$
  to $\Omega_2$ and $g'=(g'_i)_{i\in\NN_+}$ is an almost inverse family of $f'$.

  \medskip

  We now note that the case $d=0$ is trivial: by taking $\Omega_1=\Omega_2$, we can simply let
  $f_i(x)\df x_{[i]}$ for every $i\in\NN_+$ and every $x\in\cE_i(\Omega_1)$. An almost inverse
  family of $f$ is then obtained by taking $g_i(x,y)\df x_{[i]}$ for every $i\in\NN_+$ and every
  $(x,y)\in\cE_i(\Omega_2\times\Omega_2)$.

  \medskip

  We now claim that the case $d\in\NN_+$ follows from the case $d=1$. Indeed, if
  $f=(f_i)_{i\in\NN_+}$ is a $1$-realization from $\Omega_1$ to $\Omega_2$ and $g=(g_i)_{i\in\NN_+}$
  is an almost inverse family of $f$, then letting
  \begin{align*}
    f'_i & \df \bigotimes_{i=1}^d f_i, &
    g'_i & \df \bigotimes_{i=1}^d g'_i
  \end{align*}
  for every $i\in\NN_+$, it is straightforward to check that $f'=(f'_i)_{i\in\NN_+}$ is a
  $d$-realization from $\Omega_1^d$ to $\Omega_2^d$ and $g'=(g'_i)_{i\in\NN_+}$ is an almost inverse
  family of $f'$.

  \medskip

  Let us then prove the case $d=1$ when $\Omega_1=([0,1)^2,\lambda)$ and $\Omega_2=([0,1),\lambda)$.

  First, for every $i\in\NN_+$, let
  \begin{align*}
    D_i & \df \{(x,y)\in\cE_i(\Omega_1) \mid \exists A,B\in r(i), (A\neq B\land y_A = y_B)\}.
  \end{align*}
  (Recall that we naturally identify $\cE_i(\Omega_1)$ with $\cE_i(\Omega_2)\times\cE_i(\Omega_2)$.)
  Note that $D_i$ is a Borel set with $\lambda$-measure zero, so it is sufficient to define $f_i$
  only outside of $D_i$.

  For each $i\in\NN_+$, let $\theta_i$ be the uniform probability measure on $S_i$ and let
  $\tau_i\colon[0,1)\to S_i$ be defined as follows. We fix an enumeration
  $\gamma^i_1,\ldots,\gamma^i_{i!}$ of $S_i$ and let
  \begin{align}\label{eq:dreal:tau}
    \tau_i(x) & \df \gamma^i_{\floor{x\cdot i!}+1}.
  \end{align}
  Note that $\tau_i$ is a measure-preserving with respect to $\lambda$ and $\theta_i$.

  Define the function $f_i\colon\cE_i(\Omega_1)\to\Omega_2\times\cO_i$ by
  \begin{align*}
    f_i(x,y) & \df (x_{[i]}, (\tau_i(y_{[i]})\comp\sigma_y)^*({\leq}))
  \end{align*}
  for every $(x,y)\in\cE_i(\Omega_1)\setminus D_i$, where $\leq$ is the usual order on $[i]$ and
  $\sigma_y\in S_i$ is the unique permutation such that
  \begin{align}\label{eq:dreal:sigmay}
    y_{[i]\setminus\{\sigma_y^{-1}(1)\}} < y_{[i]\setminus\{\sigma_y^{-1}(2)\}} < \cdots <
    y_{[i]\setminus\{\sigma_y^{-1}(i)\}}.
  \end{align}
  We also define $f_i$ arbitrarily but measurably in $D_i$.

  Clearly, item~\ref{def:dreal:indep} of Definition~\ref{def:dreal} holds for $f=(f_i)_{i\in\NN_+}$.

  \medskip

  Let us now show that $f$ is symmetric. Let $i\in\NN_+$, let $(x,y)\in\cE_i(\Omega)\setminus D_i$,
  let $\pi\in S_i$ and note that
  \begin{align*}
    (f_i\comp\pi^*)(x,y)
    & =
    \bigl(x_{[i]}, (\tau_i(y_{[i]})\comp\sigma_{\pi^*(y)})^*({\leq})\bigr)
    \\
    \bigl((\id_{\Omega_2}\otimes\pi^*)\comp f_i\bigr)(x,y)
    & =
    \Bigl(x_{[i]}, \pi^*\bigl((\tau_i(y_{[i]})\comp\sigma_y)^*({\leq})\bigr)\Bigr)
    =
    \bigl(x_{[i]}, (\tau_i(y_{[i]})\comp\sigma_y\comp\pi)^*({\leq})\bigr),
  \end{align*}
  so it suffices to show that $\sigma_y\comp\pi = \sigma_{\pi^*(y)}$.

  But indeed, we know that $\sigma_{\pi^*(y)}$ is the unique element of $S_i$ such that
  \begin{align*}
    \pi^*(y)_{[i]\setminus\{\sigma_{\pi^*(y)}^{-1}(1)\}}
    <
    \pi^*(y)_{[i]\setminus\{\sigma_{\pi^*(y)}^{-1}(2)\}}
    <
    \cdots
    <
    \pi^*(y)_{[i]\setminus\{\sigma_{\pi^*(y)}^{-1}(i)\}}
  \end{align*}
  and since for every $j\in[i]$, we have
  \begin{align*}
    \pi^*(y)_{[i]\setminus\{\sigma_{\pi^*(y)}^{-1}(j)\}}
    & =
    y_{[i]\setminus\{(\pi\comp\sigma_{\pi^*(y)}^{-1})(j)\}}
    =
    y_{[i]\setminus\{(\sigma_{\pi^*(y)}\comp\pi^{-1})^{-1}(j)\}},
  \end{align*}
  it follows that $\sigma_y = \sigma_{\pi^*(y)}\comp\pi^{-1}$, hence $\sigma_y\comp\pi =
  \sigma_{\pi^*(y)}$ as desired. Therefore $f_i$ is symmetric.

  \medskip

  Let us now show that $f_i$ is measure-preserving on h.o.a. Fix $(x,y)\in\cE_{i,i-1}(\Omega_1)$ and
  note that showing that the function $f_i(x,y,\place)\colon\Omega_1\to\Omega_2\times\cO_i$ is
  measure-preserving is equivalent to showing that if $\rn{w}$ and $\rn{z}$ are picked
  i.i.d.\ uniformly at random in $[0,1)$, then $(\rn{u},{\rn{\lhd}})\df f_i((x,\rn{w}),(y,\rn{z}))$
  is uniformly distributed in $[0,1)\times\cO_i$.

  The definition of $f_i$ gives $\rn{u} = \rn{w}$ and $\rn{\lhd}$ is $\rn{z}$-measurable with
  \begin{align*}
    \rn{\lhd} & = (\tau_i(\rn{z})\comp\sigma_{y,\rn{z}})^*({\leq}),
  \end{align*}
  but note that $\sigma_{y,\rn{z}}$ does not depend on $\rn{z}$ at all and $\tau_i(\rn{z})$ is
  uniformly distributed in $S_i$, hence $\rn{\lhd}$ is uniformly distributed in $\cO_i$ and is
  independent from $\rn{u}$. Therefore $(\rn{u},\rn{\lhd})$ is uniformly distributed in
  $[0,1)\times\cO_i$.

  \medskip
  
  It remains to show that there exists an almost inverse family of $f$. For each $i\in\NN_+$, let
  \begin{align*}
    D'_i
    & \df
    \{(y,{\lhd})\in\cE_i^{(1)}(\Omega_2) \mid
    \exists A,B\in r(i), (A\neq B\land y_A = y_B)\}.
  \end{align*}
  Clearly $D'_i$ is an $S_i$-invariant Borel set with zero $\lambda$-measure.

  We now define $S_i$-invariant functions
  \begin{align*}
    h_i\colon & \cE_i^{(1)}(\Omega_2)\setminus D'_i\to [0,1),\\
    k_i\colon & \cE_i^{(1)}(\Omega_2)\setminus D'_i\to [i!]
  \end{align*}
  inductively in $i\in\NN_+$ as follows. (We will later use $h_i$ in~\eqref{eq:dreal:gi} to define
  the almost inverse family $g$ of $f$.)

  To start the induction, we define $h_0$ as the constant $0$ function. Given $h_{i-1}$, for
  $(y,{\lhd})\in\cE_i^{(1)}(\Omega_2)\setminus D'_i$, we let $k_i(y,{\lhd})\in[i!]$ be the unique
  element such that
  \begin{align}\label{eq:dreal:ki}
    (\gamma^i_{k_i(y,{\lhd})}\comp\widetilde{\sigma}_{y,{\lhd}})^*({\leq}) & = {\lhd_{[i]}},
  \end{align}
  where $\leq$ is the usual order on $[i]$, $\widetilde{\sigma}_{y,{\lhd}}\in S_i$ is the unique
  permutation such that
  \begin{align}\label{eq:dreal:widetildesigmaylhd}
    [i]\ni a
    & \longmapsto
    h_{i-1}(\alpha_{[i]\setminus\{\widetilde{\sigma}_{y,{\lhd}}^{-1}(a)\},i}^*(y,{\lhd}))
    \in
    [0,1)
  \end{align}
  is increasing and for each $A\subseteq[i]$, $\alpha_{A,i}\colon[\lvert A\rvert]\to[i]$ is any
  injection with $\im(\alpha_{A,i})=A$. Note that since $h_{i-1}$ is inductively assumed to be
  $S_{i-1}$-invariant, the above does not depend on the choice of $\alpha_{A,i}$.

  We then define
  \begin{align*}
    h_i(y,{\lhd})
    & \df
    \frac{y_{[i]} + k_i(y,{\lhd}) - 1}{i!}.
  \end{align*}

  \smallskip

  Clearly, if $k_i$ is $S_i$-invariant, so is $h_i$.

  Let us then show that $k_i$ is $S_i$-invariant. Let $(y,{\lhd})\in\cE_i^{(1)}(\Omega_2)\setminus
  D'_i$, let $\pi\in S_i$. Since
  \begin{align*}
    (\gamma^i_{k_i(\pi^*(y,{\lhd}))}\comp\widetilde{\sigma}_{\pi^*(y,{\lhd})})^*({\leq}) & = \pi^*({\lhd_{[i]}}),
  \end{align*}
  we get
  \begin{align*}
    (\gamma^i_{k_i(\pi^*(y,{\lhd}))}\comp\widetilde{\sigma}_{\pi^*(y,{\lhd})}\comp\pi^{-1})^*({\leq})
    & =
    {\lhd_{[i]}},
  \end{align*}
  so it suffices to show that
  \begin{align}\label{eq:dreal:wsigma}
    \widetilde{\sigma}_{\pi^*(y,{\lhd})}\comp\pi^{-1} = \widetilde{\sigma}_{y,{\lhd}}.
  \end{align}

  To do so, note first that since for every $a\in[i]$, we have
  \begin{equation}\label{eq:dreal:hi-1}
    \begin{aligned}
      h_{i-1}(\alpha_{[i]\setminus\{\widetilde{\sigma}_{\pi^*(y,{\lhd})}^{-1}(a)\},i}^*(\pi^*(y,{\lhd})))
      & =
      h_{i-1}(\pi\comp\alpha_{[i]\setminus\{\widetilde{\sigma}_{\pi^*(y,{\lhd})}^{-1}(a)\},i})^*(y,{\lhd}))
      \\
      & =
      h_{i-1}(\alpha_{[i]\setminus\{(\widetilde{\sigma}_{\pi^*(y,{\lhd})}\comp\pi^{-1})^{-1}(a)\},i})^*(y,{\lhd}),
    \end{aligned}
  \end{equation}
  where the second equality follows since $h_{i-1}$ is inductively assumed to be
  $S_{i-1}$-invariant and
  \begin{align*}
    \im(\pi\comp\alpha_{[i]\setminus\{\widetilde{\sigma}_{\pi^*(y,{\lhd})}^{-1}(a)})
    & =
    \pi([i]\setminus\{\widetilde{\sigma}_{\pi^*(y,{\lhd})}^{-1}(a)\})
    \\
    & =
    [i]\setminus\{(\pi\comp\widetilde{\sigma}_{\pi^*(y,{\lhd})}^{-1})(a)\}
    =
    \im(\alpha_{[i]\setminus\{(\widetilde{\sigma}_{\pi^*(y,{\lhd})}\comp\pi^{-1})^{-1}(a)}).
  \end{align*}

  Since $\widetilde{\sigma}_{y,{\lhd}}$ and $\widetilde{\sigma}_{\pi^*(y,{\lhd})}$ are the unique elements of
  $S_i$ such that
  \begin{align*}
    [i]\ni a
    & \longmapsto
    h_{i-1}(\alpha_{[i]\setminus\{\widetilde{\sigma}_{y,{\lhd}}^{-1}(a)\}}^*(y,{\lhd}))
    \in
    [0,1)
    \\
    [i]\ni a
    & \longmapsto
    h_{i-1}(\alpha_{[i]\setminus\{\widetilde{\sigma}_{\pi^*(y,{\lhd})}^{-1}(a)\}}^*(\pi^*(y,{\lhd})))
    \in
    [0,1)
  \end{align*}
  are increasing, \eqref{eq:dreal:wsigma} follows from~\eqref{eq:dreal:hi-1}. Therefore $k_i$ is
  $S_i$-invariant.

  \smallskip

  We now claim that $h_i$ is measure-preserving on h.o.a.\ (or more precisely, any measurable
  extension of $h_i$ to $\cE_i^{(1)}(\Omega_2)$ is measure-preserving on h.o.a.). This is equivalent
  to saying that for $\lambda^{(1)}$-almost every $(y,{\lhd})\in\cE_{i,i-1}^{(1)}(\Omega_2)$ such
  that $y_A\neq y_B$ whenever $A\neq B$, if $\rn{z}$ and $\rn{\ll}$ are sampled independently and
  uniformly in $[0,1)$ and $\cO_i$, respectively, then $\rn{w}\df h_i((y,\rn{z}),({\lhd},\rn{\ll}))$
  is uniformly distributed in $[0,1)$.

  In turn, from the definition of $h_i$, the above is equivalent to showing that $\rn{j}\df
  k_i((y,\rn{z}),({\lhd},\rn{\ll}))$ is uniformly distributed in $[i!]$. But this follows since
  \begin{align*}
    (\gamma^i_{\rn{j}}\comp\widetilde{\sigma}_{(y,\rn{z}),({\lhd},\rn{\ll})})^*({\leq}) & = {\rn{\ll}},
  \end{align*}
  since $\widetilde{\sigma}_{(y,\rn{z}),({\lhd},\rn{\ll})}$ does not depend on $\rn{z}$ or
  $\rn{\ll}$ and since $\rn{\ll}$ is uniformly distributed on $\cO_i$.

  \smallskip

  We now finally define $g_i\colon\cE_i^{(1)}(\Omega_2\times\Omega_2)\to\Omega_1$ ($i\in\NN_+$) as
  \begin{align}\label{eq:dreal:gi}
    g_i(x,y,{\lhd}) & \df (x_{[i]}, h_i(y,{\lhd})),
  \end{align}
  (using an arbitrary measurable extension of $h_i$).

  It is straightforward to check that the fact that $h_i$ is $S_i$-invariant implies that $g_i$ is
  symmetric and the fact that $h_i$ is measure-preserving on h.o.a.\ implies that $g_i$ is measure
  preserving on h.o.a.

  We now claim that $g=(g_i)_{i\in\NN_+}$ is an almost inverse family of $f$. For this, we have to
  show that for every $i\in\NN_+$ and $\lambda^{(1)}$-almost every
  $(x,y,{\lhd})\in\cE_i^{(1)}(\Omega_2\times\Omega_2)$, we have
  \begin{align}\label{eq:dreal:gialmostinverse}
    \widehat{f}_i(\widehat{g}_i(x,y,{\lhd})) & = (x,{\lhd}).
  \end{align}

  Since $D'_i$ has zero measure, we may suppose that $(y,{\lhd})\notin D'_i$. Note that for $A\in
  r(i)$, we have
  \begin{align*}
    \widehat{f}_i(\widehat{g}_i(x,y,{\lhd}))_A
    & =
    (\overline{\alpha}_{A,i}^{-1})^*\biggl(f_{\lvert A\rvert}\Bigl(
    \alpha_{A,i}^*\bigl(\widehat{g}_i(x,y,{\lhd})
    \bigr)\Bigr)\biggr)
    =
    (\overline{\alpha}_{A,i}^{-1})^*\biggl(f_{\lvert A\rvert}\Bigl(
    \bigl(\widehat{g}_{\lvert A\rvert}(\alpha_{A,i}^*(x,y,{\lhd}))
    \bigr)\Bigr)\biggr),
  \end{align*}
  almost everywhere, where the second equality follows from
  Lemma~\ref{lem:widehat}\ref{lem:widehat:equivariance}.

  For every $B\in r(\lvert A\rvert)$, let
  \begin{align}
    z_B
    & \df
    \alpha_{B,\lvert A\rvert}^*\bigl(\alpha_{A,i}^*(x)\bigr)_{[\lvert B\rvert]}
    =
    \alpha_{A,i}^*(x)_B
    =
    x_{\alpha_{A,i}(B)},
    \notag
    \\
    w_B
    & \df
    h_{\lvert B\rvert}\Bigl(\alpha_{B,\lvert A\rvert}^*\bigl(\alpha_{A,i}^*(y,{\lhd})\bigr)\Bigr)
    =
    \frac{
      y_{\alpha_{A,i}(B)}
      + k_{\lvert B\rvert}\bigl((\alpha_{A,i}\comp\alpha_{B,\lvert A\rvert})^*(y,{\lhd})\bigr)
      - 1
    }{
      \lvert B\rvert!
    },
    \label{eq:dreal:w}
  \end{align}
  so that $\widehat{g}_{\lvert A\rvert}(\alpha_{A,i}^*(x,y,{\lhd}))=(z,w)$.

  We claim that
  \begin{align}\label{eq:dreal:sigmawsigma}
    \sigma_w & = \widetilde{\sigma}_{\alpha_{A,i}^*(y,{\lhd})}.
  \end{align}
  To see this, note first that for every $a\in[\lvert A\rvert]$, we have
  \begin{align}\label{eq:dreal:wh}
    w_{[\lvert A\rvert]\setminus\{\sigma_w^{-1}(a)\}}
    & =
    h_{\lvert A\rvert - 1}\Bigl(
    \alpha_{[\lvert A\rvert]\setminus\{\sigma_w^{-1}(a)\},\lvert A\rvert}^*
    \bigl(\alpha_{A,i}^*(y,{\lhd})\bigr)
    \Bigr).
  \end{align}

  The definition of $\sigma_w$ in~\eqref{eq:dreal:sigmay} is as the unique permutation in $S_{\lvert
    A\rvert}$ such that mapping $a\in[\lvert A\rvert]$ to the left-hand side of~\eqref{eq:dreal:wh}
  yields an increasing function. On the other hand, the definition of
  $\widetilde{\sigma}_{\alpha_{A,i}^*(y,{\lhd})}$ in~\eqref{eq:dreal:widetildesigmaylhd} is as the
  unique permutation in $S_{\lvert A\rvert}$ such that mapping $a\in[\lvert A\rvert]$ to
  \begin{align*}
    h_{\lvert A\rvert-1}\Bigl(
    \alpha_{[\lvert A\rvert]\setminus\{\widetilde{\sigma}_{\alpha_{A,i}^*(y,{\lhd})}(a)\}}^*\bigl(
    \alpha_{A,i}^*(y,{\lhd})\bigr)\Bigr)
  \end{align*}
  yields an increasing function; comparing this to the right-hand side of~\eqref{eq:dreal:wh}, we
  conclude that $\sigma_w=\widetilde{\sigma}_{\alpha_{A,i}^*(y,{\lhd})}$ (i.e.,
  \eqref{eq:dreal:sigmawsigma} holds).

  Note now that the definitions of $\tau_{\lvert A\rvert}$ and $w$ (in~\eqref{eq:dreal:tau}
  and~\eqref{eq:dreal:w}, respectively) imply that
  \begin{align}\label{eq:dreal:taugamma}
    \tau_{\lvert A\rvert}(w_{[\lvert A\rvert]})
    & =
    \gamma^{\lvert A\rvert}_{\floor{w_{[\lvert A\rvert]}\cdot\lvert A\rvert!}+1}
    =
    \gamma^{\lvert A\rvert}_{k_{\lvert A\rvert}(\alpha_{A,i}^*(y,{\lhd}))}
  \end{align}
  almost everywhere simply because $y_A\in[0,1)$.

  Thus, almost everywhere we have
  \begin{align*}
    \widehat{f}_i(\widehat{g}_i(x,y,{\lhd}))_A
    & =
    (\overline{\alpha}_{A,i}^{-1})^*\bigl(f_{\lvert A\rvert}(z,w)\bigr)
    \\
    & =
    \Bigl(z_{[\lvert A\rvert]},
    (\overline{\alpha}_{A,i}^{-1})^*\bigl(
    (\tau_{\lvert A\rvert}(w_{[\lvert A\rvert]})\comp\sigma_w)^*({\leq})
    \bigr)\Bigr)
    \\
    & =
    \Bigl(x_A,
    (\overline{\alpha}_{A,i}^{-1})^*\bigl(
    (\gamma^{\lvert A\rvert}_{k_{\lvert A\rvert}(\alpha_{A,i}^*(y,{\lhd}))}
    \comp\widetilde{\sigma}_{\alpha_{A,i}^*(x,y,{\lhd})})^*
    ({\leq})
    \bigr)\Bigr)
    \\
    & =
    \Bigl(x_A,
    (\overline{\alpha}_{A,i}^{-1})^*\bigl(
    \alpha_{A,i}\down_{[\lvert A\rvert]}^*({\lhd_A})
    \bigr)\Bigr)
    \\
    & =
    (x_A,{\lhd_A}),
  \end{align*}
  where the third equality follows from~\eqref{eq:dreal:sigmawsigma} and~\eqref{eq:dreal:taugamma},
  the fourth equality follows from~\eqref{eq:dreal:ki} and the fifth equality follows since
  $\alpha_{A,i}\down_{[\lvert A\rvert]} = \overline{\alpha}_{A,i}$. Thus~\eqref{eq:dreal:gialmostinverse}
  holds almost everywhere, so $g$ is an almost inverse family of $f$.
\end{proof}

The next lemma shows how to use families of symmetric measure-preserving on h.o.a.\ functions, and in
particular $d$-realizations to change representations of limits.

\begin{lemma}\label{lem:widehatEuclidean}
  Let $\Omega_1=(X_1,\mu_1)$ and $\Omega_2=(X_2,\mu_2)$ be atomless standard probability spaces, let
  $f = (f_i)_{i\in\NN_+}$ be a family of symmetric measure-preserving on h.o.a.\ functions
  $f_i\colon\cE_i^{(d_1)}(\Omega_1)\to\Omega_2\times\cO_i^{d_2}$ and let $\cN$ be a $d_2$-Euclidean
  structure in $\cL$ over $\Omega_2$. Then the following hold.
  \begin{enumerate}
  \item\label{lem:widehatEuclidean:rk} $\rk(f^*(\cN))\leq\rk(\cN)$.
  \item\label{lem:widehatEuclidean:phi} $\phi_{f^*(\cN)} = \phi_\cN$.
  \item\label{lem:widehatEuclidean:dreal} If $d_1=0$, $f$ is a $d_2$-realization and $\cN$ is
    $\ell$-independent, then $f^*(\cN)$ is a.e.\ equal to an $(\ell-1)$-independent $0$-Euclidean
    structure.
  \item\label{lem:widehatEuclidean:dreal2} If $d_1=0$, $f$ is a $d_2$-realization and $\cN$ is
    $\ell$-independent and does not depend on any coordinate in $\cO_A$ with $\lvert
    A\rvert=\ell+1$, then $f^*(\cN)$ is a.e.\ equal to an $\ell$-independent $0$-Euclidean
    structure.
  \end{enumerate}
\end{lemma}

\begin{proof}
  Item~\ref{lem:widehatEuclidean:rk} follows directly from definitions.

  \medskip

  For item~\ref{lem:widehatEuclidean:phi}, note that for every finite set $V$, for almost every
  $x\in\cE_V^{(d_1)}(\Omega_1)$, every $P\in\cL$ and every $\alpha\in (V)_{k(P)}$, we have
  \begin{align*}
    (M^{f^*(\cN)}_V(x)\vDash P(\alpha))
    & \iff
    \alpha^*(x)\in f^*(\cN)_P
    \iff
    \alpha^*(x)\in\widehat{f}_{k(P)}^{-1}(\cN_P)
    \iff
    (\widehat{f}_{k(P)}\comp\alpha^*)(x)\in\cN_P
    \\
    & \iff
    \alpha^*(\widehat{f}_V(x))\in\cN_P
    \iff
    (M^\cN_V(\widehat{f}_V(x))\vDash P(\alpha)),
  \end{align*}
  where the third equivalence follows from Lemma~\ref{lem:widehat}\ref{lem:widehat:equivariance}.

  Since $\widehat{f}_V$ is measure-preserving by Lemma~\ref{lem:widehat}\ref{lem:widehat:measpres},
  we get that for $K\in\cK_V$, we have
  \begin{align*}
    \tind\bigl(K,f^*(\cN)\bigr)
    & =
    \mu^{(d_1)}\Bigl(\Tind\bigl(K,f^*(\cN)\bigr)\Bigr)
    =
    \mu^{(d_1)}\bigl((M_V^{f^*(\cN)})^{-1}(\{K\})\bigr)
    =
    \mu^{(d_1)}\Bigl(\widehat{f}_V^{-1}\bigr((M_V^\cN)^{-1}(\{K\})\bigr)\Bigr)
    \\
    & =
    \mu^{(d_2)}\bigl((M_V^\cN)^{-1}(\{K\})\bigr)
    =
    \mu^{(d_2)}\bigl(\Tind(K,\cN)\bigr)
    =
    \tind(K,\cN).
  \end{align*}
  Therefore $\phi_{f^*(\cN)}=\phi_\cN$.

  \medskip

  For item~\ref{lem:widehatEuclidean:dreal}, since each $d_2$-$P$-on $\cN_P$ is $\ell$-independent
  and item~\ref{def:dreal:indep} of Definition~\ref{def:dreal} says that, except for a zero
  $\mu_1$-measure set, $f_i$ depends only on coordinates indexed by sets in $\binom{[i]}{>i-2}$, it
  follows that, except for a zero $\mu_1$-measure set, each $P$-on $f^*(\cN_P)$ depends only on
  coordinates indexed by sets in $\binom{[k(P)]}{>\ell-1}$.

  \medskip

  Item~\ref{lem:widehatEuclidean:dreal2} follows with the same logic as
  item~\ref{lem:widehatEuclidean:dreal}, except that since $\cN$ also does not depend on coordinates
  in $\cO_A$ with $\lvert A\rvert=\ell+1$, it follows that, except for a zero $\mu_1$-measure
  set, each $P$-on $f^*(\cN_P)$ depends only on coordinates indexed by sets in
  $\binom{[k(P)]}{>\ell}$ as the order coordinates of size $\ell+1$ are now irrelevant.
\end{proof}

\begin{remark}\label{rmk:altkqrO}
  The alternative representation $\cH^{\kqrO}$ in~\eqref{eq:altkqrO} of the quasirandom
  $k$-orientation is obtained from the standard representation $\cN^{\kqrO}$ as $\cH^{\kqrO} =
  f^*(\cN^{\kqrO})$ (up to zero-measure) for the $d$-realization constructed in
  Lemma~\ref{lem:dreal} (when $\Omega=([0,1),\lambda)$).
\end{remark}

As a direct corollary, we conclude that $\ast$-$\ell$-independence implies $\Independence[\ell']$
for every $\ell'<\ell$:

\begin{corollary}\label{cor:astindep->indep}
  If $\ell',\ell\in\NN$ are such that $\ell' < \ell$, then every $\ast$-$\ell$-independent limit
  $\phi$ satisfies $\Independence[\ell']$.
\end{corollary}

\begin{proof}
  It suffices to show the case $\ell'=\ell-1$. If $\phi$ is $\ast$-$\ell$-independent, then there
  exist $d\in\NN$ and an $\ell$-independent $d$-Euclidean structure $\cN$ over some space
  $\Omega=(X,\mu)$ such that $\phi=\phi_\cN$. By Lemma~\ref{lem:dreal}, there exists a
  $d$-realization $f$ from $\Omega$ to $\Omega$ and by Lemma~\ref{lem:widehatEuclidean},
  items~\ref{lem:widehatEuclidean:phi} and~\ref{lem:widehatEuclidean:dreal}, we have
  $\phi=\phi_{f^*(\cN)}$ and $f^*(\cN)$ is a $0$-Euclidean structure that is a.e.\ equal to some
  $(\ell-1)$-independent $0$-Euclidean structure $\cH$. In particular, this implies that for the
  $(\ell-1)$-independent $0$-Euclidean structure $\cH$, we have $\phi=\phi_\cH$, so $\phi$ satisfies
  $\Independence[\ell-1]$.
\end{proof}

Another direct corollary is that the rank of a limit can be alternatively defined as the minimum
rank of any $d$-Euclidean structure representing it (instead taking the minimum over only
$0$-Euclidean structures representing it):

\begin{corollary}\label{cor:rk}
  For every limit $\phi$, we have
  \begin{align*}
    \rk(\phi) & = \min_\cN \rk(\cN),
  \end{align*}
  where the minimum is over all $d$-Euclidean structures $\cN$ over some space $\Omega$ and for some
  $d\in\NN$ such that $\phi=\phi_\cN$.
\end{corollary}

\begin{proof}
  Let $m$ be the minimum on the right-hand side. From the definition of rank, it is clear that
  $\rk(\phi)\geq m$. On the other hand, if $\cN$ is a $d$-Euclidean structure over some space
  $\Omega$ with $\phi=\phi_\cN$ and $\rk(\cN)=m$, then by Lemma~\ref{lem:dreal}, there exists a
  $d$-realization $f$ from $\Omega$ to $\Omega$ and by Lemma~\ref{lem:widehatEuclidean},
  items~\ref{lem:widehatEuclidean:phi} and~\ref{lem:widehatEuclidean:rk}, we have
  $\phi=\phi_{f^*(\cN)}$ and $\rk(\phi)\leq\rk(f^*(\cN))\leq\rk(\cN)=m$ (as $f^*(\cN)$ is a
  $0$-Euclidean structure).
\end{proof}

\section{Uniqueness for $d$-theons}
\label{sec:uniqueness}

This section is devoted to proving the two easiest non-trivial implications of
Theorem~\ref{thm:UCouple}, namely, the
implications~\ref{thm:UCouple:astellindep}$\implies$\ref{thm:UCouple:UCouple}, which says that any
$\ast$-$\ell$-independent limit satisfies $\UCouple[\ell]$
(Proposition~\ref{prop:astellindep->UCouple}) and the
implication~\ref{thm:UCouple:kqrOcoupling}$\implies$\ref{thm:UCouple:UCouple}, which says that any
limit obtained as a quantifier-free definition of an independent coupling of an
$(\ell+1)$-quasirandom orientation with an $\ell$-independent limit satisfies $\UCouple[\ell]$
(Corollary~\ref{cor:kqrOcoupling->UCouple}).

The proof idea is the same as the proof that $\Independence[\ell]\implies\UCouple[\ell]$
from~\cite[Theorem~3.2]{CR23}: we first show that if $\rk(\psi)\leq\ell$, then every $d$-Euclidean
structure $\cN$ representing $\psi$ must be a.e.\ equal to a $d$-Euclidean structure of rank at most
$\ell$ (Lemma~\ref{lem:rk}). This means that for a coupling $\xi$ of a $\ast$-$\ell$-independent
limit $\phi$ with $\psi$, we can represent $\xi$ by some $d$-Euclidean structure $\cH$ such that the
$\phi$ part is $\ell$-independent and the $\psi$ part has ``almost everywhere'' rank at most $\ell$;
thus $\tind(\place,\cH)$ behaves multiplicatively so $\xi$ must be the independent coupling
$\phi\otimes\psi$.

The main ingredient for such a proof is $d$-theon uniqueness that characterizes when two $d$-theons
represent the same limit. We start by stating the $0$-theon uniqueness proved
in~\cite{CR20}\footnote{Let us note that the statement in~\cite{CR20} is slightly different from the
  below in two aspects: First, the measurability is with respect to the completion of the spaces
  $\cE_k(\Omega)$ on the domains of the functions (i.e., the analogue of Lebesgue measurability as
  opposed to Borel measurability); however, one can get Borel measurability by simply changing the
  functions on a zero measure set. Second, the notions of symmetric and measure-preserving are not
  a.e.\ but everywhere notions in~\cite{CR20}, but this is clearly not an issue (see
  Lemma~\ref{lem:widehatEuclidean}\ref{lem:widehatEuclidean:phi}).}.

\begin{theorem}[Theon Uniqueness~\protect{\cite[Theorems~3.9 and~3.11, Proposition~7.7]{CR20}}]\label{thm:TU}
  The following are equivalent for $T$-ons $\cN^1$ and $\cN^2$ over spaces $\Omega_1$ and
  $\Omega_2$, respectively.
  \begin{enumerate}
  \item We have $\phi_{\cN^1} = \phi_{\cN^2}$ (i.e., $\cN^1$ and $\cN^2$ represent the same limit).
  \item For some (equivalently, every) atomless standard probability space $\Omega$, there exist
    families $f=(f_i)_{i\in\NN_+}$ and $g=(g_i)_{i\in\NN_+}$ of symmetric measure-preserving on
    h.o.a.\ functions $f_i\colon\cE_i(\Omega)\to\Omega_1$ and $g_i\colon\cE_i(\Omega)\to\Omega_2$
    such that $f^*(\cN^1)$ is a.e.\ equal to $g^*(\cN^2)$.
  \item There exists a family $h=(h_i)_{i\in\NN_+}$ of symmetric measure-preserving on
    h.o.a.\ functions ($h_i\colon\cE_i(\Omega_1\times\Omega_1)\to\Omega_2$) such that $h^*(\cN^2)$
    is a.e.\ equal to the the $T$-on $\cH$ over $\Omega_1\times\Omega_1$ given by
    \begin{align*}
      \cH_P & \df \{(x,y)\in\cE_{k(P)}(\Omega\times\Omega) \mid x\in\cN^1_P\}
      \qquad (P\in\cL),
    \end{align*}
    where $\cL$ is the language of $T$.
  \end{enumerate}
\end{theorem}

The avid reader might have noticed that we did not use almost inverse families of $d$-realizations
in any of the applications in Section~\ref{sec:dreal}; this is where they will become useful: we
will reduce $d$-theon uniqueness to theon uniqueness via $d$-realizations and their almost inverse
families.

\begin{theorem}[$d$-Theon Uniqueness]\label{thm:dTU}
  The following are equivalent for a $d_1$-$T$-on $\cN^1$ over $\Omega_1$ and a $d_2$-$T$-on $\cN^2$
  over $\Omega_2$.
  \begin{enumerate}
  \item\label{thm:dTU:phi} We have $\phi_{\cN^1} = \phi_{\cN^2}$ (i.e., $\cN^1$ and $\cN^2$
    represent the same limit).
  \item\label{thm:dTU:everyfg} For every $d\in\NN$ and every atomless standard probability space
    $\Omega$, there exist families $f=(f_i)_{i\in\NN_+}$ and $g=(g_i)_{i\in\NN_+}$ of symmetric
    measure-preserving on h.o.a.\ functions
    $f_i\colon\cE_i^{(d)}(\Omega)\to\Omega_1\times\cO_i^{d_1}$ and
    $g_i\colon\cE_i^{(d)}(\Omega)\to\Omega_2\times\cO_i^{d_2}$ such that $f^*(\cN^1)$ is a.e.\ equal
    to $g^*(\cN^2)$.
  \item\label{thm:dTU:existsfg} Item~\ref{thm:dTU:existsfg} holds for some $d$ and $\Omega$.
  \item\label{thm:dTU:h} There exists a family $h=(h_i)_{i\in\NN_+}$ of symmetric measure-preserving
    on h.o.a.\ functions
    ($h_i\colon\cE_i^{(d_1)}(\Omega_1\times\Omega_1)\to\Omega_2\times\cO_i^{d_2}$) such that
    $h^*(\cN^2)$ is a.e.\ equal to the $d_1$-$T$-on $\cH$ over $\Omega_1\times\Omega_1$ given by
    \begin{align*}
      \cH_P
      & \df
      \{(x,y,{\lhd})\in\cE_{k(P)}^{(d_1)}(\Omega\times\Omega) \mid (x,{\lhd})\in\cN^1_P\}
      \qquad (P\in\cL),
    \end{align*}
    where $\cL$ is the language of $T$.
  \end{enumerate}
\end{theorem}

\begin{proof}
  The implication~\ref{thm:dTU:everyfg}$\implies$\ref{thm:dTU:existsfg} is trivial.

  \medskip

  The implications~\ref{thm:dTU:existsfg}$\implies$\ref{thm:dTU:phi}
  and~\ref{thm:dTU:h}$\implies$\ref{thm:dTU:phi} follow from
  Lemma~\ref{lem:widehatEuclidean}\ref{lem:widehatEuclidean:phi}.

  \medskip

  Let us now show the implication~\ref{thm:dTU:phi}$\implies$\ref{thm:dTU:everyfg}.

  By Lemma~\ref{lem:dreal}, for $j\in[2]$, there exists a $d_j$-realization $F_j$ from $\Omega_j$ to
  $\Omega_j$. Since by Lemma~\ref{lem:widehatEuclidean}\ref{lem:widehatEuclidean:phi} we have
  \begin{align*}
    \phi_{F_1^*(\cN^1)} & = \phi_{\cN^1} = \phi_{\cN^2} = \phi_{F_2^*(\cN^2)}
  \end{align*}
  for the $0$-$T$-ons $F_1^*(\cN^1)$ and $F_2^*(\cN^2)$, by Theorem~\ref{thm:TU}, there exist
  families $f'=(f'_i)_{i\in\NN_+}$ and $g'=(g'_i)_{i\in\NN_+}$ of symmetric measure-preserving on
  h.o.a.\ functions $f'_i\colon\cE_i(\Omega)\to\Omega_1$ and $g'_i\colon\cE_i(\Omega)\to\Omega_2$
  such that $(f')^*(F_1^*(\cN^1))$ is a.e.\ equal to $(g')^*(F_2^*(\cN^2))$.

  Define $f_i\colon\cE_i^{(d)}(\Omega)\to\Omega_1\times\cO_i^{d_1}$ and
  $g_i\colon\cE_i^{(d)}(\Omega)\to\Omega_2\times\cO^{d_2}$ by
  \begin{align*}
    f_i(x,{\leq}) & \df ((F_1)_i\comp\widehat{f'}_i)(x), &
    g_i(x,{\leq}) & \df ((F_2)_i\comp\widehat{g'}_i)(x).
  \end{align*}
  (Note that both $f_i$ and $g_i$ completely ignore the order variables.) It is straightforward to
  check that $f_i$ and $g_i$ are symmetric and measure-preserving on h.o.a.\ and that $f^*(\cN^1)$
  is a.e.\ equal to $g^*(\cN^2)$.

  \medskip

  It remains to show the implication~\ref{thm:dTU:phi}$\implies$\ref{thm:dTU:h}.

  Let $F_1$ be a $d_1$-realization from $\Omega_1$ to $\Omega_1$ and let $G_1$ be an almost inverse
  family of $F$. Let also $F_2$ be a $d_2$-realization from $\Omega_2$ to $\Omega_2$. Since by
  Lemma~\ref{lem:widehatEuclidean}\ref{lem:widehatEuclidean:phi} we have
  \begin{align*}
    \phi_{F_1^*(\cN^1)} & = \phi_{\cN^1} = \phi_{\cN^2} = \phi_{F_2^*(\cN^2)}
  \end{align*}
  for the $0$-$T$-ons $F_1^*(\cN^1)$ and $F_2^*(\cN^2)$, by Theorem~\ref{thm:dTU}, there exists a
  family $h'=(h'_i)_{i\in\NN_+}$ of symmetric measure-preserving on h.o.a.\ functions
  $h'_i\colon\cE_i(\Omega_1\times\Omega_1)\to\Omega_2$ such that $(h')^*(F_2^*(\cN^2))$ is
  a.e.\ equal to the $T$-on $\cH'$ over $\Omega_1\times\Omega_1$ given by
  \begin{align*}
    \cH'_P & \df \{(x,y)\in\cE_{k(P)}(\Omega_1\times\Omega_1) \mid x\in F_1^*(\cN^1)_P\}
    \qquad (P\in\cL).
  \end{align*}

  Let now $H\colon\Omega_1\to\Omega_1\times\Omega_1$ be a measure-isomorphism modulo $0$ and let
  $H^1,H^2\colon\Omega_1\to\Omega_1$ be its coordinate functions (which are measure-preserving) and
  define functions $\widetilde{H}^1_i,\widetilde{H}^2_i\colon\cE_i(\Omega_1)\to\cE_i(\Omega_1)$ by
  \begin{align*}
    \widetilde{H}^j_i(y)_A & \df H^j(y_A) \qquad (j\in[2], y\in\cE_i(\Omega_1), A\in r(i)).
  \end{align*}
  Clearly $\widetilde{H}^j_i$ is measure-preserving.

  Finally, define $h_i\colon\cE_i^{(d)}(\Omega_1\times\Omega_1)\to\Omega_2\times\cO_i^{d_2}$ by
  \begin{align*}
    h_i(x,y,{\leq})
    & \df
    (F_2)_i\Bigl(
    \widehat{h'}_i\bigl(\widehat{(G_1)}_i(x,\widetilde{H}^1_i(y),{\leq}), \widetilde{H}^2_i(y)
    \bigr)\Bigr).
  \end{align*}
  It is straightforward to check that $h_i$ is symmetric and measure-preserving on h.o.a.

  Let us now show that $h^*(\cN^2)$ is a.e.\ equal to the $d_1$-$T$-on $\cH$ over
  $\Omega_1\times\Omega_1$ given by
  \begin{align*}
    \cH_P
    & \df
    \{(x,y,{\lhd})\in\cE_{k(P)}^{(d_1)}(\Omega\times\Omega) \mid (x,{\lhd})\in\cN^1_P\}
    \qquad (P\in\cL).
  \end{align*}

  Fix $P\in\cL$ and note that for $(\mu_1\otimes\mu_1)^{(d_1)}$-almost every
  $(x,y,{\leq})\in\cE_{k(P)}^{(d_1)}(\Omega_1\times\Omega_1)$, we have
  \begin{align*}
    & \!\!\!\!\!\!
    (x,y,{\leq})\in\cH_P
    \\
    & \iff
    (x,{\leq})\in\cN^1_P
    \\
    & \iff
    \widehat{(F_1)}_{k(P)}\Bigl(\widehat{(G_1)}_{k(P)}\bigl(x,\widetilde{H}_{k(P)}^1(y),{\leq}
    \bigr)\Bigr)
    \in\cN^1_P
    \\
    & \iff
    \widehat{(G_1)}_{k(P)}\bigl(x,\widetilde{H}_{k(P)}^1(y),{\leq}\bigr)\in F_1^*(\cN^1)_P
    \\
    & \iff
    \Bigl(\widehat{(G_1)}_{k(P)}\bigl(x,\widetilde{H}_{k(P)}^1(y),{\leq}\bigr),\widetilde{H}^2_{k(P)}(y)
    \Bigr)
    \in
    \cH'_P
    \\
    & \iff
    \Bigl(\widehat{(G_1)}_{k(P)}\bigl(x,\widetilde{H}_{k(P)}^1(y),{\leq}\bigr),
    \widetilde{H}^2_{k(P)}(y)
    \Bigr)
    \in
    (h')^*(F_2^*(\cN^2))_P
    \\
    & \iff
    (\widehat{F_2})_{k(P)}\Biggl(
    \biggl(\widehat{h'}_{k(P)}
    \Bigl(\widehat{(G_1)}_{k(P)}\bigl(x,\widetilde{H}_{k(P)}^1(y),{\leq}\bigr),
    \widetilde{H}^2_{k(P)}(y)
    \Bigr)\biggr)\Biggr)
    \in\cN^2_P
    \\
    & \iff
    \widehat{h}_{k(P)}(x,y,{\leq})\in\cN^2_P
    \\
    & \iff
    (x,y,{\leq})\in h^*(\cN^2)_P,
  \end{align*}
  as desired, where the second equivalence follows since $G$ is an almost inverse family to $F$ and
  $\widetilde{H}^1_{k(P)}$ is measure-preserving, the fourth equivalence follows by the definition
  of $\cH'$, the fifth equivalence follows since $\cH'$ is a.e.\ equal to $(h')^*(F_2^*(\cN^2))$ and
  the seventh equivalence follows by the definition of $h$.
\end{proof}

The next lemma shows that except for zero measure changes, the notion of rank does not depend on the
representation.

\begin{lemma}\label{lem:rk}
  If $\cN$ is a $d$-Euclidean structure over $\Omega$, then it is a.e.\ equal to a
  $d$-Euclidean structure $\cN'$ over $\Omega$ with $\rk(\cN')=\rk(\phi_\cN)$. Moreover, if $\cN$ is
  $\ell$-independent from some $\ell$, then $\cN'$ can be taken to also be $\ell$-independent.
\end{lemma}

\begin{proof}
  Write $\Omega=(X,\mu)$ and let $r\df\rk(\phi_\cN)$. Note that it suffices to prove the case when
  $\cL$ has a single predicate symbol $P$.

  Let $k\df k(P)$ and define $W\colon\cE_{k,r}^{(d)}(\Omega)\to[0,1]$ by
  \begin{align*}
    W(x)
    & \df
    \mu^{(d)}\left(\left\{y\in\prod_{A\in\binom{[k]}{>r}} (X\times\cO_A^d)
    \;\middle\vert\;
    (x,y)\in\cN_P\right\}\right)
  \end{align*}
  and let
  \begin{align*}
    \cH_P & \df W^{-1}(1)\times\prod_{A\in\binom{[k]}{>r}} (X\times\cO_A^d).
  \end{align*}
  Clearly $\rk(\cH)\leq r$. Also note that if $\cN$ is $\ell$-independent, then so is $\cH$.

  Thus, it suffices to show that $\cN$ is a.e.\ equal to $\cH$. In turn, this is equivalent to
  showing that $W$ is $\{0,1\}$-valued $\mu^{(d)}$-a.e.

  Since $\rk(\phi_\cN)\leq r$, we know that there exists a $0$-Euclidean structure $\cG$ over some
  space $\Omega'=(X',\mu')$ such that $\phi_\cG=\phi_\cN$ and $\rk(\cG) = r$. By
  Theorem~\ref{thm:dTU}, there exist families $f=(f_i)_{i\in\NN_+}$ and $g=(g_i)_{i\in\NN_+}$ of
  symmetric measure-preserving on h.o.a.\ functions
  $f_i\colon\cE_i(\Omega)\to\Omega\times\cO_i^{(d)}$ and $g_i\colon\cE_i(\Omega)\to\Omega'$ such
  that $f^*(\cN)$ is a.e.\ equal to $g^*(\cG)$.
  
  From the structure of the functions $\widehat{f}_k$ and $\widehat{g}_k$, we can decompose them as
  \begin{align*}
    \widehat{f}_k(x,y) & = (F_1(x), F_2(x,y)),\\
    \widehat{g}_k(x,y) & = (G_1(x), G_2(x,y)),
  \end{align*}
  where $x\in\cE_{k,r}(\Omega)$, $y\in X^{\binom{[k]}{>r}}$ and
  \begin{align*}
    F_1\colon & \cE_{k,r}(\Omega)\to\cE_{k,r}^{(d)}(\Omega), &
    F_2\colon & \cE_k(\Omega)\to\prod_{A\in\binom{[k]}{>r}}(X\times\cO_A^d),
    \\
    G_1\colon & \cE_{k,r}(\Omega)\to\cE_{k,r}(\Omega'), &
    G_2\colon & \cE_k(\Omega)\to X^{\binom{[k]}{>r}}.
  \end{align*}

  The fact that the $f_i$ and $g_i$ are measure-preserving on h.o.a.\ implies that $F_1$ and $G_1$
  are measure-preserving and that for almost every $x\in\cE_{k,r}(\Omega)$, the restrictions
  $F_2(x,\place)$ and $G_2(x,\place)$ are also measure-preserving.

  Since $f^*(\cN)$ is a.e.\ equal to $g^*(\cG)$, by Fubini's Theorem, we get
  \begin{align*}
    W(F_1(x))
    & =
    \mu\left(\left\{y\in X^{\binom{[k]}{>r}}
    \;\middle\vert\;
    (G_1(x), G_2(x,y))\in\cG_P\right\}\right)
  \end{align*}
  for $\mu$-almost every $x\in\cE_{k,r}(\Omega)$. Since $\rk(\cG)=r$, the measure above can only be
  $0$ or $1$ (as $G_2(x,y)$ only contains coordinates indexed by sets in $\binom{[k]}{>r}$). Using
  again that $F_1$ is measure-preserving, we conclude that $W(z)\in\{0,1\}$ for $\mu^{(d)}$-almost
  every $z\in\cE_{k,r}^{(d)}(\Omega)$, as desired.
\end{proof}

We can finally prove the
implication~\ref{thm:UCouple:astellindep}$\implies$\ref{thm:UCouple:UCouple} of
Theorem~\ref{thm:UCouple} in the proposition below.

\begin{proposition}[Theorem~\ref{thm:UCouple}\ref{thm:UCouple:astellindep}$\implies$\ref{thm:UCouple:UCouple}]\label{prop:astellindep->UCouple}
  If $\phi\in\HomT$ is $\ast$-$\ell$-independent, then $\phi$ satisfies $\UCouple[\ell]$.
\end{proposition}

\begin{proof}
  Let $\cL$ be the language of $T$. Since $\phi$ is $\ast$-$\ell$-independent, there exist $d\in\NN$
  and an $\ell$-independent $d$-$T$-on over some space $\Omega$ such that $\phi=\phi_\cN$.

  Let now $\psi$ be any limit in some language $\cL'$ with $\rk(\psi)\leq\ell$, let $\xi$ be a
  coupling of $\phi$ and $\psi$ and let $\cG$ be any $0$-Euclidean structure in $\cL'$ over some
  space $\Omega'$ such that $\xi=\phi_\cG$.

  Since $\xi\rest_\cL=\phi$, we have $\phi_{\cG\rest_\cL}=\phi_\cN$, so by Theorem~\ref{thm:dTU},
  there exists a family $h=(h_i)_{i\in\NN_+}$ of symmetric measure-preserving on h.o.a.\ functions
  $h_i\colon\cE_i^{(d)}(\Omega\times\Omega)\to\Omega'$ such that $h^*(\cG\rest_\cL)$ is a.e.\ equal
  to the $d$-$T$-on $\cH$ over $\Omega\times\Omega$ given by
  \begin{align*}
    \cH_P
    & \df
    \{(x,y,{\lhd})\in\cE_{k(P)}^{(d)}(\Omega\times\Omega) \mid (x,{\lhd})\in\cN_P\}
    \qquad (P\in\cL).
  \end{align*}

  We now extend $\cH$ to a $d$-Euclidean structure $\cH'$ in $\cL\disjcup\cL'$ over
  $\Omega\times\Omega$ by defining for every $P\in\cL'$ the $P$-on
  \begin{align*}
    \cH'_P & \df \widehat{h}_{k(P)}^{-1}(\cG_P).
  \end{align*}

  Note that since $h^*(\cG\rest_\cL)$ is a.e.\ equal to $\cH$, it follows that $h^*(\cG)$ is
  a.e.\ equal to $\cH'$, hence by Lemma~\ref{lem:widehatEuclidean}\ref{lem:widehatEuclidean:phi}, we
  get $\phi_{\cH'}=\phi_\cG=\xi$.

  Note now that, since $\cN$ is $\ell$-independent, it follows that $\cH'\rest_\cL=\cH$ is
  $\ell$-independent and since
  $\rk(\phi_{\cH'\rest_{\cL'}})=\rk(\xi\rest_{\cL'})=\rk(\psi)\leq\ell$, by Lemma~\ref{lem:rk}, by
  possibly changing only a zero measure set, we may suppose that $\rk(\cH'\rest_{\cL'})\leq\ell$.

  Let now $V$ be a finite set and note that if $z\in\cE_V^{(d)}(\Omega\times\Omega)$, then since
  $\cH'\rest_\cL$ is $\ell$-independent, it follows that $M^\cN_V(z)\rest_\cL$ only depends on the
  coordinates of $z$ that are indexed by sets in $\binom{V}{>\ell}$. Similarly, since
  $\rk(\cH'\rest_{\cL'})\leq\ell$, $M^\cN_V(z)\rest_{\cL'}$ only depends on the coordinates of $z$
  that are indexed by sets in $r(V,\ell)$.

  We conclude that if $\rn{z}$ is picked at random in $\cE_V^{(d)}(\Omega\times\Omega)$ according to
  $(\mu\otimes\mu)^{(d)}$, then $M^{\cH'}_V(\rn{z})\rest_\cL$ is independent from
  $M^{\cH'}_V(\rn{z})\rest_{\cL'}$, so for every $K\in\cK_V[\cL\disjcup\cL']$, we have
  \begin{align*}
    \tind(K,\cH') & = \tind(K\rest_\cL,\cH'\rest_\cL)\cdot\tind(K\rest_{\cL'},\cH'\rest_{\cL'}),
  \end{align*}
  that is, $\phi_{\cH'}=\xi$ is the independent coupling $\phi\otimes\psi$. Therefore $\phi$
  satisfies $\UCouple[\ell]$.
\end{proof}

As a corollary, we get that the quasirandom $k$-orientation (see Definition~\ref{def:kqrO})
satisfies $\UCouple[k-1]$:
\begin{corollary}\label{cor:kqrO}
  The quasirandom $k$-orientation $\phi^{\kqrO}$ satisfies $\UCouple[k-1]$.
\end{corollary}

\begin{proof}
  Since $\phi^{\kqrO}=\phi_{\cN^{\kqrO}}$ for the $(k-1)$-independent $1$-$\TkO$-on $\cN^{\kqrO}$,
  it is $\ast$-$(k-1)$-independent, so it satisfies $\UCouple[k-1]$ by
  Proposition~\ref{prop:astellindep->UCouple}.
\end{proof}

As a further consequence, we can also obtain the
implication~\ref{thm:UCouple:kqrOcoupling}$\implies$\ref{thm:UCouple:UCouple} of
Theorem~\ref{thm:UCouple} from the facts that $\Independence[\ell]\implies\UCouple[\ell]$ and that
$\UCouple[\ell]$ is preserved by independent couplings and quantifier-free definitions.

\begin{corollary}[Theorem~\ref{thm:UCouple}\ref{thm:UCouple:kqrOcoupling}$\implies$\ref{thm:UCouple:UCouple}]\label{cor:kqrOcoupling->UCouple}
  If $I\colon T\leadsto T'\cup\TkO[(\ell+1)]$ is a quantifier-free definition and $\psi\in\HomT[T']$
  satisfies $\Independence[\ell]$, then $I^*(\psi\otimes\phi^{\kqrO[(\ell+1)]})$ satisfies
  $\UCouple[\ell]$.
\end{corollary}

\begin{proof}
  By~\cite[Theorem~3.2]{CR23} and Corollary~\ref{cor:kqrO}, both $\psi$ and $\phi^{\kqrO[(\ell+1)]}$
  satisfy $\UCouple[\ell]$, hence by~\cite[Theorem~3.4]{CR23} so does their independent coupling
  $\psi\otimes\phi^{\kqrO[(\ell+1)]}$ and by~\cite[Theorem~3.3]{CR23} so does
  $I^*(\psi\otimes\phi^{\kqrO[(\ell+1)]})$.
\end{proof}

\section{Types}
\label{sec:types}

We now move to arguably the hardest implication of Theorem~\ref{thm:UCouple},
\ref{thm:UCouple:UCouple}$\implies$\ref{thm:UCouple:1ellindep}, whose full argument will only be
concluded in Section~\ref{sec:amlg}.

Recall from the introduction that the main idea is to show that a limit satisfying $\UCouple[\ell]$
has a different representation in which we sample ``amalgamating types'' of points instead of the
points directly. For this, we will need these ``amalgamating types'' to have the following
properties:
\begin{enumerate}
\item\label{it:unique} (Uniqueness for $\UCouple[\ell]$) If $\cN$ is a $d$-Euclidean structure that is
  $(\ell-1)$-independent and satisfies $\UCouple[\ell]$, then $\cN$ has a unique amalgamating
  $\ell$-type.
\item\label{it:strongamalg} (Strong amalgamation:) If $A$ is a finite set then the distribution of
  the random amalgamating $A$-type is completely determined when we are given the amalgamating $B$-types for
  all $B\in r(A,\lvert A\rvert-1)$.
\end{enumerate}
As we have previously mentioned, amalgamating types will be obtained as something intermediate
between what we call ``dissociated types'' and ``overlapping types''. The objective of this section
is to define these latter concepts and prove basic properties about them, culminating in the proof
that they satisfy the uniqueness of item~\ref{it:unique} above
(Proposition~\ref{prop:typeuniqueness}). We will see in Section~\ref{sec:dsctovlp} that they also
satisfy a weak version of item~\ref{it:strongamalg} (Proposition~\ref{prop:weakamalgdsctovlp}).

\begin{definition}
  Given a finite set $A$ and a countable set $V$ disjoint from $A$, we define the sets
  \begin{align*}
    r^{\dsct}(A,V) & \df r(V),
    \\
    \overline{r}^{\dsct}(A,V)
    & \df
    r(A\cup V)\setminus(r(A)\cup r^{\dsct}(A,V))
    =
    \{B\in r(A\cup V) \mid B\cap A\neq\varnothing\land B\cap V\neq\varnothing\},
    \\
    r^{\ovlp}(A,V)
    & \df
      \{B\in r(A\cup V) \mid A = \varnothing\lor (B\cap V\neq\varnothing \land A\not\subseteq B)\},
    \\
    \overline{r}^{\ovlp}(A,V)
    & \df
    r(A\cup V)\setminus(r(A)\cup r^{\ovlp}(A,V))
    =
    \{B\in r(A\cup V) \mid \varnothing\neq A\subsetneq B\}.
  \end{align*}
  (Note that when $A=\varnothing$, we have $r^{\dsct}(\varnothing,V)=r^{\ovlp}(\varnothing,V)=r(V)$
  and $\overline{r}^{\dsct}(\varnothing,V)=\overline{r}^{\ovlp}(\varnothing,V)=\varnothing$.)
\end{definition}

The dissociated type of some $d$-Euclidean structure $\mathcal{N}$ will be determined by the
information contained in the coordinates whose indices are in $r(A)\cup r^{\dsct}(A,V)$. The
overlapping type will contain slightly more information, since it will instead use the coordinates
whose indices are in $r(A)\cup r^{\ovlp}(A,V)$, a slightly larger set. (The sets
$\overline{r}^{\dsct}$ and $\overline{r}^{\ovlp}$ are simply the complementary information.)

\begin{definition}
  Given further an atomless standard probability space $\Omega=(X,\mu)$ and $d\in\NN$, we let
  \begin{align*}
    \cE_{A,V}^{(d),\dsct}(\Omega)
    & \df
    \prod_{B\in r^{\dsct}(A,V)} (X\times\cO_B^d),
    &
    \overline{\cE}_{A,V}^{(d),\dsct}(\Omega)
    & \df
    \prod_{B\in \overline{r}^{\dsct}(A,V)} (X\times\cO_B^d),
    \\
    \cE_{A,V}^{(d),\ovlp}(\Omega)
    & \df
    \prod_{B\in r^{\ovlp}(A,V)} (X\times\cO_B^d),
    &
    \overline{\cE}_{A,V}^{(d),\ovlp}(\Omega)
    & \df
    \prod_{B\in \overline{r}^{\ovlp}(A,V)} (X\times\cO_B^d),
  \end{align*}
  and we equip these spaces with the product measures
  \begin{align*}
    \bigotimes_{B\in r^{\dsct}(A,V)}\left(\mu\otimes\bigotimes_{i=1}^d \nu_B\right),
    & &
    \bigotimes_{B\in \overline{r}^{\dsct}(A,V)}\left(\mu\otimes\bigotimes_{i=1}^d \nu_B\right),
    \\
    \bigotimes_{B\in r^{\ovlp}(A,V)}\left(\mu\otimes\bigotimes_{i=1}^d \nu_B\right),
    & &
    \bigotimes_{B\in \overline{r}^{\ovlp}(A,V)}\left(\mu\otimes\bigotimes_{i=1}^d \nu_B\right),
  \end{align*}
  which by abuse of notation, we denote all by $\mu^{(d)}$ (which is also the notation for the
  measure on $\cE_A^{(d)}$ and $\cE_{A,\ell}^{(d)}$).

  Note that an injection $\alpha\colon A\cup V\to U$ contra-variantly induces maps
  \begin{align*}
    \alpha^*\colon &
    \cE_{\alpha(A),U\setminus\alpha(A)}^{(d),\dsct}
    \to
    \cE_{A,V}^{(d),\dsct},
    &
    \alpha^*\colon &
    \overline{\cE}_{\alpha(A),U\setminus\alpha(A)}^{(d),\dsct}
    \to
    \overline{\cE}_{A,V}^{(d),\dsct},
    \\
    \alpha^*\colon &
    \cE_{\alpha(A),U\setminus\alpha(A)}^{(d),\ovlp}
    \to
    \cE_{A,V}^{(d),\ovlp},
    &
    \alpha^*\colon &
    \overline{\cE}_{\alpha(A),U\setminus\alpha(A)}^{(d),\ovlp}
    \to
    \overline{\cE}_{A,V}^{(d),\ovlp},
  \end{align*}
  all given by the formula $\alpha^*(x)_B\df x_{\alpha(B)}$. With even more abuse of notation, we
  also use $\alpha^*$ for the product maps of the first two or the last two functions above (which
  are also contra-variant definitions).

  Given further a $d$-Euclidean structure $\cN$ in $\cL$ over $\Omega$, we define the functions
  \begin{align*}
    f^{\dsct,\cN}_{A,V}\colon &
    \cE_A\times\cE_{A,V}^{(d),\dsct}\to\cP(\cK_{A\cup V}),
    \\
    f^{\ovlp,\cN}_{A,V}\colon & \cE_A\times\cE_{A,V}^{(d),\ovlp}\to\cP(\cK_{A\cup V})
  \end{align*}
  as follows: for $(x,z)\in\cE_A^{(d)}\times\cE_{A,V}^{(d),\dsct}$, we let $\rn{y}$ be picked at
  random in $\overline{\cE}_{A,V}^{(d),\dsct}$ according to $\mu^{(d)}$ and let
  $f^{\dsct,\cN}_{A,V}(x,z)$ be the distribution of the random $\cL$-structure $M^\cN_{A\cup
    V}(x,\rn{y},z)$ over $A\cup V$. We define $f^{\ovlp,\cN}_{A,V}$ analogously by replacing $\dsct$
  with $\ovlp$ throughout.

  For a point $x\in\cE_A^{(d)}$, the \emph{dissociated $(A,V)$-type} of $x$ in $\cN$, denoted by
  $\pp^{\dsct,\cN}_{A,V}(x)$, is the function $\cE_{A,V}^{(d),\dsct}\to\cP(\cK_{A\cup V})$ obtained
  from $f^{\dsct,\cN}_{A,V}$ by fixing its $\cE_A^{(d)}$ argument to be $x$, up to
  $\mu^{(d)}$-almost everywhere equivalence (in the argument in $\cE_{A,V}^{(d),\dsct}$).

  We also let $\cS_{A,V}^{(d),\dsct}(\Omega,\cL)$ be the space of all functions
  $p\colon\cE_{A,V}^{(d),\dsct}\to\cP(\cK_{A\cup V})$ up to $\mu^{(d)}$-almost everywhere
  equivalence such that:
  \begin{itemize}
  \item For every cylinder set $C(U,K)\df\{M\in\cK_{A\cup V} \mid M\rest_U=K\}$ of $\cK_{A\cup V}$
    (see~\eqref{eq:cylinderset}), the map $x\mapsto p(x)(B)$ is measurable.
  \item There exists $M^{\dsct}_{A,V}(p)\in\cK_A$ such that for every $x\in\cE_{A,V}^{(d),\dsct}$, we have
    \begin{align*}
      p(x)(\{K\in\cK_{A\cup V} \mid K\rest_A = M^{\dsct}_{A,V}(p)\}) = 1.
    \end{align*}
  \end{itemize}

  Clearly such $M^{\dsct}_{A,V}(p)$ is unique and is called the \emph{underlying model} of $p$ (and we
  view $M^{\dsct}_{A,V}$ as a function $\cS_{A,V}^{(d),\dsct}\to\cK_A$). Note also that we can view
  $\pp^{\dsct,\cN}_{A,V}$ as the unique function $\cE_A^{(d)}\to\cS_{A,V}^{(d),\dsct}$ such that
  \begin{align*}
    \pp^{\dsct,\cN}_{A,V}(x)(z) & = f^{\dsct,\cN}_{A,V}(x,z)
    \qquad (x\in\cE_A^{(d)}, z\in\cE_{A,V}^{(d),\dsct}).
  \end{align*}

  We define \emph{overlapping $(A,V)$-types} and
  $\pp^{\ovlp,\cN}_{A,V}\colon\cE_A^{(d)}\to\cS_{A,V}^{(d),\ovlp}$ analogously by replacing $\dsct$
  with $\ovlp$ throughout.
\end{definition}

Our first goal is to show in the lemma below that the definitions of $f^{\dsct,\cN}_{A,V}$ and
$f^{\ovlp,\cN}_{A,V}$ are equivariant. An important consequence is that the dissociated (overlapping,
respectively) $(A,V)$-type of $x$ is completely determined by the dissociated (overlapping,
respectively) $(A,U)$-type of $x$ for any fixed countably infinite set $U$ (disjoint from
$A$). Thus, we will use the term dissociated (overlapping, respectively) $A$-type when we do not
want to specify $V$.

\begin{lemma}[Equivariance of $f^{\dsct,\cN}$ and $f^{\ovlp,\cN}$]\label{lem:equivf}
  Let $\cN$ be a $d$-Euclidean structure in $\cL$ over $\Omega$, let $A$ be a finite set, let $V$ be
  a countable set disjoint from $A$. Then the definition of $f^{\dsct,\cN}_{A,V}$ is equivariant in
  the sense that for every injection $\alpha\colon A\cup V\to U$, the following diagram commutes:
  \begin{equation}\label{eq:diagpp}
    \begin{tikzcd}[column sep={2cm}]
      \cE_{\alpha(A)}^{(d)}\times\cE_{\alpha(A),U\setminus\alpha(A)}^{(d),\dsct}
      \arrow[r, "f^{\dsct,\cN}_{\alpha(A),U\setminus\alpha(A)}"]
      \arrow[d, "\alpha^*"']
      &
      \cP(\cK_U)
      \arrow[d, "\alpha^*"]
      \\
      \cE_A^{(d)}\times\cE_{A,V}^{(d),\dsct}
      \arrow[r, "f^{\dsct,\cN}_{A,V}"]
      &
      \cP(\cK_{A\cup V})
    \end{tikzcd}
  \end{equation}

  In particular, if $V$ is infinite, then for every set $A'$ with $\lvert A'\rvert=\lvert A\rvert$
  and every $V'$ countable and disjoint from $A'$, the
  function $\pp^{\dsct,\cN}_{A',V'}$ is completely determined by $\pp^{\dsct,\cN}_{A,V}$ via
  \begin{align}\label{eq:ppalpha*x}
    \pp^{\dsct,\cN}_{A',V'}(\alpha\down_{A'}^*(x))\comp\alpha^* & = \alpha^*\comp\pp^{\dsct,\cN}_{A,V}(x)
    \qquad (x\in\cE_{A,V}^{(d),\dsct}),
  \end{align}
  where $\alpha\colon A'\cup V'\to A\cup V$ is any injection with $\alpha(A')=A$.

  The analogous statements hold for $f^{\ovlp,\cN}$ and overlapping types.
\end{lemma}

\begin{proof}
  The proofs for the dissociated and overlapping cases are analogous to each other, so we prove only
  the dissociated case.

  Let $\alpha\colon A\cup V\to U$ be an injection, let $x\in\cE_{\alpha(A)}^{(d)}$ and
  $z\in\cE_{\alpha(A),U\setminus\alpha(A)}^{(d),\dsct}$. Our objective is to show that
  \begin{align*}
    \alpha^*(f^{\dsct,\cN}_{\alpha(A),U\setminus\alpha(A)}(x,z))
    & =
    f^{\dsct,\cN}_{A,V}(\alpha^*(x,z)).
  \end{align*}
  For this, we let $\rn{y}$ and $\rn{w}$ be picked at random in
  $\overline{\cE}_{\alpha(A),U\setminus\alpha(A)}^{(d),\dsct}$ and
  $\overline{\cE}_{A,V}^{(d),\dsct}$, respectively according to $\mu^{(d)}$ and we need to show that
  \begin{align}\label{eq:alpha*M}
    \alpha^*(M^\cN_U(x,\rn{y},z)) & \sim M^\cN_{A\cup V}(\alpha^*(x),\rn{w},\alpha^*(z)).
  \end{align}

  But note that since $M^\cN$ is equivariant, we have
  \begin{align*}
    \alpha^*(M^\cN_U(x,\rn{y},z))
    & =
    M^\cN_U(\alpha^*(x),\alpha^*(\rn{y}),\alpha^*(z))
  \end{align*}
  and since the middle $\alpha^*$ is measure-preserving (with respect to both versions of
  $\mu^{(d)}$), we have $\alpha^*(\rn{y})\sim\rn{w}$, so~\eqref{eq:alpha*M} follows.

  \medskip

  For the final assertion, first note that since $V$ is (countably) infinite and $V'$ is countable,
  an injection $\alpha\colon A'\cup V'\to A\cup V$ is guaranteed to exist.

  We claim that $\alpha\down_{A'}^*$ and $\alpha^*$ that appear on the left-hand side
  of~\eqref{eq:ppalpha*x} are surjective (which implies that~\eqref{eq:ppalpha*x} indeed completely
  determines $\pp^{\dsct,\cN}_{A',V'}$).

  For $\alpha\down_{A'}^*\colon\cE_A^{(d)}\to\cE_{A'}^{(d)}$, since $\lvert A'\rvert=\lvert
  A\rvert$, $\alpha\down_{A'}$ is bijective so if $x'\in\cE_{A'}^{(d)}$, then $x\in\cE_A^{(d)}$
  defined by $x_B\df x_{\alpha^{-1}(B)}$ ($B\in r(A)$) satisfies $\alpha\down_{A'}^*(x)=x'$.

  For $\alpha^*\colon\cE_{A',V'}^{(d),\dsct}\to\cE_{A,V}^{(d),\dsct}$, if
  $z'\in\cE_{A',V'}^{(d),\dsct}$, then $z\in\cE_{A,V}^{(d),\dsct}$ defined by $z_B\df
  z_{\alpha^{-1}(B)}$ when $B\subseteq\im(\alpha)$ and defined arbitrarily for $B\in
  r^{\dsct}(A,V)\setminus\im(\alpha)$ satisfies $\alpha^*(z)=z'$.

  Thus the claim is proved.

  Let us now show that~\eqref{eq:ppalpha*x} holds: for $z\in\cE_{A',V'}^{(d),\dsct}$, we have
  \begin{align*}
    \pp^{\dsct,\cN}_{A',V'}(\alpha\down_{A'}^*(x))(\alpha^*(z))
    & =
    f^{\dsct,\cN}_{A',V'}(\alpha^*(x,z))
    =
    \alpha^*(f^{\dsct,\cN}_{A,V}(x,z))
    =
    \alpha^*(\pp^{\dsct,\cN}_{A,V}(x)(z)),
  \end{align*}
  where the second equality follows from the fact that~\eqref{eq:diagpp} commutes and $\alpha(A)=A'$
  since $\lvert A\rvert=\lvert A'\rvert < \infty$ and $\alpha$ is injective.
\end{proof}

Our next objective is to show that injective functions also induce contra-variant maps between the
spaces of types (both dissociated and overlapping).

\begin{definition}\label{def:contravartypes}
  Let $\Omega=(X,\mu)$ be an atomless standard probability space, let $d\in\NN$, let $\cL$ be a
  finite relational language, let $\alpha\colon A\to B$ be an injection between finite sets and let
  $V$ be a countable set disjoint from both $A$ and $B$.

  We define the map $\alpha^*\colon\cS_{B,V}^{(d),\dsct}\to\cS_{A,V}^{(d),\dsct}$ as follows: recall
  that $\overline{\alpha}$ denotes the bijection $A\to\alpha(A)$ obtained from $\alpha$ by
  restricting its codomain. In turn, this contra-variantly induces the map
  $(\overline{\alpha}^{-1})^*\colon\cE_{A,V}^{(d),\dsct}\to\cE_{\alpha(A),V}^{(d),\dsct}$ via
  $(\overline{\alpha}^{-1})^*(x)_C = x_{\overline{\alpha}^{-1}(C)}$.

  For $p\in\cS_{B,V}^{(d),\dsct}$, that is, $p\colon\cE_{B,V}^{(d),\dsct}\to\cP(\cK_{B\cup V})$ we
  need to define $\alpha^*(p)\in\cS_{A,V}^{(d),\dsct}$, that is,
  $\alpha^*(p)\colon\cE_{A,V}^{(d),\dsct}\to\cP(\cK_{A\cup V})$. For this, given
  $x\in\cE_{A,V}^{(d),\dsct}$, we sample $\rn{y}$ at random in
  \begin{align*}
    \prod_{C\in r^{\dsct}(B,V)\setminus r^{\dsct}(\alpha(A),V)} (X\times\cO_C^d)
  \end{align*}
  according to $\bigotimes_{C\in r^{\dsct}(B,V)\setminus
    r^{\dsct}(\alpha(A),V)}(\mu\otimes\bigotimes_{i=1}^d \nu_C)$, we sample $\rn{M}$ at random in
  $\cK_{B\cup V}$ according to
  \begin{align*}
    p((\overline{\alpha}^{-1})^*(x),\rn{y})
  \end{align*}
  (conditionally on $\rn{y}$) and we let $\alpha^*(p)(x)\in\cP(\cK_{A\cup V})$ be the distribution of
  \begin{align*}
    \alpha\up_V^*(\rn{M})
  \end{align*}
  (recall that $\alpha\up_V\colon A\cup V\to B\cup V$ is the extension of $\alpha$ that acts identically on
  $V$, which gives a contra-variantly induced map $\alpha\up_V^*\colon\cK_{B\cup V}\to\cK_{A\cup V}$).

  Note that
  \begin{align*}
    \PP[\alpha\up_V^*(\rn{M})\rest_A = \alpha^*(M^{\dsct}_{B,V}(p))]
    & =
    \PP[\alpha^*(\rn{M}\rest_B) = \alpha^*(M^{\dsct}_{B,V}(p))]
    \geq
    \PP[\rn{M}\rest_B = M^{\dsct}_{B,V}(p)]
    =
    1,
  \end{align*}
  so indeed $\alpha^*(p)$ is an element of $\cS_{A,V}^{(d),\dsct}$ and we have
  \begin{align}\label{eq:equivMdsct}
    M^{\dsct}_{A,V}(\alpha^*(p)) & = \alpha^*(M^{\dsct}_{B,V}(p)),
  \end{align}
  that is, $M^{\dsct}$ is equivariant.

  We define $\alpha^*\colon\cS_{B,V}^{(d),\ovlp}\to\cS_{A,V}^{(d),\ovlp}$ analogously by replacing
  $\dsct$ with $\ovlp$ throughout.
\end{definition}

\begin{lemma}\label{lem:contravartypes}
  The definitions of the maps $\alpha^*$ of Definition~\ref{def:contravartypes} are contra-variant,
  that is, if $\alpha\colon A\to B$ and $\beta\colon B\to C$ are injections and $V$ is a countable
  set disjoint from $A$, $B$ and $C$, then $(\beta\comp\alpha)^* = \alpha^*\comp\beta^*$.
\end{lemma}

\begin{proof}
  We prove only the overlapping case as the dissociated case has an analogous proof.

  Let $p\in\cS_{C,V}^{(d),\ovlp}$ and $x\in\cE_{A,V}^{(d),\ovlp}$ and let $\rn{y}$ and $\rn{z}$ be
  picked at random in
  \begin{align*}
    \prod_{D\in r^{\ovlp}(B,V)\setminus r^{\ovlp}(\alpha(A),V)} (X\times\cO_D^d),
    & &
    \prod_{D\in r^{\ovlp}(C,V)\setminus r^{\ovlp}(\beta(B),V)} (X\times\cO_D^d),
  \end{align*}
  respectively, according to
  \begin{align*}
    \bigotimes_{D\in r^{\ovlp}(B,V)\setminus r^{\ovlp}(\alpha(A),V)}
    \left(\mu\otimes\bigotimes_{i=1}^d \nu_D\right),
    & &
    \bigotimes_{D\in r^{\ovlp}(C,V)\setminus r^{\ovlp}(\beta(B),V)}
    \left(\mu\otimes\bigotimes_{i=1}^d \nu_D\right)
  \end{align*}
  and independently from each other.

  By sampling $\rn{M}$ in $\cK_{C\cup V}$ according to
  \begin{align*}
    p\Bigl((\overline{\beta}^{-1})^*\bigl((\overline{\alpha}^{-1})^*(x), \rn{y}\bigr),
      \rn{z}\Bigr)
  \end{align*}
  (conditionally on $(\rn{y},\rn{z})$), we have that the conditional distributions given $\rn{y}$ of
  $\beta\up_V^*(\rn{M})$ and $\beta^*(p)((\overline{\alpha}^{-1})^*(x),\rn{y})$ are the same. In turn, this
  means that $\alpha\up_V^*(\beta\up_V^*(\rn{M}))\sim \alpha^*(\beta^*(p))(x)$.

  Now let $\rn{w}$ be picked at random in
  \begin{align*}
    \prod_{D\in r^{\ovlp}(C,V)\setminus r^{\ovlp}((\beta\comp\alpha)(A),V)} (X\times\cO_D^d)
  \end{align*}
  according to
  \begin{align*}
    \bigotimes_{D\in r^{\ovlp}(C,V)\setminus r^{\ovlp}((\beta\comp\alpha)(A),V)}
    \left(\mu\otimes\bigotimes_{i=1}^d \nu_D\right)
  \end{align*}
  and note that
  \begin{align*}
    \rn{w} & \sim \bigl((\overline{\beta}^{-1})^*(\rn{y}),\rn{z}\bigr),
  \end{align*}
  so if we sample $\rn{N}$ in $\cK_{C\cup V}$ according to
  \begin{align*}
    p\Bigl((\overline{\beta}^{-1})^*\bigl((\overline{\alpha}^{-1})^*(x)\bigr),\rn{w}\Bigr)
    & =
    p\bigl((\overline{\beta\comp\alpha}^{-1})^*(x),\rn{w}\bigr)
  \end{align*}
  (conditionally on $\rn{w}$), then we have $\rn{N}\sim\rn{M}$.

  Thus we get
  \begin{align*}
    (\beta\comp\alpha)^*(p)(x)
    & \sim
    (\beta\comp\alpha)\up_V^*(\rn{N})
    \sim
    (\beta\comp\alpha)\up_V^*(\rn{M})
    =
    \alpha\up_V^*(\beta\up_V^*(\rn{M}))
    \sim
    \alpha^*(\beta^*(p))(x),
  \end{align*}
  so we conclude that $(\beta\comp\alpha)^* = \alpha^*\comp\beta^*$.
\end{proof}

We also define natural maps that move between overlapping types and dissociated types.

\begin{definition}\label{def:typeproj}
  Let $\Omega=(X,\mu)$ be an atomless standard probability space, let $d\in\NN$, let $\cL$ be a
  finite relational language, let $A$ be a finite set and let $V$ be a countable set disjoint from
  $A$.

  We define the map $\pi^{\ovlp}_{A,V}\colon\cS_{A,V}^{(d),\ovlp}\to\cS_{A,V}^{(d),\dsct}$ as follows: for
  $p\in\cS_{A,V}^{(d),\ovlp}$, that is, $p\colon\cE_{A,V}^{(d),\ovlp}\to\cP(\cK_{A\cup V})$ and
  $x\in\cE_{A,V}^{(d),\dsct}$, we sample $\rn{y}$ at random in $\prod_{B\in r^{\ovlp}(A,V)\setminus
    r^{\dsct}(A,V)} (X\times\cO_B^d)$ according to $\bigotimes_{B\in r^{\ovlp}(A,V)\setminus
    r^{\dsct}(A,V)}(\mu\otimes\bigotimes_{i=1}^d\nu_B)$, we sample $\rn{M}$ at random in $\cK_{A\cup
    V}$ according to 
  \begin{align*}
    p(x,\rn{y})
  \end{align*}
  (conditionally on $\rn{y}$) and we let $\pi^{\ovlp}_{A,V}(p)(x)$ be the distribution of
  $\rn{M}$. Lemma~\ref{lem:equivpp} below shows that $\pi^{\ovlp}_{A,V}(p)$ is indeed an element of
  $\cS_{A,V}^{(d),\dsct}$ by showing that $\PP[\rn{M}\rest_A = M^{\ovlp}_{A,V}(p)] = 1$.
\end{definition}

The next lemma shows that both $\pp^{\dsct,\cN}$ and $\pp^{\ovlp,\cN}$ have the expected
equivariances and appropriately commute with the maps $M^{\dsct}$, $M^{\ovlp}$ and $\pi^{\ovlp}$.

\begin{lemma}\label{lem:equivpp}
  Let $\cN$ be a $d$-Euclidean structure in $\cL$ over $\Omega$, let $A$ and $B$ be finite sets, let $V$ be
  a countable set disjoint from $A$ and $B$ and let $\alpha\colon A\to B$ be an injection. Then the
  following diagram commutes:
  
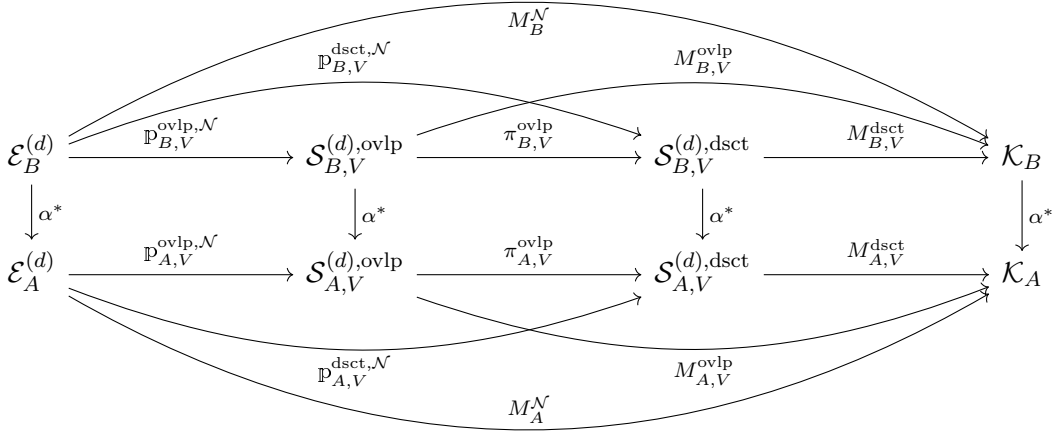
\begin{equation*}
    \begin{tikzcd}[column sep={3cm}]
      \cE_B^{(d)}
      \arrow[r, "\pp^{\ovlp,\cN}_{B,V}"]
      \arrow[rr, bend left=20, "\pp^{\dsct,\cN}_{B,V}"]
      \arrow[rrr, bend left=30, "M^\cN_B"']
      \arrow[d, "\alpha^*"]
      &
      \cS_{B,V}^{(d),\ovlp}
      \arrow[r, "\pi^{\ovlp}_{B,V}"]
      \arrow[rr, bend left=20, "M^{\ovlp}_{B,V}"]
      \arrow[d, "\alpha^*"]
      &
      \cS_{B,V}^{(d),\dsct}
      \arrow[r, "M^{\dsct}_{B,V}"]
      \arrow[d, "\alpha^*"]
      &
      \cK_B
      \arrow[d, "\alpha^*"]
      \\
      \cE_A^{(d)}
      \arrow[r, "\pp^{\ovlp,\cN}_{A,V}"]
      \arrow[rr, bend right=20, "\pp^{\dsct,\cN}_{A,V}"']
      \arrow[rrr, bend right=30, "M^\cN_A"]
      &
      \cS_{A,V}^{(d),\ovlp}
      \arrow[r, "\pi^{\ovlp}_{A,V}"]
      \arrow[rr, bend right=20, "M^{\ovlp}_{A,V}"']
      &
      \cS_{A,V}^{(d),\dsct}
      \arrow[r, "M^{\dsct}_{A,V}"]
      &
      \cK_A
    \end{tikzcd}
  \end{equation*}
\end{lemma}

\begin{proof}
  Let us first prove the commutative properties involving only one of the rows, which we can
  assume without loss of generality to be the second row.

  The equalities
  \begin{align*}
    M^{\dsct}_{A,V}\comp\pp^{\dsct,\cN}_{A,V} & = M^\cN_A, &
    M^{\ovlp}_{A,V}\comp\pp^{\ovlp,\cN}_{A,V} & = M^\cN_A
  \end{align*}
  follow straightforwardly from definitions.

  \medskip

  We now prove that
  \begin{align}\label{eq:pipp}
    \pi^{\ovlp}_{A,V}\comp\pp^{\ovlp,\cN}_{A,V} & = \pp^{\dsct,\cN}_{A,V}.
  \end{align}
  Let $z\in\cE_A^{(d)}$, let $x\in\cE_{A,V}^{(d),\dsct}$, sample $\rn{y}$ at random in $\prod_{C\in
    r^{\ovlp}(A,V)\setminus r^{\dsct}(A,V)} (X\times\cO_C^d)$ according to $\bigotimes_{C\in
    r^{\ovlp}(A,V)\setminus r^{\dsct}(A,V)} (\mu\otimes\bigotimes_{i=1}^d\nu_C)$ and sample $\rn{M}$
  at random in $\cK_{A\cup V}$ according to
  \begin{align*}
    \pp^{\ovlp,\cN}_{A,V}(z)(x,\rn{y})
  \end{align*}
  (conditionally on $\rn{y}$) so that $\pi^{\ovlp}_{A,V}(\pp^{\ovlp,\cN}_{A,V}(z))(x)$ is the distribution
  of $\rn{M}$.

  By the definition of $\pp^{\ovlp,\cN}_{A,V}$, if we sample $\rn{w}$ in
  $\overline{\cE}_{A,V}^{(d),\ovlp}$ according to $\mu^{(d)}$, then $\rn{M}$ has the same
  distribution as $M^\cN_{A\cup V}(z,x,\rn{y},\rn{w})$. Since the distribution of $(\rn{y},\rn{w})$
  as an element of $\overline{\cE}_{A,V}^{(d),\dsct}$ is exactly $\mu^{(d)}$, by the definition of
  $\pp^{\dsct,\cN}_{A,V}$, it follows that $\rn{M}$ has the same distribution as
  $\pp^{\dsct,\cN}_{A,V}(z)(x)$, so we conclude that
  $\pi^{\ovlp}_{A,V}(\pp^{\ovlp,\cN}_{A,V}(z))(x)=\pp^{\dsct,\cN}_{A,V}(z)(x)$, from which~\eqref{eq:pipp}
  follows.

  \medskip

  To conclude the second row, it remains to show that
  \begin{align}\label{eq:Mpi}
    M^{\dsct}_{A,V}\comp\pi^{\ovlp}_{A,V} & = M^{\ovlp}_{A,V}.
  \end{align}
  Let $p\in\cS_{A,V}^{(d),\ovlp}$, let $x\in\cE_{A,V}^{(d),\dsct}$, sample $\rn{y}$ at random in
  $\prod_{C\in r^{\ovlp}(A,V)\setminus r^{\dsct}(A,V)} (X\times\cO_C^d)$ according to
  $\bigotimes_{C\in r^{\ovlp}(A,V)\setminus r^{\dsct}(A,V)} (\mu\otimes\bigotimes_{i=1}^d\nu_C)$ and
  sample $\rn{M}$ at random in $\cK_{A\cup V}$ according to
  \begin{align*}
    p(x,\rn{y})
  \end{align*}
  (conditionally on $\rn{y}$) so that $\pi^{\ovlp}_{A,V}(p)(x)$ is the distribution of $\rn{M}$.

  Then
  \begin{align*}
    \PP[\rn{M}\rest_A = M^{\ovlp}_{A,V}(p)] & = 1,
  \end{align*}
  so we have $M^{\dsct}_{A,V}(\pi^{\ovlp}_{A,V}(p)) = M^{\ovlp}_{A,V}(p)$, from which~\eqref{eq:Mpi} follows.

  \medskip

  It remains to show that all squares involving exactly two elements of each row commute, that is,
  we need to show equivariance of $\pp^{\ovlp,\cN}$, $\pp^{\dsct,\cN}$, $M^\cN$, $\pi^{\ovlp}_{A,V}$,
  $M^{\ovlp}$ and $M^{\dsct}$.

  The fact that $\pp^{\ovlp,\cN}$ and $\pp^{\dsct,\cN}$ are equivariant follows by applying
  Lemma~\ref{lem:equivf} with $\alpha\up_V$.

  Equivariance of $M^\cN$ is straightforward.

  Equivariance of $M^{\dsct}$ was observed in~\eqref{eq:equivMdsct} and equivariance of
  $M^{\ovlp}$ is analogous.

  It remains to show equivariance of $\pi^{\ovlp}$.

  Let $p\in\cS_{B,V}^{(d),\ovlp}$, let $x\in\cE_{A,V}^{(d),\dsct}$ and let us describe the
  distributions of $(\pi^{\ovlp}_{A,V}\comp\alpha^*)(p)(x)$ and
  $(\alpha^*\comp\pi^{\ovlp}_{B,V})(p)(x)$. Throughout this, all measures will be product measures of
  copies of $\mu$ with appropriate copies of $\nu_C$, which can be deduced from the underlying space
  of the random variables, so we will omit the measures for simplicity.

  First, we sample $\rn{y}$ and $\rn{z}$ independently in
  \begin{align*}
    \prod_{C\in r^{\ovlp}(A,V)\setminus r^{\dsct}(A,V)} (X\times\cO_C^d),
    & &
    \prod_{C\in r^{\ovlp}(B,V)\setminus r^{\ovlp}(\alpha(A),V)} (X\times\cO_C^d),
  \end{align*}
  respectively and we sample $\rn{M}$ according to
  \begin{align*}
    p((\overline{\alpha}^{-1})^*(x,\rn{y}),\rn{z})
  \end{align*}
  (conditionally on $\rn{y}$ and $\rn{z}$) so that the conditional distribution of
  $\alpha\up_V^*(\rn{M})$ given $\rn{y}$ is $\alpha^*(p)(x,\rn{y})$, which in turn means that
  the unconditional distribution of $\alpha\up_V^*(\rn{M})$ is $(\pi^{\ovlp}_{A,V}\comp\alpha^*)(p)(x)$.

  For the other distribution, we sample $\rn{u}$ and $\rn{v}$ independently in
  \begin{align*}
    \prod_{C\in r^{\dsct}(B,V)\setminus r^{\dsct}(\alpha(A),V)} (X\times\cO_C^d),
    & &
    \prod_{C\in r^{\ovlp}(B,V)\setminus r^{\dsct}(B,V)} (X\times\cO_C^d),
  \end{align*}
  respectively and we sample $\rn{N}$ according to
  \begin{align*}
    p((\overline{\alpha}^{-1})^*(x),\rn{u},\rn{v})
  \end{align*}
  (conditionally on $\rn{u}$ and $\rn{v}$) so that the conditional distribution of $\rn{N}$ given
  $\rn{u}$ is $\pi^{\ovlp}_{B,V}(p)((\overline{\alpha}^{-1})^*(x),\rn{u})$, which in turn means that the
  unconditional distribution of $\alpha\up_V^*(\rn{N})$ is $(\alpha^*\comp\pi^{\ovlp}_{B,V})(p)(x)$.

  Now we note that since all distributions are products of copies of $\mu$ with appropriate copies
  of $\nu_C$ and we have appropriate independences between random elements, it follows that
  \begin{align*}
    ((\overline{\alpha}^{-1})^*(\rn{y}),\rn{z}) & \sim (\rn{u},\rn{v})
  \end{align*}
  as random elements of
  \begin{align}\label{eq:equivpispace}
    \prod_{C\in r^{\ovlp}(B,V)\setminus r^{\dsct}(\alpha(A),V)} (X\times\cO_C^d),
  \end{align}
  so both distributions of $\rn{M}$ and $\rn{N}$ are equal to the distribution of $\rn{K}$ defined
  as follows: we sample $\rn{w}$ in the space of~\eqref{eq:equivpispace} above and sample $\rn{K}$
  according to
  \begin{align*}
    p((\overline{\alpha}^{-1})^*(x),\rn{w})
  \end{align*}
  (conditionally on $\rn{w}$). Clearly $\rn{M}\sim\rn{K}\sim\rn{N}$, from which we get
  $\alpha\up_V^*(\rn{M})\sim\alpha\up_V^*(\rn{K})\sim\alpha\up_V^*(\rn{N})$, hence
  $(\pi^{\ovlp}_{A,V}\comp\alpha^*)(p)(x) = (\alpha^*\comp\pi^{\ovlp}_{B,V})(p)(x)$, concluding the
  proof.
\end{proof}

We conclude this section by showing uniqueness of types of size $\ell$ for $(\ell-1)$-independent
$d$-Euclidean structures satisfying $\UCouple[\ell]$.

\begin{proposition}\label{prop:typeuniqueness}
  If $\cN$ is a $d$-Euclidean structure in $\cL$ that is $(\ell-1)$-independent and satisfies
  $\UCouple[\ell]$, $A$ is a set of size at most $\ell$ and $V$ is a countable set disjoint from
  $A$, then for almost every $(x,x')\in\cE_A^{(d)}\times\cE_A^{(d)}$, we have
  $\pp^{\ovlp,\cN}_{A,V}(x) = \pp^{\ovlp,\cN}_{A,V}(x')$.
\end{proposition}

Before we start the proof, let us point out that the analogous result for $\pp^{\dsct,\cN}$ can be
derived from the above and the fact that
$\pp^{\dsct,\cN}_{A,V}=\pi^{\ovlp}_{A,V}\comp\pp^{\ovlp,\cN}_{A,V}$ from Lemma~\ref{lem:equivpp}.

\begin{proof}
  By equivariance of $\pp^{\ovlp,\cN}$ (Lemma~\ref{lem:equivpp}), it suffices to show only the case
  when $A=[\ell]$. Furthermore, by using measure-isomorphisms modulo $0$, it suffices to show the
  case when $\Omega$ is the interval $[0,1]$ and $\mu$ is the Lebesgue measure $\lambda$.

  In turn, it suffices to show that for every finite $U\subseteq[\ell]\cup V$, every $K\in\cK_U$,
  and for the function $g\colon\cE_\ell^{(d)}(\Omega)\times\cE_{\ell,V}^{(d),\ovlp}(\Omega)\to[0,1]$
  given by
  \begin{align*}
    g(x,y) & \df \pp^{\ovlp,\cN}_{\ell,V}(x)(y)(\{M\in\cK_{[\ell]\cup V} \mid M\rest_U = K\})
    \qquad (x\in\cE_\ell^{(d)}(\Omega), y\in\cE_{\ell,V}^{(d),\ovlp}(\Omega)),
  \end{align*}
  the function $x\mapsto g(x,y)$ is a.e.\ constant for a.e.~$y\in\cE_{\ell,V}^{(d),\ovlp}(\Omega)$.

  Without loss of generality, we may suppose that $[\ell]\subseteq U$. Furthermore, since the
  definition of $\pp^{\ovlp,\cN}_{\ell,V}$ implies that $g$ only depends on the coordinates of $y$ that
  are indexed by sets in $r(U)$, we may assume without loss of generality that $V =
  U\setminus[\ell]$.

  Recalling that overlapping types are factored up to almost everywhere equivalence, by Fubini's
  Theorem and Lebesgue's Differentiation Theorem, it suffices to show that
  \begin{align*}
    \frac{1}{\lambda^{(d)}(B'_\ell\times B_\ell)}\cdot
    \int_{B'_\ell\times B_\ell}
    \int_{B'_V\times B_V}
    g(x,y)
    \ d\lambda^{(d)}(y)
    \ d\lambda^{(d)}(x)
    & =
    \int_{\cE_\ell^{(d)}(\Omega)}
    \int_{B'_V\times B_V}
    g(x,y)
    \ d\lambda^{(d)}(y)
    \ d\lambda^{(d)}(x)
  \end{align*}
  for all positive measure sets
  \begin{align*}
    B'_\ell & \subseteq\cE_{\ell,\ell-1}^{(d)}(\Omega),\\
    B_\ell & \subseteq [0,1]\times\cO_\ell^d,\\
    B'_V & \subseteq \cE_{\ell,V,\ell-1}^{(d),\ovlp}(\Omega)
    \df\prod_{\substack{C\in r^{\ovlp}(\ell,V)\\\lvert C\rvert\leq\ell-1}} (X\times\cO_C^d),\\
    B_V & \subseteq \prod_{\substack{C\in r^{\ovlp}(\ell,V)\\\lvert C\rvert\geq\ell}} (X\times\cO_C^d).
  \end{align*}

  Since $\cN$ is $(\ell-1)$-independent, it follows that $g$ does not depend on the coordinates
  indexed by sets of size at most $\ell-1$, so it suffices to show only the case when
  $B'_\ell=\cE_{\ell,\ell-1}^{(d)}(\Omega)$ and $B'_V=\cE_{\ell,V,\ell-1}^{(d)}(\Omega)$:
  \begin{align*}
    & \!\!\!\!\!\!
    \frac{1}{\lambda^{(d)}(B_\ell)}\cdot
    \int_{\cE_{\ell,\ell-1}^{(d)}(\Omega)\times B_\ell}
    \int_{\cE_{\ell,V,\ell-1}^{(d),\ovlp}(\Omega)\times B_V}
    g(x,y)
    \ d\lambda^{(d)}(y)
    \ d\lambda^{(d)}(x)
    \\
    & =
    \int_{\cE_\ell^{(d)}(\Omega)}
    \int_{\cE_{\ell,V,\ell-1}^{(d),\ovlp}(\Omega)\times B_V}
    g(x,y)
    \ d\lambda^{(d)}(y)
    \ d\lambda^{(d)}(x),
  \end{align*}
  which in turn is equivalent to
  \begin{equation}\label{eq:typeuniqueness:objective}
    \begin{aligned}
      & \!\!\!\!\!\!
      \lambda^{(d)}(\{w\in\Tind(K,\cN)\mid \pi_\ell(w)\in B_\ell\land \pi_V(w)\in B_V\})
      \\
      & =
      \lambda^{(d)}(\{w\in\Tind(K,\cN)\mid\pi_V(w)\in B_V\})
      \cdot
      \lambda^{(d)}(B_\ell),
    \end{aligned}
  \end{equation}
  where $\pi_\ell\colon\cE_U^{(d)}(\Omega)\to [0,1]\times\cO_\ell^d$ is the projection onto the coordinate
  indexed by $[\ell]$ and $\pi_V$ is the projection onto the coordinates indexed by sets $C$ in
  $r^{\ovlp}(\ell,V)$ with $\lvert C\rvert\geq\ell$.

  Let $\cL'$ be obtained from $\cL$ by adding a predicate symbol $Q$ of arity $\ell$ and let $\cN'$
  be the $d$-Euclidean structure in $\cL'$ extending $\cN$ by letting
  \begin{align*}
    \cN'_Q & \df \cE_{\ell,\ell-1}^{(d)}(\Omega)\times B_\ell.
  \end{align*}
  Clearly $\phi_{\cN'}$ is a coupling of $\phi_\cN$ with the $(\ell-1)$-independent $d$-peon
  $\cH\df\cN'_Q$ of arity $\ell$, so $\UCouple[\ell]$ gives $\phi_{\cN'} = \phi_{\cN\otimes\cH}$.

  By Theorem~\ref{thm:dTU}, there exists a family $h=(h_i)_{i\in\NN_+}$ of symmetric
  measure-preserving on h.o.a.\ functions $h_i\colon\cE_i^{(2d)}(\Omega^4)\to[0,1]\times\cO_i^d$
  such that $h^*(\cN')$ is a.e.\ equal to the $2d$-$T$-on $\cG$ over $\Omega^4$ given by
  \begin{align*}
    \cG_P
    & \df
    \{(x,y,z)\in\cE_{k(P)}^{(d)}(\Omega)\times\cE_{k(P)}^{(d)}(\Omega)\times\cE_{k(P)}(\Omega^2)
    \mid (x,y)\in(\cN\otimes\cH)_P\}
    \qquad
    (P\in\cL).
  \end{align*}

  From the product structure of $\cN\otimes\cH$ and the definition of $\cN'$, we conclude that
  \begin{align}\label{eq:typeuniqueness:dTU}
    x\in\cN_P & \iff \widehat{h}_{k(P)}(x,y,z)\in\cN_P
  \end{align}
  for every $P\in\cL$ and almost every
  $(x,y,z)\in\cE_{k(P)}^{(d)}(\Omega)\times\cE_{k(P)}^{(d)}(\Omega)\times\cE_{k(P)}(\Omega^2)$ and
  \begin{align}\label{eq:typeuniqueness:dTU2}
    y\in\cH & \iff \widehat{h}_\ell(x,y,z)\in\cH
  \end{align}
  for almost every $(x,y,z)\in\cE_\ell^{(d)}(\Omega)\times\cE_\ell^{(d)}(\Omega)\times\cE_\ell(\Omega^2)$.

  To simplify notation, in all expressions from now on, we will take $x$ and $y$ in
  $\cE_U^{(d)}(\Omega)$ and take $z$ in $\cE_U(\Omega^2)$.
  
  By Lemma~\ref{lem:widehat}\ref{lem:widehat:measpres}, $\widehat{h}_U$ is measure-preserving, so we
  get
  \begin{equation}\label{eq:typeuniqueness:dTUapplication}
    \begin{aligned}
      & \!\!\!\!\!\!
      \lambda^{(d)}\Bigl(\bigl\{
      w\in\Tind(K,\cN)\mid \pi_\ell(w)\in B_\ell\land \pi_V(w)\in B_V
      \bigr\}\Bigr)
      \\
      & =
      \lambda^{(2d)}\biggl(\Bigl\{
      (x,y,z)
      \;\Big\vert\;
      \widehat{h}_U(x,y,z)\in\Tind(K,\cN)
      \land\pi_\ell\bigl(\widehat{h}_U(x,y,z)\bigr)\in B_\ell
      \land\pi_V\bigl(\widehat{h}_U(x,y,z)\bigr)\in B_V
      \Bigr\}\biggr)
      \\
      & =
      \lambda^{(2d)}\biggl(\Bigl\{
      (x,y,z)
      \;\Big\vert\;
      \widehat{h}_U(x,y,z)\in\Tind(K,\cN)
      \land\iota_{[\ell],U}^*\bigl(\widehat{h}_U(x,y,z)\bigr)\in\cH
      \land\pi_V\bigl(\widehat{h}_U(x,y,z)\bigr)\in B_V
      \Bigr\}\biggr)
      \\
      & =
      \lambda^{(2d)}\biggl(\Bigl\{
      (x,y,z)
      \;\Big\vert\;
      x\in\Tind(K,\cN)
      \land\iota_{[\ell],U}^*(y)\in\cH
      \land\pi_V\bigl(\widehat{h}_U(x,y,z)\bigr)\in B_V
      \Bigr\}\biggr),
    \end{aligned}
  \end{equation}
  where the last equality follows from~\eqref{eq:typeuniqueness:dTU}
  and~\eqref{eq:typeuniqueness:dTU2} (recall that $\iota_{[\ell],U}\colon[\ell]\to U$ is the
  inclusion map).

  Note now that since $\cH$ is $(\ell-1)$-independent, the condition $\iota_{[\ell],U}^*(y)\in\cH$ depends
  only on the coordinate of $y$ indexed by $[\ell]$. On the other hand, the definition of $\pi_V$
  implies that the condition $\pi_V(\widehat{h}_U(x,y,z))\in B_V$ depends only on coordinates of
  $\widehat{h}_U(x,y,z)$ indexed by sets in $r^{\ovlp}(\ell,V)$, which by the definition of
  $\widehat{h}_U$ means that the condition depends only on coordinates of $(x,y,z)$ indexed by sets
  in $r^{\ovlp}(\ell,V)\cup r(\ell,\ell-1)$, which does not contain the set $[\ell]$. Finally, since
  the condition $x\in\Tind(K,\cN)$ clearly does not depend on $y$, we conclude that
  \begin{align*}
    & \!\!\!\!\!\!
    \lambda^{(2d)}(\{(x,y,z)\mid
    x\in\Tind(K,\cN)\land\iota_{[\ell],U}^*(y)\in B_\ell
    \land \pi_V(\widehat{h}_U(x,y,z))\in B_V\}),
    \\
    & =
    \lambda^{(2d)}(\{(x,y,z)\mid
    x\in\Tind(K,\cN)\land\pi_V(\widehat{h}_U(x,y,z))\in B_V\})
    \cdot
    \lambda^{(2d)}(\{(x,y,z)\mid\iota_{[\ell],U}^*(y)\in B_\ell\}).
  \end{align*}

  Plugging this in~\eqref{eq:typeuniqueness:dTUapplication} and using the fact that $\widehat{h}_U$
  is measure-preserving again, we get~\eqref{eq:typeuniqueness:objective}, concluding the proof.
\end{proof}

\section{Topologies of Markov kernels and of types}
\label{sec:top}

Recall that our plan is to provide a way of sampling models from a $d$-Euclidean structure by
instead sampling types. This means that we need to make sense of the notion of a ``random type'',
which in turn requires putting a $\sigma$-algebra on the spaces $\cS^{(d),\ovlp}$ and
$\cS^{(d),\dsct}$. Furthermore, to finally produce an equivalent $d$-Euclidean structure that
corresponds to sampling types, we will need the resulting measurable space to be a standard Borel
space. With this in mind, the objective of this section is to construct Polish topologies on the
spaces of dissociated and overlapping types that have enough measurable sets to measure all that we
will need for the notion of ``random types''.

Recall that a Markov kernel from a measurable space $X$ to a measurable space $Y$ is a map
$\rho\colon X\to\cP(Y)$ such that for every measurable $B\subseteq Y$, the map $X\ni
x\mapsto\rho(x)(B)\in[0,1]$ is measurable. Note that both dissociated and overlapping types are
Markov kernels, so it will suffice to produce a Polish topology on the space of Markov kernels.

Given a Markov kernel $\rho\colon X\to\cP(Y)$ and a probability measure $\mu$ on $X$,
\Caratheodory's Extension Theorem implies that there exists a unique probability measure $\rho[\mu]$
on $X\times Y$ such that
\begin{align}\label{eq:rhomu}
  \rho[\mu](A\times B) & = \int_A \rho(x)(B)\ d\mu(x)
\end{align}
for all measurable sets $A\subseteq X$ and $B\subseteq Y$. In particular, if
$(\rn{x},\rn{y})\sim\rho[\mu]$, then
\begin{align*}
  \PP[\rn{y}\in B\given \rn{x}\in A] & = \frac{1}{\mu(A)}\cdot\int_A\rho(x)(B)\ d\mu(x)
\end{align*}
and more generally, if $h\colon Y\to\RR$ is bounded and measurable, then
\begin{align*}
  \EE[h(\rn{y})\given\rn{x}] & = \int_Y h(\rn{y})\ d\rho(\rn{x})(y)
\end{align*}
with probability $1$ (over $\rn{x}$).

Informally, $\rho(x)$ can be interpreted as the conditional probability of $\rn{y}$ given the
potentially zero measure event $\rn{x}=x$. This informal intuition is made formal by the
Disintegration Theorem below (for a proof, see e.g.~\cite[Theorem~10.6.6 and Corollary~10.6.7]{Bog07}).
\begin{theorem}[Disintegration Theorem]\label{thm:DT}
  Let $X$ and $Y$ be Borel spaces, let $\nu\in\cP(Y)$, let $\pi\colon Y\to X$ be a measurable
  function and let $\mu\df\pi_*(\nu)\in\cP(X)$ be the pushforward measure (i.e.,
  $\pi_*(\nu)(A)\df\nu(\pi^{-1}(A))$). Then there exists a Markov kernel $\rho\colon X\to\cP(Y)$
  such that the set
  \begin{align*}
    R & \df \{x\in X\mid \rho(x)(\pi^{-1}(x)) = 1\}
  \end{align*}
  has $\mu$-measure $1$ and for every measurable function $h\colon Y\to[0,\infty]$, we have
  \begin{align*}
    \int_Y h(y)\ d\nu(y)
    & =
    \int_R \int_{\pi^{-1}(x)} h(y) \ d\rho(x)(y)\ d\mu(x).
  \end{align*}

  In particular, the marginal of $\rho[\mu]$ on $Y$ is $\nu$, that is, we have
  \begin{align*}
    \nu(B) & = \rho[\mu](X\times B) = \int_X \rho(x)(B)\ d\mu(x)
  \end{align*}
  for every measurable $B\subseteq Y$.
\end{theorem}

When $Y$ is a product space and $\pi$ is a projection, the Disintegration Theorem takes the
following form:
\begin{theorem}[Disintegration Theorem, projection version]\label{thm:DTproj}
  Let $X$ and $Y$ be Borel spaces, let $\nu\in\cP(X\times Y)$, let $\pi\colon X\times Y\to X$ be the
  projection function and $\mu\df\pi_*(\nu)$ be the marginal of $\nu$ on $X$. Then there exists a
  Markov kernel $\rho\colon X\to\cP(Y)$ such that for every measurable function $h\colon X\times
  Y\to[0,\infty]$, we have
  \begin{align*}
    \int_{X\times Y} h(x,y)\ d\nu(x,y)
    & =
    \int_X \int_Y h(x,y) \ d\rho(x)(y)\ d\mu(x).
  \end{align*}

  In particular, $\rho[\mu]=\nu$.
\end{theorem}

\begin{proof}
  Let $\rho'\colon X\to\cP(X\times Y)$ be provided by Theorem~\ref{thm:DT}. Define $\rho\colon
  X\to\cP(Y)$ by
  \begin{align*}
    \rho(x)(B) & \df \rho'(x)(\{x\}\times B)
  \end{align*}
  whenever $x\in R$ and define $\rho(x)$ to be an arbitrary fixed probability measure on $Y$ when
  $x\notin R$. The definition of $R$ ensures that $\rho$ is indeed a Markov kernel.

  Note now that if $h\colon X\times Y\to[0,\infty]$ is a measurable function, then
  \begin{align*}
    \int_{X\times Y} h(x,y)\ d\nu(x,y)
    & =
    \int_R \int_{\pi^{-1}(x)} h(w,y) \ d\rho'(x)(w,y)\ d\mu(x)
    \\
    & =
    \int_R \int_{\pi^{-1}(x)} h(x,y) \ d\rho(x)(y)\ d\mu(x)
    \\
    & =
    \int_X \int_Y h(x,y) \ d\rho(x)(y)\ d\mu(x),
  \end{align*}
  where the second equality follows since when $x\in R$, then every $(w,y)\in\pi^{-1}(x)$ satisfies
  $x=w$ and $\rho(x)(B) = \rho'(x)(\{x\}\times B)$ for every $B\subseteq Y$ measurable and the third
  equality follows since when $x\in R$, the set $\pi^{-1}(x)$ has $\rho(x)$-measure $1$ and the set
  $R$ itself has $\mu$-measure $1$.

  Finally, note that for all measurable $A\subseteq X$ and $B\subseteq Y$, we have
  \begin{align*}
    \rho[\mu](A\times B)
    & =
    \int_A \rho(x)(B)\ d\mu(x)
    =
    \int_X \One[x\in A]\int_Y \One[y\in B] \ d\rho(x)(y)\ d\mu(x)
    \\
    & =
    \int_X \int_Y \One[(x,y)\in A\times B] \ d\rho(x)(y)\ d\mu(x)
    =
    \int_{X\times Y} \One[(x,y)\in A\times B]\ d\nu(x,y)
    \\
    & =
    \nu(A\times B),
  \end{align*}
  so $\rho[\mu]=\nu$.
\end{proof}

\begin{definition}\label{def:top}
  Let $\Omega=(X,\mu)$ be a standard probability space and let $Y$ be a Borel space. We denote by
  $M(\Omega,Y)$ the space of all Markov kernels from $\Omega$ to $Y$ up to almost everywhere
  equivalence.

  For measurable sets $A\subseteq X$ and $B\subseteq Y$ and a real $r\in\RR$, we let
  \begin{align*}
    U(A,B,r) & \df \{\rho\in M(\Omega,Y) \mid I_{A,B}(\rho) < r\},\\
    L(A,B,r) & \df \{\rho\in M(\Omega,Y) \mid I_{A,B}(\rho) > r\},
  \end{align*}
  where
  \begin{align*}
    I_{A,B}(\rho) & \df \int_A \rho(x)(B)\ d\mu(x).
  \end{align*}

  Let $\cB$ be a collection of measurable sets of $Y$. We let $\tau(\Omega,Y,\cB)$ be the topology
  on $M(\Omega,Y)$ generated by all sets of the form $U(A,B,r)$ or $L(A,B,r)$ for some measurable
  $A\subseteq X$, some real $r\in\RR$ and some $B\in\cB$.
\end{definition}

The next proposition proves basic properties about the topology on the space of Markov kernels.

\begin{proposition}\label{prop:top}
  Let $\Omega=(X,\mu)$ be a standard probability space, let $Y$ be a Polish space and let $\cB_X$
  and $\cB_Y$ be countable Boolean algebras of measurable sets of $X$ and $Y$, respectively that
  generate their $\sigma$-algebras. Let further $\cC$ be the collection of all sets of the form
  $U(A,B,r)$ or $L(A,B,r)$ for some $A\in\cB_X$, some $B\in\cB_Y$ and some rational $r\in\QQ$. Then
  the following hold.
  \begin{enumerate}
  \item\label{prop:top:gen} $\cC$ generates the topology $\tau(\Omega,Y,\cB_Y)$.
  \item\label{prop:top:sep} $\tau(\Omega,Y,\cB_Y)$ is separable.
  \item\label{prop:top:met} $\tau(\Omega,Y,\cB_Y)$ is metrizable.
  \item\label{prop:top:Pol} If $Y$ is compact, all sets in $\cB_Y$ are clopen and $X$ is either
    finite or uncountable, then $(M(\Omega,Y),\tau(\Omega,Y,\cB_Y))$ is a Polish space.
  \end{enumerate}
\end{proposition}

\begin{proof}
  To prove item~\ref{prop:top:gen}, it suffices to show that for all measurable sets $A\subseteq X$
  and $B\subseteq Y$ and all real numbers $r\in\RR$, the following hold:
  \begin{itemize}
  \item If $\rho\in U(A,B,r)$, then there exists $C\in\cC$ with $\rho\in C\subseteq U(A,B,r)$.
  \item If $\rho\in L(A,B,r)$, then there exists $C\in\cC$ with $\rho\in C\subseteq L(A,B,r)$.
  \end{itemize}

  We prove only the first item as the second has an analogous proof.

  Since $\rho\in U(A,B,r)$, we can let $r'\in\QQ$ be such that $I_{A,B}(\rho) < r' < r$ and let
  $\epsilon\df\min\{r'-I_{A,B}(\rho), r-r'\} > 0$.

  Since $\cB_X$ is a Boolean algebra that generates the $\sigma$-algebra of $X$, there exists
  $A'\in\cB_X$ such that $\mu(A\symdiff A')\leq\epsilon/2$. We claim that $C\df U(A',B,r')$
  satisfies the required property $\rho\in C\subseteq U(A,B,r)$. First note that since
  $\mu(A\symdiff A')\leq\epsilon/2$ and $0\leq\rho(x)(B)\leq 1$ for every $x\in X$, we get
  \begin{align*}
    I_{A',B}(\rho) & \leq I_{A,B}(\rho) + \frac{\epsilon}{2} < r',
  \end{align*}
  hence $\rho\in C$.

  On the other hand, if $\rho'\in C$, then
  \begin{align*}
    I_{A,B}(\rho') & \leq I_{A',B}(\rho') + \frac{\epsilon}{2} < r' + \frac{\epsilon}{2} < r,
  \end{align*}
  so $\rho'\in U(A,B,r)$, hence $C\subseteq U(A,B,r)$ as desired.

  Therefore, $\cC$ generates the topology $\tau(\Omega,Y,\cB_Y)$.

  \medskip

  For item~\ref{prop:top:sep}, since $\cC$ is clearly countable, it follows that
  $\tau(\Omega,Y,\cB_Y)$ is second-countable, which in particular implies it is separable.

  \medskip

  We now prove item~\ref{prop:top:met} by constructing an explicit metric that generates the
  topology $\tau(\Omega,Y,\cB_Y)$. For this, let $(A_n,B_n)_{n\in\NN}$ be an enumeration of
  $\cB_X\times\cB_Y$ and for $\rho_1,\rho_2\in M(\Omega,Y)$, we let
  \begin{align*}
    \Delta(\rho_1,\rho_2)
    & \df
    \sum_{n\in\NN} \frac{\lvert I_{A_n,B_n}(\rho_1) - I_{A_n,B_n}(\rho_2)\rvert}{2^n}.
  \end{align*}
  (Note that the sum above is convergent as $0\leq I_{A_n,B_n}(\rho_i)\leq\mu(A_n)\leq 1$.)
  
  Obviously $\Delta$ is a non-negative symmetric function and it is straightforward to check that
  $\Delta$ satisfies triangle inequality.

  Let us now show that if $\rho_1,\rho_2\in M(\Omega,Y)$ are such that $\Delta(\rho_1,\rho_2)=0$,
  then we must have $\rho_1=\rho_2$ a.e. Since $\cB_X$ is a Boolean algebra that generates the
  $\sigma$-algebra of $X$, the set
  \begin{align*}
    \{x\in X \mid \forall B\in\cB_Y, \rho_1(x)(B) = \rho_2(x)(B)\}
  \end{align*}
  must have $\mu$-measure $1$ by the uniqueness of \Caratheodory's Extension Theorem. In turn, by
  the same theorem but applied to $\cB_Y$, the set above is equal to
  \begin{align*}
    \{x\in X \mid \rho_1(x) = \rho_2(x)\},
  \end{align*}
  hence $\rho_1 = \rho_2$ almost everywhere.

  Therefore $\Delta$ is a metric on $M(\Omega,Y)$.

  \smallskip

  Let us now show that $\Delta$ generates the topology $\tau(\Omega,Y,\cB_Y)$. To show this, it
  suffices to show that the following hold:
  \begin{enumerate}[label={\alph*.}, ref={(\alph*)}]
  \item\label{it:top:met:Deltafiner} For every $C\in\tau(\Omega,Y,\cB_Y)$ and every $\rho\in C$,
    there exists $\epsilon > 0$ such that $B_\epsilon(\rho)\subseteq C$, where $B_\epsilon(C)$ is
    the ball of radius $\epsilon$ centered at $\rho$ (with respect to $\Delta$).
  \item\label{it:top:met:Deltacoarser} For every $\rho\in M(\Omega,Y)$ and every $\epsilon > 0$,
    there exists $C\in\tau(\Omega,Y,\cB_Y)$ such that $\rho\in C\subseteq B_\epsilon(\rho)$.
  \end{enumerate}

  We start with item~\ref{it:top:met:Deltafiner}. Since $\cC$ generates the topology
  $\tau(\Omega,Y,\cB_Y)$, it suffices to show only the case when $C\in\cC$, that is, it suffices to
  show only the case when $C$ is of the form $U(A,B,r)$ or $L(A,B,r)$ for some $A\in\cB_X$, some
  $B\in\cB_Y$ and some rational $r\in\QQ$. We show only the case $C = U(A,B,r)$ as the other case
  has an analogous proof.

  Since $(A,B)\in\cB_X\times\cB_Y$, we know that $(A,B)=(A_n,B_n)$ for some $n\in\NN$ and since
  $\rho\in U(A,B,r)$, we have $I_{A,B}(\rho) < r$. Let then $\epsilon\df 2^{-n}(r - I_{A,B}(\rho)) >
  0$ and note that if $\rho'\in B_\epsilon(\rho)$, then
  \begin{align*}
    I_{A,B}(\rho') & \leq I_{A,B}(\rho) + 2^n\cdot\Delta(\rho,\rho') < r,
  \end{align*}
  so $\rho'\in U(A,B,r)$, hence $B_\epsilon(\rho)\subseteq U(A,B,r)$ as desired.

  Let us now show item~\ref{it:top:met:Deltacoarser}. Let $n_0\in\NN$ be large enough so that
  $\sum_{n > n_0} 2^{-n} < \epsilon/2$ and consider the set
  \begin{align*}
    C & \df
    \bigcap_{n=0}^{n_0} U\left(A_n,B_n,I_{A_n,B_n}(\rho) + \frac{\epsilon}{4}\right)
    \cap
    \bigcap_{n=0}^{n_0} L\left(A_n,B_n,I_{A_n,B_n}(\rho) - \frac{\epsilon}{4}\right).
  \end{align*}
  Clearly $C\in\tau(\Omega,Y,\cB_Y)$ and $\rho\in C$.

  On the other hand, if $\rho'\in C$, then
  \begin{align*}
    \Delta(\rho,\rho')
    & =
    \sum_{n\in\NN}\frac{\lvert I_{A_n,B_n}(\rho) - I_{A_n,B_n}(\rho')\rvert}{2^n}
    \leq
    \frac{\epsilon}{4}\cdot\sum_{n=0}^{n_0} 2^{-n} + \sum_{n > n_0} 2^{-n}
    <
    \epsilon,
  \end{align*}
  hence $\rho'\in B_\epsilon(\rho)$, thus $C\subseteq B_\epsilon(\rho)$ as desired.

  Therefore $\Delta$ generates the topology $\tau(\Omega,Y,\cB_Y)$.

  \medskip

  It remains to show item~\ref{prop:top:Pol}. Since $X$ is either finite or uncountable and is a
  Borel space, we may put a Polish topology on $X$ that both generates the $\sigma$-algebra of
  $\Omega$ and is \emph{compact}. Furthermore, we may assume that $\cB_X$ is a countable Boolean
  algebra of \emph{clopen} sets of $X$ that generates the $\sigma$-algebra of $\Omega$ (in the
  finite case, any choice works as all subsets are clopen; in the uncountable case, this can be done
  as $\Omega$ is Borel-isomorphic to the Cantor space). Note that even though this changes the
  metric $\Delta$, it does not change the topology $\tau(\Omega,Y,\cB_Y)$ as it does not depend on
  $\cB_X$ nor on the topology on $X$ (and the new metric $\Delta$ still generates
  $\tau(\Omega,Y,\cB_Y)$ by our proof of item~\ref{prop:top:met}). With this choice, it now suffices
  to show that $\Delta$ is complete.

  Let $(\rho_n)_{n\in\NN}$ be a Cauchy sequence with respect to $\Delta$. It is straightforward to
  check that this implies that for every $A\in\cB_X$ and every $B\in\cB_Y$, the limit
  $\lim_{n\to\infty} I_{A,B}(\rho_n)$ exists. For each $n\in\NN$, let $\theta_n\df\rho_n[\mu]$ so
  that for every $A\in\cB_X$ and every $B\in\cB_Y$, we have
  \begin{align*}
    \theta_n(A\times B) & = \int_A \rho_n(x)(B)\ d\mu(x).
  \end{align*}

  Let $\cD$ be the (countable) Boolean algebra generated by sets of the form $A\times B$ for
  $A\in\cB_X$ and $B\in\cB_Y$. Clearly, $\cD$ generates the product $\sigma$-algebra on $X\times
  Y$. Furthermore, the product topology on $X\times Y$ is compact and all elements of $\cD$ are
  clopen.

  Since every set $C$ in $\cD$ can be written as a finite disjoint union $\bigcup_{i=1}^m A_i\times
  B_i$ with $A\in\cB_X$ and $B\in\cB_Y$, we can define
  \begin{align*}
    \theta(C) & \df \sum_{i=1}^m \lim_{n\to\infty} I_{A_i,B_i}(\rho_n)
  \end{align*}
  (it is straightforward to check that this does not depend on the particular choice of sets
  $(A_i,B_i)_{i=1}^m$).

  We claim that $\theta$ is a pre-measure on $\cD$. It is clear that $\theta$ is finitely additive and
  assigns measure $0$ to the empty set. To see that $\theta$ is also $\sigma$-additive, note that since
  elements of $\cC$ are both compact and open, any countably infinite disjoint union of elements of
  $\cC$ that is in $\cC$ must necessarily contain only finitely many non-empty sets, so
  $\sigma$-additivity follows from finite additivity.

  By \Caratheodory's Extension Theorem, we can uniquely extend $\theta$ to a measure on $X\times Y$
  (which we denote by $\theta$ as well). Note that for every $A\in\cB_X$, we have
  \begin{align*}
    \theta(A\times Y)
    & =
    \lim_{n\to\infty} \theta_n(A\times Y)
    =
    \lim_{n\to\infty}\int_A\rho_n(x)(Y)\ d\mu(x)
    =
    \mu(A),
  \end{align*}
  which in particular implies that the marginal of $\theta$ on $X$ is $\mu$ and that $\theta$ is a
  probability measure on $X\times Y$.

  By the Disintegration Theorem, Theorem~\ref{thm:DTproj}, there exists a Markov kernel $\rho\in
  M(\Omega,Y)$ such that
  \begin{align*}
    \theta(A\times B) & = \int_A\rho(x)(B)\ d\mu(x)
  \end{align*}
  for every $A\in\cB_X$ and every $B\in\cB_Y$, which in particular implies that
  \begin{align*}
    \lim_{n\to\infty} \Delta(\rho_n,\rho)
    & =
    \lim_{n\to\infty} \sum_{m\in\NN}\frac{\lvert I_{A_m,B_m}(\rho_n) - I_{A_m,B_m}(\rho)\rvert}{2^m}
    \\
    & =
    \lim_{n\to\infty} \sum_{m\in\NN}
    \frac{\lvert \theta_n(A_m\times B_m) - \theta(A_m\times B_m)\rvert}{2^m}
    =
    0,
  \end{align*}
  where the last equality follows since the $m$th term in the sum is bounded by $2^{-m}$ and
  $\lim_{n\to\infty}\theta_n(A\times B) = \theta(A\times B)$ for every $A\in\cB_X$ and every
  $B\in\cB_Y$. Thus, $\Delta$ is complete as desired.
\end{proof}

We are now ready to describe the topology and $\sigma$-algebra of the spaces of dissociated and
overlapping types, which in turn allows us to make sense of the notion of ``random type''.

\begin{definition}
  Let $\Omega=(X,\mu)$ be an atomless standard probability space, let $d\in\NN$, let $\cL$ be a
  finite relational language, let $V$ be a countable set, let $A$ be a finite set disjoint from $V$.

  Consider the space $M(\cE_{A,V}^{(d),\dsct},\cK_{A\cup V})$ of Markov kernels from
  $\cE_{A,V}^{(d),\dsct}$ to $\cK_{A\cup V}$ equipped with the topology
  $\tau(\cE_{A,V}^{(d),\dsct},\cK_{A\cup V},\cB_{A\cup V})$ (see Definition~\ref{def:top}), where
  $\cB_{A\cup V}$ is the (countable) Boolean algebra generated by all cylinder sets of $\cK_{A\cup
    V}$ (see~\eqref{eq:cylinderset}). We equip the set of dissociated types
  $\cS_{A,V}^{(d),\dsct}\subseteq M(\cE_{A,V}^{(d),\dsct},\cK_{A\cup V})$ with the subspace topology
  and with the Borel $\sigma$-algebra.

  Given a random element $\rn{p}$ of $\cS_{A,V}^{(d),\dsct}$, a measurable set
  $E\subseteq\cE_{A,V}^{(d),\dsct}$, a finite set $U\subseteq A\cup V$ and $K\in\cK_U$, we define
  the shorthand notation
  \begin{align*}
    \EE[\rn{p}[E]\rest_U = K]
    & \df
    \EE\left[\int_E\rn{p}(x)(C(U,K))\ d\mu^{(d)}(x)\right],
  \end{align*}
  where $C(U,K)\df\{M\in\cK_{A\cup V}\mid M\rest_U=K\}$ is the cylinder set as
  in~\eqref{eq:cylinderset} (note that this is well-defined as the function
  $\cS_{A,V}^{(d),\dsct}\ni p\mapsto\int_E p(x)(C(U,K))\ d\mu^{(d)}(x)\in\RR$ is continuous, hence
  measurable).

  We define the analogous notions for overlapping types similarly.
\end{definition}

The next proposition collects all topological properties of dissociated and overlapping types that
we will need. A direct consequence of it is that all maps of Lemma~\ref{lem:equivpp} are measurable.

\begin{proposition}\label{prop:toptypes}
  Let $\Omega=(X,\mu)$ be an atomless standard probability space, let $d\in\NN$, let $\cL$ be a
  finite relational language, let $V$ be a countable set, let $A$ be a finite set disjoint from
  $V$. Then the following hold.
  \begin{enumerate}
  \item\label{prop:toptypes:closed} The sets $\cS_{A,V}^{(d),\dsct}$ and $\cS_{A,V}^{(d),\ovlp}$ are
    closed subsets of $M(\cE_{A,V}^{(d),\dsct},\cK_{A\cup V})$ and
    $M(\cE_{A,V}^{(d),\ovlp},\cK_{A\cup V})$ in the topologies
    $\tau(\cE_{A,V}^{(d),\dsct},\cK_{A\cup V},\cB_{A\cup V})$ and
    $\tau(\cE_{A,V}^{(d),\ovlp},\cK_{A\cup V},\cB_{A\cup V})$, respectively.
  \item\label{prop:toptypes:Pol} $\cS_{A,V}^{(d),\dsct}$ and $\cS_{A,V}^{(d),\ovlp}$ are Polish
    spaces, hence standard Borel spaces.
  \item\label{prop:toptypes:M} The maps $M^{\dsct}_{A,V}\colon\cS_{A,V}^{(d),\dsct}\to\cK_A$ and
    $M^{\ovlp}_{A,V}\colon\cS_{A,V}^{(d),\ovlp}\to\cK_A$ are continuous (when $\cK_A$ is equipped
    with discrete topology).
  \item\label{prop:toptypes:pi} The map
    $\pi^{\ovlp}_{A,V}\colon\cS_{A,V}^{(d),\ovlp}\to\cS_{A,V}^{(d),\dsct}$ is continuous.
  \item\label{prop:toptypes:alpha} If $B$ is a finite set disjoint from $V$ and $\alpha\colon A\to
    B$ is an injection, then the maps $\alpha^*\colon\cS_{B,V}^{(d),\dsct}\to\cS_{A,V}^{(d),\dsct}$
    and $\alpha^*\colon\cS_{B,V}^{(d),\ovlp}\to\cS_{A,V}^{(d),\ovlp}$ are continuous.
  \item\label{prop:toptypes:pp} If $\cN$ is a $d$-Euclidean structure in $\cL$ over $\Omega$, then
    the maps $\pp^{\dsct,\cN}_{A,V}$ and $\pp^{\ovlp,\cN}_{A,V}$ are measurable.
  \end{enumerate}
\end{proposition}

\begin{proof}
  For item~\ref{prop:toptypes:closed}, note that
  \begin{align*}
    \cS_{A,V}^{(d),\dsct}
    & =
    \bigl\{p\in M(\cE_{A,V}^{(d),\dsct},\cK_{A\cup V}) \mathrel{\big\vert}
    \exists K\in\cK_A, \forall x\in\cE_{A,V}^{(d),\dsct},
    p(x)(\{K'\in\cK_{A\cup V}\mid K'\rest_A=K\})=1
    \bigr\}
    \\
    & =
    \left.
    M(\cE_{A,V}^{(d),\dsct},\cK_{A\cup V})
    \middle\backslash
    \bigcap_{K\in\cK_A} U\bigl(\cE_{A,V}^{(d),\dsct}, C(A,K), 1\bigr)
    \right.
  \end{align*}
  and since the last intersection is finite (as $\cK_A$ is finite since $\cL$ is finite), the set
  above is closed.

  The proof that $\cS_{A,V}^{(d),\ovlp}$ is also closed is analogous.

  \medskip

  For item~\ref{prop:toptypes:Pol}, first we equip $\cK_{A\cup V}$ with the topology generated by
  cylinder sets so that its $\sigma$-algebra is simply the Borel $\sigma$-algebra. Note that since
  $\cL$ is finite, all cylinder sets are clopen and $\cK_{A\cup V}$ is compact. Also,
  $\cE_{A,V}^{(d),\dsct}$ is either finite (e.g., when $V = \varnothing$) or uncountable, so
  Proposition~\ref{prop:top}\ref{prop:top:Pol} gives that $(M(\cE_{A,V}^{(d),\dsct},\cK_{A\cup V}),
  \tau(\cE_{A,V}^{(d),\dsct},\cK_{A\cup V},\cB_{A\cup V}))$ is a Polish space. Since
  $\cS_{A,V}^{(d),\dsct}$ is a closed subspace (with subspace topology), it follows that it is also
  a Polish space.

  The proof that $\cS_{A,V}^{(d),\ovlp}$ is also a Polish space is analogous.

  \medskip

  For item~\ref{prop:toptypes:M}, continuity of the map
  $M^{\dsct}_{A,V}\colon\cS_{A,V}^{(d),\dsct}\to\cK_A$ follows since for every $K\in\cK_A$, the
  preimage
  \begin{align*}
    (M^{\dsct}_{A,V})^{-1}(\{K\})
    & =
    \bigl\{p\in M(\cE_{A,V}^{(d),\dsct},\cK_{A\cup V}) \mathrel{\big\vert}
    \forall x\in\cE_{A,V}^{(d)},
    p(x)(\{K'\in\cK_{A\cup V}\mid K'\rest_A=K\})=1\bigr\}
    \\
    & =
    \cS_{A,V}^{(d),\dsct}\setminus U(\cE_{A,V}^{(d),\dsct}, C(A,K), 1)
  \end{align*}
  is a closed set. The proof of continuity of $M^{\ovlp}_{A,V}$ is analogous.

  \medskip

  For item~\ref{prop:toptypes:pi}, continuity of the map
  $\pi^{\ovlp}_{A,V}\colon\cS_{A,V}^{(d),\ovlp}\to\cS_{A,V}^{(d),\dsct}$ follows since for every
  measurable $E\subseteq\cE_{A,V}^{(d),\dsct}$, every finite set $W\subseteq A\cup V$, every
  $K\in\cK_W$ and every $r\in\RR$, we have
  \begin{align*}
    (\pi^{\ovlp}_{A,V})^{-1}\Bigl(U\bigl(E, C(W,K), r\bigr)\cap\cS_{A,V}^{(d),\dsct}\Bigr)
    & =
    U\bigl(E\times\cR, C(W,K), r\bigr)\cap\cS_{A,V}^{(d),\ovlp},
    \\
    (\pi^{\ovlp}_{A,V})^{-1}\Bigl(L\bigl(E, C(W,K), r\bigr)\cap\cS_{A,V}^{(d),\dsct}\Bigr)
    & =
    L\bigl(E\times\cR, C(W,K), r\bigr)\cap\cS_{A,V}^{(d),\ovlp},
  \end{align*}
  where
  \begin{align*}
    \cR & \df \prod_{C\in r^{\ovlp}(A,V)\setminus r^{\dsct}(A,V)}(X\times\cO_C^d).
  \end{align*}
  Thus, the pre-images above are open (relative to $\cS_{A,V}^{(d),\ovlp}$).

  \medskip

  For item~\ref{prop:toptypes:alpha}, to show continuity of the map
  $\alpha^*\colon\cS_{B,V}^{(d),\dsct}\to\cS_{A,V}^{(d),\dsct}$, let
  \begin{align*}
    Y & \df \prod_{C\in r^{\dsct}(B,V)\setminus r^{\dsct}(\alpha(A),V)}(X\times\cO_C^d)
  \end{align*}
  and with a small abuse of notation, let us denote by $\mu^{(d)}$ the measure on $Y$ that is the
  product of several copies of $\mu$ with the appropriate copies of $\nu_C$.

  Note that for every measurable $E\subseteq\cE_{A,V}^{(d),\dsct}$, every finite set $W\subseteq
  A\cup V$, every $K\in\cK_W$ and every $r\in\RR$, the preimage
  \begin{align*}
    & \!\!\!\!\!\!
    (\alpha^*)^{-1}(U(E, C(W,K), r)\cap\cS_{A,V}^{(d),\dsct})
    \\
    & =
    \left\{p\in\cS_{B,V}^{(d),\dsct} \;\middle\vert\;
    \int_E \int_Y
    p((\overline{\alpha}^{-1})^*(x),y)(\{M\in\cK_{B\cup V} \mid \alpha\up_V^*(M)\rest_W = K\})
    \ d\mu^{(d)}(y)
    \ d\mu^{(d)}(x)
    < r
    \right\}
    \\
    & =
    U\bigl(
    (\overline{\alpha}^{-1})^*(E)\times Y, C(\alpha\up_V(W), (\alpha\up_V\down_W^{-1})^*(K)), r
    \bigr)\cap\cS_{A,V}^{(d),\dsct}
  \end{align*}
  is open (relative to $\cS_{A,V}^{(d),\dsct}$).

  A similar calculation shows that $(\alpha^*)^{-1}(L(E, C(W,K), r)\cap\cS_{A,V}^{(d),\dsct})$ is
  also open (relative to $\cS_{A,V}^{(d),\dsct}$).

  The proof of continuity of $\alpha^*\colon\cS_{B,V}^{(d),\ovlp}\to\cS_{A,V}^{(d),\ovlp}$ is analogous.

  \medskip

  For the final item~\ref{prop:toptypes:pp}, to show measurability of the map
  $\pp^{\dsct,\cN}_{A,V}$, we note that for every measurable $E\subseteq\cE_{A,V}^{(d),\dsct}$,
  every finite set $W\subseteq A\cup V$, every $K\in\cK_W$ and every $r\in\RR$, we have
  \begin{align*}
    (\pp^{\dsct,\cN}_{A,V})^{-1}(\cS_{A,V}^{(d),\dsct}\cap U(E,C(W,K),r))
    & =
    \left\{x\in\cE_A^{(d)} \;\middle\vert\;
    \int_E f^{\dsct,\cN}(x,y)(C(W,K))\ d\mu^{(d)}(y) < r
    \right\},
  \end{align*}
  which is measurable by Fubini's Theorem. A similar calculation shows that
  $(\pp^{\dsct,\cN}_{A,V})^{-1}(\cS_{A,V}^{(d),\dsct}\cap L(E,C(W,K),r))$ is measurable.

  The proof of measurability of $\pp^{\ovlp,\cN}_{A,V}$ is analogous.
\end{proof}

\section{Weak amalgamation of dissociated and overlapping types}
\label{sec:dsctovlp}

The objective of this section is to prove in Proposition~\ref{prop:weakamalgdsctovlp} a weak
amalgamation property of dissociated and overlapping types that (in particular) says that the
distribution of a dissociated $A$-type is completely determined by all overlapping $B$-types when
$B$ ranges in $r(A,\lvert A\rvert-1)$.

The next lemma provides a formula for the ``expected value'' of a random dissociated $A$-type in
terms of a Lebesgue Differentiation expression involving the overlapping $B$-type for a fixed
$B\subsetneq A$. The lemma below as stated only works when $\Omega=([0,1],\lambda)$ (but as
expected, we will be able to transfer its consequence to arbitrary $\Omega$ via a
measure-isomorphism modulo $0$).

\begin{lemma}\label{lem:randomtypeLDT}
  Let $\cN$ be a $d$-Euclidean structure in $\cL$ over $\Omega=([0,1],\lambda)$, let $A$ be a finite
  set, let $B\subsetneq A$, let $k\df\lvert B\rvert$, let $V$ be a countable set disjoint from $A$,
  let $U\subseteq A\cup V$ be a finite set, let $K\in\cK_U$ and let
  $E\subseteq\cE_{A,V}^{(d),\dsct}$ be measurable.

  Let also $\rn{y}$ be picked at random in $\prod_{C\in\binom{A}{>k}} ([0,1]\times\cO_C^d)$
  according to $\bigotimes_{C\in\binom{A}{>k}} (\lambda\otimes\bigotimes_{i=1}^d\nu_C)$.

  For every $x\in\cE_{A,k}^{(d)}$, decompose its coordinates as $x = (x^B,
  \widehat{x}, {\lhd})$, where
  \begin{align*}
    x^B & \in \cE_B^{(d)}, &
    \widehat{x} & \in [0,1]^{r(A,k)\setminus r(B)}, &
    {\lhd} & \in \prod_{C\in r(A,k)\setminus r(B)}\cO_C^d.
  \end{align*}
  Define also
  \begin{align*}
    \rn{p}_x & \df \pp^{\dsct,\cN}_{A,V}(x,\rn{y}).
  \end{align*}

  Then for $\lambda^{(d)}$-almost every $x\in\cE_{A,k}^{(d)}$, we have
  \begin{align*}
    & \!\!\!\!\!\!
    \EE[\rn{p}_x[E]\rest_U = K]
    \\
    & =
    \lim_{r\to 0}\frac{1}{\lambda(B_r(\widehat{x}))}\cdot
    \int_{B_r(\widehat{x})}\int_{E_B} 
    \pp^{\ovlp,\cN}_{B,V_B}(x^B)(z,t,{\lhd})(\{M\in\cK_{A\cup V} \mid M\rest_U = K\})
    \ d\lambda^{(d)}(z)
    \ d\lambda(t),
  \end{align*}
  where
  \begin{align*}
    V_B
    & \df
    V\cup (A\setminus B),
    \\
    E_B
    & \df
    E\times\prod_{C\in r^{\ovlp}(B,V_B)\setminus (r(A,k)\cup r^{\dsct}(A,V))} ([0,1]\times\cO_C^d)
  \end{align*}
  and $B_r(\widehat{x})$ is the $\ell^\infty$-ball of radius $r$ centered at $\widehat{x}$ (and by
  abuse of notation, the measure $\lambda^{(d)}$ over $E_B$ denotes the product of several copies of
  $\lambda$ with the appropriate copies of $\nu_C$).
\end{lemma}

\begin{proof}
  Let
  \begin{align*}
    R
    & \df
    r(A\cup V)\setminus (r(A,k)\cup r^{\dsct}(A,V))
    \\
    & =
    \{C\subseteq A\cup V \mid A\cap C\neq\varnothing\land (C\subseteq A\to\lvert C\rvert > k)\},
    \\
    R_B
    & \df
    r^{\ovlp}(B,V_B)\setminus (r(A,k)\cup r^{\dsct}(A,V))
    \\
    & =
    \{C\in A\cup V \mid
    (B = \varnothing\lor (C\cap V_B\neq\varnothing\land B\not\subseteq C))
    \land A\cap C\neq\varnothing
    \land (C\subseteq A\to\lvert C\rvert>k)\},
  \end{align*}
  so that $E_B = E\times\prod_{C\in R_B} ([0,1]\times\cO_C^d)$. Note also that $R_B\subseteq R$ and
  \begin{align*}
    R\setminus R_B
    & =
    \{C\subseteq A\cup V \mid
    A\cap C\neq\varnothing
    \land (C\subseteq A\to\lvert C\rvert > k)
    \land B\neq\varnothing
    \land (C\cap V_B=\varnothing\lor B\subseteq C)\}
    \\
    & =
    \{C\subseteq A\cup V \mid
    A\cap C\neq\varnothing
    \land (C\subseteq A\to\lvert C\rvert > k)
    \land \varnothing\neq B\subseteq C\}
    \\
    & =
    \{C\in r(A\cup V) \mid \varnothing\neq B\subsetneq C\}
    =
    \overline{r}^{\ovlp}(B,V_B),
  \end{align*}
  where the second equality follows since $\lvert B\rvert=k$ makes the condition $C\cap V_B=\varnothing$
  incompatible with $C\subseteq A\to\lvert C\rvert < k$, so it can be dropped.
  
  With a small abuse of notation, whenever we have a product of several copies of $[0,1]$ with
  several copies of $\cO_C^d$, let us also denote by $\lambda^{(d)}$ the measure that is the product
  of the appropriate number of copies of $\lambda$ with the appropriate copies of $\nu_C$.

  Then by Fubini's Theorem, we have
  \begin{align*}
    & \!\!\!\!\!\!
    \EE[\rn{p}_x[E]\rest_U = K]
    \\
    & =
    \int_{\prod_{C\in\binom{A}{>k}}([0,1]\times\cO_C^d)}
    \int_E
    \pp^{\dsct,\cN}_{A,V}(x,y)(u)(\{M\in\cK_{A\cup V} \mid M\rest_U=K\})
    \ d\lambda^{(d)}(u)
    \ d\lambda^{(d)}(y)
    \\
    & =
    \int_E \int_{\prod_{C\in R}([0,1]\times\cO_C^d)}
    \One[M^\cN_{A\cup V}(x,u,v)\rest_U = K]
    \ d\lambda^{(d)}(v)
    \ d\lambda^{(d)}(u)
    \\
    & =
    \int_{E_B} \int_{\prod_{C\in R\setminus R_B} ([0,1]\times\cO_C^d)}
    \One[M^\cN_{A\cup V}(x,z,w)\rest_U = K]
    \ d\lambda^{(d)}(w)
    \ d\lambda^{(d)}(z).
  \end{align*}

  Recalling now the decomposition $x = (x^B,\widehat{x},{\lhd})$, where $x^B$ contains the
  coordinates indexed by sets in $r(B)$ and $(\widehat{x},{\lhd})$ the coordinates indexed by sets
  in $r(A,k)\setminus r(B)$, since the latter set is a subset of $r^{\ovlp}(B,V_B)$ and $R\setminus
  R_B=\overline{r}^{\ovlp}(B,V_B)$, we get
  \begin{align*}
    \EE[\rn{p}_x[E]\rest_U = K]
    & =
    \int_{E_B}
    \pp^{\ovlp,\cN}_{B,V_B}(x^B)(z,\widehat{x},{\lhd})
    \ d\lambda^{(d)}(z).
  \end{align*}

  By the Lebesgue Differentiation Theorem (and Fubini's Theorem), we get
  \begin{align*}
    \EE[\rn{p}_x[E]\rest_U = K]
    & =
    \lim_{r\to 0}\frac{1}{\lambda(B_r(\widehat{x}))}\cdot
    \int_{B_r(\widehat{x})}\int_{E_B}
    \pp^{\ovlp,\cN}_{B,V_B}(x^B)(z,t,{\lhd})
    \ d\lambda^{(d)}(z)
    \ d\lambda(t),
  \end{align*}
  for $\lambda^{(d)}$-almost every $x\in\cE_{A,k}^{(d)}$ as desired.
\end{proof}

The next lemma is a stepping stone of the weak amalgamation property we aim to prove in
Proposition~\ref{prop:weakamalgdsctovlp}.

\begin{lemma}\label{lem:weakamalgstep}
  Let $\cN$ be a $d$-Euclidean structure in $\cL$ over $\Omega=(X,\mu)$, let $A$ be a finite set,
  let $B\subsetneq A$, let $k\df\lvert B\rvert$ and let $V$ be a countably infinite set disjoint
  from $A$.

  Let also $\rn{y}$ be picked at random in $\prod_{C\in\binom{A}{>k}} ([0,1]\times\cO_C^d)$
  according to $\bigotimes_{C\in\binom{A}{>k}} (\lambda\otimes\bigotimes_{i=1}^d\nu_C)$ and let
  \begin{align*}
    R_B & \df r(A,k)\setminus r(B).
  \end{align*}

  Then there exists a Markov kernel
  \begin{align*}
    \rho\colon\cS_{B,V}^{(d),\ovlp}\times\prod_{C\in R_B}(X\times\cO_C^d)\to\cP(\cS_{A,V}^{(d),\dsct})
  \end{align*}
  such that for $\mu^{(d)}$-almost every $x\in\cE_{A,k}^{(d)}$, we have
  \begin{align}\label{eq:weakamalgstep}
    \pp^{\dsct,\cN}_{A,V}(x,\rn{y})
    & \sim
    \rho(\pp^{\ovlp,\cN}_{B,V}(x^B), x^{R_B}),
  \end{align}
  where $x^B$ and $x^{R_B}$ are the projections of $x$ onto the coordinates indexed by $r(B)$ and
  $R_B$, respectively.

  In particular, if $\rn{x}$ is picked at random in $\cE_{A,k}^{(d)}$ according to
  $\mu^{(d)}$ independently from $\rn{y}$, then $\pp^{\dsct,\cN}_{A,V}(\rn{x},\rn{y})$ is
  conditionally independent from $\rn{x}^B$ given $\pp^{\ovlp,\cN}_{B,V}(\rn{x}^B)$ and $\rn{x}^{R_B}$.
\end{lemma}

\begin{proof}
  First note that the conditional independence statement follows from~\eqref{eq:weakamalgstep} as
  its left-hand side is precisely the conditional distribution of
  $\pp^{\dsct,\cN}_{A,V}(\rn{x},\rn{y})$ given $\rn{x}$ and the right-hand side gives an expression
  for the same conditional distribution depending only on $\pp^{\ovlp,\cN}_{B,V}(\rn{x}^B)$ and
  $\rn{x}^{R_B}$.

  \medskip
  
  Now we claim that it suffices to show the result when $\Omega=([0,1],\lambda)$. To see this how
  the result for an arbitrary $\Omega=(X,\mu)$ follows from the result for
  $\Omega'=([0,1],\lambda)$, let $G\colon\Omega\to\Omega'$ be a measure-isomorphism modulo $0$ and
  for every countable set $C$, let
  $\widetilde{G}_C\colon\cE_C^{(d)}(\Omega)\to\cE_C^{(d)}(\Omega')$ be the function that acts
  as $G$ in all coordinates that have the space $X$ and acts identically in all other
  coordinates. We define
  \begin{align*}
    \widetilde{G}_{A,V}^{\dsct}
    \df
    \widetilde{G}_V\colon
    &
    \cE_{A,V}^{(d),\dsct}(\Omega)\to\cE_{A,V}^{(d),\dsct}(\Omega'),
    \\
    \widetilde{G}_{B,V}^{\ovlp}\colon
    &
    \cE_{B,V}^{(d),\ovlp}(\Omega)\to\cE_{B,V}^{(d),\ovlp}(\Omega'),
  \end{align*}
  similarly.
  
  Let $\cH$ be the $d$-Euclidean structure in $\cL$ over $\Omega'$ given by
  \begin{align*}
    \cH_P & \df \widetilde{G}_{k(P)}^{-1}(\cN_P)
  \end{align*}
  and note that
  \begin{align*}
    \pp^{\dsct,\cN}_{A,V}(u)(v) & = \pp^{\dsct,\cH}_{A,V}(\widetilde{G}_A(u))(\widetilde{G}_{A,V}^{\dsct}(v)),
    \\
    \pp^{\ovlp,\cN}_{B,V}(u)(v) & = \pp^{\ovlp,\cH}_{B,V}(\widetilde{G}_B(u))(\widetilde{G}_{B,V}^{\ovlp}(v)),
  \end{align*}
  so the result for $\cN$ follows from the result from $\cH$ and the fact that all $\widetilde{G}$
  are measure-preserving and bijective (except for a zero measure set).

  \medskip

  Let us now prove the case $\Omega=([0,1],\lambda)$. Fix a countable Boolean algebra $\cB$ that
  generates the $\sigma$-algebra of $\cE_{A,V}^{(d)}$ and let $\cC$ be the Boolean algebra of
  subsets of $\cS_{A,V}^{(d),\dsct}$ generated by all sets of the form $\cS_{A,V}^{(d),\dsct}\cap
  U(E,C(W,K),r)$ or $\cS_{A,V}^{(d),\dsct}\cap L(E,C(W,K),r)$ for some $E\in\cB$, some $W\subseteq
  A\cup V$ finite, some $K\in\cK_W$ and some $r\in\QQ$, where $C(W,K)\df\{M\in\cK_{A\cup V}\mid
  M\rest_W=K\}$ are the cylinder sets of~\eqref{eq:cylinderset} and $U$ and $L$ are as
  Definition~\ref{def:top} of the topology of $\cS_{A,V}^{(d),\dsct}$.

  Clearly $\cC$ is countable and by Proposition~\ref{prop:top}\ref{prop:top:gen}, we know that $\cC$
  generates the $\sigma$-algebra of $\cS_{A,V}^{(d),\dsct}$. Since $\cS_{A,V}^{(d)}$ is a Polish
  space (by Proposition~\ref{prop:toptypes}\ref{prop:toptypes:Pol}), it follows that distributions
  on $\cS_{A,V}^{(d)}$ are completely determined by their values on $\cC$.

  For every $x\in\cE_{A,k}^{(d)}$, let $\rn{p}_x\df\pp^{\dsct,\cN}_{A,V}(x,\rn{y})$. Given further
  $E\in\cB$, a finite set $W\subseteq A\cup V$ and $K\in\cK_W$, let
  \begin{align*}
    \rn{Z}_{x,E,W,K}
    & \df
    I_{E,C(W,K)}(\rn{p}_x)
    =
    \int_E \rn{p}_x(z)(C(W,K))\ d\lambda^{(d)}(z)
  \end{align*}
  and note that
  \begin{align*}
    \rn{p}_x\in U(E,C(W,K),r) & \iff \rn{Z}_{x,E,W,K} < r, &
    \rn{p}_x\in L(E,C(W,K),r) & \iff \rn{Z}_{x,E,W,K} > r.
  \end{align*}

  Thus, the distribution of $\rn{p}_x$ is completely determined by the joint distribution of the
  random variables $(\rn{Z}_{x,E,W,K})_{E,W,K}$. In turn, since the $\rn{Z}_{x,E,W,K}$ have ranges
  contained in $[0,1]$ (and there are countably many triples $(E,W,K)$), their joint distribution is
  completely determined by their joint moments.

  Let then $E_1,\ldots,E_n\in\cB$, $W_1,\ldots,W_n\subseteq A\cup V$ be finite and $K_i\in\cK_{W_i}$
  ($i\in[n]$) and let us compute the moment
  \begin{align*}
    \EE\left[\prod_{i=1}^n \rn{Z}_{x,E_i,W_i,K_i}\right].
  \end{align*}

  For each $i\in[n]$, let $\pi_i\colon A\cup V\to A\cup (V\times[n])$ be given by
  \begin{align*}
    \pi_i(a) & \df
    \begin{dcases*}
      a, & if $a\in A$,\\
      (a,i), & if $a\in V$,
    \end{dcases*}
  \end{align*}
  and let
  \begin{align*}
    E & \df \bigcap_{i=1}^n (\pi_i^*)^{-1}(E_i) \subseteq\cE_{A,V\times[n]}^{(d)}.
  \end{align*}

  For each $i\in[n]$, we also let $W'_i\df\pi_i(W_i)$ and let
  $K'_i\df(\pi_i\down_{W_i}^{-1})^*(K_i)$, that is, $K'_i$ is the unique $\cL$-structure on $W'_i$
  such that $\pi_i\down_{W_i}$ is an isomorphism from $K_i$ to $K'_i$ (recall that
  $\pi_i\down_{W_i}$ is the bijection obtained from $\pi_i$ by restricting its domain to $W_i$ and
  its codomain to $\pi_i(W_i)$). We also define
  \begin{align*}
    \rn{q}_x & \df \pp^{\dsct,\cN}_{A,V\times[n]}(x,\rn{y}).
  \end{align*}

  By~\eqref{eq:ppalpha*x} in Lemma~\ref{lem:equivf} applied to $\pi_i$, since $\pi_i\down_A =
  \id_A$, we have
  \begin{equation}\label{eq:px->qx}
    \begin{aligned}
      \int_{E_i} \rn{p}_x(z)(C(W_i,K_i))\ d\lambda^{(d)}(z)
      & =
      \int_{(\pi_i^*)^{-1}(E_i)}
      \rn{p}_x(\pi_i^*(w))(C(W_i,K_i))
      \ d\lambda^{(d)}(w)
      \\
      & =
      \int_{(\pi_i^*)^{-1}(E_i)}
      \pi_i^*(\rn{q}_x(w))(C(W_i,K_i))
      \ d\lambda^{(d)}(w)
      \\
      & =
      \int_{(\pi_i^*)^{-1}(E_i)}
      \rn{q}_x(w)(C(W'_i,K'_i))
      \ d\lambda^{(d)}(w).
    \end{aligned}
  \end{equation}
  
  We now let
  \begin{align*}
    \widehat{W} & \df \bigcup_{i=1}^n W'_i, &
    \widehat{\cK} & \df \{K\in\cK_{\widehat{W}} \mid \forall i\in[n], K\rest_{W'_i}=K'_i\},
  \end{align*}
  and note that
  \begin{align*}
    \EE\left[\prod_{i=1}^n \rn{Z}_{x,E_i,W_i,K_i}\right]
    & =
    \EE\left[
      \prod_{i=1}^n\left(
      \int_{E_i} \rn{p}_x(z)(C(W_i,K_i))\ d\lambda^{(d)}(z)
      \right)
      \right]
    \\
    & =
    \EE\left[
      \prod_{i=1}^n\left(
      \int_{(\pi_i^*)^{-1}(E_i)}
      \rn{q}_x(w)(C(W'_i,K'_i))
      \ d\lambda^{(d)}(w)
      \right)
      \right]
    \\
    & =
    \sum_{K\in\widehat{\cK}}
    \EE\left[
      \int_E\rn{q}_x(w)(C(\widehat{W},K))\ d\lambda^{(d)}(w)
      \right]
    \\
    & =
    \sum_{K\in\widehat{\cK}}
    \EE[\rn{q}_x[E]\rest_{\widehat{W}} = K]
  \end{align*}
  where the second equality follows from~\eqref{eq:px->qx} and the third equality follows from
  linearity of expectation.

  Let us decompose $x^{R_B}=(\widehat{x},{\lhd})$, where all order variables are collected in ${\lhd}$,
  let $V_B$ and $E_B$ be as in Lemma~\ref{lem:randomtypeLDT} using $V\times[n]$ in place of $V$, so
  that applying the lemma to the above yields
  \begin{align*}
    \EE\left[\prod_{i=1}^n \rn{Z}_{x,E_i,W_i,K_i}\right]
    & =
    \sum_{K\in\widehat{\cK}}
    \lim_{r\to 0}\frac{1}{\lambda(B_r(\widehat{x}))}\cdot
    \int_{B_r(\widehat{x})}\int_{E_B}
    \pp^{\ovlp,\cN}_{B,V_B}(x^B)(z,t,{\lhd})(C(\widehat{W},K))
    \ d\lambda^{(d)}(z)
    \ d\lambda(t)
  \end{align*}
  for almost every $x\in\cE_{A,k}^{(d)}$.

  Note now that the expression above only depends on $x$ via the overlapping $(B,V_B)$-type
  $\pp^{\ovlp,\cN}_{B,V_B}(x^B)$ and the projection $x^{R_B}=(\widehat{x},{\lhd})$. Since
  Lemma~\ref{lem:equivf} says that the overlapping $(B,V_B)$-type $\pp^{\ovlp,\cN}_{B,V_B}(x^B)$ is
  completely determined by the overlapping $(B,V)$-type $\pp^{\ovlp,\cN}_{B,V}(x^B)$ (and there are
  only countably many joint moments of $(\rn{Z}_{x,E,W,K})_{E,W,K}$), it follows that there exists a
  Markov kernel
  \begin{align*}
    \rho'\colon
    \im(\pp^{\ovlp,\cN}_{B,V})\times\prod_{C\in R_B} ([0,1]\times\cO_C^d)
    \to
    \cP(\cS_{A,V}^{(d),\dsct})
  \end{align*}
  such that~\eqref{eq:weakamalgstep} holds for $\lambda^{(d)}$-almost every $x\in\cE_{A,k}^{(d)}$
  with $\rho'$ in place of $\rho$.

  One final technicality is that $\rho'$ is not defined when its first argument is not in
  $\im(\pp^{\ovlp,\cN}_{B,V})$ and we need to define the Markov kernel $\rho$ measurably. This can
  be done as follows: since by Proposition~\ref{prop:toptypes}\ref{prop:toptypes:pp}
  $\pp^{\ovlp,\cN}_{B,V}$ is measurable, its image is universally measurable (as it is an analytic
  set); in particular, it is measurable with respect to the completion of the pushforward measure
  $\theta\df(\pp^{\ovlp,\cN}_{B,V})_*(\lambda^{(d)})$, so we can find a Borel set
  $D\subseteq\im(\pp^{\ovlp,\cN}_{B,V})$ that differs from $\im(\pp^{\ovlp,\cN}_{B,V})$ only by a
  zero $\theta$-measure set. We can then let $\rho$ act as $\rho'$ in $D\times\prod_{C\in
    R_B}([0,1]\times\cO_C^d)$ and to be constant equal to any fixed distribution in its complement
  and~\eqref{eq:weakamalgstep} follows.
\end{proof}

Before we proceed, let us recall the following standard lemma from probability theory:

\begin{lemma}\label{lem:mk}
  Let $\Omega_X$, $\Omega_Y$ and $\Omega_Z$ be Borel spaces, let $\rn{X}$, $\rn{Y}$ and $\rn{Z}$ be
  random variables with values in $\Omega_X$, $\Omega_Y$ and $\Omega_Z$, respectively. Then the
  following are equivalent:
  \begin{enumerate}
  \item\label{lem:mk:ci} $\rn{X}$ is conditionally independent from $\rn{Y}$ given $\rn{Z}$.
  \item\label{lem:mk:product} There exists a Markov kernel $\rho\colon\Omega_Z\to\cP(\Omega_X)$ such
    that if $\rn{\widetilde{X}}$ is picked at random according to $\rho(\rn{Z})$ conditionally on
    $\rn{Z}$ and conditionally independently from $\rn{Y}$, then
    $(\rn{\widetilde{X}},\rn{Y},\rn{Z})\sim(\rn{X},\rn{Y},\rn{Z})$.
  \end{enumerate}
\end{lemma}

\begin{proof}
  The implication~\ref{lem:mk:product}$\implies$\ref{lem:mk:ci} is obvious as $\rn{\widetilde{X}}$
  and $\rn{Y}$ are conditionally independent given $\rn{Z}$.

  \medskip

  For the implication~\ref{lem:mk:ci}$\implies$\ref{lem:mk:product}, by the Disintegration Theorem,
  Theorem~\ref{thm:DTproj}, there exists a Markov kernel $\rho\colon\Omega_Z\to\cP(\Omega_X)$ such
  that for all measurable $A\subseteq\Omega_X$ and measurable $C\subseteq\Omega_Z$, we have
  \begin{align*}
    \PP[\rn{X}\in A, \rn{Z}\in C]
    & =
    \int_C \rho(z)(A)\ d\mu_Z(z),
  \end{align*}
  where $\mu_Z\in\cP(\Omega_Z)$ is the distribution of $\rn{Z}$. Thus, with probability $1$ we have
  \begin{align*}
    \PP[\rn{X}\in A\given\rn{Y},\rn{Z}]
    & =
    \PP[\rn{X}\in A\given\rn{Z}]
    =
    \rho(\rn{Z})(A)
    =
    \PP[\rn{\widetilde{X}}\in A\given\rn{Z}]
    =
    \PP[\rn{\widetilde{X}}\in A\given\rn{Y},\rn{Z}],
  \end{align*}
  where the first equality follows since $\rn{X}$ is conditionally independent from $\rn{Y}$ given
  $\rn{Z}$ and the last equality follows since $\rn{\widetilde{X}}$ is conditionally independent
  from $\rn{Y}$ given $\rn{Z}$. Thus we get
  $(\rn{\widetilde{X}},\rn{Y},\rn{Z})\sim(\rn{X},\rn{Y},\rn{Z})$ as desired.
\end{proof}

We can now prove the weak amalgamation property of dissociated and overlapping types.

\begin{proposition}\label{prop:weakamalgdsctovlp}
  Let $\cN$ be a $d$-Euclidean structure in $\cL$ over $\Omega=(X,\mu)$, let $A$ be a finite set and
  let $V$ be a countably infinite set disjoint from $A$. Let also $k\in\{0,\ldots,\lvert
  A\rvert-1\}$ and let $\rn{y}$ be picked at random in $\prod_{B\in\binom{A}{>k}} (X\times\cO_A^d)$
  according to $\bigotimes_{B\in\binom{A}{>k}}(\mu\otimes\bigotimes_{i=1}^d\nu_B)$.

  Then there exists a Markov kernel
  \begin{align*}
    \rho\colon\prod_{B\in r(A,k)}\cS_{B,V}^{(d),\ovlp}\to\cP(\cS_{A,V}^{(d),\dsct})
  \end{align*}
  such that for $\mu^{(d)}$-almost every $x\in\cE_{A,k}$, we have
  \begin{align}\label{eq:weakamalgdsctovlp}
    \pp^{\dsct,\cN}_{A,V}(x,\rn{y})
    & \sim
    \rho\biggl(
    \Bigl(\pp^{\ovlp,\cN}_{B,V}\bigl(\iota_{B,A}^*(x)\bigr)
    \mathrel{\Big\vert}
    B\in r(A,k)\Bigr)
    \biggr),
  \end{align}
  (recall that $\iota_{B,A}\colon B\to A$ is the inclusion map).

  In particular, if $\rn{x}$ is picked at random in $\cE_{A,k}^{(d)}$ according to
  $\mu^{(d)}$ independently from $\rn{y}$, then $\pp^{\dsct,\cN}_{A,V}(\rn{x},\rn{y})$ is
  conditionally independent from $\rn{x}$ given
  $(\pp^{\ovlp,\cN}_{B,V}(\iota_{B,A}(\rn{x}))\mid B\in r(A,k))$.
\end{proposition}

\begin{proof}
  First, let us show how~\eqref{eq:weakamalgdsctovlp} follows from the conditional independence
  statement and Lemma~\ref{lem:mk}.

  Let $\rn{p}\df\pp^{\dsct,\cN}_{A,V}(\rn{x},\rn{y})$ and let $\rn{q}$ be the random element of
  $\prod_{B\in\binom{A}{k}}\cS_{B,V}^{(d),\ovlp}$ given by
  \begin{align*}
    \rn{q}_B & \df \pp^{\ovlp,\cN}_{B,V}\bigl(\iota_{B,A}(\rn{x})\bigr)
    \qquad \left(B\in \binom{A}{k}\right).
  \end{align*}

  Since $\rn{p}$ is conditionally independent from $\rn{x}$ given $\rn{q}$, by
  Lemma~\ref{lem:mk:ci}, there exists a Markov kernel
  \begin{align*}
    \rho\colon\prod_{B\in\binom{A}{k}}\cS_{B,V}^{(d),\ovlp}\to\cP(\cS_{A,V}^{(d),\dsct})
  \end{align*}
  such that if $\rn{\widetilde{p}}$ is picked at random according to $\rho(\rn{q})$ conditionally on
  $\rn{q}$ and conditionally independently from $\rn{x}$, then
  $(\rn{\widetilde{p}},\rn{x},\rn{q})\sim(\rn{p},\rn{x},\rn{q})$. In particular this implies that
  with probability $1$, for every measurable $P\subseteq\cS_{A,V}^{(d),\dsct}$ we have
  \begin{align*}
    \PP[\pp^{\dsct,\cN}_{A,V}(\rn{x},\rn{y})\in A\given\rn{x}]
    & =
    \PP[\rn{p}\in P\given\rn{x}]
    =
    \PP[\rn{\widetilde{p}}\in P\given\rn{x}]
    \\
    & =
    \rho(\rn{q})(P)
    =
    \rho\left(
    \left(\pp^{\ovlp,\cN}_{B,V}\bigl(\iota_{B,A}(\rn{x})\bigr)\;\middle\vert\; B\in\binom{A}{k}\right)
    \right)(P),
  \end{align*}
  from which~\eqref{eq:weakamalgdsctovlp} follows.

  \medskip

  Let us now prove that $\rn{p}$ is conditionally independent from $\rn{x}$ given $\rn{q}$.

  For this, it suffices to show that for all measurable $P\subseteq\cS_{A,V}^{(d),\dsct}$ and all
  measurable $E_C\subseteq X\times\cO_C^d$ ($C\in r(A,k)$), we have
  \begin{align*}
    \PP[\rn{p}\in P\land\rn{x}\in E_{r(A,k)} \given \rn{q}]
    & =
    \PP[\rn{p}\in P\given\rn{q}]\cdot\PP[\rn{x}\in E_{r(A,k)}\given\rn{q}],
  \end{align*}
  where $E_{r(A,k)}\df\prod_{C\in r(A,k)} E_C$.

  For every $W\subseteq r(A,k)$, let
  \begin{align*}
    E_W & \df \prod_{C\in W} E_C
  \end{align*}
  (note that this is compatible with the notation $E_{r(A,k)}$) and let $\rn{x}_W$ be the projection
  of $\rn{x}$ onto the coordinates indexed by sets in $W$.

  Let us enumerate the elements of $\binom{A}{k}$ as $B_1,\ldots,B_m$ where $m\df\binom{\lvert
    A\rvert}{k}$ and for each $n\in[m]$, define inductively
  \begin{align*}
    W_n & \df \left.r(B_n)\middle\backslash\bigcup_{i=1}^{n-1} W_i\right..
  \end{align*}
  Note that the $W_n$ form a partition of $r(A,k)$. We also let $R_n\df r(A,k)\setminus r(B_n)$ for
  every $n\in[m]$ and note that for every $i\in\{n+1,\ldots,m\}$, we have $W_i\subseteq R_n$.

  Note also that for every $n\in[m]$, we have
  \begin{align*}
    & \!\!\!\!\!\!
    \PP[\rn{p}\in P\land\forall i\in\{n,\ldots,m\},\rn{x}_{W_i}\in E_{W_i}
      \given\rn{q}]
    \\
    & =
    \EE\bigl[
      \PP[
        \rn{p}\in P\land\rn{x}_{W_n}\in E_{W_n}
        \given
        \rn{q}_{B_n},\rn{x}_{R_n}
      ]
      \cdot
      \One[\forall i\in\{n+1,\ldots,m\},\rn{x}_{W_i}\in E_{W_i}]
      \given[\big]\rn{q}
      \bigr]
    \\
    & =
    \EE\bigl[
      \PP[
        \rn{p}\in P
        \given
        \rn{q}_{B_n},\rn{x}_{R_n}
      ]
      \cdot
      \PP[
        \rn{x}_{W_n}\in E_{W_n}
        \given
        \rn{q}_{B_n},\rn{x}_{R_n}
      ]
      \cdot
      \One[\forall i\in\{n+1,\ldots,m\},\rn{x}_{W_i}\in E_{W_i}]
      \given[\big]\rn{q}
      \bigr]
    \\
    & =
    \EE\bigl[
      \PP[
        \rn{p}\in P\land\forall i\in\{n+1,\ldots,m\},\rn{x}_{W_i}\in E_{W_i}
        \given
        \rn{q}_{B_n},\rn{x}_{R_n}
      ]
      \cdot
      \PP[
        \rn{x}_{W_n}\in E_{W_n}
        \given
        \rn{q}_{B_n}
      ]
      \given[\big]\rn{q}
      \bigr]
    \\
    & =
    \PP[\rn{p}\in P\land\forall i\in\{n+1,\ldots,m\},\rn{x}_{W_i}\in E_{W_i}
      \given\rn{q}]
    \cdot
    \PP[\rn{x}_{W_n}\given\rn{q}_{B_n}],
  \end{align*}
  where the first equality follows since for each $i\in\{n+1,\ldots,m\}$,
  $\rn{x}_{W_i}$ is $\rn{x}_{R_n}$-measurable (as $W_i\subseteq R_n$), the second equality follows
  since Lemma~\ref{lem:weakamalgstep} says that $\rn{p}$ is conditionally independent from
  $\rn{x}^{B_n}$ (hence also from $\rn{x}_{W_n}$) given $(\rn{q}_{B_n},\rn{x}_{R_n})$, the third
  equality follows since for each $i\in\{n+1,\ldots,m\}$, $\rn{x}_{W_i}$ is
  $\rn{x}_{R_n}$-measurable and since $\rn{x}_{W_n}$ is conditionally independent from
  $\rn{x}_{R_n}$ given $\rn{q}_{B_n}$ (this is because both $\rn{x}_{W_n}$ and $\rn{q}_{B_n}$ are
  $\rn{x}^{B_n}$-measurable and $\rn{x}^{B_n}$ is independent from $\rn{x}_{R_n}$).

  With a simple inductive argument, we conclude that
  \begin{align*}
    \PP[\rn{p}\in P\land\forall i\in[m], \rn{x}_{W_i}\in E_{W_i}\given\rn{q}]
    & =
    \PP[\rn{p}\in P\given\rn{q}]\cdot
    \prod_{n=1}^m \PP[\rn{x}_{W_n}\given\rn{q}_{B_n}]
  \end{align*}
  Applying the above with both $P=P$ and with $P=\cS_{A,V}^{(d),\dsct}$ then yields
  \begin{align*}
    \PP[\rn{p}\in P\land\forall i\in[m], \rn{x}_{W_i}\in E_{W_i}\given\rn{q}]
    & =
    \PP[\rn{p}\in P\given\rn{q}]\cdot
    \PP[\forall i\in[m], \rn{x}_{W_i}\in E_{W_i}\given\rn{q}]
    \\
    & =
    \PP[\rn{p}\in P\given\rn{q}]\cdot\PP[\rn{x}\in E_{r(A)}\given\rn{q}],
  \end{align*}
  concluding the proof.
\end{proof}

\section{Measure theory interlude}
\label{sec:meas}

Recall that our goal is to provide a way of sampling $M^\cN_V(\rn{x})\sim\mu^\cN_V$ by inductively
sampling types of larger and larger finite sets instead of sampling points in $\cE_V^{(d)}$. This
means that we would have liked to have a version of Proposition~\ref{prop:weakamalgdsctovlp} that
either allowed us to correctly sample the dissociated $A$-type given the dissociated $B$-types for
all $B\in\binom{A}{\lvert A\rvert-1}$, or that allowed us to sample the overlapping $A$-type given
the overlapping $B$-types for all $B\in\binom{A}{\lvert A\rvert-1}$. However,
Proposition~\ref{prop:weakamalgdsctovlp} is an awkward in-between: to correctly sample the
dissociated $A$-type, we need to know the overlapping $B$-types for all $B\in\binom{A}{\lvert
  A\rvert-1}$, which means that we cannot use Proposition~\ref{prop:weakamalgdsctovlp} inductively.

Instead, as mentioned in the introduction, we will define the notion of ``amalgamating type'' which
is something between the dissociated and overlapping type and that has the correct amalgamation
property: to correctly sample the amalgamating $A$-type, it suffices to know the amalgamating
$B$-types for all $B\in\binom{A}{\lvert A\rvert-1}$.

Intuitively, we want the amalgamating $A$-type $p$ to know the dissociated $A$-type and know a
recipe that is able to retrieve the overlapping $A$-type if it is given a point $x\in\cE_{A,\lvert
  A\rvert-1}^{(d)}$ whose amalgamating $B$-types are consistent with $p$. However, it is important
that the amalgamating $A$-type does not store too much extra information; namely, it is obvious that
the overlapping $A$-type has the property above (as $\pi^{\ovlp}_{A,V}$ recovers the dissociated
$A$-type), but it has too much information as it fails to amalgamate.

To make sense out of this, the definition of amalgamating $A$-types needs to be inductive and must
itself require a weak amalgamation property: to correctly sample the dissociated $A$-type, it
suffices to know the amalgamating $B$-types for all $B\in\binom{A}{\lvert A\rvert-1}$.

Let us assume for a second that we have an even stronger property: assume that if
$x,x'\in\cE_{A,\lvert A\rvert-1}^{(d)}$ have the same amalgamating $B$-types for all
$B\in\binom{A}{\lvert A\rvert-1}$, then there exists a measure-preserving function $f_{x,x'}\colon
X\times\cO_A^d\to X\times\cO_A^d$ such that
\begin{align*}
  \pp^{\dsct,\cN}_{A,V}(x,y)
  & =
  \pp^{\dsct,\cN}_{A,V}(x', f_{x,x'}(y))
\end{align*}
for almost every $y\in X\times\cO_A^d$. Intuitively, this means that we would like the amalgamating
$A$-type of some $(x,y)\in\cE_{A,\lvert A\rvert-1}^{(d)}\times (X\times\cO_A^d)$ to be
$\pp^{\dsct,\cN}_{A,V}(x,y)$ along with the function
\begin{align*}
  x'
  \mapsto
  \pp^{\ovlp,\cN}(x', f_{x,x'}(y))
\end{align*}
defined only when $x'\in\cE_{A,\lvert A\rvert-1}^{(d)}$ has the same amalgamating $B$-types for all
$B\in\binom{A}{\lvert A\rvert-1}$.

Here is where the measure-theoretic technicalities start creeping in:
\begin{enumerate}[label={\arabic*.}]
\item To ensure consistency of the definition, we would like the functions above to satisfy a
  cocycle condition of the form $f_{x,x'}\comp f_{x',x''} = f_{x,x''}$.
\item The weak amalgamation only ensures that the existence of measure-preserving functions
  $f_{x,x'}\colon X\times\cO_A^d\to (X\times\cO_A^d)\times(X\times\cO_A^d)$ such that
  \begin{align*}
    \pp^{\dsct,\cN}_{A,V}(x,y)
    & =
    \pp^{\dsct,\cN}_{A,V}(x', f_{x,x'}(y,y'))
  \end{align*}
  for almost every $(y,y')\in (X\times\cO_A^d)\times(X\times\cO_A^d)$, that is, we might actually need
  some dummy variable $y'$ to properly change one function into the other.

  In particular, this makes the cocycle condition of the previous item even more complicated.
\item When we eventually define the space $\cS_{A,V}^{(d),\amlg}$ of amalgamating $A$-types, we will
  also define a function $\pp^{\amlg,\cN}_{A,V}\colon\cE_A^{(d)}\to\cS_{A,V}^{(d),\amlg}$ that maps
  a point to its amalgamating $A$-type, a function
  $\pi^{\amlg}_{A,V}\colon\cS_{A,V}^{(d),\amlg}\to\cS_{A,V}^{(d),\dsct}$ that retrieves the dissociated
  $A$-type from the amalgamating $A$-type and a function $\qq^{\amlg,\cN}_{A,V}\colon\cE_{A,\lvert
    A\rvert-1}^{(d)}\times\cS_{A,V}^{(d),\amlg}\to\cS_{A,V}^{(d),\ovlp}$ that retrieves the
  overlapping $A$-type when given lower pointwise information in the sense
  \begin{align*}
    \pp^{\ovlp,\cN}_{A,V}(x,y) & = \qq^{\amlg,\cN}_{A,V}(x,\pp^{\amlg,\cN}_{A,V}(x,y)) \qquad
    (x\in\cE_{A,\lvert A\rvert-1}^{(d)}, y\in X\times\cO_A^d).
  \end{align*}
  Then we want to equip $\cS_{A,V}^{(d),\amlg}$ with a $\sigma$-algebra that turns it into a Borel
  space such that the functions $\pp^{\amlg,\cN}_{A,V}$, $\pi^{\amlg}_{A,V}$ and $\qq^{\amlg,\cN}_{A,V}$
  are measurable.

  In fact, we will also want to include the $\pp^{\amlg,\cN}_{A,V}$ and $\pi^{\amlg}_{A,V}$ in the
  commutative diagram of Lemma~\ref{lem:equivpp}. This means that will also need to define
  measurable contra-variant maps $\alpha^*\colon\cS^{(d),\amlg}_{B,V}\to\cS^{(d),\amlg}_{A,V}$ for
  all injections $\alpha\colon A\to B$, which in particular means that the amalgamating $A$-type
  must know all amalgamating $C$-types for $C\subseteq A$.
\end{enumerate}

To avoid all technicalities of having to make sense of a Polish topology over a space of functions
$f_{x,x'}$ satisfying a complicated cocycle condition, we will instead define the space of
amalgamating $A$-types $\cS_{A,V}^{(d),\amlg}$ in a more pragmatic way: intuitively, since
$\cS_{A,V}^{(d),\ovlp}$ is a Borel space, any sort of information that the amalgamating $A$-type
needs to know about the overlapping $A$-type that is not already present in the dissociated $A$-type
can be encoded by a point of $[0,1]$, so we can define inductively
\begin{align*}
  \cS_{A,V}^{(d),\amlg}
  & \df
  \left(\prod_{B\in r(A,\lvert A\rvert-1)}\cS_{B,V}^{(d),\amlg}\right)\times\cS_{A,V}^{(d),\dsct}\times[0,1].
\end{align*}

In this way, $\cS_{A,V}^{(d),\amlg}$ has the natural product topology (which is Polish), we can
retrieve all amalgamating $B$-types for $B\in r(A,\lvert A\rvert-1)$ measurably by simply projecting
to the first coordinate and we can retrieve the dissociated $A$-type by projecting to the second
coordinate (this gives us $\pi^{\amlg}_{A,V}$). We only need to be careful in defining the functions
$\pp^{\amlg,\cN}_{A,V}$ and $\qq^{\amlg,\cN}_{A,V}$ so that
\begin{align*}
  \pp^{\ovlp,\cN}_{A,V}(x,y) & = \qq^{\amlg,\cN}_{A,V}(x,\pp^{\amlg,\cN}_{A,V}(x,y)),
  \\
  \pp^{\dsct,\cN}_{A,V}(x,y) & = \pi^{\amlg}_{A,V}(\pp^{\amlg,\cN}_{A,V}(x,y))
\end{align*}
for almost every $(x,y)\in\cE_{A,\lvert A\rvert-1}^{(d)}\times(X\times\cO_A^d)$ and so that weak
amalgamation holds.

The product definition of $\cS_{A,V}^{(d),\amlg}$ is also very convenient to understand how
$\pp^{\amlg,\cN}_{A,V}$ has to be defined in order for the (strong) amalgamation property to
hold. To elaborate, fix $x\in\cE_{A,\lvert A\rvert-1}^{(d)}$, let $\rn{y}$ be picked at random in
$X\times\cO_A^d$ according to $\mu\otimes\bigotimes_{i=1}^d\nu_A$ and write
\begin{align*}
  \pp^{\amlg,\cN}_{A,V}(x,\rn{y}) & = (q_x, \rn{p}_x, \rn{t}_x),
\end{align*}
where $q_x$ contains the lower amalgamating types (which is not random as it is completely
determined by $x$), $\rn{p}_x$ is the dissociated $A$-type and $\rn{t}_x$ is the extra
information. Amalgamation then reduces to saying that if $x$ and $x'$ satisfy $q_x=q_{x'}$, then
$(\rn{p}_x,\rn{t}_x)\sim(\rn{p}_{x'},\rn{t}_{x'})$ (and weak amalgamation amounts only to the
marginal conclusion $\rn{p}_x\sim\rn{p}_{x'}$).

Since we have the freedom of choosing how the extra information is encoded by points of $[0,1]$, we
can be pragmatic again and insist that $\rn{p}_x$ is independent from $\rn{t}_x$ and $\rn{t}_x$ is
distributed uniformly in $[0,1]$. If we indeed can define $\pp^{\amlg,\cN}_{A,V}$ so as to ensure
that, then amalgamation will immediately follow from weak amalgamation, which in turn should follow
from the dissociated vs.\ overlapping amalgamation of Proposition~\ref{prop:weakamalgdsctovlp}.

Finally, we point out that to properly ensure weak amalgamation, we will still need to add ``dummy
variables''. This is because for amalgamation to hold, then it must be true that, except for a zero
measure change, if $x\in\cE_{A,\lvert A\rvert-1}^{(d)}$ and $p\in\cS_{A,V}^{(d),\dsct}$, then the
cardinality of the set
\begin{align*}
  \{\pp^{\ovlp,\cN}_{A,V}(x,y) \mid y\in X\times\cO_A^d\land\pp^{\dsct,\cN}_{A,V}(x,y)=p\}
\end{align*}
can only depend on $x$ via the amalgamating $B$-types for $B\in\binom{A}{\lvert A\rvert-1}$. If we
do not add dummy variables, this can fail terribly as Example~\ref{ex:semitwist} below shows.

\begin{example}\label{ex:semitwist}
  Inspired by $\cN^{\qr}$ and $\cN^{\twist}$ from Examples~\ref{ex:qr} and~\ref{ex:twist}, consider
  the $2$-hypergraphon $\cN^{\semitwist}$ given by
  \begin{align*}
    \cN^{\semitwist}_E
    \df
    \Biggl\{x\in\cE_2 \;\Bigg\vert\;
    &
    \left(\min\{x_{\{1\}},x_{\{2\}}\} < \frac{1}{2}\land x_{\{1,2\}} < \frac{1}{2}\right)
    \\
    & \lor
    \left(\min\{x_{\{1\}},x_{\{2\}}\}\geq \frac{1}{2}\land (x_{\{1\}} + x_{\{2\}} +
    x_{\{1,2\}})\bmod 1 < \frac{1}{2}\right)
    \Biggr\}.
  \end{align*}

  It is straightforward to check that $\cN^{\semitwist}$ has a unique overlapping/dissociated
  $1$-type. This can be done directly, or one can note that
  $\phi_{\cN^{\semitwist}}=\phi_{\cN^{\qr}}$, so $\cN^{\qr}$ satisfies $\UCouple[1]$ (and is
  obviously $0$-independent) hence uniqueness of overlapping/dissociated $1$-types follows from
  Proposition~\ref{prop:typeuniqueness}.

  We now consider $x,x'\in\cE_{2,1}$ given by
  \begin{align*}
    x_{\{1\}} & = x_{\{2\}} = \frac{1}{4}, &
    x'_{\{1\}} & = x'_{\{2\}} = \frac{3}{4}.
  \end{align*}
  
  Note that for every $y\in[0,1]$ and every $w\in\cE_{2,V}^{(d),\ovlp}$, we have that 
  $\pp^{\ovlp,\cN^{\semitwist}}_{2,V}(x,y)(w)$ is
  the distribution of the random graph $\rn{G}$ on $[2]\cup V$ such that:
  \begin{itemize}
  \item We have $\{1,2\}\in E(\rn{G})$ if and only if $y < 1/2$.
  \item If $\{u,v\}\in\binom{[2]\cup V}{2}$ is such that $\{u,v\}\cap V\neq\varnothing$ and
    $[2]\not\subseteq\{u,v\}$, then $\{u,v\}\in E(\rn{G})$ if and only if $w_{\{u,v\}}<1/2$.
  \item All other pairs are declared to be edges independently at random with
    probability $1/2$.
  \end{itemize}

  This in particular implies that the set
  \begin{align*}
    \{\pp^{\ovlp,\cN^{\semitwist}}_{2,V}(x,y) \mid y\in[0,1]\}
  \end{align*}
  has size $2$: one of the overlapping $2$-types happens when $y<1/2$ and the other when $y\geq 1/2$.

  On the other hand, note that if $y\in[0,1]$, $u\in V$ and $w\in\cE_{2,V}^{(d),\ovlp}$ is such that
  $w_{\{u\}}\geq 1/2$, then by sampling $\rn{G}$ according to
  $\pp^{\ovlp,\cN^{\semitwist}}_{2,V}(x',y)(w)$, we see that
  \begin{align*}
    \PP[\{1,u\}\in E(\rn{G})] & =
    \begin{dcases*}
      1, & if $(x'_{\{1\}} + w_{\{u\}} + w_{\{1,u\}})\bmod 1 < 1/2$,\\
      0, & otherwise,
    \end{dcases*}
    \\
    \PP[\{2,u\}\in E(\rn{G})] & =
    \begin{dcases*}
      1, & if $(x'_{\{2\}} + w_{\{u\}} + w_{\{2,u\}})\bmod 1 < 1/2$,\\
      0, & otherwise,
    \end{dcases*}
    \\
    \PP[\{1,2\}\in E(\rn{G})] & =
    \begin{dcases*}
      1, & if $(x'_{\{1\}} + x'_{\{2\}} + y)\bmod 1 < 1/2$,\\
      0, & otherwise,
    \end{dcases*}
  \end{align*}
  which means that whenever $y\neq y'$, we have
  $\pp^{\ovlp,\cN^{\semitwist}}_{2,V}(x',y)\neq\pp^{\ovlp,\cN^{\semitwist}}_{2,V}(x',y')$. In
  particular, the set
  \begin{align*}
    \{\pp^{\ovlp,\cN^{\semitwist}}_{2,V}(x',y) \mid y\in[0,1]\}
  \end{align*}
  has size continuum.
\end{example}

\medskip

The measure-theoretic formalization of the ideas necessary to make the construction of the
amalgamating types possible is done in Lemma~\ref{lem:keymeas} below. To ease notation, all lemmas
in this section are proved with single letters for the different objects, but for the reader's
convenience, we point out now how each of these will be used:
\begin{itemize}
\item $(X,\mu)$ will be the set of all lower pointwise data along with dummy variables:
  $\cE_{A,\lvert A\rvert-1}^{(d)}(\Omega^2)$ along with the measure $(\mu^2)^{(d)}$.
\item $Y$ will be the set of overlapping $A$-types: $\cS_{A,V}^{(d),\ovlp}$.
\item $(Z,\nu)$ will be an extra dummy variable (that is being added indexed by $A$): $(X,\mu)$.
\item $U$ will be the set of all lower amalgamating types: $\prod_{B\in r(A,\lvert
  A\rvert-1)}\cS_{B,V}^{(d),\amlg}$.
\item $V$ will be the set of dissociated $A$-types: $\cS_{A,V}^{(d),\dsct}$.
\item $W$ will be the set of extra information that the amalgamating $A$-type carries besides the
  dissociated $A$-type information: $[0,1]$.
\item $\gamma\colon\cE_{A,\lvert A\rvert-1}^{(d)}\to\cP(\cS_{A,V}^{(d),\ovlp})$ will be the Markov
  kernel that gives the conditional distribution of $\pp^{\ovlp,\cN}_{A,V}(\rn{x},\rn{y})$ given the
  lower pointwise data $\rn{x}$.
\item $\theta\colon\prod_{B\in r(A,\lvert
  A\rvert-1)}\cS_{B,V}^{(d),\amlg}\to\cP(\cS_{A,V}^{(d),\dsct})$ will be the Markov kernel that
  gives the conditional distribution of $\pp^{\dsct,\cN}_{A,V}(\rn{x},\rn{y})$ given the lower
  amalgamating types $(\pp^{\amlg,\cN}_{B,V}(\iota_{B,A}^*(\rn{x})))_{B\in\binom{A}{\lvert
      A\rvert-1}}$ (recall that $\iota_{B,A}\colon B\to A$ is the inclusion map); the existence of
  this is weak amalgamation and will be proven in Lemma~\ref{lem:amlg}\ref{lem:amlg:weak}.
\item $f$ will be the function that maps a point $x\in\cE_{A,\lvert A\rvert-1}^{(d)}$ to all the
  lower amalgamating types: $f(x)\df(\pp^{\amlg,\cN}_{B,V}(\iota_{B,A}^*(x)))_{B\in r(A,\lvert
    A\rvert-1)}$.
\item $g$ will be the function that takes the lower pointwise data along with the overlapping
  $A$-type $(x,p)\in\cE_{A,\lvert A\rvert-1}^{(d)}\times\cS_{A,V}^{(d),\dsct}$ and returns the
  dissociated $A$-type (ignoring the pointwise data): $g(x,p)\df\pi^{\ovlp}_{A,V}(p)$.
\end{itemize}

Lemma~\ref{lem:keymeas} will then provide us objects that will be used as follows:
\begin{itemize}
\item $G$ will be used to construct the function $\pp^{\amlg,\cN}_{A,V}$ that maps a point of
  $\cE_A^{(d)}$ to its amalgamating $A$-type.
\item $G'$ will be the function $\qq^{\amlg,\cN}_{A,V}$ that retrieves the overlapping $A$-type from
  the amalgamating $A$-type plus the lower pointwise data.
\item $\Theta$ will be the Markov kernel that gives the conditional distribution of the random
  amalgamating $A$-type $\pp^{\amlg,\cN}_{A,V}(\rn{x},\rn{y})$ given the lower amalgamating types
  $(\pp^{\amlg,\cN}_{B,V}(\iota_{B,A}^*(\rn{x})))_{B\in\binom{A}{\lvert A\rvert-1}}$ (the fact that this
  exists gives our desired (strong) amalgamation property).
\end{itemize}

Finally, we point out that for technical reasons, it will be convenient to work with partial
functions that are defined almost everywhere. We use the notation $f\colon A\pto B$ for such partial
functions and $\dom(f)\subseteq A$ for their domain. As usual, composition of partial functions
$f\colon A\pto B$ and $g\colon B\pto C$ is the partial function $g\comp f\colon A\pto C$ defined
only for $x\in f^{-1}(\dom(g))$ as $(g\comp f)(x)\df g(f(x))$.

By a measurable partial function, we mean a partial function $f$ such that $\dom(f)$ is measurable
and the restriction $f\rest_{\dom(f)}$ (which is a total function) is measurable. It is
straightforward to check that the Disintegration Theorem, Theorem~\ref{thm:DT} still holds if $\pi$
is a partial function as long as it is defined almost everywhere (i.e., $\nu(\dom(\pi))=1$): this is
done by extending $\pi$ arbitrarily (but measurably) to be defined everywhere.

By a partial Markov kernel from a Borel space $X$ to a Borel space $Y$, we mean a partial function
$\rho\colon X\pto\cP(Y)$ such that for every measurable $B\subseteq Y$, the partial function $X\ni
x\mapsto\rho(x)(B)\in[0,1]$ is measurable. If we are given further a probability measure $\mu$ on
$X$ such that $\mu(\dom(\rho))=1$, then it remains true that there is a unique probability measure
$\rho[\mu]$ such that
\begin{align*}
  \rho[\mu](A\times B) & = \int_A \rho(x)(B)\ d\mu(x)
\end{align*}
for all measurable sets $A\subseteq X$ and $B\subseteq Y$ (this is because the expression under the
integral is defined $\mu$-almost everywhere).

The next lemma will serve as an inductive step for the end goal of this section, Lemma~\ref{lem:keymeas}.

\begin{lemma}\label{lem:measstep}
  Let $(X,\mu)$, $(Z,\nu)$ and $(W,\eta)$ be atomless standard probability spaces, let $Y$ and $U$
  be Borel spaces with $Y$ uncountable, let $f\colon X\pto U$ and $g\colon X\times Y\pto V$ be
  measurable partial functions, let $\gamma\colon X\pto\cP(Y)$ and $\theta\colon U\pto\cP(V)$ be
  partial Markov kernels with $\mu(\dom(\gamma))=1$, let $A\subseteq\dom(f)\subseteq X$ and
  $B\subseteq\dom(g)\subseteq X\times Y$ be measurable sets with $\mu(A)=\gamma[\mu](B)=1$
  (see~\eqref{eq:rhomu}). Suppose further that $A\subseteq f^{-1}(\dom(\theta))\cap\dom(\gamma)$ and
  for every $x\in A$, we have
  \begin{align*}
    g(x,\place)_*(\gamma(x)) & = \theta(f(x)).
  \end{align*}
  Finally, let $\pi\colon V\times W\to V$ be the natural projection.

  Then there exist partial functions $G\colon X\times Y\times Z\pto V\times W$ and $G'\colon X\times
  V\times W\pto Y$ and measurable sets $A'\subseteq X$ and $B'\subseteq X\times Y\times Z$ such that
  letting $\Theta\colon U\pto\cP(V\times W)$ be the partial Markov kernel given by
  \begin{align*}
    \Theta(u) & \df \theta(u)\otimes\eta \qquad (u\in U),
  \end{align*}
  the following hold:
  \begin{enumerate}
  \item\label{lem:measstep:A'} $A'\subseteq A$.
  \item\label{lem:measstep:B'A'} $B'\subseteq ((A'\times Y)\cap B)\times Z$.
  \item\label{lem:measstep:B'G} $B'\subseteq\dom(G)$.
  \item\label{lem:measstep:G'G} For every $(x,y,z)\in B'$, we have $G'(x,G(x,y,z))=y$.
  \item\label{lem:measstep:piG} For every $(x,y,z)\in B'$, we have $\pi(G(x,y,z)) = g(x,y)$.
  \item\label{lem:measstep:gammamunuB'} $(\gamma[\mu]\otimes\nu)(B')=1$.
  \item\label{lem:measstep:muA'} $\mu(A')=1$.
  \item\label{lem:measstep:Ggammanu} For every $x\in A'$, we have
    $G(x,\place,\place)_*(\gamma(x)\otimes\nu)=\Theta(f(x))$.
  \end{enumerate}
\end{lemma}

\begin{proof}
  First we prove the case when $(X,\mu)=(Z,\nu)=(W,\eta)=([0,1],\lambda)$, where $\lambda$ is the Lebesgue
  measure and $Y = [0,1]$.

  Let $\phi\colon X\times Y\times Z\pto X\times V$ be given by
  \begin{align*}
    \phi(x,y,z) & \df (x,g(x,y)) \qquad (x\in X, y\in Y, z\in Z).
  \end{align*}
  (note that $\phi$ does not actually depend on the last coordinate $z\in Z$). Since $\dom(\phi) =
  \dom(g)\times Z\supseteq B$ and $\gamma[\mu](B)=1$, it follows that $\phi$ is defined
  $\gamma[\mu]\otimes\nu$-almost everywhere.

  We now apply Theorem~\ref{thm:DT} to $\phi$ to obtain a Markov kernel $\rho\colon X\times
  V\pto\cP(X\times Y\times Z)$ so that the set
  \begin{align*}
    R & \df \{x\in X\mid \rho(x)(\phi^{-1}(x)) = 1\},
  \end{align*}
  has $\sigma(R)=1$, where $\sigma=\phi_*(\gamma[\mu]\otimes\nu)$ and for every measurable $h\colon
  X\times Y\times Z\to[0,\infty]$, we have
  \begin{align}\label{eq:measstep:DT}
    \int_{X\times Y\times Z} h(s)\ d(\gamma[\mu]\otimes\nu)(s)
    & =
    \int_R \int_{\phi^{-1}(r)} h(s)\ d\rho(r)(s)\ d\sigma(r).
  \end{align}

  Define the partial function $G\colon X\times Y\times Z\pto V\times W$ by
  \begin{align}\label{eq:measstep:G}
    G(x,y,z)
    & \df
    \Bigl(g(x,y), \rho\bigl(\phi(x,y,z)\bigr)\bigl(L(x,y,z)\bigr)\Bigr)
    \qquad (x\in X, y\in Y, z\in Z),
  \end{align}
  where
  \begin{align*}
    L(x,y,z)
    & \df
    \{(x',y',z')\in X\times Y\times Z \mid
    \phi(x',y',z') = \phi(x,y,z)
    \land (y',z')\leq_L (y,z)\},
  \end{align*}
  and $\leq_L$ is the lexicographic order on $Y\times Z=[0,1]^2$ (note that the definition of $\phi$
  forces $x' = x$ in the above). It is obvious that $G$ is measurable and
  $\dom(G)=\dom(\phi)=\dom(g)\times Z\supseteq B\times Z$. Note also that for $(x,y,z)\in\dom(G)$,
  we have $\pi(G(x,y,z))=g(x,y)$ (recall $\pi\colon V\times W\to V$ is the natural projection).

  We now define a partial function $G'\colon X\times V\times W\pto Y$ by
  \begin{align}\label{eq:measstep:G'}
    G'(x,v,w)
    & \df
    \inf\{y\in Y \mid \phi(x,y,1) = (x,v)\land\rho(x,v)(L(x,y,1))\geq w\}
    \qquad (x\in X, v\in V, w\in W),
  \end{align}
  leaving $G'$ undefined if the set above is empty.

  Let us now prove some basic properties about these definitions.

  \begin{claim}\label{clm:measstep:domG'}
    We have $R\times [0,1)\subseteq\dom(G')$.
  \end{claim}

  \begin{proofof}{Claim~\ref{clm:measstep:domG'}}
    For $(x,v,w)\in R\times[0,1)$, the definition of $R$ ensures that
    $\rho(x,v)(\phi^{-1}(x,v))=1$. Let
    \begin{align*}
      \widetilde{y} & \df \sup\{y\in Y \mid \phi(x,y,1) = (x,v)\}.
    \end{align*}

    We claim that there exists a non-decreasing sequence $(y_n)_{n\in\NN}$ in the set above such
    that $\phi^{-1}(x,v) = \bigcup_{n\in\NN} L(x,y_n,1)$. Indeed, since $\phi$ does not actually
    depend on its last coordinate, if the supremum $\widetilde{y}$ is actually attained by some $y$,
    then we can simply take $y_n=y$ for every $n\in\NN$, otherwise, we can simply take a
    non-decreasing sequence $(y_n)_{n\in\NN}$ in the set converging to $\widetilde{y}$.

    Since $(y_n)_{n\in\NN}$ is non-decreasing, we have $L(x,y_n,1)\subseteq L(x,y_{n+1},1)$ and
    since $\rho(x,v)(\phi^{-1}(x,v)) = 1 > w$, there must exist $n\in\NN$ such that
    $\rho(x,y)(L(x,y_n,1))\geq w$, so the set in~\eqref{eq:measstep:G'} is non-empty, hence
    $(x,v,w)\in\dom(G')$.
  \end{proofof}

  \begin{claim}\label{clm:measstep:G'L}
    For every $(x,v)\in R$ and every $(x',y,z)\in\phi^{-1}(x,v)$, we have
    \begin{align*}
      G'\Bigl(x,v,\rho(x,v)\bigl(L(x',y,z)\bigr)\Bigr) & \leq y.
    \end{align*}
  \end{claim}

  \begin{proofof}{Claim~\ref{clm:measstep:G'L}}
    Let $w\df\rho(x,v)(L(x',y,z))$ so that the claim amounts to $G'(x,v,w)\leq y$.
    
    Since the definition of $\phi$ forces $x'=x$ and $\phi$ does not depend on its last coordinate,
    it follows that $L(x,y,1)\supseteq L(x',y,z)$, which in particular implies
    $\rho(x,v)(L(x,y,1))\geq\rho(x,v)(L(x',y,z))=w$.
    
    Since $\phi(x,y,1)=(x,v)=\phi(x,y,z)$, we get that $y$ is in the set on~\eqref{eq:measstep:G'}, hence
    \begin{align*}
      G'\Bigl(x,v,\rho(x,v)\bigl(L(x',y,z)\bigr)\Bigr)
      =
      G'(x,v,w)
      & \leq
      y
    \end{align*}
    as desired.
  \end{proofof}

  We now define the set
  \begin{align}\label{eq:measstep:B''}
    B''
    & \df
    \{(x,y,z)\in X\times Y\times Z \mid
    (x,y)\in B\land\phi(x,y,z)\in R\land G'(x,G(x,y,z))=y\}.
  \end{align}

  After we define the set $A'$, we will let $B'\df B''\cap(A'\times Y\times Z)$, which will ensure
  that $B'\subseteq((A'\times Y)\cap B)\times Z$ yielding item~\ref{lem:measstep:B'A'}. Note also
  that once we show item~\ref{lem:measstep:muA'} (i.e., $\mu(A')=1)$, then
  items~\ref{lem:measstep:B'G}, \ref{lem:measstep:G'G}, \ref{lem:measstep:piG}
  and~\ref{lem:measstep:gammamunuB'} will follow from their counterparts for $B''$, so let us prove that
  these items do hold for $B''$.

  We have already observed that $\dom(G)=\dom(\phi)=\dom(g)\times Z\supseteq B\times Z$ and since
  $B''\subseteq B\times Z$, item~\ref{lem:measstep:B'G} holds for $B''$ ($B''\subseteq\dom(G)$). It
  is also clear from the definition of $B''$ that it satisfies item~\ref{lem:measstep:G'G}
  ($G'(x,G(x,y,z))=y$ for every $(x,y,z)\in B''$). Item~\ref{lem:measstep:piG} for $B''$ ($\pi(G(x,y,z)) =
  g(x,y)$ for every $(x,y,z)\in B''$) follows from the definition of $G$ and the already proved
  $B''\subseteq\dom(G)$.

  \medskip

  Let us now prove item~\ref{lem:measstep:gammamunuB'} for $B''$
  ($(\gamma[\mu]\otimes\nu)(B'')=1$). By~\eqref{eq:measstep:DT} applied to $\One_{B''}$, we have
  \begin{align*}
    (\gamma[\mu]\otimes\nu)(B'')
    & =
    \int_R \int_{\phi^{-1}(r)} \One_{B''}(s) \ d\rho(r)(s)\ d\sigma(r).
  \end{align*}

  Since for every $r\in R$ we have $\rho(r)(\phi^{-1}(r))=1$ and using Claim~\ref{clm:measstep:G'L},
  we get that to show $(\gamma[\mu]\otimes\nu)(B'')=1$ it suffices to show that for every $(x,v)\in R$ and
  for the set
  \begin{align*}
    U(x,v)
    & \df
    \left\{(x',y,z)\in\phi^{-1}(x,v) \;\middle\vert\;
    G'\Bigl(x,v,\rho(x,v)\bigl(L(x,y,z)\bigr)\Bigr) < y\right\},
  \end{align*}
  we have $\rho(x,v)(U(x,v))=0$.

  Fix $(x,v)\in R$ and since we will be working with points of the form $(x',y,z)\in\phi^{-1}(x,v)$,
  which forces $x'=x$, we will simplify notation by writing $(x,y,z)\in\phi^{-1}(x,v)$. We will also
  use the shorthand $U\df U(x,v)$.

  Let us define by transfinite induction sequences $(y_\alpha)_{\alpha < \alpha_0}$,
  $(y'_\alpha)_{\alpha<\alpha_0}$, $(z_\alpha)_{\alpha<\alpha_0}$ and $(M_\alpha)_{\alpha<\alpha_0}$
  in $[0,1]$ with the following properties:
  \begin{enumerate}[label={\alph*.}, ref={(\alph*)}]
  \item\label{it:yy'} $y_\alpha < y'_\alpha$.
  \item\label{it:atomless} If $(x,y'_\alpha,z_\alpha)\in U$, then $(x,y'_\alpha,z_\alpha)$ is
    \emph{not} an atom of the measure $\rho(x,v)$, that is, we have
    $\rho(x,v)(\{(x,y'_\alpha,z_\alpha)\})=0$.
  \item\label{it:rhoLM} For the set
    \begin{align}\label{eq:measstep:Ualpha}
      U_\alpha & \df \{(x,y,z)\in U \mid (y_\alpha,1) <_L (y,z) \leq_L (y'_\alpha,z_\alpha)\},
    \end{align}
    we have $\rho(x,v)(L(x,y,z))=M_\alpha$ for every $(x,y,z)\in U_\alpha$.
  \item\label{it:rhoU} $\rho(x,v)(U_\alpha)=0$ for $U_\alpha$ given by~\eqref{eq:measstep:Ualpha}.
  \item\label{it:intervals} For every $\beta < \alpha$, we have $(y_\beta,y'_\beta)\not\supseteq
    (y_\alpha,y'_\alpha)$ (as open intervals in $[0,1]$).
  \end{enumerate}

  Given $(y_\beta)_{\beta<\alpha}$, $(y'_\beta)_{\beta<\alpha}$ and $(z_\beta)_{\beta<\alpha}$, if
  $U\subseteq\bigcup_{\beta<\alpha} U_\beta$, then stop the construction setting
  $\alpha_0\df\alpha$; otherwise, let $(x,y'_\alpha,\widetilde{z}_\alpha)\in
  U\setminus\bigcup_{\beta<\alpha} U_\beta$ and define
  \begin{align*}
    M_\alpha & \df \rho(x,v)\bigl(L(x,y'_\alpha,\widetilde{z}_\alpha)\bigr),
    \\
    y_\alpha & \df G'(x,v,M_\alpha),
    \\
    z_\alpha
    & \df
    \sup\left\{z\in Z\;\middle\vert\;
    \rho(x,v)\bigl(L(x,y'_\alpha,z)\bigr) = M_\alpha\right\}.
  \end{align*}
  (Note that in the definition of $z_\alpha$, all points $z$ in the set satisfy
  $\phi(x,y'_\alpha,z)=(x,v)$ since $(x,y'_\alpha,\widetilde{z}_\alpha)\in U$.)

  Let us check that all promised items are inductively satisfied.

  First note that since $(x,y'_\alpha,\widetilde{z}_\alpha)\in U$, we must have
  \begin{align*}
    y_\alpha & = G'\Bigl(x,v,\rho(x,v)\bigl(L(x,y,z)\bigr)\Bigr) < y'_\alpha,
  \end{align*}
  so item~\ref{it:yy'} holds.

  \medskip

  We prove item~\ref{it:atomless} by the contra-positive. Suppose that $(x,y'_\alpha,z_\alpha)$ is
  an atom of $\rho(x,v)$ and note that the definition of $G'$ in~\eqref{eq:measstep:G'} implies that
  \begin{align*}
    G'\Bigl(x,v,\rho(x,v)\bigl(L(x,y'_\alpha,z_\alpha)\bigr)\Bigr) = y'_\alpha
  \end{align*}
  as for any $y < y'_\alpha$ with $\phi(x,y,1)=(x,v)$, the set $L(x,y,1)$ has strictly less
  $\rho(x,v)$-measure than $L(x,y'_\alpha,1)$ as it excludes the atom
  $(x,y'_\alpha,z_\alpha$). Thus, we get $(x,y'_\alpha,z_\alpha)\notin U$.

  \smallskip

  To show item~\ref{it:rhoLM}, first note that if $(x,y,z)\in U_\alpha$, then we must have $y >
  y_\alpha$, so there exists $y'$ with $y_\alpha\leq y' < y$, $\phi(x,y',1)=(x,v)$ and
  $\rho(x,v)(L(x,y',1))\geq M_\alpha$ and since $L(x,y,z)\supseteq L(x,y',1)$, we conclude
  \begin{align*}
    \rho(x,v)\bigl(L(x,y,z)\bigr) & \geq \rho(x,v)\bigl(L(x,y',1)\bigr) \geq M_\alpha.
  \end{align*}

  To prove the other inequality, first suppose $(y,z) <_L (y'_\alpha,z_\alpha)$ and note that the
  definition of $z_\alpha$ implies that there exists $z'\in Z$ such that $(y,z) <_L (y'_\alpha,z')$
  and $\rho(x,v)(L(x,y'_\alpha,z')) = M_\alpha$. Since $L(x,y,z)\subseteq L(x,y'_\alpha,z')$, we get
  $\rho(x,v)(L(x,y,z))\leq M_\alpha$.

  Suppose now that $(y,z)=(y'_\alpha,z_\alpha)$ and note that the definition of $z_\alpha$ and the
  facts that $\phi$ does not depend on its last coordinate and that
  $(x,y'_\alpha,\widetilde{z}_\alpha)\in U$ implies $\phi(x,y'_\alpha,\widetilde{z}_\alpha)=(x,v)$
  give us that for every $z' < z_\alpha$, we have $\rho(x,v)(L(x,y'_\alpha,z'))\leq
  M_\alpha$. Finally, note that
  \begin{align*}
    \rho(x,v)\bigl(L(x,y,z)\bigr)
    & \leq
    \rho(x,v)(\{(x,y,z)\})
    +
    \sup_{z' < z_\alpha} \rho(x,v)\bigl(L(x,y'_\alpha,z')\bigr)
    \leq
    M_\alpha
  \end{align*}
  since $(x,y,z)=(x,y'_\alpha,z_\alpha)\in U$ and item~\ref{it:atomless} says that $(x,y,z)$ is not
  an atom of $\rho(x,v)$. This concludes the proof of item~\ref{it:rhoLM}.

  \smallskip

  Let us now show item~\ref{it:rhoU}. The definition of $y_\alpha$ and $z_\alpha$ as infimum and
  supremum, respectively, ensure the existence of a non-increasing sequence $(y_\alpha^n)_{n\in\NN}$
  and a non-decreasing sequence $(z_\alpha^n)_{n\in\NN}$ converging to $y_\alpha$ and $z_\alpha$,
  respectively, such that
  \begin{gather*}
    \phi(x,y_\alpha^n,1) = \phi(y'_\alpha,z_\alpha^n) = (x,v),\\
    \rho(x,v)\bigl(L(x,y_\alpha^n,1)\bigr) = \rho(x,v)\bigl(L(x,y'_\alpha,z_\alpha^n)\bigr) = M_\alpha,
  \end{gather*}
  from which it follows that for every $n\in\NN$, letting
  \begin{align*}
    U_\alpha^n
    & \df
    L(x,y'_\alpha,z_\alpha^n)\setminus L(x,y_\alpha^n,1)
    \\
    & =
    \{(x,y,z)\in X\times Y\times Z \mid
    \phi(x,y,z)=(x,v)\land
    (y_\alpha^n,1) <_L (y,z) \leq_L (y'_\alpha,z_\alpha^n)\},
  \end{align*}
  we get $\rho(x,v)(U_\alpha^n) = 0$.

  Note now that
  \begin{align*}
    U_\alpha
    & \subseteq
    \bigcup_{n\in\NN} U_\alpha^n\cup
    \{(x,y'_\alpha,z_\alpha) \mid\text{ if }(x,y'_\alpha,z_\alpha)\in U\}.
  \end{align*}
  Since each $U_\alpha^n$ has zero $\rho(x,v)$-measure and $(x,y'_\alpha,z_\alpha)\in U$ implies
  that it is not an atom of $\rho(x,v)$ by the already proved item~\ref{it:atomless}, it follows
  that $\rho(x,v)(U_\alpha)=0$, concluding the proof of item~\ref{it:rhoU}.

  \smallskip

  Finally, to show item~\ref{it:intervals}, it suffices to show that if the intervals
  $(y_\beta,y'_\beta)$ and $(y_\alpha,y'_\alpha)$ have non-empty intersection, then we have
  $y'_\alpha > y'_\beta$.

  Suppose for a contradiction that $y'_\alpha\leq y'_\beta$ and that $(y_\beta,y'_\beta)\cap
  (y_\alpha,y'_\alpha)\neq\varnothing$. Note that $y'_\alpha > y_\beta$ and since the construction
  of $y'_\alpha$ ensures $(x,y'_\alpha,\widetilde{z}_\alpha)\in U\setminus U_\beta$, we must have
  $(y'_\alpha,\widetilde{z}_\alpha) >_L (y'_\beta,z_\beta)$. Since $y'_\alpha\leq y'_\beta$, we get
  $y'_\alpha = y'_\beta$ and $\widetilde{z}_\alpha > z_\beta$.

  Since $\widetilde{z}_\beta\leq z_\beta < \widetilde{z}_\alpha\leq z_\alpha$ and $y_\alpha <
  y'_\alpha = y'_\beta$ from item~\ref{it:yy'}, we get $(y'_\beta,\widetilde{z}_\beta)\in U_\alpha$,
  so item~\ref{it:rhoLM} gives
  \begin{align*}
    M_\beta & = \rho(x,v)\bigl(L(x,y'_\beta,\widetilde{z}_\beta)\bigr) = M_\alpha,
  \end{align*}
  which in turn implies $y_\alpha = y_\beta$. Finally, since $y'_\alpha=y'_\beta$ and
  $M_\alpha=M_\beta$, we also get $z_\beta = z_\alpha\geq\widetilde{z}_\alpha$, contradicting
  $\widetilde{z}_\alpha > z_\beta$.

  This concludes the inductive construction.

  \smallskip

  Note now that item~\ref{it:intervals} implies that the construction must stop (as the union of the
  intervals $(y_\alpha,y'_\alpha)$ is getting strictly larger and is always a subset of $[0,1]$), so
  $\alpha_0$ gets defined. We claim that $\alpha_0$ is countable. Indeed, otherwise, for the first
  uncountable ordinal $\omega_1$, the union $\bigcup_{\alpha <
    \omega_1}(y_\alpha,y'_\alpha)\subseteq[0,1]$ has a cover by open sets that does not contain a
  countable subcover, contradicting the fact that the usual topology on $[0,1]$ is second-countable.

  By construction, we know that $U = \bigcup_{\alpha < \alpha_0} U_\alpha$, so item~\ref{it:rhoU}
  along with the fact that $\alpha_0$ is countable implies $\rho(x,v)(U) = 0$. Thus we have proven
  that $\rho(r)(U) = 0$ for every $r\in R$ so item~\ref{lem:measstep:gammamunuB'} for $B''$
  ($(\gamma[\mu]\otimes\nu)(B'')=0$) holds.

  \medskip

  Our next objective is to construct $A'$ so that the remaining items (items~\ref{lem:measstep:A'},
  \ref{lem:measstep:muA'} and~\ref{lem:measstep:Ggammanu}) hold.

  To do so, let us first define for every $(x,v)\in R$ and every $a\in(0,1)$ the set
  \begin{align*}
    C_a(x,v)
    & \df
    \left\{(x',y',z')\in\phi^{-1}(x,v) \;\middle\vert\;
    \rho(x,v)\bigl(L(x',y',z')\bigr) \leq a\right\}.
  \end{align*}
  We want to show that for $\sigma$-almost every $(x,v)\in R$ we have $\rho(x,v)(C_a(x,v))=a$. For
  this, we define the following points:
  \begin{align*}
    y_{x,v,a}
    & \df
    \sup\{y\in[0,1] \mid (x,y,0)\in C_a(x,v)\}\cup\{0\},
    \\
    z_{x,v,a}
    & \df
    \sup\{z\in[0,1] \mid (x,y_{x,v,a},z)\in C_a(x,v)\}\cup\{0\}
    \\
    y'_{x,v,a}
    \df
    G'(x,v,a)
    & =
    \inf\left\{y\in[0,1] \;\middle\vert\;
    \phi(x,v,1)=(x,v)\land \rho(x,v)\bigl(L(x,y,1)\bigr)\geq a
    \right\}
    \\
    z'_{x,v,a}
    & \df
    \inf\left\{z\in[0,1] \;\middle\vert\;
    \phi(x,y'_{x,v,a},z)=(x,v)\land\rho(x,v)\bigl(L(x,y,1)\bigr)\geq a\right\}
    \cup\{1\}
  \end{align*}
  (Note that Claim~\ref{clm:measstep:domG'} along with the fact that $a < 1$ ensure that
  $y'_{x,v,a}$ is actually defined, i.e., that $(x,v,a)\in\dom(G')$.)

  We first prove the following weaker claim about the $\rho(x,v)$-measure of $C_a(x,v)$.

  \begin{claim}\label{clm:measstep:rhoCa}
    Let $(x,v)\in R$. Then $\rho(x,v)(C_a(x,v))\leq a$. Furthermore, if $\rho(x,v)(C_a(x,v)) < a$,
    then $(x,y'_{x,v,a},z'_{x,v,a})$ is an atom of $\rho(x,v)$.
  \end{claim}

  \begin{proofof}{Claim~\ref{clm:measstep:rhoCa}}
    If $C_a(x,v)$ is empty, the we clearly have $\rho(x,v)(C_a(x,v))\leq a$. Suppose then that
    $C_a(x,v)$ is non-empty and note that the definitions of $y_{x,v,a}$ and $z_{x,v,a}$ as suprema
    ensure the existence of non-decreasing sequences $(y_n)_{n\in\NN}$ and $(z_n)_{n\in\NN}$
    such that $(x,y_n,z_n)\in C_a(x,v)$ and $C_a(x,v) = \bigcup_{n\in\NN} L(x,y_n,z_n)$ (when any of
    the suprema is actually a maximum, the corresponding sequence should be taken constant equal to
    the maximum; otherwise, the corresponding sequence should approximate the supremum).

    Since both sequences are non-decreasing, we get $L(x,y_n,z_n)\subseteq L(x,y_{n+1},z_{n+1})$ for
    every $n\in\NN$, so we conclude
    \begin{align*}
      \rho(x,v)\bigl(C_a(x,v)\bigr)
      & =
      \rho(x,v)\left(\bigcup_{n\in\NN} L(x,y_n,z_n)\right)
      =
      \lim_{n\to\infty} \rho(x,v)(L(x,y_n,z_n))
      \leq
      a,
    \end{align*}
    where the last inequality follows since $(x,y_n,z_n)\in C_a(x,v)$.

    \medskip

    Suppose now that $\rho(x,v)(C_a(x,v)) < a$ and note that the definitions of $y'_{x,v,a}$ and
    $z'_{x,v,a}$ as infima ensure the existence of non-increasing sequences $(y'_n)_{n\in\NN}$ and
    $(z'_n)_{n\in\NN}$ converging to $y'_{x,v,a}$ and $z'_{x,v,a}$, respectively and such that
    $\phi(x,y'_n,z'_n)=(x,v)$, $\rho(x,v)(L(x,y'_n,z'_n))\geq a$ and
    \begin{align*}
      \bigcap_{n\in\NN} L(x,y'_n,z'_n)
      & =
      \left\{(x,y,z) \;\middle\vert\;
      \phi(x,y,z)=(x,v)\land\rho(x,v)\bigl(L(x,y,z)\bigr)\geq a
      \right\}.
    \end{align*}
    (Again, whenever any of the infima is actually a minimum, the corresponding sequence should be
    taken constant equal to the minimum.)

    Since both sequences are non-increasing, we get $L(x,y'_n,z'_n)\supseteq
    L(x,y'_{n+1},z'_{n+1})$, from which we conclude
    \begin{align*}
      \rho(x,v)\left(\bigcap_{n\in\NN} L(x,y'_n,z'_n)\right)
      & =
      \lim_{n\to\infty} \rho(x,v)\bigl(L(x,y'_n,z'_n)\bigr)
      \geq
      a.
    \end{align*}

    On the other hand, since $y'_n$ and $z'_n$ converge to $y'_{x,v,a}$ and $z'_{x,v,a}$
    respectively, we have
    \begin{align*}
      \bigcap_{n\in\NN} L(x,y'_n,z'_n)
      & =
      \{(x',y',z')\in\phi^{-1}(x,v) \mid (y',z')\leq_L (y'_{x,v,a},z'_{x,v,a})\}.
    \end{align*}
    Let $C'_a$ be the set above and note that since $\rho(x,v)(C_a(x,v)) < a\leq \rho(x,v)(C'_a)$,
    we must have $C_a(x,v)\subsetneq C'_a$ (since both these sets are subsets of $\phi^{-1}(x,v)$
    that are closed downward with respect to $\leq_L$ within $\phi^{-1}(x,v)$).

    We claim that $C'_a\setminus C_a(x,v)=\{(x,y'_{x,v,a},z'_{x,v,a})\}$. To see this, let
    $(x',y',z')\in C'_a\setminus C_a(x,v)$ and note that the fact that $(x',y',z')\notin C_a(x,v)$
    gives $\rho(x,v)(L(x',y',z')) > a$. From $\phi(x',y',z')=(x,v)$ along with the definitions of
    $y'_{x,v,a}$ and $z'_{x,v,a}$ as infima, we get $x'=x$ and
    $(y',z')\geq_L(y'_{x,v,a},z'_{x,v,a})$. In turn, since $(x',y',z')\in C'_a$, we conclude
    $(y',z')\leq_L (y'_{x,v,a},z'_{x,v,a})$. Thus, $(x',y',z') = (x,y'_{x,v,a},z'_{x,v,a})$, hence
    $C'_a\setminus C_a(x,v)=\{(x,y'_{x,v,a},z'_{x,v,a})\}$.

    Finally, since
    \begin{align*}
      \rho(x,v)(\{x,y'_{x,v,a},z'_{x,v,a}\})
      & =
      \rho(x,v)(C'_a\setminus C_a(x,v))
      >
      0,
    \end{align*}
    we conclude that $(x,y'_{x,v,a},z'_{x,v,a})$ is an atom of $\rho(x,v)$.
  \end{proofof}

  For every $a\in(0,1)$, let
  \begin{align*}
    S_a & \df \{(x,y,z)\in B'' \mid y=y'_{\phi(x,y,z),a}\land z=z'_{\phi(x,y,z),a}\}.
  \end{align*}

  Since $\nu=\lambda$ is atomless and for each $(x,y)$ there exists at most one $z$ such that
  $(x,y,z)\in S_a$, by Fubini's Theorem, the set $S_a$ must have zero $\gamma[\mu]\otimes\nu$-measure. On
  the other hand, by~\eqref{eq:measstep:DT} with $h\df\One_{S_a}$, we have
  \begin{align*}
    0
    & =
    (\gamma[\mu]\otimes\nu)(S_a)
    =
    \int_R \int_{\phi^{-1}(x,v)} \One_{S_a}(x',y',z')
    \ d\rho(x,v)(x',y',z')\ d\sigma(x,v)
    \\
    & =
    \int_R \rho(x,v)(\{(x,y'_{x,v,a},z'_{x,v,a})\})\ d\sigma(x,v),
  \end{align*}
  from which we conclude that for $\sigma$-almost every $(x,v)\in R$, we have
  $\rho(x,v)(\{x,y'_{x,v,a},z'_{x,v,a}\}) = 0$, which along with Claim~\ref{clm:measstep:rhoCa}
  allows us to conclude that
  \begin{align*}
    \rho(x,v)(C_a(x,v)) = a
  \end{align*}
  for $\sigma$-almost every $(x,v)\in R$.

  We now define
  \begin{align*}
    R'_a
    & \df
    \left\{(x,v)\in R\;\middle\vert\;
    \rho(x,v)\bigl(C_a(x,v)\bigr) = a\right\}
    \qquad \bigl(a\in(0,1)\bigr)
    \\
    R'
    & \df
    \bigcap_{a\in(0,1)\cap\QQ} R'_a.
  \end{align*}
  We have just proved that $\sigma(R'_a)=1$ for all $a\in(0,1)$, which also implies $\sigma(R')=1$.

  In particular, we get
  \begin{align*}
    1
    =
    \sigma(R')
    & =
    \int_R \One_{R'}(r)\ d\sigma(r)
    \\
    & =
    \int_R \int_{\phi^{-1}(r)} \One_{\phi^{-1}(R')}(s)\ d\rho(r)(s)\ d\sigma(r)
    \\
    & =
    \int_{X\times Y\times Z} \One_{\phi^{-1}(R')}(s) \ d(\gamma[\mu]\otimes\nu)(s)
    \\
    & =
    \int_X (\gamma(x)\otimes\nu)(\{(y,z)\in Y\times Z \mid \phi(x,y,z)\in R'\})\ d\mu(x)
  \end{align*}
  where the second equality follows since $\sigma(R)=1$, the third equality follows since when $r\in
  R$, we have $\rho(r)(\phi^{-1}(r))=1$, the third equality follows from~\eqref{eq:measstep:DT} with
  $h\df\One_{\phi^{-1}(R')}$ and the last equality follows from Fubini's Theorem.

  Thus, we conclude that defining
  \begin{align*}
    A'
    & \df
    \left\{x\in A \;\middle\vert\;
    (\gamma(x)\otimes\nu)\bigl(\{(y,z)\in Y\times Z \mid \phi(x,y,z)\in R'\}\bigr) = 1
    \right\},
  \end{align*}
  we have $\mu(A') = 1$, that is, item~\ref{lem:measstep:muA'} is satisfied. It is also obvious from
  the definition of $A'$ that item~\ref{lem:measstep:A'} ($A'\subseteq A$) holds.

  \medskip

  It remains to show item~\ref{lem:measstep:Ggammanu}, that is, we want to show that for every $x\in
  A'$, we have $G(x,\place,\place)_*(\gamma(x)\otimes\nu)=\Theta(f(x))$. To do so, we pick
  $(\rn{y},\rn{z})$ according to $\gamma(x)\otimes\nu$, we let $(\rn{v},\rn{w})\df
  G(x,\rn{y},\rn{z})$ and we need to show that $(\rn{v},\rn{w})$ is distributed according to
  $\Theta(f(x)) = \theta(f(x))\otimes\eta$ (recall $\eta=\lambda$).

  The hypothesis $g(x,\place)_*(\gamma(x))=\theta(f(x))$ for every $x\in A$ along with the already
  proved items~\ref{lem:measstep:A'} and~\ref{lem:measstep:piG} ($A'\subseteq A$ and
  $\pi(G(x,y,z))=g(x,y)$) implies that $\rn{v}$ is distributed according to $\theta(f(x))$. On the
  other hand, the definition of $G$ in~\eqref{eq:measstep:G} implies that for every $a\in[0,1]$, we
  have the following conditional probability equality:
  \begin{align*}
    \PP[\rn{w}\leq a \given\rn{v}]
    & =
    \PP\bigl[\rho(x,\rn{v})(L(x,\rn{y},\rn{z}))\leq a\given[\big]\rn{v}\bigr]
    =
    \rho(x,\rn{v})\bigl(C_a(x,\rn{v})\bigr).
  \end{align*}

  On the other hand, by the definition of $A'$, we know that for every $a\in(0,1)\cap\QQ$, we have
  \begin{align*}
    \PP\Bigl[\rho(x,\rn{v})\bigl(C_a(x,\rn{v})\bigr) = a\Bigr],
  \end{align*}
  from which we get that the conditional distribution of $\rn{w}$ given $\rn{v}$ is $\lambda$ (with
  probability $1$ over $\rn{v}$), so $(\rn{w},\rn{v})$ has distribution $\theta(f(x))\otimes\lambda
  = \Theta(f(x))$, as desired.

  \medskip

  Finally, as mentioned after the definition of $B''$ in~\eqref{eq:measstep:B''}, letting $B'\df
  B''\cap(A'\times Y\times Z)$ concludes the proof in the case when
  $(X,\mu)=(Z,\nu)=([0,1],\lambda)$ and $W=Y=[0,1]$.

  \bigskip

  We now prove that the general case reduces to the previous case. Let $I_X\colon[0,1]\to X$,
  $I_Z\colon [0,1]\to Z$ and $I_W\colon [0,1]\to W$ be measure-isomorphisms modulo $0$ from
  $([0,1],\lambda)$ to $(X,\mu)$, $(Z,\nu)$ and $(W,\eta)$, respectively. Since $Y$ is uncountable,
  there exists a Borel-isomorphism $J_Y\colon[0,1]\to Y$.

  Let now $\widetilde{f}\colon[0,1]\pto U$, $\widetilde{g}\colon[0,1]\times[0,1]\pto V$,
  $\widetilde{\gamma}\colon[0,1]\pto\cP([0,1])$, $\widetilde{A}\subseteq[0,1]$ and
  $\widetilde{B}\subseteq[0,1]\times[0,1]$ be given by
  \begin{gather*}
    \begin{aligned}
      \widetilde{f} & \df f\comp I_X, &
      \widetilde{g} & \df g\comp(I_X\otimes J_Y),
      \\
      \widetilde{A} & \df I_X^{-1}(A), &
      \widetilde{B} & \df (I_X\otimes J_Y)^{-1}(B),
    \end{aligned}
    \\
    \widetilde{\gamma}(x)
    \df
    (J_Y^{-1})_*\Bigl(\gamma\bigl(I_X(x)\bigr)\Bigr).
  \end{gather*}

  Let then $\widetilde{G}\colon[0,1]\times[0,1]\times[0,1]\pto V\times[0,1]$, $\widetilde{G}'\colon
  [0,1]\times V\times[0,1]\pto [0,1]$, $\widetilde{A}'\subseteq\widetilde{A}$ and
  $\widetilde{B}'\subseteq\widetilde{B}\times[0,1]$ be the $(G,G',A',B')$ provided by the
  lemma for $(\widetilde{f},\widetilde{g},\widetilde{\gamma},\widetilde{A},\widetilde{B})$ (it is
  straightforward to check that the hypotheses are satisfied and this is the already proved case).

  Then define
  \begin{align*}
    G
    & \df
    (\id_V\otimes I_W)\comp\widetilde{G}\comp(I_X^{-1}\otimes J_Y^{-1}\otimes I_Z^{-1}),
    \\
    G' & \df J_Y\comp\widetilde{G}'\comp(I_X^{-1}\otimes\id_V\otimes I_W^{-1}),
    \\
    A' & \df I_X(\widetilde{A}')\cap\dom(I_X^{-1}),
    \\
    B'
    & \df
    (I_X\otimes J_Y\otimes I_Z)(\widetilde{B}')
    \cap(\dom(I_X^{-1})\times Y\times\dom(I_Z^{-1})),
  \end{align*}

  It is straightforward to check that the conclusions of the lemma for $(G,G',A',B')$ follow from
  the conclusions of the lemma for $(\widetilde{G},\widetilde{G}',\widetilde{A}',\widetilde{B}')$.
\end{proof}

The next lemma is the key lemma of this section that will be used to formalize the construction of
the amalgamating types. It upgrades Lemma~\ref{lem:measstep} in two ways: the first is so that the
resulting $G'(x,\place,\place)$ is an almost everywhere left-inverse of $G(x,\place,\place)$ (i.e.,
it is able to retrieve both the $Y$ and $Z$ coordinates instead of only the former); the second is
an added equivariant version of the construction.

\begin{lemma}\label{lem:keymeas}
  Let $(X,\mu)$, $(Z,\nu)$ and $(W,\eta)$ be atomless standard probability spaces, let $Y$ and $U$
  be Borel spaces with $Y$ uncountable, let $f\colon X\pto U$ and $g\colon X\times Y\pto V$ be
  measurable partial functions, let $\gamma\colon X\pto\cP(Y)$ and $\theta\colon U\pto\cP(V)$ be
  partial Markov kernels with $\mu(\dom(\gamma))=1$ and let $A\subseteq\dom(f)\subseteq X$ and
  $B\subseteq\dom(g)\subseteq X\times Y$ be measurable sets with $\mu(A)=\gamma[\mu](B)=1$
  (see~\eqref{eq:rhomu}). Suppose further that $A\subseteq f^{-1}(\dom(\theta))\cap\dom(\gamma)$ and
  for every $x\in A$, we have
  \begin{align*}
    g(x,\place)_*(\gamma(x)) & = \theta(f(x)).
  \end{align*}
  Finally, let $\pi\colon V\times W\to V$ be the natural projection.

  Then there exist partial functions $G\colon X\times Y\times Z\pto V\times W$ and $G'\colon X\times
  V\times W\pto Y\times Z$ and measurable sets $A'\subseteq X$ and $B'\subseteq X\times Y\times Z$
  such that letting $\Theta\colon U\pto\cP(V\times W)$ be the partial Markov kernel given by
  \begin{align*}
    \Theta(u) & \df \theta(u)\otimes\eta \qquad (u\in U),
  \end{align*}
  the following hold:
  \begin{enumerate}
  \item\label{lem:keymeas:A'} $A'\subseteq A$.
  \item\label{lem:keymeas:B'A'} $B'\subseteq ((A'\times Y)\cap B)\times Z$.
  \item\label{lem:keymeas:B'G} $B'\subseteq\dom(G)$.
  \item\label{lem:keymeas:G'G} For every $(x,y,z)\in B'$, we have $G'(x,G(x,y,z))=(y,z)$.
  \item\label{lem:keymeas:piG} For every $(x,y,z)\in B'$, we have $\pi(G(x,y,z)) = g(x,y)$.
  \item\label{lem:keymeas:gammamunuB'} $(\gamma[\mu]\otimes\nu)(B')=1$.
  \item\label{lem:keymeas:muA'} $\mu(A')=1$.
  \item\label{lem:keymeas:Ggammanu} For every $x\in A'$, we have
    $G(x,\place,\place)_*(\gamma(x)\otimes\nu)=\Theta(f(x))$.
  \end{enumerate}

  Furthermore, if $H$ is a finite group such that the following hold:
  \begin{itemize}
  \item $H$ acts measurably on $X$, $Y$, $Z$, $U$, $V$ and $W$.
  \item There exists a measurable set $X'\subseteq X$ such that $H$ the orbits $(\sigma\cdot
    X')_{\sigma\in H}$ of $X'$ are pairwise disjoint and $\mu(\bigcup_{\sigma\in H}\sigma\cdot
    X')=1$.
  \item The $H$-actions on $Z$ and $W$ are trivial.
  \item $\mu$ is $H$-invariant.
  \item $f$, $g$, $\gamma$ and $\theta$ are $H$-equivariant (when we equip product spaces with the
    diagonal action and $\cP(V)$ and $\cP(Y)$ with the induced actions).
  \end{itemize}
  Then $G$, $G'$, $\Theta$, $A'$ and $B'$ can be taken to further satisfy:
  \begin{enumerate}[resume*]
  \item\label{lem:keymeas:invariance} $A'$ and $B'$ are $H$-invariant (with respect to the diagonal
    actions).
  \item\label{lem:keymeas:equivariance} $G$, $G'$ and $\Theta$ are $H$-equivariant (with respect to
    the diagonal actions and the induced action on $\cP(V\times W)$).
  \end{enumerate}
\end{lemma}

\begin{proof}
  We will first show all the statements that do not involve the group $H$.
  
  We first prove the case when $(Z,\nu) = (\widetilde{Z}^{\NN_+},\widetilde{\nu}^{\NN_+})$ and
  $(W,\eta) = (\widetilde{W}^{\NN_+},\widetilde{\eta}^{\NN_+})$ for some atomless standard
  probability spaces $(\widetilde{Z},\widetilde{\nu})$ and $(\widetilde{W},\widetilde{\eta})$.

  For every $m,n\in\NN$ with $m\geq n$, let $\tau_{m,n}\colon\widetilde{Z}^m\to\widetilde{Z}^n$,
  $\tau_{\NN_+,n}\colon\widetilde{Z}^{\NN_+}\to\widetilde{Z}^n$,
  $\pi_{m,n}\colon\widetilde{W}^m\to\widetilde{W}^n$ and
  $\pi_{\NN_+,n}\colon\widetilde{W}^{\NN_+}\to\widetilde{W}^n$ be the projections onto the first $n$
  coordinates.

  Define
  \begin{align*}
    G_{-1} & \df g, &
    \Gamma_{-1} & \df \gamma, &
    \Theta_{-1} & \df \theta,
    \\
    A_{-1} & \df A, &
    B_{-1} & \df B
  \end{align*}
  and let us construct inductively in $n\in\NN$ sequences of measurable partial functions $G_n\colon
  X\times Y\times\widetilde{Z}^{n+1}\pto V\times\widetilde{W}^{n+1}$ and $G'_n\colon X\times
  V\times\widetilde{W}^{n+1}\pto X\times Y\times\widetilde{Z}^n$, sequences of partial Markov
  kernels $\Gamma_n\colon X\pto\cP(Y\times\widetilde{Z}^{n+1})$ and $\Theta_n\colon
  U\pto\cP(V\times\widetilde{W}^{n+1})$ and sequences of measurable sets $A_n\subseteq X$ and
  $B_n\subseteq X\times Y\times\widetilde{Z}^{n+1}$ such that the following hold:
  \begin{enumerate}[label={\Roman*.}, ref={(\Roman*)}]
  \item\label{lem:keymeas:An} $A_n\subseteq A_{n-1}$.
  \item\label{lem:keymeas:BnAn} $B_n\subseteq ((A_n\times Y\times\widetilde{Z}^n)\cap
    B_{n-1})\times\widetilde{Z}$.
  \item\label{lem:keymeas:BnGn} $B_n\subseteq\dom(G_n)$.
  \item\label{lem:keymeas:G'nGn} For every $(x,y,z)\in B_n$, we have $G'_n(x,G_n(x,y,z)) =
    (y,\tau_{n+1,n}(z))$.
  \item\label{lem:keymeas:piGn} For every $(x,y,z)\in B_n$, we have
    $(\id_V\otimes\pi_{n+1,n})(G_n(x,y,z)) = G_{n-1}(x,y,\tau_{n+1,n}(z))$.
  \item\label{lem:keymeas:GammanmuBn} $\Gamma_n[\mu](B_n)=1$.
  \item\label{lem:keymeas:muAn} $\mu(A_n)=1$.
  \item\label{lem:keymeas:GnGammanThetan} For every $x\in A_n$, we have
    $G_n(x,\place,\place)_*(\Gamma_n(x)) = \Theta_n(f(x))$.
  \item\label{lem:keymeas:Andomf} $A_n\subseteq\dom(f)$.
  \item\label{lem:keymeas:mudomGamman} $\mu(\dom(\Gamma_n))=1$.
  \end{enumerate}

  Assume that we are given $(G_{n-1},\Gamma_{n-1},\Theta_{n-1},A_{n-1},B_{n-1})$ satisfying
  items~\ref{lem:keymeas:BnGn}, \ref{lem:keymeas:GammanmuBn}, \ref{lem:keymeas:muAn},
  \ref{lem:keymeas:GnGammanThetan}, \ref{lem:keymeas:Andomf} and~\ref{lem:keymeas:mudomGamman} (note
  that these indeed hold for $(G_{-1},\Gamma_{-1}, \Theta_{-1}, A_{-1}, B_{-1})$) and note that the
  hypotheses of Lemma~\ref{lem:measstep} are satisfied with the following choice of parameters:
  \begin{align*}
    (X,\mu) & \leftarrow (X,\mu), &
    (Z,\nu) & \leftarrow (\widetilde{Z},\widetilde{\nu}), &
    Y & \leftarrow Y\times\widetilde{Z}^n,
    \\
    U & \leftarrow U, &
    V & \leftarrow V\times\widetilde{W}^n, &
    (W,\eta) & \leftarrow (\widetilde{W},\widetilde{\eta}),
    \\
    f & \leftarrow f, &
    g & \leftarrow G_{n-1}, &
    \gamma & \leftarrow \Gamma_{n-1},
    \\
    \theta & \leftarrow \Theta_{n-1}, &
    A & \leftarrow A_{n-1}, &
    B & \leftarrow B_{n-1}.
  \end{align*}

  Let then $(G_n,G'_n,\Theta_n,A_n,B_n)$ be the $(G,G',\Theta,A',B')$ provided by
  Lemma~\ref{lem:measstep} with the choice of parameters above. Define also the partial Markov
  kernel $\Gamma_n\colon X\pto\cP(Y\times\widetilde{Z}^{n+1})$ by
  \begin{align*}
    \Gamma_n(x) & \df \Gamma_{n-1}(x)\otimes\widetilde{\nu} \qquad (x\in X)
  \end{align*}
  (which by induction is equivalent to $\Gamma_n(x) = \gamma(x)\otimes\widetilde{\nu}^{n+1}$). Note
  that we have the following implications between the items of Lemma~\ref{lem:measstep} and the
  items required by our inductive construction:
  \begin{center}
    \begin{tabular}{*{4}{r@{$\implies$}l}}
    \ref{lem:measstep:A'} & \ref{lem:keymeas:An}, &
    \ref{lem:measstep:B'A'} & \ref{lem:keymeas:BnAn}, &
    \ref{lem:measstep:B'G} & \ref{lem:keymeas:BnGn}, &
    \ref{lem:measstep:G'G} & \ref{lem:keymeas:G'nGn},
    \\
    \ref{lem:measstep:piG} & \ref{lem:keymeas:piGn}, &
    \ref{lem:measstep:gammamunuB'} & \ref{lem:keymeas:GammanmuBn}, &
    \ref{lem:measstep:muA'} & \ref{lem:keymeas:muAn}, &
    \ref{lem:measstep:Ggammanu} & \ref{lem:keymeas:GnGammanThetan}.
    \end{tabular}
  \end{center}

  Note also that item~\ref{lem:keymeas:Andomf} follows from the same item for $A_{n-1}$ and the
  already proven item~\ref{lem:keymeas:An} and item~\ref{lem:keymeas:mudomGamman} follows from the
  same item for $\Gamma_{n-1}$.

  \medskip

  Define then the sets
  \begin{align*}
    A' & \df \bigcap_{n\in\NN} A_n,
    &
    B' & \df \bigcap_{n\in\NN} (\id_X\otimes\id_Y\otimes\tau_{\NN_+,n})^{-1}(B_{n-1}).
  \end{align*}
  Since item~\ref{lem:keymeas:An} implies $A_n\subseteq A_{-1}=A$ for every $n\in\NN$,
  item~\ref{lem:keymeas:A'} ($A'\subseteq A)$ holds.

  An inductive application of item~\ref{lem:keymeas:BnAn} shows that item~\ref{lem:keymeas:B'A'}
  ($B'\subseteq((A'\times Y)\cap B)\times Z$) holds.

  Define the measurable partial function $G\colon X\times Y\times Z\pto V\times W$ by letting
  \begin{align*}
    G(x,y,z) & \df \bigl(g(x,y), w\bigr) \qquad (x\in X, y\in Y, z\in Z),\\
    w_n & \df G_{n-1}(x,y,\tau_{\NN_+,n}(z))_n \qquad (n\in\NN_+).
  \end{align*}
  Note that item~\ref{lem:keymeas:BnGn} implies item~\ref{lem:keymeas:B'G} ($B'\subseteq\dom(G)$)
  holds.

  Define now the measurable partial function $G'\colon X\times V\times W\pto Y\times Z$ by letting
  \begin{align*}
    G'(x,v,w)
    & \df
    \Bigl(G'_0\bigl(x,v,\pi_{\NN_+,1}(w)\bigr), z) \qquad (x\in X, v\in V, w\in W),
    \\
    z_n
    & \df
    \widetilde{\tau}_n\Bigl(G'_n\bigl(x,v,\pi_{\NN_+,n+1}(w)\bigr)\Bigr)
    \qquad (n\in\NN),
  \end{align*}
  where $\widetilde{\tau}_n\colon Y\times\widetilde{Z}^n\to\widetilde{Z}$ is the projection onto the
  last coordinate of $\widetilde{Z}^n$. A simple induction shows that item~\ref{lem:keymeas:G'nGn}
  implies item~\ref{lem:keymeas:G'G} ($G'(x,(G(x,y,z)))=(y,z)$ whenever $(x,y,z)\in B'$) holds.

  It is obvious from the definition of $G$ that item~\ref{lem:keymeas:piG} ($\pi(G(x,y,z))=g(x,y)$
  whenever $(x,y,z)\in B'$) holds.

  Since $\Gamma_n(x)=\gamma(x)\otimes\widetilde{\nu}^{n+1}$, an inductive application of
  item~\ref{lem:keymeas:GammanmuBn} implies item~\ref{lem:keymeas:gammamunuB'}
  ($(\gamma[\mu]\otimes\nu)(B')=1$) holds.

  It is also clear that item~\ref{lem:keymeas:muAn} implies that item~\ref{lem:keymeas:muA'}
  ($\mu(A')=1$) holds.

  Recalling now that the partial Markov kernel $\Theta\colon U\pto\cP(V\times W)$ is given by
  \begin{align*}
    \Theta(u) & \df \theta(u)\otimes\eta = \theta(u)\otimes\widetilde{\eta}^{\NN_+} \qquad (u\in U),
  \end{align*}
  by item~\ref{lem:keymeas:GnGammanThetan}, it follows that item~\ref{lem:keymeas:Ggammanu}
  ($G(x,\place,\place)_*(\gamma(x)\otimes\nu)=\Theta(f(x))$) holds.

  \medskip

  Let us now prove that the general case reduces to the previous case. Let $I_Z\colon Z^{\NN_+}\to
  Z$ and $I_W\colon W^{\NN_+}\to W$ be measure-isomorphisms modulo $0$ from
  $(Z^{\NN_+},\nu^{\NN_+})$ to $(Z,\nu)$ and from $(W^{\NN_+},\eta^{\NN_+})$ to $(W,\eta)$.

  Let $\widetilde{G}\colon X\times Y\times Z^{\NN_+}\pto V\times W^{\NN_+}$, $\widetilde{G}'\colon
  X\times V\times W^{\NN+}\pto Y\times Z^{\NN_+}$, $\widetilde{A}'\subseteq A$,
  $\widetilde{B}'\subseteq B\times Z^{\NN_+}$ be the $(G,G',A',B')$ provided by the lemma in the
  previous case.

  It is then straightforward to check that setting
  \begin{gather*}
    \begin{aligned}
      G & \df (\id_V\otimes I_W)\comp\widetilde{G}\comp(\id_X\otimes\id_Y\otimes I_Z^{-1}), &
      G' & \df (\id_Y\otimes I_Z)\comp\widetilde{G}'\comp(\id_X\otimes\id_V\otimes I_W^{-1}),
    \end{aligned}
    \\
    \begin{aligned}
      A' & \df\widetilde{A}', &
      B' & \df (\id_X\otimes\id_Y\otimes I_Z)(\widetilde{B}')\cap(X\times Y\times\dom(I_Z^{-1}))
    \end{aligned}
  \end{gather*}
  gives the desired result.

  \medskip

  We now prove that the equivariant version of the lemma involving the group $H$ reduces to the
  non-equivariant version.

  First let
  \begin{align*}
    A_H & \df \bigcap_{\sigma\in H} (\sigma\cdot A)\cap X', &
    B_H & \df \bigcap_{\sigma\in H} (\sigma\cdot B)\cap (X'\times Y).
  \end{align*}
  It will also be convenient to set $X_H\df\bigcup_{\sigma\in H}\sigma\cdot X'$.

  Note that since the orbits $(\sigma\cdot X')_{\sigma\in H}$ of $X'$ are pairwise disjoint and
  $\mu(X_H)=1$, it follows that
  \begin{align*}
    \mu(A_H) = \gamma[\mu](B_H) = \mu(X') & = \frac{1}{\lvert H\rvert}.
  \end{align*}

  Note that for all measurable $C\subseteq X$ and all measurable $D\subseteq Y$, we have
  \begin{align*}
    \gamma[\mu](C\times D)
    & =
    \int_C \gamma(x)(D)\ d\mu(x)
    =
    \int_C \gamma(\sigma\cdot x)(\sigma\cdot D)\ d\mu(x)
    \\
    & =
    \int_{\sigma\cdot C} \gamma(z)(\sigma\cdot D)\ d\mu(z)
    =
    \gamma[\mu](\sigma\cdot(C\times D)),
  \end{align*}
  where the second equality follows since $\gamma$ is $H$-equivariant and the third equality follows
  since $\mu$ is $H$-invariant. Thus $\gamma[\mu]$ is $H$-invariant. In particular, since
  $\gamma[\mu](B)=1$, it follows that $\gamma[\mu](\bigcap_{\sigma\in H}\sigma\cdot B)=1$.

  We now let $\mu_H$ be the probability measure on $X'$ that is the rescaled restriction of $\mu$,
  that is, $\mu_H\df\lvert H\rvert\cdot\mu\rest_{X'}$. It is also clear from definitions that
  $\gamma\rest_{X'}[\mu_H]=\lvert H\rvert\cdot\gamma[\mu]\rest_{X'\times Y}$. In particular, we have
  \begin{align*}
    \gamma\rest_{X'}[\mu_H](B_H)
    & =
    \lvert H\rvert\cdot
    \gamma[\mu]\left(\bigcap_{\sigma\in H}(\sigma\cdot B)\cap (X'\times Y)\right)
    =
    \lvert H\rvert\cdot
    \gamma[\mu](X'\times Y)
    =
    1,
  \end{align*}
  where the second equality follows since $\gamma[\mu](\bigcap_{\sigma\in H}\sigma\cdot B)=1$ as
  $\gamma[\mu](B)=1$ and $\gamma[\mu]$ is $H$-invariant.

  This means that we can apply the current lemma without the equivariance/invariance conclusions
  with the following choice of parameters:
  \begin{align*}
    (X,\mu) & \leftarrow (X',\mu_H), &
    (Z,\nu) & \leftarrow (Z,\nu), &
    Y & \leftarrow Y,
    \\
    U & \leftarrow U, &
    V & \leftarrow V, &
    (W,\eta) & \leftarrow (W,\eta),
    \\
    f & \leftarrow f\rest_{X'}, &
    g & \leftarrow g\rest_{X'\times Y}, &
    \gamma & \leftarrow \gamma\rest_{X'}
    \\
    \theta & \leftarrow \theta, &
    A & \leftarrow A_H, &
    B & \leftarrow B_H.
  \end{align*}

  Let then $(G_H,G'_H,A'_H,B'_H)$ be the $(G,G',A',B')$ provided by the lemma with
  the choice of parameters above.

  We then define $(G,G',A',B')$ by
  \begin{align*}
    G(x,y,z)
    & \df
    \tau_x^{-1}\cdot G_H\bigl(\tau_x\cdot (x,y,z)\bigr)
    \qquad (x\in X_H, y\in Y, z\in Z)
    \\
    G'(x,v,w)
    & \df
    \tau_x^{-1}\cdot G'_H\bigl(\tau_x\cdot (x,v,w)\bigr),
    \qquad (x\in X_H, v\in V, w\in W),
    \\
    A' & \df \bigcup_{\sigma\in H} \sigma\cdot A'_H,
    \\
    B' & \df \bigcup_{\sigma\in H} \sigma\cdot B'_H,
  \end{align*}
  where $\tau_x\in H$ is the unique group element such that $\tau_x\cdot x\in X'$ (recall that the
  orbits $(\sigma\cdot X')_{\sigma\in H}$ of $X'$ are pairwise disjoint).

  It is obvious that $A'$ and $B'$ are $H$-invariant (i.e., item~\ref{lem:keymeas:invariance}
  holds). Note also that for $x\in X_H$ and $\sigma\in H$, we have $\tau_{\sigma\cdot x} =
  \tau_x\comp\sigma^{-1}$, from which it follows that
  \begin{align*}
    G\bigl(\sigma\cdot(x,y,z)\bigr)
    & =
    \tau_{\sigma\cdot x}^{-1}\cdot G_H\bigl((\tau_{\sigma\cdot x}\comp\sigma)\cdot(x,y,z)\bigr)
    =
    \sigma\cdot\Bigl(\tau_x^{-1}\cdot G_H\bigl(\tau_x\cdot(x,y,z)\bigr)\Bigr)
    =
    \sigma\cdot G(x,y,z),
    \\
    G'\bigl(\sigma\cdot(x,v,w)\bigr)
    & =
    \tau_{\sigma\cdot x}^{-1}\cdot G'_H\bigl((\tau_{\sigma\cdot x}\comp\sigma)\cdot(x,v,w)\bigr)
    =
    \sigma\cdot\Bigl(\tau_x^{-1}\cdot G'_H\bigl(\tau_x\cdot(x,v,w)\bigr)\Bigr)
    =
    \sigma\cdot G'(x,v,w),
  \end{align*}
  for every $y\in Y$, every $z\in Z$, every $v\in V$ and every $w\in W$. Thus, $G$ and $G'$ are
  $H$-equivariant. Finally, the definition of $\Theta$ along with the fact that $\theta$ is
  $H$-equivariant and $\eta$ is $H$-invariant (as the $H$-action on $W$ is trivial) implies that
  $\Theta$ is $H$-equivariant. Thus, item~\ref{lem:keymeas:equivariance} holds.

  Let us now check that all items not concerning invariance/equivariance follow from their
  counterparts for $(G_H,G'_H,A'_H,B'_H)$.

  Item~\ref{lem:keymeas:A'} is straightforward:
  \begin{align*}
    A'
    & =
    \bigcup_{\sigma\in H} \sigma\cdot A'_H
    \subseteq
    \bigcup_{\sigma\in H} \sigma\cdot A_H
    =
    \left(\bigcap_{\sigma\in H} \sigma\cdot A\right)\cap X_H
    \subseteq
    A.
  \end{align*}

  Item~\ref{lem:keymeas:B'A'} is also straightforward:
  \begin{align*}
    B'
    & =
    \bigcup_{\sigma\in H} \sigma\cdot B'_H
    \subseteq
    \bigcup_{\sigma\in H} \sigma\cdot \Bigl(\bigl((A'_H\times Y)\cap B_H\bigr)\times Z\Bigr)
    \\
    & =
    \left(\left(\left(\bigcup_{\sigma\in H} \sigma\cdot A'_H\right)\times Y\right)
    \cap B_H\right)\times Z
    \\
    & =
    \left((A'\times Y)\cap
    \left(\bigcap_{\sigma\in H}\sigma\cdot B\right)\cap (X_H\times Y)\right)
    \times Z
    \\
    & \subseteq
    \bigl((A'\times Y)\cap B\bigr)\times Z.
  \end{align*}

  Item~\ref{lem:keymeas:B'G} is straightforward too:
  \begin{align*}
    \dom(G)
    & =
    \bigcup_{\sigma\in H}\sigma\cdot\dom(G_H)
    \supseteq
    \bigcup_{\sigma\in H} \sigma\cdot B'_H
    =
    B'.
  \end{align*}

  For item~\ref{lem:keymeas:G'G}, if $(x,y,z)\in B'$, then we have $\tau_x\cdot(x,y,z)\in B'_H$,
  so we get
  \begin{align*}
    G'\bigl(x,G(x,y,z)\bigr)
    & =
    \tau_x^{-1}\cdot\Biggl(
    G'\biggl(\tau_x\cdot\Bigl(x,\tau_x^{-1}\cdot G\bigl(\tau_x\cdot(x,y,z)\bigr)\Bigr)\biggr)
    \Biggr)
    \\
    & =
    \tau_x^{-1}\cdot\biggl(
    G'\Bigl(\tau_x\cdot x, G\bigl(\tau_x\cdot x, \tau_x\cdot(y,z)\bigr)\Bigr)
    \biggr)
    \\
    & =
    \tau_x^{-1}\cdot\bigl(\tau_x\cdot(y,z)\bigr)
    =
    (y,z).
  \end{align*}

  For item~\ref{lem:keymeas:piG}, if $(x,y,z)\in B'$, then (since $\tau_x\cdot(x,y,z)\in B'_H$)
  we have
  \begin{align*}
    \pi\bigl(G(x,y,z)\bigr)
    & =
    \pi\biggl(\tau_x^{-1}\cdot\Bigl(G\bigl(\tau_x\cdot(x,y,z)\bigr)\Bigr)\biggr)
    =
    \tau_x^{-1}\cdot\pi\Bigl(G\bigl(\tau_x\cdot(x,y),\tau_x\cdot z\bigr)\Bigr)
    \\
    & =
    \tau_x^{-1}\cdot g\bigl(\tau_x\cdot(x,y)\bigr)
    =
    g(x,y),
  \end{align*}
  where the last equality follows since $g$ is $H$-equivariant.

  For item~\ref{lem:keymeas:gammamunuB'}, we have
  \begin{align*}
    (\gamma[\mu]\otimes\nu)(B')
    & =
    (\gamma[\mu]\otimes\nu)\left(\bigcup_{\sigma\in H}\sigma\cdot B'_H\right)
    =
    \lvert H\rvert\cdot(\gamma[\mu]\otimes\nu)(B'_H)
    =
    (\gamma\rest_{X'}[\mu_H]\otimes\nu)(B'_H)
    =
    1,
  \end{align*}
  where the second equality follows since $\gamma[\mu]$ is $H$-invariant, the $H$-action on $Z$ is
  trivial and the orbits $(\sigma\cdot X')_{\sigma\in H}$ of $X'$ are pairwise disjoint.

  For item~\ref{lem:keymeas:muA'}, we have
  \begin{align*}
    \mu(A')
    & =
    \mu\left(\bigcup_{\sigma\in H}\sigma\cdot A'_H\right)
    =
    \lvert H\rvert\cdot\mu(A'_H)
    =
    \mu_H(A'_H)
    =
    1,
  \end{align*}
  where the second equality follows since $\mu$ is $H$-invariant and $A'_H\subseteq A_H\subseteq X'$
  and the orbits $(\sigma\cdot X')_{\sigma\in H}$ of $X'$ are pairwise disjoint.

  Finally, for item~\ref{lem:keymeas:Ggammanu}
  ($G(x,\place,\place)_*(\gamma(x)\otimes\nu)=\Theta(f(x))$ whenever $x\in A'$), we need to show
  that if $x\in A'$ and $(\rn{y},\rn{z})\sim(\gamma(x)\otimes\nu)$ then
  $G(x,\rn{y},\rn{z})\sim\Theta(f(x))$.

  Indeed, first note that since $\gamma$ is $H$-equivariant and $\nu$ is $H$-invariant (as the
  $H$-action on $Z$ is trivial), it follows that
  \begin{align*}
    \tau_x\cdot(\rn{y},\rn{z})
    & \sim
    \tau_x\cdot(\gamma(x)\otimes\nu)
    =
    \gamma(\tau_x\cdot x)\otimes\nu
  \end{align*}
  so we get
  \begin{align*}
    G(x,\rn{y},\rn{z})
    & =
    \tau_x^{-1}\cdot G_H\bigl(\tau_x\cdot(x,\rn{y},\rn{z})\bigr)
    \sim
    \tau_x^{-1}\cdot\Theta\bigl(f\rest_{X'}(\tau_x\cdot x)\bigr)
    \\
    & =
    \tau_x^{-1}\cdot(\theta\bigl(f(\tau_x\cdot x)\bigr)\otimes\eta)
    =
    \theta(f(x))\otimes\eta
    =
    \Theta(f(x)),
  \end{align*}
  where the distributional equality follows from item~\ref{lem:keymeas:Ggammanu} applied to
  $(G_H,\gamma\rest_{X'})$ as $\tau_x\cdot x\in X'$, the second equality follows from the definition
  of $\Theta$, the third equality follows since $f$ and $\theta$ are $H$-equivariant and $\eta$ is
  $H$-invariant (as the $H$-action on $W$ is trivial).
\end{proof}

\section{Amalgamating types}
\label{sec:amlg}

In this section we finally define amalgamating types and prove its amalgamation property. Before we
give the definition, let us point out that one more technical detail was not covered in
Section~\ref{sec:meas}. Recall that the whole purpose of sampling via amalgamating types is that we
want $(\ell-1)$-independent $d$-Euclidean structures $\cN$ satisfying $\UCouple[\ell]$ to have a
unique amalgamating $\ell$-type. We have seen in Proposition~\ref{prop:typeuniqueness} that such
$\cN$ have a unique overlapping $\ell$-type, which in particular implies that they also have a
unique dissociated $\ell$-type; since our intuition is that amalgamating $\ell$-types carry an
intermediate amount of information between their dissociated and overlapping counterparts, we would
expect that they are also unique. However, since to make Lemma~\ref{lem:keymeas} go through we had
encoded (potential) extra information within the set $[0,1]$, it is impossible to get uniqueness
from the output of Lemma~\ref{lem:keymeas}.

We solve this issue by being pragmatic again: since Proposition~\ref{prop:typeuniqueness} says that
an $(\ell-1)$-independent $d$-Euclidean structure $\cN$ satisfying $\UCouple[\ell]$ has a unique
overlapping $\ell$-type, we can define amalgamating $A$-types in $\cN$ by splitting into cases: if
$\lvert A\rvert\leq\ell$, then we simply let the amalgamating $A$-type be the overlapping $A$-type
and if $\lvert A\rvert > \ell$, we actually apply Lemma~\ref{lem:keymeas} (in both cases, we also
add information about lower amalgamating types). Since we also have to keep track of the additional
countable set $V$ of the dissociated and overlapping types and the same $\cN$ can satisfy
$(\ell-1)$-independence and $\UCouple[\ell]$ for multiple different $\ell$, we call this
construction the amalgamating $(A,V,\ell)$-type.

Finally, we point out that since Lemma~\ref{lem:keymeas} requires the addition of dummy variables,
our definition of $(A,V,\ell)$-type requires us to go from $\cN$ to an auxiliary $\widehat{\cN}$
that has extra dummy variables.

Before we formalize this construction, let us introduce some terminology regarding partial
functions: we say that a diagram of partial functions is \emph{weakly commutative} if for any two
composition paths $(\alpha_1,\ldots,\alpha_n)$ and $(\beta_1,\ldots,\beta_m)$ between from the same
$A$ to the same $B$, we have
\begin{align*}
  (\alpha_n\comp\cdots\comp\alpha_1)\rest_D & = (\beta_m\comp\cdots\comp\beta_1)_D,
\end{align*}
where
\begin{align*}
  D & \df \dom(\alpha_n\comp\cdots\comp\alpha_1)\cap\dom(\beta_m\comp\cdots\comp\beta_1)_D,
\end{align*}
that is, the two compositions match in their common domain (but might not match as partial functions
because of points in which only one of them is defined); weak equivariance is defined
similarly. Since in all of our weakly commutative diagrams the partial functions will be defined
almost everywhere, these can be thought of ``almost everywhere commutative diagrams''.

We will also need the following invariant/equivariant version of Lemma~\ref{lem:mk}:

\begin{lemma}\label{lem:mkequiv}
  Let $\Omega_X$, $\Omega_Y$ and $\Omega_Z$ be Borel spaces, let $\rn{X}$, $\rn{Y}$ and $\rn{Z}$ be
  random variables with values in $\Omega_X$, $\Omega_Y$ and $\Omega_Z$, respectively. Suppose also
  that $H$ is a countable group acting measurably on $\Omega_X$, $\Omega_Y$ and $\Omega_Z$ and that
  the distribution $\mu_Z$ of $\rn{Z}$ is $H$-invariant. Then the following are equivalent:
  \begin{enumerate}
  \item\label{lem:mkequiv:ci} $\rn{X}$ is conditionally independent from $\rn{Y}$ given $\rn{Z}$ and the
    distribution of $(\rn{X},\rn{Z})$ is $H$-invariant.
  \item\label{lem:mkequiv:product} There exists an $H$-equivariant partial Markov kernel
    $\rho\colon\Omega_Z\pto\cP(\Omega_X)$ with $\mu_Z(\dom(\rho))=1$ such that if
    $\rn{\widetilde{X}}$ is picked at random according to $\rho(\rn{Z})$ conditionally on $\rn{Z}$
    and conditionally independently from $\rn{Y}$, then
    $(\rn{\widetilde{X}},\rn{Y},\rn{Z})\sim(\rn{X},\rn{Y},\rn{Z})$.
  \end{enumerate}
\end{lemma}

\begin{proof}
  The implication~\ref{lem:mkequiv:product}$\implies$\ref{lem:mkequiv:ci} is obvious as
  $\rn{\widetilde{X}}$ and $\rn{Y}$ are conditionally independent given $\rn{Z}$ and the facts that
  $\rho$ is $H$-equivariant and that $\mu_Z$ is $H$-invariant imply that the distribution of
  $(\rn{\widetilde{X}},\rn{Z})$ is $H$-invariant.

  \medskip

  For the implication~\ref{lem:mkequiv:ci}$\implies$\ref{lem:mkequiv:product}, we apply
  Lemma~\ref{lem:mk} to obtain a (not necessarily $H$-equivariant) Markov kernel
  $\rho'\colon\Omega_Z\to\cP(\Omega_X)$ such that if $\rn{\widetilde{X}}$ is picked at random
  according to $\rho'(\rn{Z})$ conditionally on $\rn{Z}$ and conditionally independently from
  $\rn{Y}$, then $(\rn{\widetilde{X}},\rn{Y},\rn{Z})\sim(\rn{X},\rn{Y},\rn{Z})$.

  We now argue that $\rho'$ is $\mu_Z$-almost everywhere $H$-equivariant, that is, we argue that for
  $\mu_Z$-almost every $z\in\Omega_Z$ and every $\tau\in H$, we have $\rho'(\tau\cdot z) =
  \tau\cdot\rho'(z)$. For this, it suffices to show that for every measurable $A\subseteq\Omega_Z$
  and every measurable $B\subseteq\Omega_X$, we have
  \begin{align*}
    \int_A \rho'(\tau\cdot z)(B)\ d\mu_Z(z) & = \int_A (\tau\cdot\rho'(z))(B)\ d\mu_Z(z).
  \end{align*}

  But indeed, note that since $\mu_Z$ is $H$-invariant, we have
  \begin{align*}
    \int_A \rho'(\tau\cdot z)(B)\ d\mu_Z(z)
    & =
    \int_{\tau^{-1}(A)} \rho'(z)(B)\ d\mu_Z(z)
    =
    \PP[\rn{Z}\in\tau^{-1}(A), \rn{\widetilde{X}}\in B]
    \\
    & =
    \PP[\rn{Z}\in A, \rn{\widetilde{X}}\in \tau(B)]
    =
    \int_A \rho'(z)(\tau(B))\ d\mu_Z(z)
    =
    \int_A (\tau\cdot\rho'(z))(B)\ d\mu_Z(z),
  \end{align*}
  where the third equality follows since $(\rn{\widetilde{X}},\rn{Z})\sim(\rn{X},\rn{Z})$ and this
  latter distribution is $H$-invariant.

  Thus $\rho'$ is $\mu_Z$-almost everywhere $H$-equivariant. Since $H$ is countable, this means that
  the set
  \begin{align*}
    D & \df \{z\in\Omega_Z \mid \forall\tau\in H, \rho'(\tau\cdot z) = \tau\cdot\rho'(z)\}
  \end{align*}
  has $\mu_Z$-measure $1$, so letting $\rho\colon\Omega_Z\pto\cP(\Omega_X)$ be the partial Markov
  kernel defined to be equal to $\rho$ on $D$ and not defined anywhere else yields the result.
\end{proof}

We are now finally ready to define amalgamating types.

\begin{definition}\label{def:amlg}
  Let $A$ be a finite set, let $V$ be a countably infinite\footnote{Differently from dissociated or
    overlapping types, amalgamating types do not make sense when $V$ is finite. This is because for
    the definition to go through (see Lemma~\ref{lem:amlg} below), we make use of
    Proposition~\ref{prop:weakamalgdsctovlp}, which requires $V$ to be infinite.} set disjoint from
  $A$ and let $\ell\in\NN$. Let also $\Omega=(X,\mu)$ be an atomless standard probability space and
  $d\in\NN$. The space of \emph{amalgamating $(A,V,\ell)$-types} is defined inductively as
  \begin{align*}
    \cS_{A,V}^{(d),\amlg,\ell}(\Omega)
    & \df
    \begin{dcases*}
      \left(\prod_{B\in r(A,\lvert A\rvert-1)}\cS_{B,V}^{(d),\amlg}(\Omega)\right)
      \times\cS_{A,V}^{(d),\ovlp}(\Omega),
      & if $\lvert A\rvert\leq\ell$,
      \\
      \left(\prod_{B\in r(A,\lvert A\rvert-1)}\cS_{B,V}^{(d),\amlg}(\Omega)\right)
      \times\cS_{A,V}^{(d),\dsct}(\Omega)\times[0,1],
      & if $\lvert A\rvert>\ell$,
    \end{dcases*}
  \end{align*}
  equipped with the product topology (which makes it a Polish space, hence a Borel space, see
  Proposition~\ref{prop:toptypes}\ref{prop:toptypes:Pol}). We think of
  $\cS_{A,V}^{(d),\ovlp}(\Omega)$ and $\cS_{A,V}^{(d),\dsct}(\Omega)\times[0,1]$ in the above as the
  coordinate indexed by $A$.

  We also define $\pi^{\amlg,\ell}_{A,V}\colon\cS_{A,V}^{(d),\amlg,\ell}\to\cS_{A,V}^{(d),\dsct}$ by
  letting it be the natural projection when $\lvert A\rvert > \ell$ and letting it be the
  composition of the natural projection onto $\cS_{A,V}^{(d),\ovlp}$ with
  $\pi^{\ovlp}_{A,V}\colon\cS_{A,V}^{(d),\ovlp}\to\cS_{A,V}^{(d),\dsct}$ when $\lvert
  A\rvert\leq\ell$.

  We further define $M^{\amlg,\ell}_{A,V}\colon\cS_{A,V}^{(d),\amlg,\ell}\to\cK_A$ as the
  composition $M^{\amlg,\ell}_{A,V}\df M^{\dsct}_{A,V}\comp\pi^{\amlg,\ell}_{A,V}$; as before, for
  an amalgamating type $p\in\cS_{A,V}^{(d),\amlg,\ell}$, we call $M^{\amlg,\ell}_{A,V}(p)$ the
  \emph{underlying model} of $p$.

  For an injection $\alpha\colon A\to B$ with $B$ disjoint from $V$ we define a map
  $\alpha^*\colon\cS_{B,V}^{(d),\amlg,\ell}\to\cS_{A,V}^{(d),\amlg,\ell}$ inductively on the size of
  $A$ by
  \begin{align*}
    \alpha^*(p)_C
    & \df
    \begin{dcases*}
      \alpha\down_C^*(p_{\alpha(C)}), & if $\lvert C\rvert < \lvert A\rvert$,\\
      \alpha^*(p_B), & if $C = A$ and $\lvert A\rvert=\lvert B\rvert$,\\
      \alpha\down_A^*((p_{\alpha(A)})_{\alpha(A)}), & if $C = A$ and $\lvert A\rvert < \lvert B\rvert$,
    \end{dcases*}
  \end{align*}
  for every $p\in\cS_{B,V}^{(d),\amlg,\ell}$ and every $C\in r(A)$, where in the second case,
  $\alpha^*$ is either the map $\cS_{B,V}^{(d),\ovlp}\to\cS_{A,V}^{(d),\ovlp}$ defined
  contra-variantly in Definition~\ref{def:contravartypes} when $\lvert B\rvert\leq\ell$ or the
  extension of the map $\cS_{B,V}^{(d),\dsct}\to\cS_{A,V}^{(d),\dsct}$ to a map
  $\cS_{B,V}^{(d),\dsct}\times[0,1]\to\cS_{A,V}^{(d),\dsct}\times[0,1]$ by acting identically on the
  second coordinate when $\lvert B\rvert > \ell$; in the third case, $\alpha\down_A^*$ is defined
  similarly and we take the $\alpha(A)$th coordinate twice, the first time corresponds to the
  projection $\cS_{B,V}^{(d),\amlg,\ell}\to\cS_{\alpha(A),V}^{(d),\amlg,\ell}$ and the second time
  corresponds to the projection of $\cS_{\alpha(A),V}^{(d),\amlg,\ell}$ onto either
  $\cS_{\alpha(A),V}^{(d),\ovlp}$ when $\lvert A\rvert\leq\ell$ or
  $\cS_{\alpha(A),V}^{(d),\dsct}\times[0,1]$ when $\lvert A\rvert > \ell$.

  It is straightforward to check that the fact that the maps of Definition~\ref{def:contravartypes}
  are contra-variant (see Lemma~\ref{lem:contravartypes}) implies that the maps defined above are
  also contra-variant.

  In turn, recall that this also contra-variantly induces maps
  \begin{align*}
    \alpha^*\colon\prod_{C\in r(B,\lvert B\rvert-1)}\cS_{C,V}^{(d),\amlg,\ell}\to
    \prod_{C\in r(A,\lvert A\rvert-1)}\cS_{C,V}^{(d),\amlg,\ell}
  \end{align*}
  given by
  \begin{align*}
    \alpha^*(p)_C & \df (\alpha\down_C)^*(p_{\alpha(C)})
    \qquad (p\in\cS_{B,V}^{(d),\amlg,\ell}, C\in r(A,\lvert A\rvert-1)).
  \end{align*}

  Finally, it is obvious that the maps $\pi^{\amlg,\ell}$ and $M^{\amlg,\ell}$ are equivariant, that
  is, we have
  \begin{align}\label{eq:equivpiamlgMamlg}
    \pi^{\amlg,\ell}_{A,V}\comp\alpha^* & = \alpha^*\comp\pi^{\amlg,\ell}_{B,V}, &
    M^{\amlg,\ell}_{A,V}\comp\alpha^* & = \alpha^*\comp M^{\amlg,\ell}_{B,V}, &
  \end{align}

  \medskip

  Let now $\cN$ be a $d$-Euclidean structure in $\cL$ over $\Omega$ and let $\widehat{\cN}$ be the
  $d$-Euclidean structure in $\cL$ over $\Omega^2$ given by
  \begin{align}\label{eq:amlgwidehatcN}
    \widehat{\cN}_P & \df \{(x,x')\in\cE_{k(P)}^{(d)}(\Omega)\times\cE_{k(P)}(\Omega) \mid x\in\cN_P\}
    \qquad (P\in\cL).
  \end{align}
  (In the above, all order variables are contained within $x$.)

  Clearly $\phi_\cN=\phi_{\widehat{\cN}}$. In fact, note that for every
  $(x,x')\in\cE_A^{(d)}(\Omega)\times\cE_A(\Omega)$ we have
  \begin{equation}\label{eq:dummy}
    \begin{aligned}
      \pp^{\dsct,\widehat{\cN}}_{A,V}(x,x')(y,y')
      & =
      \pp^{\dsct,\cN}_{A,V}(x)(y)
      \qquad ((y,y')\in\cE_{A,V}^{(d),\dsct}(\Omega)\times\cE_{A,V}^{(0),\dsct}(\Omega)),
      \\
      \pp^{\ovlp,\widehat{\cN}}_{A,V}(x,x')(y,y')
      & =
      \pp^{\ovlp,\cN}_{A,V}(x)(y)
      \qquad ((y,y')\in\cE_{A,V}^{(d),\ovlp}(\Omega)\times\cE_{A,V}^{(0),\ovlp}(\Omega)).
    \end{aligned}
  \end{equation}

  Suppose now that $\cN$ is $(\ell-1)$-independent and satisfies $\UCouple[\ell]$ and note that
  $\widehat{\cN}$ inherits the same properties. We will define
  $\pp^{\amlg,\ell,\widehat{\cN}}_{A,V}$, $\qq^{\amlg,\ell,\widehat{\cN}}_{A,V}$ and
  $\Theta^{\amlg,\ell,\widehat{\cN}}_{A,V}$ by induction on the size of $A$ so that as to satisfy
  the following properties:
  \begin{enumerate}[label={\alph*.}, ref={(\alph*)}]
  \item\label{def:amlg:pp}
    $\pp^{\amlg,\ell,\widehat{\cN}}_{A,V}\colon\cE_A^{(d)}(\Omega^2)\pto\cS_{A,V}^{(d),\amlg,\ell}(\Omega^2)$
    is a measurable partial function that is defined $\mu^{(d)}$-almost everywhere. Furthermore, if
    $\lvert A\rvert\leq\ell$, then $\pp^{\amlg,\ell,\widehat{\cN}}$ is constant on its
    domain. ($\pp^{\amlg,\ell,\widehat{\cN}}_{A,V}$ gives the amalgamating $(A,V,\ell)$-type of
    $x\in\cE_A^{(d)}(\Omega^2)$.)
  \item\label{def:amlg:dompp} When the amalgamating $(A,V,\ell)$-type is defined, then all lower
    amalgamating types are defined: if $x\in\dom(\pp^{\amlg,\ell,\widehat{\cN}}_{A,V})$, then for
    every $B\in r(A,\lvert A\rvert-1)$, we have
    $\iota_{B,A}^*(x)\in\dom(\pp^{\amlg,\ell,\widehat{\cN}}_{B,V})$ (recall that $\iota_{B,A}\colon
    B\to A$ is the inclusion map).
  \item\label{def:amlg:qq} $\qq^{\amlg,\ell,\widehat{\cN}}_{A,V}\colon\cE_{A,\lvert
    A\rvert-1}^{(d)}(\Omega^2)\times\cS_{A,V}^{(d),\amlg,\ell}(\Omega^2)\pto\cS_{A,V}^{(d),\ovlp}(\Omega^2)$
    is a measurable partial function such that for every
    $x\in\dom(\pp^{\amlg,\ell,\widehat{\cN}}_{A,V})$, we have
    \begin{align*}
      \qq^{\amlg,\ell,\widehat{\cN}}_{A,V}
      \bigl(\tau_A(x), \pp^{\amlg,\ell,\widehat{\cN}}_{A,V}(x)\bigr)
      & =
      \pp^{\ovlp,\widehat{\cN}}_{A,V}(x),
    \end{align*}
    where $\tau_A\colon\cE_A^{(d)}(\Omega^2)\to\cE_{A,\lvert A\rvert-1}^{(d)}(\Omega^2)$ is the
    natural projection ($\qq^{\amlg,\ell,\widehat{\cN}}_{A,V}$ is the function that retrieves the
    overlapping type from the amalgamating type along with lower pointwise data).
  \item\label{def:amlg:Theta} $\Theta^{\amlg,\ell,\widehat{\cN}}_{A,V}\colon\prod_{B\in r(A,\lvert
    A\rvert-1)}\cS_{B,V}^{(d),\amlg,\ell}(\Omega^2)\pto\cP(\cS_{A,V}^{(d),\amlg,\ell}(\Omega^2))$ is
    a partial Markov kernel that is $S_A$-equivariant and is such that for the natural projection
    $\tau_A\colon\cE_A^{(d)}(\Omega^2)\to\cE_{A,\lvert A\rvert-1}^{(d)}(\Omega^2)$, if $\rn{x}$ is
    sampled in $\cE_A^{(d)}(\Omega^2)$ according to $\mu^{(d)}$, then the conditional distribution
    of $\pp^{\amlg,\ell,\widehat{\cN}}_{A,V}(\rn{x})$ given $\tau_A(\rn{x})$ is
    \begin{align*}
      \Theta^{\amlg,\ell,\widehat{\cN}}_{A,V}\Bigl(
      \pp^{\amlg,\ell,\widehat{\cN}}_{B,V}\bigl(\iota_{B,A}^*(\rn{x})\bigr)
      \mathrel{\Big\vert}
      B\in r(A,\lvert A\rvert-1)
      \Bigr).
    \end{align*}

    In particular, $\pp^{\amlg,\ell,\widehat{\cN}}_{A,V}(\rn{x})$ is conditionally independent from
    $\tau_A(\rn{x})$ given $(\pp^{\amlg,\ell,\widehat{\cN}}_{B,V}(\iota_{B,A}^*(\rn{x}))\mid B\in
    r(A,\lvert A\rvert-1))$ ($\Theta^{\amlg,\ell,\widehat{\cN}}_{A,V}$ then gives how to sample the
    amalgamating $(A,V,\ell)$-type from the lower amalgamating types).

    Furthermore, if $\lvert A\rvert\leq\ell$, then $\Theta^{\amlg,\ell,\widehat{\cN}}_{A,V}$ is the
    constant function that maps every point to the distribution concentrated on the unique point in
    the image of $\pp^{\amlg,\ell,\widehat{\cN}}_{A,V}$ (see item~\ref{def:amlg:pp} above).
  \item\label{def:amlg:commutative} For every injection $\beta\colon A'\to A$ with $A'$ also
    disjoint from $V$, the diagram
    
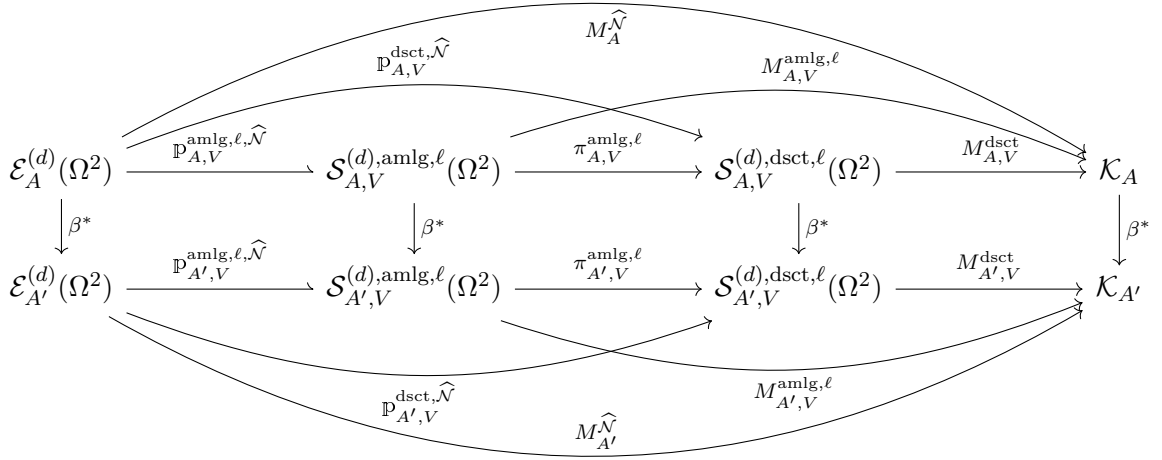
\begin{equation*}
      \begin{tikzcd}[column sep={2.5cm}]
        \cE_A^{(d)}(\Omega^2)
        \arrow[r, harpoon, "\pp^{\amlg,\ell,\widehat{\cN}}_{A,V}"]
        \arrow[rr, bend left=20, "\pp^{\dsct,\widehat{\cN}}_{A,V}"]
        \arrow[rrr, bend left=30, "M^{\widehat{\cN}}_A"']
        \arrow[d, "\beta^*"]
        &
        \cS_{A,V}^{(d),\amlg,\ell}(\Omega^2)
        \arrow[r, "\pi^{\amlg,\ell}_{A,V}"]
        \arrow[rr, bend left=20, "M^{\amlg,\ell}_{A,V}"]
        \arrow[d, "\beta^*"]
        &
        \cS_{A,V}^{(d),\dsct,\ell}(\Omega^2)
        \arrow[r, "M^{\dsct}_{A,V}"]
        \arrow[d, "\beta^*"]
        &
        \cK_A
        \arrow[d, "\beta^*"]
        \\
        \cE_{A'}^{(d)}(\Omega^2)
        \arrow[r, harpoon, "\pp^{\amlg,\ell,\widehat{\cN}}_{A',V}"]
        \arrow[rr, bend right=20, "\pp^{\dsct,\widehat{\cN}}_{A',V}"']
        \arrow[rrr, bend right=30, "M^{\widehat{\cN}}_{A'}"]
        &
        \cS_{A',V}^{(d),\amlg,\ell}(\Omega^2)
        \arrow[r, "\pi^{\amlg,\ell}_{A',V}"]
        \arrow[rr, bend right=20, "M^{\amlg,\ell}_{A',V}"']
        &
        \cS_{A',V}^{(d),\dsct,\ell}(\Omega^2)
        \arrow[r, "M^{\dsct}_{A',V}"]
        &
        \cK_{A'}
      \end{tikzcd}
    \end{equation*}
    is weakly commutative.
  \end{enumerate}
  
  The final property is weak amalgamation, which is a joint property of all sets of size $\lvert
  A\rvert$:
  \begin{enumerate}[resume*]
  \item\label{def:amlg:weakamalg} Weak amalgamation: if $B$ is a set with $\lvert B\rvert=\lvert
    A\rvert+1$ and $B\cap V=\varnothing$ and $\rn{x}$ is sampled in $\cE_B^{(d)}(\Omega^2)$
    according to $\mu^{(d)}$, then $\pp^{\dsct,\widehat{\cN}}_{B,V}(\rn{x})$ is conditionally
    independent from $\tau_B(\rn{x})$ given
    \begin{align*}
      \bigl(\pp^{\amlg,\ell,\widehat{\cN}}_{A',V}(\iota_{A',B}(\rn{x}))
      \mathrel{\big\vert}
      A'\in r(B,\lvert B\rvert-1)\bigr),
    \end{align*}
    where $\tau_B\colon\cE_B^{(d)}(\Omega^2)\to\cE_{B,\lvert B\rvert-1}^{(d)}(\Omega^2)$ is the
    natural projection.
  \end{enumerate}

  In this definition, we focus on describing the construction, deferring the proof of some of the
  claimed properties to Lemmas~\ref{lem:amlgcomm} and~\ref{lem:amlg}.

  \medskip

  First we consider the case when $\lvert A\rvert\leq\ell$. Since $\widehat{\cN}$ is
  $(\ell-1)$-independent and satisfies $\UCouple[\ell]$, Proposition~\ref{prop:typeuniqueness} says
  that there exists an overlapping $(A,V)$-type $q_A\in\cS_{A,V}^{(d),\ovlp}(\Omega^2)$ such that
  for $\mu^{(d)}$-almost every $x\in\cE_A^{(d)}(\Omega^2)$, we have
  $\pp^{\ovlp,\widehat{\cN}}_{A,V}(x) = q_A$. Note also that by inductive hypothesis, we know that
  for every $B\in r(A,\lvert A\rvert-1)$, for almost every $x\in\cE_A^{(d)}(\Omega^2)$, we have
  $\iota_{B,A}^*(x)\in\dom(\pp^{\amlg,\ell,\widehat{\cN}}_{B,V})$. Thus, the set
  \begin{align*}
    D \df
    \{x\in\cE_A^{(d)}(\Omega^2)
    & \mid
    \forall\beta\in S_A, \pp^{\ovlp,\widehat{\cN}}_{A,V}(\beta^*(x)) = q_A
    \\
    & \land
    \forall B\in r(A,\lvert A\rvert-1), \iota_B^*(x)\in\dom(\pp^{\amlg,\ell,\widehat{\cN}}_{B,V})
    \}
  \end{align*}
  has $\mu^{(d)}$-measure $1$, so we define
  \begin{align*}
    \pp^{\amlg,\ell,\widehat{\cN}}_{A,V}(x)
    & \df
    \biggl(
    \Bigl(\pp^{\amlg,\ell,\widehat{\cN}}_{B,V}\bigl(\iota_{B,A}^*(x)\bigr)
    \mathrel{\Big\vert}
    B\in r(A,\lvert A\rvert-1)\Bigr), q_A
    \biggr)
  \end{align*}
  whenever $x\in D$ and leave $\pp^{\amlg,\ell,\widehat{\cN}}_{A,V}(x)$ undefined when $x\notin D$.

  Since inductively all $\pp^{\amlg,\ell,\widehat{\cN}}_{B,V}$ for $B\in r(A,\lvert A\rvert-1)$ are
  constant on their domains, it follows that $\pp^{\amlg,\ell,\widehat{\cN}}_{A,V}$ is also constant
  on its domain $D$ (i.e., item~\ref{def:amlg:pp} holds).

  By definition it is also clear that item~\ref{def:amlg:dompp} holds (i.e., if
  $\pp^{\amlg,\ell,\widehat{\cN}}_{A,V}$ is defined, then all lower amalgamating types are also
  defined).

  We now define $\qq^{\amlg,\ell,\widehat{\cN}}_{A,V}\colon\cE_{A,\lvert
    A\rvert-1}^{(d)}(\Omega^2)\times\cS_{A,V}^{(d),\amlg,\ell}(\Omega^2)\pto\cS_{A,V}^{(d),\ovlp}(\Omega^2)$
  by
  \begin{align*}
    \qq^{\amlg,\ell,\widehat{\cN}}_{A,V}(x,p)
    & \df
    \rho_A(p),
    \qquad (x\in\cE_{A,\lvert A\rvert-1}^{(d)}(\Omega^2), p\in\cS_{A,V}^{(d),\amlg,\ell}(\Omega^2)),
  \end{align*}
  where $\rho_A\colon\cS_{A,V}^{(d),\amlg,\ell}(\Omega^2)\to\cS_{A,V}^{(d),\ovlp}(\Omega^2)$ is the
  natural projection.

  Item~\ref{def:amlg:qq}
  ($\qq^{\amlg,\ell,\widehat{\cN}}_{A,V}(\tau_A(x),\pp^{\amlg,\ell,\widehat{\cN}}_{A,V}(x)) =
  \pp^{\ovlp,\widehat{\cN}}_{A,V}(x)$ whenever $x\in\dom(\pp^{\amlg,\ell,\widehat{\cN}}_{A,V})$)
  follows immediately from definitions.

  Finally, define the partial Markov kernel
  \begin{align*}
    \Theta^{\amlg,\ell,\widehat{\cN}}_{A,V}\colon
    \prod_{B\in r(A,\lvert A\rvert-1)}\cS_{B,V}^{(d),\amlg,\ell}(\Omega^2)
    & \pto
    \cP(\cS_{A,V}^{(d),\amlg,\ell}(\Omega^2)
  \end{align*}
   to be the constant function whose value is the distribution on
   $\cS_{A,V}^{(d),\amlg,\ell}(\Omega^2)$ that is concentrated on the unique point in the image of
   $\pp^{\amlg,\ell,\widehat{\cN}}_{A,V}$. Obviously, item~\ref{def:amlg:Theta} holds (i.e.,
   $\Theta^{\amlg,\ell,\widehat{\cN}}_{A,V}$ correctly computes the conditional distribution of the
   amalgamating $(A,V,\ell)$-type given the lower amalgamating types).

  We defer the proofs of items~\ref{def:amlg:commutative} and~\ref{def:amlg:weakamalg} to
  Lemmas~\ref{lem:amlgcomm} and~\ref{lem:amlg}, respectively.

  \medskip

  We now consider the case when $A = [a]$ for some $a\in\NN$ with $a > \ell$. Our objective is to
  setup for applying Lemma~\ref{lem:keymeas}. First, we take the following choice of parameters for
  the spaces:
  \begin{align*}
    (X,\mu) & \leftarrow (\cE_{a,a-1}^{(d)}(\Omega^2),\mu^{(d)}), &
    (Z,\nu) & \leftarrow (\Omega,\mu), &
    (W,\eta) & \leftarrow ([0,1],\lambda),
    \\
    Y & \leftarrow \cS_{a,V}^{(d),\ovlp}(\Omega^2), &
    U & \leftarrow \prod_{B\in r(a,a-1)}\cS_{B,V}^{(d),\amlg,\ell}(\Omega^2), &
    V & \leftarrow \cS_{a,V}^{(d),\dsct}(\Omega^2).
  \end{align*}

  Note that when we consider any of the contra-variant definitions only for members of $S_a$ (i.e.,
  injections $\beta\colon[a]\to[a]$, which are necessarily bijections), contra-variance translates
  into a right action of $S_a$ on the corresponding space.

  It is obvious that the action of $S_a$ on $\cE_{a,a-1}^{(d)}(\Omega^2)$ is measurable and
  $\mu^{(d)}$ is $S_a$-invariant. We claim that there exists a measurable set
  $X'\subseteq\cE_{a,a-1}^{(d)}(\Omega^2)$ such that the orbits $(\sigma\cdot X')_{\sigma\in S_a}$
  of $X'$ under the action of $S_a$ are pairwise disjoint and $\mu(\bigcup_{\sigma\in
    S_a}\sigma\cdot X')=1$. Indeed, if we fix a measurable order $\leq$ on $\Omega^2$ (this can be
  done by noting that $\Omega^2$ is Borel-isomorphic to $[0,1]$ and pulling back the usual order of
  $[0,1]$ through the Borel-isomorphism), then we can take
  \begin{align*}
    X' & \df \{x\in\cE_{a,a-1}^{(d)}(\Omega^2) \mid x_{\{1\}} < x_{\{2\}} < \cdots < x_{\{a\}}\}.
  \end{align*}
  The fact that $\mu^{(d)}(\bigcup_{\sigma\in S_a}\sigma\cdot X')=1$ follows since $\mu$ is
  atomless.

  By Proposition~\ref{prop:toptypes}\ref{prop:toptypes:alpha}, the actions of $S_a$ on
  $\cS_{a,V}^{(d),\ovlp}(\Omega^2)$ and $\cS_{a,V}^{(d),\dsct}(\Omega^2)$ are continuous, hence
  measurable. Similarly, the same proposition (along with the fact that
  $\cS_{a,V}^{(d),\amlg}(\Omega^2)$ is equipped with product topology) implies that the action of
  $S_a$ on $\prod_{B\in r(a,a-1)}\cS_{B,V}^{(d),\amlg,\ell}(\Omega^2)$ is continuous, hence
  measurable.

  Also, obviously all trivial actions are measurable. Thus, we have covered all hypotheses of
  Lemma~\ref{lem:keymeas} concerning action measurability and invariance (but not yet equivariance
  of $f$, $g$, $\gamma$ and $\theta$).

  Let us now setup the final objects $f$, $g$, $\gamma$ and $\theta$ of Lemma~\ref{lem:keymeas}.

  Note that if $\rn{x}$ is picked at random in $\cE_a^{(d)}(\Omega^2)$ according to $\mu^{(d)}$,
  then Lemma~\ref{lem:equivpp} implies that the distribution of
  $\pp^{\dsct,\widehat{\cN}}_{a,V}(\rn{x})$ is $S_a$-invariant (as $\mu^{(d)}$ is
  $S_a$-invariant). Since inductively all lower sized amalgamating types satisfy the weak
  amalgamation property of item~\ref{def:amlg:weakamalg}, Lemma~\ref{lem:mkequiv} gives us an
  $S_a$-equivariant partial Markov kernel
  \begin{align*}
    \theta_a\colon
    \prod_{B\in r(a,a-1)}\cS_{B,V}^{(d),\amlg,\ell}(\Omega^2)
    \pto
    \cP(\cS_{a,V}^{(d),\dsct}(\Omega^2))
  \end{align*}
  such that the conditional distribution of $\pp^{\dsct,\widehat{\cN}}_{a,V}(\rn{x})$ given
  $\tau_a(\rn{x})$ is
  \begin{align*}
    \theta_a\Bigl(\pp^{\amlg,\ell,\widehat{\cN}}_{B,V}\bigl(\iota_{B,[a]}^*(\rn{x})\bigr)
    \mathrel{\Big\vert}
    B\in r(a,a-1)\Bigr).
  \end{align*}

  On the other hand, Lemma~\ref{lem:equivpp} also implies that the distribution of
  $\pp^{\ovlp,\widehat{\cN}}_{A,V}(\rn{x})$ is $S_a$-invariant as well, so Lemma~\ref{lem:mkequiv}
  (with $\rn{Y}$ trivial) further gives us an $S_a$-equivariant partial Markov kernel
  \begin{align*}
    \gamma_a\colon\cE_{a,a-1}^{(d)}(\Omega^2)\pto\cP(\cS_{a,V}^{(d),\ovlp}(\Omega^2))
  \end{align*}
  such the conditional distribution of $\pp^{\ovlp,\widehat{\cN}}_{a,V}(\rn{x})$ given
  $\tau_a(\rn{x})$ is $\gamma_a(\tau_a(\rn{x}))$ (in particular, we have
  $\mu^{(d)}(\dom(\gamma_a))=1$).

  We then take the following parameters for Lemma~\ref{lem:keymeas}:
  \begin{align*}
    \theta & \leftarrow \theta_a, &
    \gamma & \leftarrow \gamma_a,
  \end{align*}
  we define the partial functions
  \begin{align*}
    f\colon\cE_{a,a-1}^{(d)}(\Omega^2)
    & \pto
    \prod_{B\in r(a,a-1)}\cS_{B,V}^{(d),\amlg,\ell}(\Omega^2),
    \\
    g\colon\cE_{a,a-1}^{(d)}(\Omega^2)\times\cS_{a,V}^{(d),\ovlp}(\Omega^2)
    & \pto
    \cS_{a,V}^{(d),\dsct}(\Omega^2)
  \end{align*}
  by
  \begin{align}
    f(x) & \df (\pp^{\amlg,\ell,\widehat{\cN}}_{B,V}(\iota_{B,[a]}^*(x)) \mid B\in r(a,a-1)),\notag\\
    g(x,p) & \df \pi^{\ovlp}_{a,V}(p),\label{eq:amlg:g}
  \end{align}
  for every $x\in\cE_{a,a-1}^{(d)}(\Omega^2)$ and every $p\in\cS_{a,V}^{(d),\ovlp}(\Omega^2)$ and we
  set
  \begin{align*}
    A & \df \dom(f)\cap f^{-1}(\dom(\theta_a))\cap\dom(\gamma_a), &
    B & \df \dom(g) = \cE_{a,a-1}^{(d)}(\Omega^2)\times\cS_{a,V}^{(d),\ovlp}(\Omega^2)
  \end{align*}
  ($g$ is actually a total function). Our inductive hypothesis of item~\ref{def:amlg:pp} implies
  that $\mu^{(d)}(A)=\gamma_a[\mu^{(d)}](B)=1$.

  Let us now check that for every $x\in A$, we have
  \begin{align*}
    g(x,\place)_*(\gamma_a(x)) & = \theta_a(f(x)).
  \end{align*}
  Since $\mu^{(d)}(A)=1$ and $A\subseteq f^{-1}(\dom(\theta_a))\cap\dom(\gamma_a)$, this is of
  course equivalent to showing that for $\rn{x}$ picked at random according to $\mu^{(d)}$, we have
  \begin{align*}
    g(\rn{x},\place)_*(\gamma_a(\rn{x})) & = \theta_a(f(\rn{x})).
  \end{align*}
  By construction, the right-hand side of the above is the conditional distribution of
  $\pp^{\dsct,\widehat{\cN}}_{a,V}(\rn{x})$ given $\tau_a(\rn{x})$ and the left-hand side is the
  conditional distribution of $\pi^{\ovlp}_{a,V}(\pp^{\ovlp,\widehat{\cN}}_{a,V}(\rn{x}))$ given
  $\tau_a(\rn{x})$, so their equality follows from
  $\pi^{\ovlp}_{a,V}\comp\pp^{\ovlp,\widehat{\cN}}_{a,V} = \pp^{\dsct,\widehat{\cN}}_{a,V}$ of
  Lemma~\ref{lem:equivpp}.

  Note also that Lemma~\ref{lem:equivpp} implies that $g$ is $S_a$-equivariant and our inductive
  hypothesis of item~\ref{def:amlg:commutative} implies that $f$ is $S_a$-equivariant.

  Thus, all hypotheses of Lemma~\ref{lem:keymeas} are satisfied. Let then $(G,G',A',B',\Theta)$ be
  given by the lemma and define
  \begin{align*}
    \pp^{\amlg,\ell,\widehat{\cN}}_{a,V}\colon
    \cE_a^{(d)}(\Omega^2)
    & \pto
    \cS_{a,V}^{(d),\amlg,\ell}(\Omega^2),
    \\
    \qq^{\amlg,\ell,\widehat{\cN}}_{a,V}\colon
    \cE_{a,a-1}^{(d)}(\Omega^2)\times\cS_{a,V}^{(d),\amlg,\ell}(\Omega^2)
    & \pto
    \cS_{a,V}^{(d),\ovlp}(\Omega^2),
    \\
    \Theta^{\amlg,\ell,\widehat{\cN}}_{a,V}\colon
    \prod_{B\in r(a,a-1)}\cS_{B,V}^{(d),\amlg,\ell}(\Omega^2)
    & \pto
    \cP(\cS_{A,V}^{(d),\amlg,\ell}(\Omega^2))
  \end{align*}
  by
  \begin{align*}
    \pp^{\amlg,\ell,\widehat{\cN}}_{a,V}(x,(y,z))
    & \df
    \biggl(
    \Bigl(\pp^{\amlg,\ell,\widehat{\cN}}_{B,V}\bigl(\iota_{B,[a]}^*(x)\bigr) \mid B\in r(a,a-1)\Bigr),
    G\Bigl(x, \pp^{\ovlp,\widehat{\cN}}_{a,V}\bigl(x,(y,z)\bigr), z\Bigr)
    \biggr)
  \end{align*}
  for every $x\in\cE_{a,a-1}^{(d)}(\Omega^2)$, every $y\in\Omega\times\cO_A^d$ and every
  $z\in\Omega$ (so that $(x,y,z)\in\cE_a^{(d)}(\Omega^2)$, with $(y,z)$ containing the coordinate
  indexed by $A$) such that
  \begin{align*}
    \Bigr(x, \pp^{\ovlp,\widehat{\cN}}_{a,V}\bigl(x,(y,z)\bigr), z\Bigr) & \in B',
  \end{align*}
  otherwise, we leave $\pp^{\amlg,\ell,\widehat{\cN}}_{a,V}(x,(y,z))$ undefined;
  \begin{align*}
    \qq^{\amlg,\ell,\widehat{\cN}}_{a,V}(x, p)
    & \df
    \rho'_a(G'(x, \rho_a(p))),
  \end{align*}
  for every $x\in\cE_{a,a-1}^{(d)}(\Omega^2)$ and every
  $p\in\cS_{a,a-1}^{(d),\amlg,\ell}(\Omega^2)$, where
  $\rho_a\colon\cS_{A,V}^{(d),\amlg,\ell}(\Omega^2)\to\cS_{A,V}^{(d),\dsct}(\Omega^2)\times[0,1]$
  and $\rho'_a\colon\cS_{A,V}^{(d),\ovlp}(\Omega^2)\times\Omega\to\cS_{A,V}^{(d),\ovlp}(\Omega^2)$
  are the natural projections; and
  \begin{align*}
    \Theta^{\amlg,\ell,\widehat{\cN}}_{a,V}(q)
    & \df
    \delta_q\otimes\Theta(q),
  \end{align*}
  for every $q\in\prod_{B\in r(a,a-1)}\cS_{B,V}^{(d),\amlg,\ell}(\Omega^2)$, where $\delta_q$ is the
  Dirac measure concentrated on $q$.

  We now argue why Lemma~\ref{lem:keymeas} gives all required properties to continue the inductive
  construction, with the exception of properties~\ref{def:amlg:commutative}
  and~\ref{def:amlg:weakamalg} whose proofs are deferred to Lemmas~\ref{lem:amlgcomm}
  and~\ref{lem:amlg}, respectively.

  Since $G$ is measurable and is defined $\gamma_a[\mu^{(d)}]$-almost everywhere (by
  items~\ref{lem:keymeas:B'G} and~\ref{lem:keymeas:gammamunuB'} of Lemma~\ref{lem:keymeas}), it
  follows that $\pp^{\amlg,\ell,\widehat{\cN}}$ is measurable and is defined almost everywhere, so
  item~\ref{def:amlg:pp} holds.
  
  Items~\ref{lem:keymeas:A'} and~\ref{lem:keymeas:B'A'} of Lemma~\ref{lem:keymeas} imply that when
  the amalgamating $(a,V,\ell)$-type is defined, then all lower amalgamating types are also defined,
  that is, item~\ref{def:amlg:dompp} holds.

  The fact that $\qq^{\amlg,\ell,\widehat{\cN}}_{a,V}$ correctly retrieves the overlapping type from
  the amalgamating type along with lower pointwise data (item~\ref{def:amlg:qq}) follows from
  item~\ref{lem:keymeas:G'G} in Lemma~\ref{lem:keymeas}: for every
  $(x,y,z)\in\dom(\pp^{\amlg,\ell,\widehat{\cN}}_{a,V})$ with $x\in\cE_{a,a-1}^{(d)}(\Omega)$,
  $y\in\Omega\times\cO_a^d$ and $z\in\Omega$, we have
  \begin{align*}
    \qq^{\amlg,\ell,\widehat{\cN}}_{a,V}\Bigl(
    x, \pp^{\amlg,\ell,\widehat{\cN}}_{a,V}\bigl(x,(y,z)\bigr)
    \Bigr)
    & =
    \rho'_a\Biggl(
    G'\biggl(
    x,G\Bigl(x,\pp^{\ovlp,\widehat{\cN}}_{a,V}\bigl(x,(y,z)\bigr),z\Bigr)
    \biggr)
    \Biggr)
    \\
    & =
    \rho'_a\Bigl(\pp^{\ovlp,\widehat{\cN}}_{a,V}\bigl(x,(y,z)\bigr),z\Bigr)
    =
    \pp^{\ovlp,\widehat{\cN}}_{a,V}\bigl(x,(y,z)\bigr).
  \end{align*}

  For item~\ref{def:amlg:Theta}, we have to show that if $\rn{x}$ is sampled in
  $\cE_a^{(d)}(\Omega^2)$ according to $\mu^{(d)}$, then the conditional distribution of
  $\pp^{\amlg,\ell,\widehat{\cN}}_{a,V}(\rn{x})$ given $\tau_a(\rn{x})$ is
  \begin{align*}
    \Theta^{\amlg,\ell,\widehat{\cN}}_{a,V}\Bigl(
    \pp^{\amlg,\ell,\widehat{\cN}}_{B,V}\bigl(\iota_{B,[a]}^*(\rn{x})\bigr)
    \mathrel{\Big\vert}
    B\in r(a,a-1)
    \Bigr).
  \end{align*}
  Note that the argument of $\Theta^{\amlg,\ell,\widehat{\cN}}_{a,V}$ is precisely $f(\rn{x})$.

  Since $\mu^{(d)}(A')=1$, it suffices to show that for every $x\in
  A'\subseteq\cE_{a,a-1}^{(d)}(\Omega^2)$, if $\rn{y}$ and $\rn{z}$ are picked independently at
  random in $\Omega\times\cO_a^d$ and $\Omega$ with respect to $\mu\otimes\bigotimes_{i=1}^d\nu_a$
  and $\mu$, respectively, then
  \begin{align*}
    \pp^{\amlg,\ell,\widehat{\cN}}_{a,V}(x, (\rn{y},\rn{z}))
    & \sim
    \Theta^{\amlg,\ell,\widehat{\cN}}_{A,V}\bigl(f(\rn{x})\bigr).
  \end{align*}

  By the definition of $\pp^{\amlg,\ell,\widehat{\cN}}_{a,V}$ and
  $\Theta^{\amlg,\ell,\widehat{\cN}}_{a,V}$, it is clear that all coordinates corresponding to lower
  amalgamating types correctly match, so the above is equivalent to
  \begin{align*}
    G\Bigl(x, \pp^{\ovlp,\widehat{\cN}}_{a,V}\bigl(x, (\rn{y},\rn{z})\bigr), \rn{z}\Bigr)
    & \sim
    \Theta\bigl(f(x)\bigr).
  \end{align*}

  Recalling that $\rn{z}$ is a dummy variable for $\widehat{\cN}$ (more specifically,
  using~\eqref{eq:dummy}), we note that $\pp^{\ovlp,\widehat{\cN}}_{a,V}\bigl(x,
  (\rn{y},\rn{z})\bigr)$ is independent from $\rn{z}$, so their joint distribution is precisely
  $\gamma_a(x)\otimes\mu$, so the above is equivalent to
  \begin{align*}
    G(x,\place,\place)_*(\gamma_a(x)\otimes\mu)
    & =
    \Theta\bigl(f(x)\bigr),
  \end{align*}
  which is precisely item~\ref{lem:keymeas:Ggammanu} of Lemma~\ref{lem:keymeas}.

  \medskip

  Finally, we consider the case when $\lvert A\rvert > \ell$ and $A$ is not $[a]$ for some
  $a\in\NN$. We fix a bijection $\alpha_A\colon [\lvert A\rvert]\to A$ and define
  \begin{equation}\label{eq:equivdef}
    \begin{aligned}
      \pp^{\amlg,\ell,\widehat{\cN}}_{A,V}
      & \df
      (\alpha_A^{-1})^*\comp\pp^{\amlg,\ell,\widehat{\cN}}_{\lvert A\rvert,V}\comp\alpha_A^*,
      \\
      \qq^{\amlg,\ell,\widehat{\cN}}_{A,V}
      & \df
      (\alpha_A^{-1})^*\comp\qq^{\amlg,\ell,\widehat{\cN}}_{\lvert A\rvert,V}\comp\alpha_A^*,
      \\
      \Theta^{\amlg,\ell,\widehat{\cN}}_{A,V}
      & \df
      (\alpha_A^{-1})^*\comp\Theta^{\amlg,\ell,\widehat{\cN}}_{\lvert A\rvert,V}\comp\alpha_A^*.
    \end{aligned}
  \end{equation}

  It will be convenient to let $\alpha_{[a]}\df\id_{[a]}$ for every $a\in\NN_+$ so that the
  equalities in the above remain true even in the case $A = [a]$, in which we define the partial
  functions above directly. It is also clear that the equalities above also hold when $\lvert
  A\rvert\leq\ell$.

  It is straightforward to check that the properties required by the construction for $A$ follow as
  a consequence of the properties for $[\lvert A\rvert]$, except for the two joint properties of
  items~\ref{def:amlg:commutative} and~\ref{def:amlg:weakamalg} that are being deferred.

  Note also that once we show that item~\ref{def:amlg:commutative} holds in
  Lemma~\ref{lem:amlgcomm}, a direct consequence is that the definition above does not
  depend on the choice of $\alpha_A$, except for a zero-measure set.
\end{definition}

\begin{remark}\label{rmk:amlguniqueness}
  A direct consequence of the item~\ref{def:amlg:pp} of Definition~\ref{def:amlg} is that if $\cN$
  is an $(\ell-1)$-independent $d$-Euclidean structure over $\Omega$ satisfying $\UCouple[\ell]$,
  then for the associated structure $\widehat{\cN}$ over $\Omega^2$ given
  by~\eqref{eq:amlgwidehatcN}, whenever $\lvert A\rvert\leq\ell$, there exists a unique amalgamating
  $(A,V,\ell)$-type. In turn, this uniqueness was obtained as a consequence of uniqueness of
  overlapping $(A,V)$-types given by Proposition~\ref{prop:typeuniqueness}.
\end{remark}

Let us now prove the deferred property~\ref{def:amlg:commutative}.

\begin{lemma}\label{lem:amlgcomm}
  Within Definition~\ref{def:amlg}, property~\ref{def:amlg:commutative} holds, that is, for every
  injection $\beta\colon A'\to A$ between finite sets $A'$ and $A$, both disjoint from a countably
  infinite set $V$, the diagram
  \begin{equation*}
    \begin{tikzcd}[column sep={2.5cm}]
      \cE_A^{(d)}(\Omega^2)
      \arrow[r, harpoon, "\pp^{\amlg,\ell,\widehat{\cN}}_{A,V}"]
      \arrow[rr, bend left=20, "\pp^{\dsct,\widehat{\cN}}_{A,V}"]
      \arrow[rrr, bend left=30, "M^{\widehat{\cN}}_A"']
      \arrow[d, "\beta^*"]
      &
      \cS_{A,V}^{(d),\amlg,\ell}(\Omega^2)
      \arrow[r, "\pi^{\amlg,\ell}_{A,V}"]
      \arrow[rr, bend left=20, "M^{\amlg,\ell}_{A,V}"]
      \arrow[d, "\beta^*"]
      &
      \cS_{A,V}^{(d),\dsct,\ell}(\Omega^2)
      \arrow[r, "M^{\dsct}_{A,V}"]
      \arrow[d, "\beta^*"]
      &
      \cK_A
      \arrow[d, "\beta^*"]
      \\
      \cE_{A'}^{(d)}(\Omega^2)
      \arrow[r, harpoon, "\pp^{\amlg,\ell,\widehat{\cN}}_{A',V}"]
      \arrow[rr, bend right=20, "\pp^{\dsct,\widehat{\cN}}_{A',V}"']
      \arrow[rrr, bend right=30, "M^{\widehat{\cN}}_{A'}"]
      &
      \cS_{A',V}^{(d),\amlg,\ell}(\Omega^2)
      \arrow[r, "\pi^{\amlg,\ell}_{A',V}"]
      \arrow[rr, bend right=20, "M^{\amlg,\ell}_{A',V}"']
      &
      \cS_{A',V}^{(d),\dsct,\ell}(\Omega^2)
      \arrow[r, "M^{\dsct}_{A',V}"]
      &
      \cK_{A'}
    \end{tikzcd}
  \end{equation*}
  is weakly commutative.
\end{lemma}

\begin{proof}
  To completely prove the lemma, it suffices to show all weakly commutative properties involving
  only the first row and to show weak equivariance of $\pp^{\amlg,\ell,\widehat{\cN}}$,
  $\pp^{\dsct,\widehat{\cN}}$, $M^{\widehat{\cN}}$, $\pi^{\amlg,\ell}$, $M^{\amlg,\ell}$ and
  $M^{\dsct}$.

  Equivariance of $\pp^{\dsct,\widehat{\cN}}$, $M^{\widehat{\cN}}$ and $M^{\dsct}$ was shown in
  Lemma~\ref{lem:equivpp} and equivariance of $\pi^{\amlg,\ell}$ and $M^{\amlg,\ell}$ was observed
  in~\eqref{eq:equivpiamlgMamlg}. The final weak equivariance of $\pp^{\amlg,\ell,\widehat{\cN}}$
  will be the last thing we prove.

  \medskip
  
  Let us prove the weakly commutative properties involving the first row. The equalities
  \begin{align*}
    M^{\dsct}_{A,V}\comp\pi^{\amlg,\ell}_{A,V} & = M^{\amlg,\ell}_{A,V}, &
    M^{\dsct}_{A,V}\comp\pp^{\dsct,\widehat{\cN}}_{A,V} & = M^{\widehat{\cN}}_A,
  \end{align*}
  follow straightforwardly from definitions (the latter already appeared in
  Lemma~\ref{lem:equivpp}).

  \medskip

  We now show
  \begin{align*}
    \pi^{\amlg,\ell}_{A,V}\comp\pp^{\amlg,\ell,\widehat{\cN}}_{A,V} & = \pp^{\dsct,\widehat{\cN}}_{A,V}
  \end{align*}
  within the domain of the left-hand side.

  When $\lvert A\rvert\leq\ell$, then by definition both sides evaluate to the unique dissociated
  $(A,V)$-type $q_A$ within the domain of the left-hand side.

  When $\lvert A\rvert > \ell$, for $(x,(y,z))\in\dom(\pp^{\amlg,\ell,\widehat{\cN}}_{A,V})$ where
  $x\in\cE_A^{(d)}(\Omega^2)$, $y\in\Omega\times\cO_A^d$ and $z\in\Omega$, we have
  \begin{align*}
    \pi^{\amlg,\ell}_{A,V}\bigl(\pp^{\amlg,\ell,\widehat{\cN}}_{A,V}(x,(y,z))\bigr)
    & =
    \pi^{\amlg,\ell}_{A,V}\biggl(
    (\alpha_A^{-1})^*\Bigl(
    \pp^{\amlg,\ell,\widehat{\cN}}_{\lvert A\rvert,V}\bigl(\alpha_A^*(x,(y,z))\bigr)
    \Bigr)
    \biggr)
    \\
    & =
    (\alpha_A^{-1})^*\Biggl(
    \pi^{\amlg,\ell}_{\lvert A\rvert,V}\biggl(
    \pp^{\amlg,\ell,\widehat{\cN}}_{\lvert A\rvert,V}\Bigl(\alpha_A^*(x),\bigl(\alpha_A^*(y),z\bigr)\Bigr)
    \biggr)
    \Biggr)
    \\
    & =
    (\alpha_A^{-1})^*\Biggl(
    \pi\Biggl(
    G\biggl(x,
    \pp^{\ovlp,\widehat{\cN}}_{\lvert A\rvert,V}\Bigl(\alpha_A^*(x),\bigl(\alpha_A^*(y),z\bigr)
    \Bigr),
    z
    \biggr)
    \Biggr)
    \Biggr)
    \\
    & =
    (\alpha_A^{-1})^*\Biggl(
    g\biggl(x,
    \pp^{\ovlp,\widehat{\cN}}_{\lvert A\rvert,V}\Bigl(\alpha_A^*(x),\bigl(\alpha_A^*(y),z\bigr)
    \Bigr)
    \biggr)
    \Biggr)
    \\
    & =
    (\alpha_A^{-1})^*\Biggl(
    \pi^{\ovlp}_{\lvert A\rvert,V}\biggl(
    \pp^{\ovlp,\widehat{\cN}}_{\lvert A\rvert,V}\Bigl(\alpha_A^*(x),\bigl(\alpha_A^*(y),z\bigr)
    \Bigr)
    \biggr)
    \Biggr)
    \\
    & =
    (\alpha_A^{-1})^*\biggl(
    \pp^{\dsct,\widehat{\cN}}_{\lvert A\rvert,V}\Bigl(\alpha_A^*(x),\bigl(\alpha_A^*(y),z\bigr)
    \Bigr)
    \biggr)
    \\
    & =
    \pp^{\dsct,\widehat{\cN}}_{A,V}\bigl(x,(y,z)\bigr),
  \end{align*}
  where the second equality follows from the already shown equivariance of $\pi^{\amlg,\ell}$, in
  the third equality, the undecorated $\pi$ is the projection $\pi\colon V\times W\to V$ of
  Lemma~\ref{lem:keymeas} (i.e., $\pi\colon\cS_{\lvert
    A\rvert,V}^{(d),\dsct}\times[0,1]\to\cS_{\lvert A\rvert,V}^{(d),\dsct}$), the fourth equality
  follows from Lemma~\ref{lem:keymeas}\ref{lem:keymeas:piG}, the fifth equality follows from the
  definition of $g$ in~\eqref{eq:amlg:g}, the sixth equality follows from Lemma~\ref{lem:equivpp}
  and the final equality follows from the already shown equivariance of $\pp^{\dsct,\widehat{\cN}}$
  (also shown in Lemma~\ref{lem:equivpp}).

  \medskip

  To conclude the weak commutativity of the first row, it remains to show
  \begin{align*}
    M^{\amlg,\ell}_{A,V}\comp\pp^{\amlg,\ell,\widehat{\cN}}_{A,V} & = M^{\widehat{\cN}}_{A,V}
  \end{align*}
  within the domain of the left-hand side. This follows from the other already shown weak
  commutativities of the first row:
  \begin{align*}
    M^{\amlg,\ell}_{A,V}\comp\pp^{\amlg,\ell,\widehat{\cN}}_{A,V}
    & =
    M^{\dsct,\ell}_{A,V}\comp\pi^{\amlg,\ell}_{A,V}\comp\pp^{\amlg,\ell,\widehat{\cN}}_{A,V}
    =
    M^{\dsct,\ell}_{A,V}\comp\pp^{\dsct,\widehat{\cN}}_{A,V}
    =
    M^{\widehat{\cN}}_{A,V}.
  \end{align*}

  \medskip

  It remains to show weak equivariance of $\pp^{\amlg,\ell,\widehat{\cN}}$, that is, we need to show that
  if $\beta\colon A'\to A$ is an injection, then
  \begin{align*}
    \pp^{\amlg,\ell,\widehat{\cN}}_{A',V}\comp\beta^* & = \beta^*\comp\pp^{\amlg,\ell,\widehat{\cN}}_{A,V}
  \end{align*}
  within the common domain of the two sides.

  Let $x\in\cE_{A'}^{(d)}(\Omega^2)$ be a point in the common domain of the two sides.

  We first consider the case when $A'$ and $A$ have the same cardinality, say $a$, and we prove the
  weak equivariance by induction in $a$.

  If $a\leq\ell$ (and assuming inductively that weak equivariance holds for smaller sizes), we note
  that
  \begin{align*}
    (\pp^{\amlg,\ell,\widehat{\cN}}_{A',V}\comp\beta^*)(x)
    & =
    \Biggl(
    \biggl(\pp^{\amlg,\ell,\widehat{\cN}}_{B,V}\Bigl(\iota_{B,A'}^*\bigl(\beta^*(x)\bigr)\Bigr)
    \mathrel{\bigg\vert}
    B\in r(A',a-1)\biggr), q_{A'}
    \Biggr)
    \\
    & =
    \Biggl(
    \biggl(\pp^{\amlg,\ell,\widehat{\cN}}_{B,V}\Bigl(\beta\down_B^*\bigl(\iota_{\beta(B),A}^*(x)\bigr)\Bigr)
    \mathrel{\bigg\vert}
    B\in r(A',a-1)\biggr), q_{A'}
    \Biggr)
    \\
    & =
    \Biggl(
    \biggl(\beta\down_B^*\Bigl(\pp^{\amlg,\ell,\widehat{\cN}}_{\beta(B),V}\bigl(\iota_{\beta(B),A}^*(x)\bigr)
    \Bigr)
    \mathrel{\bigg\vert}
    B\in r(A',a-1)\biggr), \beta^*(q_A)
    \Biggr)
    \\
    & =
    \beta^*\biggl(
    \Bigl(\pp^{\amlg,\ell,\widehat{\cN}}_{B,V}\bigl(\iota_{B,A}^*(x)\bigr)
    \mathrel{\Big\vert}
    B\in r(A,\lvert A\rvert-1)\Bigr), q_A
    \biggr)
    \\
    & =
    \beta^*\bigl(\pp^{\amlg,\ell,\widehat{\cN}}_{A,V}(x)\bigr),
  \end{align*}
  where the second equality follows from contra-variance and the fact that
  $\beta\comp\iota_{B,A'}=\iota_{\beta(B),A}\comp\beta\down_B$ and the third equality follows from
  inductive hypothesis along with $\beta^*(q_A)=q_{A'}$ (as $q_A$ and $q_{A'}$ are the unique up to zero measure
  overlapping types of $\widehat{\cN}$ over their respective sets).

  Suppose now that $a > \ell$ and consider first the case $A = A' = [a]$. Then a similar derivation
  holds using $S_a$-equivariance of $G$ instead: writing $x=(w,(y,z))$, where
  $w\in\cE_{a,a-1}^{(d)}(\Omega^2)$, $y\in\Omega\times\cO_a^d$ and $z\in\Omega$, we have
  \begin{align*}
    & \!\!\!\!\!\!
    (\pp^{\amlg,\ell,\widehat{\cN}}_{A',V}\comp\beta^*)(x)
    \\
    & =
    \Biggl(
    \biggl(\pp^{\amlg,\ell,\widehat{\cN}}_{B,V}\Bigl(\iota_{B,a}^*\bigl(\beta^*(x)\bigr)\Bigr)
    \mathrel{\bigg\vert}
    B\in r(a,a-1)\biggr),
    G\Bigl(\beta^*(x),\pp^{\ovlp,\widehat{\cN}}_{a,V}\bigl(\beta^*(w),(\beta^*(y),z)\bigr), z\Bigr)
    \Biggr)
    \\
    & =
    \Biggl(
    \biggl(\pp^{\amlg,\ell,\widehat{\cN}}_{B,V}\Bigl(\beta\down_B^*\bigl(\iota_{\beta(B),a}^*(x)\bigr)\Bigr)
    \mathrel{\bigg\vert}
    B\in r(a,a-1)\biggr),
    \beta^*\biggl(G\Bigl(x,\pp^{\ovlp,\widehat{\cN}}_{a,V}\bigl(w,(y,z)\bigr), z\Bigr)\biggr)
    \Biggr)
    \\
    & =
    \Biggl(
    \beta^*\biggl(
    \Bigl(\pp^{\amlg,\ell,\widehat{\cN}}_{B,V}\bigl(\iota_{\beta(B),a}^*(x)\bigr)\Bigr)
    \biggr)
    \mathrel{\bigg\vert}
    B\in r(a',a-1)\biggr),
    \beta^*\biggl(G\Bigl(x,\pp^{\ovlp,\widehat{\cN}}_{a,V}\bigl(w,(y,z)\bigr), z\Bigr)\biggr)
    \Biggr)
    \\
    & =
    (\beta^*\comp\pp^{\amlg,\ell,\widehat{\cN}}_{A,V})(x),
  \end{align*}
  where the second equality follows from $S_a$-equivariance of $G$ and $\pp^{\ovlp,\widehat{\cN}}$
  (from Lemma~\ref{lem:equivpp}) and the fact that
  $\beta\comp\iota_{B,a}=\iota_{\beta(B),a}\comp\beta\down_B$ and the third equality follows from
  inductive hypothesis.

  We now consider the case when $\lvert A\rvert = \lvert A'\rvert > \ell$ (but are not necessarily
  of the form $[a]$). Letting $a\df\lvert A\rvert$ and
  $\gamma\df\alpha_A^{-1}\comp\beta\comp\alpha_{A'}\in S_a$, we have
  \begin{align*}
    (\pp^{\amlg,\ell,\widehat{\cN}}_{A',V}\comp\beta^*)(x)
    & =
    \bigl((\alpha_{A'}^{-1})^*\comp\pp_{a,V}^{\amlg,\ell,\widehat{\cN}}\comp\alpha_{A'}^*
    \comp\beta^*\bigr)(x)
    =
    \bigl((\alpha_{A'}^{-1})^*\comp\pp_{a,V}^{\amlg,\ell,\widehat{\cN}}\comp\gamma^*
    \comp\alpha_A^*\bigr)(x)
    \\
    & =
    \bigl((\alpha_{A'}^{-1})^*\comp\gamma^*\comp\pp_{a,V}^{\amlg,\ell,\widehat{\cN}}
    \comp\alpha_A^*\bigr)(x)
    =
    \bigl(\beta^*\comp(\alpha_A^{-1})^*\comp\pp_{a,V}^{\amlg,\ell,\widehat{\cN}}
    \comp\alpha_A^*\bigr)(x)
    \\
    & =
    (\beta^*\comp\pp^{\amlg,\ell,\widehat{\cN}}_{A,V})(x),
  \end{align*}
  where the third equality follows from the previous case (for $[a]$).

  This concludes all cases in which $\lvert A'\rvert=\lvert A\rvert$ (i.e., cases in which $\beta$
  is a bijection).

  We now prove the case when $A'=[a']$ and $A=[a]$ for some $a',a\in\NN$ with $a' < a$. We need to
  show that for every $C\in r(a')$ we have
  \begin{align*}
    (\pp^{\amlg,\ell,\widehat{\cN}}_{A',V}\comp\beta^*)(x)_C
    & =
    (\beta^*\comp\pp^{\amlg,\ell,\widehat{\cN}}_{A,V})(x)_C.
  \end{align*}

  If $\lvert C\rvert < a'$, then we have
  \begin{align*}
    (\pp^{\amlg,\ell,\widehat{\cN}}_{A',V}\comp\beta^*)(x)_C
    & =
    \pp^{\amlg,\ell,\widehat{\cN}}_{C,V}\Bigl(\iota_{C,[a']}^*\bigl(\beta^*(x)\bigr)\Bigr)
    =
    \pp^{\amlg,\ell,\widehat{\cN}}_{C,V}\Bigl(\beta\down_C^*\bigl(\iota_{\beta(C),[a]}^*(x)\bigr)\Bigr)
    \\
    & =
    \beta\down_C^*\Bigl(\pp^{\amlg,\ell,\widehat{\cN}}_{\beta(C),V}\bigl(\iota_{\beta(C),[a]}^*(x)\bigr)\Bigr)
    =
    \beta\down_C^*\bigl(\pp^{\amlg,\ell,\widehat{\cN}}_{a,V}(x)_{\beta(C)}\bigr)
    =
    \beta^*\bigl(\pp^{\amlg,\ell,\widehat{\cN}}_{a,V}(x)\bigr)_C,
  \end{align*}
  where the third equality follows from the already proved case as $\beta\down_C$ is bijective.

  If $C = [a']$, then we have
  \begin{align*}
    (\pp^{\amlg,\ell,\widehat{\cN}}_{A',V}\comp\beta^*)(x)_C
    & =
    \pp^{\amlg,\ell,\widehat{\cN}}_{a',V}\bigl(\beta^*(x)\bigr)_{[a']}
    =
    \pp^{\amlg,\ell,\widehat{\cN}}_{a',V}\Bigl(\beta\down_{[a']}^*\bigl(\iota_{\beta([a']),[a]}^*(x)\bigr)\Bigr)_{[a']}
    \\
    & =
    \beta\down_{[a']}^*\Bigl(
    \pp^{\amlg,\ell,\widehat{\cN}}_{\beta([a']),V}\bigl(\iota_{\beta([a']),[a]}^*(x)\bigr)
    \Bigr)_{[a']}
    =
    \beta\down_{[a']}^*\bigl(
    \pp^{\amlg,\ell,\widehat{\cN}}_{a,V}(x)_{\beta([a'])}
    \bigr)_{[a']}
    \\
    & =
    \beta\down_{[a']}^*\Bigl(
    \bigl(\pp^{\amlg,\ell,\widehat{\cN}}_{a,V}(x)_{\beta([a'])}\bigr)_{\beta([a'])}
    \Bigr)
    =
    \beta^*\bigl(\pp^{\amlg,\ell,\widehat{\cN}}_{A,V}(x)\bigr)_C,
  \end{align*}
  where the third equality follows from the already proved case as $\beta\down_{[a']}$ is
  bijective, the fifth equality follows from the contra-variant definition of
  $\beta\down_{[a']}^*$ and the sixth equality follows from the contra-variant definition of
  $\beta^*$.

  It remains to show the case when $A'$ and $A$ are such that $\lvert A'\rvert<\lvert A\rvert$ and
  are not necessarily of the form $[a']$ and $[a]$, respectively. Let $a'\df\lvert A'\rvert$ and
  $a\df\lvert A\rvert$, consider the injection
  $\gamma\df\alpha_A^{-1}\comp\beta\comp\alpha_{A'}\colon [a']\to[a]$ and note that
  \begin{align*}
    (\pp^{\amlg,\ell,\widehat{\cN}}_{A',V}\comp\beta^*)(x)
    & =
    \bigl((\alpha_{A'}^{-1})^*\comp\pp^{\amlg,\ell,\widehat{\cN}}_{a',V}\comp\alpha_{A'}^*\comp\beta^*\bigr)(x)
    =
    \bigl((\alpha_{A'}^{-1})^*\comp\pp^{\amlg,\ell,\widehat{\cN}}_{a',V}\comp\gamma^*\comp\alpha_A^*\bigr)(x)
    \\
    & =
    \bigl((\alpha_{A'}^{-1})^*\comp\gamma^*\comp\pp^{\amlg,\ell,\widehat{\cN}}_{a',V}\comp\alpha_A^*\bigr)(x)
    =
    \bigl(\beta^*\comp(\alpha_A^{-1})^*\comp\pp^{\amlg,\ell,\widehat{\cN}}_{a',V}\comp\alpha_A^*\bigr)(x)
    \\
    & =
    \beta^*\bigl(\pp^{\amlg,\ell,\widehat{\cN}}_{A,V}(x)\bigr),
  \end{align*}
  where the third equality follows from the already proved case as $\gamma\colon[a']\to[a]$.

  This concludes weak equivariance of $\pp^{\amlg,\ell,\widehat{\cN}}$.
\end{proof}

The next lemma proves the weak amalgamation property (item~\ref{def:amlg:weakamalg} of
Definition~\ref{def:amlg}), completing Definition~\ref{def:amlg}, and proves the strong amalgamation
that we have been promising all along.

\begin{lemma}\label{lem:amlg}
  Let $\cN$ be an $(\ell-1)$-independent $d$-Euclidean structure in $\cL$ over $\Omega$ satisfying
  $\UCouple[\ell]$, let $V$ be a countably infinite set and let $A$ be a finite set disjoint from
  $V$. Let also $\widehat{\cN}$ be the $d$-Euclidean structure in $\cL$ over $\Omega^2$ as in
  Definition~\ref{def:amlg}, i.e., $\widehat{\cN}$ is given by
  \begin{align*}
    \widehat{\cN}_P & \df \{(x,x')\in\cE_{k(P)}^{(d)}(\Omega)\times\cE_{k(P)}(\Omega) \mid x\in\cN_P\}
    \qquad (P\in\cL).
  \end{align*}
  Finally, let $k\in\{0,\ldots,\lvert A\rvert-1\}$.

  Then the following hold.
  \begin{enumerate}
  \item\label{lem:amlg:weak} Weak amalgamation: if $\rn{x}$ is sampled in
    $\cE_{A,k}^{(d)}(\Omega^2)$ according to $\mu^{(d)}$, then
    $\pp^{\dsct,\widehat{\cN}}_{A,V}(\rn{x})$ is conditionally independent from $\tau_{A,k}(\rn{x})$
    given
    \begin{align*}
      \bigl(\pp^{\amlg,\ell,\widehat{\cN}}_{A',V}(\iota_{A',A}(\rn{x}))
      \mathrel{\big\vert} A'\in r(A,k)\bigr),
    \end{align*}
    where $\tau_{A,k}\colon\cE_A^{(d)}(\Omega^2)\to\cE_{A,k}^{(d)}(\Omega^2)$ is the natural
    projection.
  \item\label{lem:amlg:welldef} Amalgamating types are well-defined:
    $\pp^{\amlg,\ell,\widehat{\cN}}_{A,V}$, $\qq^{\amlg,\ell,\widehat{\cN}}_{A,V}$ and
    $\Theta^{\amlg,\ell,\widehat{\cN}}_{A,V}$ are well-defined.
  \item\label{lem:amlg:strong} Strong amalgamation: if $\rn{x}$ is sampled in
    $\cE_A^{(d)}(\Omega^2)$ according to $\mu^{(d)}$, then
    $\pp^{\amlg,\ell,\widehat{\cN}}_{A,V}(\rn{x})$ is conditionally independent from
    $\tau_{A,k}(\rn{x})$ given
    \begin{align*}
      \bigl(\pp^{\amlg,\ell,\widehat{\cN}}_{A',V}(\iota_{A',A}(\rn{x}))
      \mathrel{\big\vert} A'\in r(A,k)\bigr),
    \end{align*}
    where $\tau_{A,k}\colon\cE_A^{(d)}(\Omega^2)\to\cE_{A,k}^{(d)}(\Omega^2)$ is the
    natural projection.
  \end{enumerate}
\end{lemma}

\begin{proof}
  We will prove all assertions by induction on $\lvert A\rvert$ and the logic of the inductive step
  is:
  \begin{enumerate}[label={\Alph*.}, ref={(\Alph*)}]
  \item\label{lem:amlg:weak->welldef} Item~\ref{lem:amlg:weak} for a set $A$ and $k=\lvert
    A\rvert-1$ implies item~\ref{lem:amlg:welldef} for $A$.
  \item\label{lem:amlg:welldef->strong} Item~\ref{lem:amlg:welldef} for all sets of size at most
    $\lvert A\rvert$ implies item~\ref{lem:amlg:strong} for $A$.
  \item\label{lem:amlg:strong->weak} Item~\ref{lem:amlg:strong} for all sets of size at most $\lvert
    A\rvert-1$ implies item~\ref{lem:amlg:weak} for $A$ (and every $k < \lvert A\rvert$).
  \end{enumerate}

  Item~\ref{lem:amlg:weak->welldef} is trivial as item~\ref{lem:amlg:weak} when $k=\lvert A\rvert-1$
  is exactly item~\ref{def:amlg:weakamalg} of Definition~\ref{def:amlg}.

  \medskip

  For item~\ref{lem:amlg:welldef->strong}, we consider three random variables $\rn{p}$, $\rn{q}$ and
  $\rn{r}$ in $\prod_{A'\in r(A)} \cS_{A',V}^{(d),\amlg,\ell}(\Omega^2)$.

  For $\rn{p}$, we simply let
  \begin{align*}
    \rn{p}_{A'} & \df \pp^{\amlg,\ell,\widehat{\cN}}_{A',V}\bigl(\iota_{A',A}^*(\rn{x})\bigr)
    \qquad (A'\in r(A)).
  \end{align*}

  For $\rn{q}$, we inductively define
  \begin{align*}
    \rn{q}_{A'}
    & \df
    \Theta^{\amlg,\ell,\widehat{\cN}}_{A',V}\bigl(\rn{q}_C \mid C\in r(A',\lvert A'\rvert-1)\bigr)
    \qquad (A'\in r(A)).
  \end{align*}
  An inductive application of items~\ref{def:amlg:Theta} and~\ref{def:amlg:commutative} of
  Definition~\ref{def:amlg} shows that $\rn{p}$ and $\rn{q}$ are equidistributed.

  Finally, for $\rn{r}$, we inductively define
  \begin{align*}
    \rn{r}_{A'}
    & \df
    \begin{dcases*}
      \rn{p}_{A'}, & if $\lvert A'\rvert\leq k$,\\
      \Theta^{\amlg,\ell,\widehat{\cN}}_{A',V}\bigl(
      (\rn{r}_C \mid C\in r(A',\lvert A'\rvert-1)
      \bigr),
      & if $\lvert A'\rvert > k$,
    \end{dcases*}
    \qquad (A'\in r(A)).
  \end{align*}
  Since $\rn{p}$ and $\rn{q}$ are equidistributed, it follows that $\rn{r}$ also has the same
  distribution as $\rn{p}$ and $\rn{q}$.

  Since by construction we have $(\rn{r}_{A'}\mid A'\in r(A,k))=(\rn{p}_{A'}\mid A'\in r(A,k))$, it
  follows that the $\rn{p}$ and $\rn{r}$ have the same conditional distribution given
  $(\rn{p}_{A'}\mid A'\in r(A,k))$. On the other hand, by construction of $\rn{r}$, it is obvious
  that the conditional distribution of $\rn{r}$ given $(\rn{p}_{A'}\mid A'\in r(A,k))$ is
  (conditionally) independent from $(\rn{x}_{A'}\mid A'\in r(A,k))$, concluding the proof of
  item~\ref{lem:amlg:welldef->strong}.

  \medskip

  It remains to prove item~\ref{lem:amlg:strong->weak}. We prove item~\ref{lem:amlg:weak} by
  induction in $k < \lvert A\rvert$.

  Let $\rn{z}$ be picked at random in $\prod_{C\in\binom{A}{>k}} (X\times X\times\cO_C^d)$ according
  to $\bigotimes_{C\in\binom{A}{>k}} (\mu\otimes\mu\otimes\bigotimes_{i=1}^d \nu_C)$ and for each
  $y\in\cE_{A,k}^{(d)}(\Omega^2)$, let
  \begin{align*}
    \rn{p}^y & \df \pp^{\dsct,\widehat{\cN}}_{A,V}(y,\rn{z})
  \end{align*}
  and note that it suffices to show that there exists a $\mu^{(d)}$-measure $1$ subset $G$ of
  $\cE_{A,k}^{(d)}(\Omega^2)$ such that if $y,y'\in G$ are such that
  $\pp^{\amlg,\ell,\widehat{\cN}}_{A',V}(\iota_{A',A}(y)) =
  \pp^{\amlg,\ell,\widehat{\cN}}_{A',V}(\iota_{A',A}(y'))$ for every $A'\in r(A,k)$, then $\rn{p}^y$
  and $\rn{p}^{y'}$ are equidistributed. By our inductive hypothesis in $k$, we may further suppose
  that $y_{A'} = y'_{A'}$ for every $A'\in r(A,k-1)$.

  We now introduce new random variables $\rn{r}^y$ in $\cS_{A,V}^{(d),\dsct}(\Omega^2)$ for each
  $y\in\cE_{A,k}^{(d)}(\Omega^2)$: by Proposition~\ref{prop:weakamalgdsctovlp}, there exists a
  Markov kernel
  \begin{align*}
    \rho\colon
    \prod_{A'\in \binom{A}{k}}\cS_{A',V}^{(d),\ovlp}(\Omega^2)
    \to
    \cP\bigl(\cS_{A,V}^{(d),\dsct}(\Omega^2)\bigr)
  \end{align*}
  that allows us to correctly sample the conditional distribution of the dissociated $(A,V)$-type in
  $\widehat{\cN}$ given the overlapping $(A',V)$-types for every $A'\in\binom{A}{k}$. For each
  $A'\in\binom{A}{k}$, we let
  $\tau_{A'}\colon\cE_{A'}^{(d)}(\Omega^2)\to\cE_{A',k-1}^{(d)}(\Omega^2)$ be the natural projection
  and
  \begin{align*}
    q^y_{A'}
    & \df
    \qq^{\amlg,\ell,\widehat{\cN}}_{A',V}\Bigl(
    \tau_{A'}(y),
    \pp^{\amlg,\ell,\widehat{\cN}}_{A',V}\bigr(\iota_{A',A}^*(y)\bigr)
    \Bigr)
  \end{align*}
  be the overlapping $(A',V)$-type recovered via $\qq^{\amlg,\ell,\widehat{\cN}}_{A',V}$ from the
  amalgamating $(A',V,\ell)$-type of $\iota_{A',A}^*(y)$ along with the pointwise information
  $\tau_{A'}(y)$ of coordinates indexed by sets in $r(A',k-1)$ and finally let $\rn{r}^y$ be sampled
  according to $\rho((q^y_{A'} \mid A'\in\binom{A}{k}))$.

  By item~\ref{def:amlg:qq} in Definition~\ref{def:amlg}, we know that
  \begin{align*}
    q^y_{A'} & = \pp^{\ovlp,\widehat{\cN}}_{A',V}(\iota_{A',A}^*(y))
  \end{align*}
  whenever $\iota_{A',A}^*(y)\in\dom(\pp^{\amlg,\ell,\widehat{\cN}}_{A',V})$. In turn, since our
  $\rho$ was provided by Proposition~\ref{prop:weakamalgdsctovlp}, we conclude that
  \begin{align}\label{eq:ry}
    \rn{r}^y & \sim \pp^{\dsct,\widehat{\cN}}_{A,V}(y,\rn{z}) = \rn{p}^y
  \end{align}
  for $\mu^{(d)}$-almost every $y$ in the set
  \begin{align*}
    G & \df
    \left\{y\in\cE_{A,k}^{(d)}(\Omega^2) \;\middle\vert\;
    \forall A'\in\binom{A}{k}, \iota_{A',A}^*(y)\in\dom(\pp^{\amlg,\ell,\widehat{\cN}}_{A',V})
    \right\}.
  \end{align*}
  Clearly $\mu^{(d)}(G)=1$, so if $\widetilde{G}$ is the set of all $y\in G$ such that~\eqref{eq:ry}
  holds, then $\mu^{(d)}(\widetilde{G})=1$.

  Now let $y,y'\in\widetilde{G}$ be such that
  $\pp^{\amlg,\ell,\widehat{\cN}}_{A',V}(\iota_{A',A}(y)) =
  \pp^{\amlg,\ell,\widehat{\cN}}_{A',V}(\iota_{A',A}(y'))$ for every $A'\in r(A,k)$ and $y_C=y'_C$
  for every $C\in r(A,k-1)$. The latter condition implies that for every $A'\in\binom{A}{k}$, we get
  $\tau_{A'}(y) = \tau_{A'}(y')$, so we get $q^y_{A'} = q^{y'}_{A'}$, which in turn implies
  \begin{align*}
    \rn{p}^y & \sim \rn{r}^y = \rn{r}^{y'} \sim \rn{p}^{y'},
  \end{align*}
  as desired.
\end{proof}

Before we proceed, let us show equivariance of the definition of
$\Theta^{\amlg,\ell,\widehat{\cN}}$ with respect to bijections.

\begin{lemma}\label{lem:Thetaequiv}
  Within Definition~\ref{def:amlg}, if $\beta\colon A'\to A$ is a bijection between finite sets $A'$
  and $A$, both disjoint from a countably infinite set $V$, then the diagram
  \begin{equation*}
    \begin{tikzcd}
      \prod_{B\in r(A,\lvert A\rvert-1)}\cS_{B,V}^{(d),\amlg,\ell}(\Omega^2)
      \arrow[r, harpoon, "\Theta^{\amlg,\ell,\widehat{\cN}}_{A,V}"]
      \arrow[d, "\beta^*"']
      &
      \cP(\cS_{A,V}^{(d),\amlg,\ell}(\Omega^2))
      \arrow[d, "\beta^*"]
      \\
      \prod_{B\in r(A',\lvert A'\rvert-1)}\cS_{B,V}^{(d),\amlg,\ell}(\Omega^2)
      \arrow[r, harpoon, "\Theta^{\amlg,\ell,\widehat{\cN}}_{A',V}"]
      &
      \cP(\cS_{A',V}^{(d),\amlg,\ell}(\Omega^2))
    \end{tikzcd}
  \end{equation*}
  is weakly commutative.
\end{lemma}

\begin{proof}
  Let $a\df\lvert A\rvert=\lvert A'\rvert$ and $\gamma\df\alpha_A^{-1}\comp\beta\comp\alpha_{A'}\in
  S_a$ and note that
  \begin{align*}
    \beta^*\comp\Theta^{\amlg,\ell,\widehat{\cN}}_{A,V}
    & =
    \beta^*\comp
    (\alpha_A^{-1})^*\comp\Theta^{\amlg,\ell,\widehat{\cN}}_{a,V}\comp\alpha_A^*
    =
    (\alpha_{A'}^{-1})^*\comp\gamma^*\comp\Theta^{\amlg,\ell,\widehat{\cN}}_{a,V}\comp\alpha_A^*
    \\
    & =
    (\alpha_{A'}^{-1})^*\comp\Theta^{\amlg,\ell,\widehat{\cN}}_{a,V}\comp\gamma^*\comp\alpha_A^*
    =
    (\alpha_{A'}^{-1})^*\comp\Theta^{\amlg,\ell,\widehat{\cN}}_{a,V}\comp\alpha_{A'}^*\comp\beta^*
    \\
    & =
    \Theta^{\amlg,\ell,\widehat{\cN}}_{A',V}\comp\beta^*,
  \end{align*}
  where the third equality follows since $\Theta^{\amlg,\ell,\widehat{\cN}}_{a,V}$ is $S_a$-equivariant.
\end{proof}

We can finally prove the implication~\ref{thm:UCouple:UCouple}$\implies$\ref{thm:UCouple:1ellindep}
of Theorem~\ref{thm:UCouple}.

\begin{proposition}[Theorem~\ref{thm:UCouple}\ref{thm:UCouple:UCouple}$\implies$\ref{thm:UCouple:1ellindep}]\label{prop:UCouple->1ellindep}
  If $\phi\in\HomT$ satisfies $\UCouple[\ell]$, then there exists an $\ell$-independent
  $1$-Euclidean structure $\cN$ such that $\phi=\phi_\cN$.
\end{proposition}

\begin{proof}
  Let $\widetilde{\Omega}=([0,1],\lambda)$, where $\lambda$ is the Lebesgue measure. Since
  $\UCouple[\ell]$ implies $\UCouple[\ell-1]$ and since every $d$-Euclidean structure is
  $0$-independent, it suffices to show that if $\cN$ is an $(\ell-1)$-independent $d$-Euclidean
  structure over some $\Omega=(X,\lambda)$, then there exists an $\ell$-independent $1$-Euclidean
  structure $\cH$ over $\widetilde{\Omega}$ such that $\phi_\cN = \phi_\cH$ (in fact, it would even
  suffice to show only the case $d=1$ and $\Omega=\widetilde{\Omega}$, but the general $d\in\NN$ and
  arbitrary $\Omega$ case is not any harder and the different notation will help keep track of which
  points are in which spaces).

  Let $\widehat{\cN}$ be the $d$-Euclidean structure in $\cL$ over $\Omega^2$ as in
  Definition~\ref{def:amlg}:
  \begin{align*}
    \widehat{\cN}_P & \df \{(x,x')\in\cE_{k(P)}^{(d)}(\Omega)\times\cE_{k(P)}(\Omega) \mid x\in\cN_P\}
    \qquad (P\in\cL).
  \end{align*}

  We will first provide an alternative way of sampling amalgamating $(A,V,\ell)$-types of
  $\widehat{\cN}$ for every finite $A\subseteq\NN_+$ using i.i.d.\ variables distributed uniformly
  in $[0,1]$ (note that by equivariance of Lemma~\ref{lem:amlgcomm}, this also provides a way of
  sampling the amalgamating $(A,V,\ell)$-types even when $A$ is not a subset of $\NN_+$). Throughout
  this proof, we always assume $V$ is a countably infinite set disjoint from $\NN_+$ (which ensures
  it is disjoint from any $A$ we consider) and we also follow the notation of
  Definition~\ref{def:amlg}.

  First, since all spaces of amalgamating types are Borel spaces (see
  Proposition~\ref{prop:toptypes}\ref{prop:toptypes:Pol}), for each $a\in\NN_+$, we can associate to
  each Markov kernel
  \begin{align*}
    \Theta^{\amlg,\ell,\widehat{\cN}}_{a,V}\colon
    \prod_{B\in r(a,a-1)}\cS_{B,V}^{(d),\amlg,\ell}(\Omega^2)
    & \pto
    \cP(\cS_{a,V}^{(d),\amlg,\ell}(\Omega^2))
    \intertext{of Definition~\ref{def:amlg} a measurable partial function}
    F_{a,V}\colon
    \prod_{B\in r(a,a-1)}\cS_{B,V}^{(d),\amlg,\ell}(\Omega^2)\times[0,1]
    & \pto
    \cS_{a,V}^{(d),\amlg,\ell}(\Omega^2)
  \end{align*}
  such that
  \begin{align}\label{eq:FaVdef}
    F_{a,V}(p,\place)_*(\lambda) & = \Theta^{\amlg,\ell,\widehat{\cN}}_{a,V}(p)
  \end{align}
  for every $p\in\prod_{B\in r(a,a-1)}\cS_{B,V}^{(d),\amlg,\ell}(\Omega^2)$, that is,
  $\Theta^{\amlg,\ell,\widehat{\cN}}_{a,V}(p)$ is the distribution of $F_{a,V}(p,\rn{t})$ where
  $\rn{t}$ is sampled uniformly at random in $[0,1]$.

  Item~\ref{def:amlg:Theta} of Definition~\ref{def:amlg} says that when $a\leq\ell$, then
  $\Theta^{\amlg,\ell,\widehat{\cN}}_{a,V}$ is the constant function that maps all points to the
  distribution concentrated on the unique point $q_a$ in the image of $\pp^{(d),\amlg,\ell}_{a,V}$,
  so when $a\leq\ell$, we can further require that $F_{a,V}$ is the constant function that maps
  every point to $q_a$.

  It is worth noting that even though $\Theta^{\amlg,\ell,\widehat{\cN}}_{a,V}$ is
  $S_a$-equivariant, we cannot require $F_{a,V}$ to be $S_a$-equivariant (as there are
  $S_a$-invariant distributions that cannot be generated by $S_a$-equivariant functions; e.g., the
  quasirandom tournament, see Example~\ref{ex:qrT}).

  For $A\subseteq\NN_+$ finite non-empty that is not of the form $[a]$ for some $a\in\NN_+$, we also
  define
  \begin{align*}
    F_{A,V}\colon
    \prod_{B\in r(A,\lvert A\rvert-1)}\cS_{B,V}^{(d),\amlg,\ell}(\Omega^2)\times[0,1]
    & \pto
    \cS_{A,V}^{(d),\amlg,\ell}(\Omega^2)
  \end{align*}
  in a similar fashion to~\eqref{eq:equivdef}:
  \begin{align}\label{eq:FAV}
    F_{A,V} & \df (\alpha_A^{-1})^*\comp F_{\lvert A\rvert,V}\comp\alpha_A^*,
  \end{align}
  where $\alpha_A\colon [\lvert A\rvert]\to A$ is the same bijection used in~\eqref{eq:equivdef} in
  Definition~\ref{def:amlg} and we declare the rightmost $\alpha_A^*$ in the above to act trivially
  on the coordinate in $[0,1]$. Note also that since we set $\alpha_{[a]}=\id_a$, it follows
  that~\eqref{eq:FAV} for every $A\subseteq\NN_+$ regardless of whether it is of the form $[a]$ or
  not.

  Note that Lemma~\ref{lem:Thetaequiv} along with~\eqref{eq:FaVdef} implies that
  \begin{align}\label{eq:FAVTheta}
    F_{A,V}(p,\place)_*(\lambda) & = \Theta^{\amlg,\ell,\widehat{\cN}}_{A,V}(p)
  \end{align}
  for every $p\in\prod_{B\in r(A,\lvert A\rvert-1)}\cS_{B,V}^{(d),\amlg,\ell}(\Omega^2)$.
  
  Note also that equivariance of $\pp^{\amlg,\ell,\widehat{\cN}}$ (Lemma~\ref{lem:amlgcomm}) implies that
  $F_{A,V}$ also has the property that if $\lvert A\rvert\leq\ell$, then $F_{A,V}$ is the constant
  function that maps every point to the unique point in the image of
  $\pp^{\amlg,\ell,\widehat{\cN}}_{A,V}$.

  Again, $F_{A,V}$ is not necessarily $S_A$-equivariant. To address this, we will use order
  variables to define a partial function
  \begin{align*}
    G_{A,V}\colon
    \prod_{B\in r(A,\lvert A\rvert-1)}\cS_{B,V}^{(d),\amlg,\ell}(\Omega^2)\times[0,1]\times\cO_A
    & \pto
    \cS_{A,V}^{(d),\amlg,\ell}(\Omega^2)
  \end{align*}
  by  
  \begin{align*}
    G_{A,V}(p,t,{\lhd})
    & \df
    (\sigma_\lhd^{-1})^*\Bigl(F_{A,V}\bigl(\sigma_\lhd^*(p,t)\bigr)\Bigr)
    \qquad
    \left(
    p\in\prod_{B\in r(A,\lvert A\rvert-1)}\cS_{B,V}^{(d),\amlg,\ell}(\Omega^2),
    t\in[0,1], {\lhd}\in\cO_A
    \right),
  \end{align*}
  where $\sigma_\lhd\in S_A$ is the unique permutation such that
  \begin{align*}
    (\sigma_\lhd\comp\alpha_A)^*({\lhd}) & = {\leq_{\lvert A\rvert}},
  \end{align*}
  where $\leq_{\lvert A\rvert}$ is the usual order on $[\lvert A\rvert]$ (recalling that
  $\alpha_A\colon[\lvert A\rvert]\to A$ is the bijection used in~\eqref{eq:equivdef} in
  Definition~\ref{def:amlg}). Again, Lemma~\ref{lem:amlgcomm} implies that $G_{A,V}$ also has the
  property that if $\lvert A\rvert\leq\ell$, then $G_{A,V}$ is the constant function that maps every
  point to the unique point in the image of $\pp^{\amlg,\ell,\widehat{\cN}}_{A,V}$.

  The next claim gives a form of equivariance of $\sigma_\lhd$.
  \begin{claim}\label{clm:sigmalhd}
    If $\beta\colon A\to B$ is a bijection and ${\lhd}\in\cO_B$, then
    \begin{align*}
      \sigma_{\beta^*({\lhd})} & = \beta^{-1}\comp\sigma_\lhd\comp\alpha_B\comp\alpha_A^{-1}.
    \end{align*}
  \end{claim}

  \begin{proofof}{Claim~\ref{clm:sigmalhd}}
    Letting $a\df\lvert A\rvert=\lvert B\rvert$, the claim follows from:
    \begin{align*}
      (\beta^{-1}\comp\sigma_\lhd\comp\alpha_B\comp\alpha_A^{-1}\comp\alpha_A)^*\bigl(\beta^*({\lhd})\bigr)
      & =
      (\sigma_\lhd\comp\alpha_B)^*({\lhd})
      =
      {\leq_a}.
      \qedhere
    \end{align*}
  \end{proofof}

  The next claim says that the partial functions $G_{A,V}$ are equivariant under bijections.
  \begin{claim}\label{clm:GAV}
    If $\beta\colon A\to B$ is a bijection, then the diagram
    \begin{equation*}
      \begin{tikzcd}
        \prod_{C\in r(B,\lvert B\rvert-1)}\cS_{C,V}^{(d),\amlg,\ell}(\Omega^2)\times\cO_B
        \arrow[r, harpoon, "G_{B,V}"]
        \arrow[d, "\beta^*"']
        &
        \cS_{B,V}^{(d),\amlg,\ell}(\Omega^2)
        \arrow[d, "\beta^*"]
        \\
        \prod_{C\in r(B,\lvert B\rvert-1)}\cS_{C,V}^{(d),\amlg,\ell}(\Omega^2)\times\cO_B
        \arrow[r, harpoon, "G_{B,V}"]
        &
        \cS_{B,V}^{(d),\amlg,\ell}(\Omega^2)
      \end{tikzcd}
    \end{equation*}
    is weakly commutative.
  \end{claim}

  \begin{proofof}{Claim~\ref{clm:GAV}}
    Let $a\df\lvert A\rvert=\lvert B\rvert$ and let
    $(p,t,{\lhd})\in\dom(G_{A,V}\comp\beta^*)\cap\dom(G_{B,V})$ and note that
    \begin{align*}
      (G_{A,V}\comp\beta^*)(p,t,{\lhd})
      & =
      (\sigma^{-1}_{\beta^*({\lhd})})^*\biggl(
      F_{A,V}\Bigl(\sigma_{\beta^*({\lhd})}\bigl(\beta^*(p,t)\bigr)\Bigr)
      \biggr)
      \\
      & =
      (\alpha_A\comp\alpha_B^{-1}\comp\sigma_\lhd^{-1}\comp\beta)^*\biggl(
      F_{A,V}\Bigl(
      (\beta^{-1}\comp\sigma_\lhd\comp\alpha_B\comp\alpha_A^{-1})^*\bigl(\beta^*(p,t)\bigr)
      \Bigr)\biggr)
      \\
      & =
      \beta^*\biggl(
      (\alpha_A\comp\alpha_B^{-1}\comp\sigma_\lhd^{-1})^*\Bigl(
      F_{A,V}\bigl(
      (\sigma_\lhd\comp\alpha_B\comp\alpha_A^{-1})^*(p,t)
      \bigr)\Bigr)\biggr)
      \\
      & =
      \beta^*\biggl(
      (\alpha_B^{-1}\comp\sigma_\lhd^{-1})^*\Bigl(
      F_{a,V}\bigl(
      (\sigma_\lhd\comp\alpha_B)^*(p,t)
      \bigr)\Bigr)\biggr)
      \\
      & =
      \beta^*\biggl(
      (\sigma_\lhd^{-1})^*\Bigl(
      F_{B,V}\bigl(
      \sigma_\lhd^*(p,t)
      \bigr)\Bigr)\biggr)
      \\
      & =
      (\beta^*\comp G_{B,V})(p,t,{\lhd}),
    \end{align*}
    where the second equality follows from Claim~\ref{clm:sigmalhd} and the fourth and fifth
    equalities follow from~\eqref{eq:FAV}.
  \end{proofof}

  We proceed to show that the partial function $G_{A,V}$ can be used to correctly sample the
  conditional distribution of the amalgamating $(A,V,\ell)$-types given the lower amalgamating
  types.
  \begin{claim}\label{clm:GAVsample}
    For $\mu^{(d)}$-almost every $x\in\cE_{A,\lvert A\rvert-1}^{(d)}(\Omega^2)$ and every
    ${\lhd}\in\cO_A$, we have
    \begin{align*}
      G_{A,V}(p_x, \place, {\lhd})_*(\lambda) & = \Theta^{\amlg,\ell,\widehat{\cN}}_{A,V}(p_x),
    \end{align*}
    where $p_x\in\prod_{B\in r(A,\lvert A\rvert-1)}\cS_{B,V}^{(d),\amlg,\ell}(\Omega^2)$ is given by
    \begin{align*}
      (p_x)_B & \df \pp^{\amlg,\ell,\widehat{\cN}}\bigl(\iota_{B,A}^*(x)\bigr)
      \qquad (B\in r(A,\lvert A\rvert-1)).
    \end{align*}

    In particular, we have
    \begin{align*}
      G_{A,V}(p_x,\place)_*(\lambda\otimes\nu_A) & = \Theta^{\amlg,\ell,\widehat{\cN}}_{A,V}(p_x)
    \end{align*}
    for $\mu^{(d)}$-almost every $x\in\cE_{A,\lvert A\rvert-1}^{(d)}(\Omega^2)$.
  \end{claim}

  \begin{proofof}{Claim~\ref{clm:GAVsample}}
    It is clear that the second assertion follows from the first as $\nu_A$ is the uniform measure
    on $\cO_A$.

    In turn, to prove the first assertion, it suffices to show that if $\rn{x}$ is sampled in
    $\cE_{A,\lvert A\rvert-1}^{(d)}(\Omega^2)$ according to $\mu^{(d)}$, then
    \begin{align*}
      G_{A,V}(p_{\rn{x}},\place,{\lhd})_*(\lambda) & = \Theta^{\amlg,\ell,\widehat{\cN}}_{A,V}(p_{\rn{x}})
    \end{align*}
    with probability $1$.

    By item~\ref{def:amlg:Theta} in Definition~\ref{def:amlg}, we know that if we sample $\rn{y}$ in
    $X^2\times\cO_A^d$ according to $\mu^2\otimes\bigotimes_{i=1}^d\nu_A$ independently from
    $\rn{x}$, then $\Theta^{\amlg,\ell,\widehat{\cN}}_{A,V}(p_{\rn{x}})$ is the conditional
    distribution of $\pp^{\amlg,\ell,\widehat{\cN}}_{A,V}(\rn{x},\rn{y})$ given $\rn{x}$.

    By Lemma~\ref{lem:amlgcomm}, we also have
    \begin{align*}
      \pp^{\amlg,\ell,\widehat{\cN}}_{A,V}(\rn{x},\rn{y})
      & =
      (\sigma_\lhd^{-1})^*\Bigl(
      \pp^{\amlg,\ell,\widehat{\cN}}_{A,V}\bigl(\sigma_\lhd^*(\rn{x},\rn{y})\bigr)
      \Bigr).
    \end{align*}
    By considering conditional distributions given $\rn{x}$ of the above and applying
    item~\ref{def:amlg:Theta} in Definition~\ref{def:amlg} again, we get
    \begin{align*}
      \Theta^{\amlg,\ell,\widehat{\cN}}_{A,V}(p_{\rn{x}})
      & =
      (\sigma_\lhd^{-1})^*\bigl(
      \Theta^{\amlg,\ell,\widehat{\cN}}_{A,V}(p_{\sigma_\lhd^*(\rn{x})})
      \bigr).
    \end{align*}
    In turn, applying~\eqref{eq:FAVTheta} to the right-hand side of the above, we get
    \begin{align*}
      \Theta^{\amlg,\ell,\widehat{\cN}}_{A,V}(p_{\rn{x}})
      & =
      (\sigma_\lhd^{-1})^*\Bigl(F_{A,V}(p_{\sigma_\lhd^*(\rn{x})},\place)_*(\lambda)\Bigr)
      \\
      & =
      ((\sigma_\lhd^{-1})^*\comp F_{A,V}\comp\sigma_\lhd^*)(p_{\rn{x}},\place)_*(\lambda)
      \\
      & =
      G_{A,V}(p_{\rn{x}},\place,{\lhd})_*(\lambda),
    \end{align*}
    where the second equality follows from equivariance of $\pp^{\amlg,\ell,\widehat{\cN}}$ (see
    Lemma~\ref{lem:amlgcomm}) and since the action on $[0,1]$ is trivial.
  \end{proofof}

  Our goal is now to upgrade the correctness of the sampling via $G_{A,V}$ to joint
  distributions. To do so, for $A\in r(\NN_+)$ and $c\leq\lvert A\rvert$,
  let us define the natural ``product'' partial function
  \begin{align*}
    \widehat{G}_{A,c,V}\colon
    \prod_{B\in r(A,c-1)}\cS_{B,V}^{(d),\amlg,\ell}(\Omega^2)
    \times\prod_{C\in\binom{A}{c}} ([0,1]\times\cO_C)
    & \pto
    \prod_{C\in\binom{A}{c}}\cS_{C,V}^{(d),\amlg,\ell}(\Omega^2)
  \end{align*}
  by
  \begin{align*}
    \widehat{G}_{A,c,V}(p,u)_C
    & \df
    G_{C,V}\bigl(\iota_{C,A}^*(p), u_C\bigr)
  \end{align*}
  for every $C\in\binom{A}{c}$, every $p\in\prod_{B\in
    r(A,c-1)}\cS_{B,V}^{(d),\amlg,\ell}(\Omega^2)$ and every $u\in\prod_{C'\in\binom{A}{c}}
  ([0,1]\times\cO_{C'})$.

  The next claim is the desired upgrade of Claim~\ref{clm:GAVsample} for $\widehat{G}_{A,c,V}$.

  \begin{claim}\label{clm:GAVjointsample}
    Let $A\in r(\NN_+)$ and let $c\leq\lvert A\rvert$. For every $x\in\cE_{A,c-1}^{(d)}(\Omega^2)$
    and every $C\in\binom{A}{c}$, let $p_x\in\prod_{B\in
      r(A,c-1)}\cS_{B,V}^{(d),\amlg,\ell}(\Omega^2)$ and $p_{x,C}\in\prod_{B\in
      r(C,c-1)}\cS_{B,V}^{(d),\amlg,\ell}(\Omega^2)$ be given by
    \begin{align*}
      (p_x)_B & \df \pp^{\amlg,\ell,\widehat{\cN}}_{B,V}\bigl(\iota_{B,A}^*(x)\bigr)
      \qquad (B\in r(A,c-1)),
      \\
      (p_{x,C})_B & \df \pp^{\amlg,\ell,\widehat{\cN}}_{B,V}\bigl(\iota_{B,C}^*(x)\bigr)
      \qquad (B\in r(C,c-1)).
    \end{align*}

    Let further $\rn{x}$ be picked in $\cE_{A,c-1}^{(d)}(\Omega^2)$ according to $\mu^{(d)}$, $\rn{y}$
    be picked in $\prod_{C\in\binom{A}{c}} (X^2\times\cO_C^d)$ according to
    $\bigotimes_{C\in\binom{A}{c}}(\mu^2\otimes\bigotimes_{i=1}^d\nu_C)$ independently from $\rn{x}$
    and $\rn{u}$ be picked in $\prod_{C\in\binom{A}{c}}([0,1]\times\cO_C)$ according to
    $\bigotimes_{C\in\binom{A}{c}}(\lambda\otimes\nu_C)$ independently from $\rn{x}$.

    Finally let
    \begin{align*}
      \rn{\widehat{G}} & \df \widehat{G}_{A,c,V}(p_{\rn{x}},\rn{u}),
      &
      \rn{q}
      & \df
      \left(\pp^{\amlg,\ell,\widehat{\cN}}_{C,V}(\iota_{C,A}^*(\rn{x}),\rn{y}_C)
      \;\middle\vert\; C\in\binom{A}{c}\right).
    \end{align*}

    Then the following conditional distributions are the same:
    \begin{enumerate*}[label={(\roman*)}, ref={(\roman*)}]
    \item $\rn{\widehat{G}}$ given $p_{\rn{x}}$,
    \item $\rn{\widehat{G}}$ given $\rn{x}$,
    \item $\rn{q}$ given $p_{\rn{x}}$,
    \item $\rn{q}$ given $\rn{x}$.
    \end{enumerate*}
  \end{claim}

  \begin{proofof}{Claim~\ref{clm:GAVjointsample}}
    The fact that the conditional distributions of $\rn{\widehat{G}}$ given $p_{\rn{x}}$ or given
    $\rn{x}$ are the same follows since the definition of $\rn{\widehat{G}}$ implies that it is
    conditionally independent from $\rn{x}$ given $p_{\rn{x}}$. In fact, the definition of
    $\rn{\widehat{G}}$ along with Claim~\ref{clm:GAVsample} implies that this conditional
    distribution is exactly
    \begin{align*}
      \bigotimes_{C\in\binom{A}{c}}\Theta^{\amlg,\ell,\widehat{\cN}}_{C,V}(p_{\rn{x},C}).
    \end{align*}
    
    The fact that the conditional distributions of $\rn{q}$ given $p_{\rn{x}}$ or given $\rn{x}$ are
    the same is exactly strong amalgamation, Lemma~\ref{lem:amlg}\ref{lem:amlg:strong}: $\rn{q}$ is
    conditionally independent from $\rn{x}$ given $p_{\rn{x}}$. Since each coordinate $\rn{q}_C$ is
    clearly $(\iota_{C,A}^*(\rn{x}),\rn{y}_C)$-measurable, it follows that the coordinates of
    $\rn{q}$ are conditionally independent given $\rn{x}$.

    In turn, this conditional independence along with item~\ref{def:amlg:Theta} in
    Definition~\ref{def:amlg} implies that the conditional distribution of $\rn{q}$ given $\rn{x}$
    is exactly
    \begin{align*}
      \bigotimes_{C\in\binom{A}{c}}\Theta^{\amlg,\ell,\widehat{\cN}}_{C,V}(p_{\rn{x},C}),
    \end{align*}
    concluding the proof of the claim.
  \end{proofof}

  We next sample all amalgamating types inductively and use Claim~\ref{clm:GAVjointsample} to show
  that this gives the correct distribution. First, we let $\rn{z}$ be picked at random in
  $\cE_{\NN_+}^{(1)}(\widetilde{\Omega})$ (recall that $\widetilde{\Omega}\df([0,1],\lambda)$)
  according to $\lambda^{(1)}$ and for each $A\in r(\NN_+)$, we define inductively the random
  $(A,V,\ell)$-amalgamating type
  \begin{align}\label{eq:pA}
    \rn{p}_A & \df G_{A,V}\bigl((\rn{p}_B \mid B\in r(A,\lvert A\rvert-1)), \rn{z}_A\bigr).
  \end{align}

  \begin{claim}\label{clm:sampleamlg}
    If $A\in r(\NN_+)$ and $\rn{x}$ is sampled in $\cE_A^{(d)}(\Omega^2)$ according to $\mu^{(d)}$,
    then $\rn{p}_A$ and $\pp^{\amlg,\ell,\widehat{\cN}}_{A,V}(\rn{x})$ have the same distribution.
  \end{claim}

  \begin{proofof}{Claim~\ref{clm:sampleamlg}}
    For each $B\in r(A)$, let
    $\rn{q}_B\df\pp^{\amlg,\ell,\widehat{\cN}}_{B,V}(\iota_{B,A}^*(\rn{x}))$. Then it suffices to
    show that for every $c\in[\lvert A\rvert]$, we have
    \begin{align*}
      \bigl(\rn{p}_B \mid B\in r(A,c)\bigr)
      & \sim
      \bigl(\rn{q}_B \mid B\in r(A,c)\bigr).
    \end{align*}
    We prove this by induction in $c\in\{0,\ldots,\lvert A\rvert\}$.

    The result is trivial when $c=0$. When $c>0$, by definition, we have
    \begin{align*}
      \left(\rn{p}_C \;\middle\vert\; C\in\binom{A}{c}\right)
      & =
      \widehat{G}_{A,c,V}\left(
      \bigl(\rn{p}_B \;\big\vert\; B\in r(A,c-1)\bigr),
      \left(\rn{z}_C\;\middle\vert\; C\in\binom{A}{c}\right)
      \right).
    \end{align*}

    By inductive hypothesis, we know that the first argument of $\widehat{G}_{A,c,V}$ above has the same
    distribution as $(\rn{q}_B \mid B\in r(A,c-1))$, so we have
    \begin{multline*}
      \bigl(\rn{p}_B \;\big\vert\; B\in r(A,c)\bigr)
      \\
      \sim
      \left(
      \bigl(\rn{q}_B \;\big\vert\; B\in r(A,c-1)\bigr),
      \widehat{G}_{A,c,V}\left(
      \bigl(\rn{q}_B \;\big\vert\; B\in r(A,c-1)\bigr),
      \left(\rn{u}_C\;\middle\vert\; C\in\binom{A}{c}\right)
      \right)
      \right),
    \end{multline*}
    where $\rn{u}$ is picked in $\prod_{C\in\binom{A}{c}}([0,1]\times\cO_C)$ according to
    $\bigotimes_{C\in\binom{A}{c}}(\lambda\otimes\nu_C)$, independently from $(\rn{q}_B \mid B\in
    r(A,c-1))$.

    Conditioning on $(\rn{q}_B \mid B\in r(A,c-1))$ and using Claim~\ref{clm:GAVjointsample}, we
    conclude that the distribution of the right-hand side is the same as the distribution of
    $(\rn{q}_B \mid B\in r(A,c))$, as desired.
  \end{proofof}

  The inductive definition of $\rn{p}_A$ in~\eqref{eq:pA} implies the existence of measurable
  partial functions
  $f_A\colon\cE_A^{(1)}(\widetilde{\Omega})\pto\cS_{A,V}^{(d),\amlg,\ell}(\Omega^2)$ such that
  \begin{align}\label{eq:pAfA}
    \rn{p}_A
    & =
    f_A\bigl((\rn{z}_B \mid B\in r(A))\bigr)
    =
    f_A\bigl(\iota_{A,\NN_+}(\rn{z})\bigr).
  \end{align}
  More precisely, the partial functions $f_A$ are defined inductively as
  \begin{align}\label{eq:fA}
    f_A(z)
    & \df
    G_{A,V}\biggl(
    \Bigl(
    f_B\bigl(\iota_{B,A}^*(z)\bigr)
    \mathrel{\Big\vert}
    B\in r(A,\lvert A\rvert-1)
    \Bigr), z_A
    \biggr).
  \end{align}

  Note that a direct consequence of Claim~\ref{clm:GAVsample} is that
  $\lambda^{(1)}(\dom(f_A))=1$. Also, yet again $f_A$ inherits the property that if $\lvert
  A\rvert\leq\ell$, then $f_A$ is the constant function that maps every point to the unique point in
  the image of $\pp^{\amlg,\ell,\widehat{\cN}}_{A,V}$. The next claim says that $f_A$ also inherits
  equivariance under bijections.

  \begin{claim}\label{clm:fAequiv}
    If $\beta\colon A\to B$ is a bijection, then the diagram
    \begin{equation*}
      \begin{tikzcd}
        \cE_B^{(1)}(\widetilde{\Omega})
        \arrow[r, "f_B"]
        \arrow[d, "\beta^*"']
        &
        \cS_{B,V}^{(d),\amlg,\ell}(\Omega^2)
        \arrow[d, "\beta^*"]
        \\
        \cE_A^{(1)}(\widetilde{\Omega})
        \arrow[r, "f_A"]
        &
        \cS_{A,V}^{(d),\amlg,\ell}(\Omega^2)
      \end{tikzcd}
    \end{equation*}
    is weakly commutative.
  \end{claim}

  \begin{proofof}{Claim~\ref{clm:fAequiv}}
    The proof is by induction in $\lvert A\rvert$. Note that if
    $z\in\dom(f_B)\cap\dom(f_A\comp\beta^*)$, then
    \begin{align*}
      f_A\bigl(\beta^*(z)\bigr)
      & =
      G_{A,V}\Biggl(
      \biggl(
      f_C\Bigl(\iota_{C,A}^*\bigl(\beta^*(z)\bigr)\Bigr)
      \mathrel{\bigg\vert}
      C\in r(A,\lvert A\rvert-1)
      \biggr), z_{\beta(A)}
      \Biggr)
      \\
      & =
      G_{A,V}\Biggl(
      \biggl(
      f_C\Bigl(\beta\down_C^*\bigl(\iota_{\beta(C),B}^*(z)\bigr)\Bigr)
      \mathrel{\bigg\vert}
      C\in r(A,\lvert A\rvert-1)
      \biggr), z_B
      \Biggr)
      \\
      & =
      G_{A,V}\Biggl(
      \biggl(
      \beta\down_C^*\Bigl(
      f_{\beta(C)}\bigl(\iota_{\beta(C),B}^*(z)\bigr)
      \Bigr)
      \mathrel{\bigg\vert}
      C\in r(A,\lvert A\rvert-1)
      \biggr), z_B
      \Biggr)
      \\
      & =
      G_{A,V}\Biggl(
      \beta^*\biggl(
      \Bigl(
      f_D\bigl(\iota_{D,B}^*(z)\bigr)
      \mathrel{\Big\vert}
      D\in r(B,\lvert B\rvert-1)
      \Bigr),
      z_B
      \biggr)
      \Biggr)
      \\
      & =
      \beta^*\Biggl(
      G_{B,V}\biggl(
      \Bigl(
      f_D\bigl(\iota_{D,B}^*(z)\bigr)
      \mathrel{\Big\vert}
      D\in r(B,\lvert B\rvert-1)
      \Bigr),
      z_B
      \biggr)
      \Biggr)
      \\
      & =
      \beta^*\bigl(f_B(z)\bigr),
    \end{align*}
    where the second equality follows from $\beta\comp\iota_{C,A} =
    \iota_{\beta(C),B}\comp\beta\down_C$, the third equality follows from inductive hypothesis and
    the fifth equality follows from equivariance of $G_{A,V}$ under bijections (see
    Claim~\ref{clm:GAV}).
  \end{proofof}

  We can finally define the $1$-Euclidean structure $\cH$ over $\widetilde{\Omega}=([0,1],\lambda)$
  that we will show to be $\ell$-independent and satisfy $\phi_\cH=\phi_{\widehat{\cN}}=\phi_\cN$:
  for each $P\in\cL$, we let
  \begin{align*}
    \cH_P
    & \df
    \Bigl\{z\in\cE_{k(P)}^{(1)}(\widetilde{\Omega}) \;\Big\vert\;
    M^{\amlg,\ell}_{k(P),V}\bigl(f_{k(P)}(z)\bigr)\vDash P(1,\ldots,k(P))
    \Bigr\}.
  \end{align*}
  (Recall that $M^{\amlg,\ell}_{k(P),V}\colon\cS_{k(P),V}^{(d),\amlg,\ell}(\Omega^2)\to\cK_{k(P)}$
  is the function that maps amalgamating types to their underlying model.)

  Note that since $f_A$ is constant when $\lvert A\rvert\leq\ell$, from the inductive definition of
  $f_A$ in~\eqref{eq:fA}, we get that $\cH$ is $\ell$-independent. It remains to show that
  $\phi_\cH=\phi_{\widehat{\cN}}$.

  Recalling that $\rn{z}$ is sampled in $\cE_{\NN_+}^{(1)}(\widetilde{\Omega})$, to show that
  $\phi_\cH=\phi_{\widehat{\cN}}$, it suffices to show that for every $A\in r(\NN_+)$, we have
  \begin{align*}
    M^\cH_A(\iota_{A,\NN_+}^*(\rn{z})) & \sim M^{\widehat{\cN}}_A(\rn{x}),
  \end{align*}
  where $\rn{x}$ is sampled in $\cE_A^{(d)}(\Omega^2)$ according to $\mu^{(d)}$.

  In turn, by Lemma~\ref{lem:amlgcomm} and Claim~\ref{clm:sampleamlg}, we have
  \begin{align*}
    M^{\widehat{\cN}}_A(\rn{x})
    & =
    M^{\amlg,\ell}_{A,V}\bigl(\pp^{\amlg,\ell,\widehat{\cN}}_{A,V}(\rn{x})\bigr)
    \sim
    M^{\amlg,\ell}_{A,V}(\rn{p}_A),
  \end{align*}
  so it suffices to show $M^\cH_A(\iota_{A,\NN_+}^*(\rn{z})) = M^{\amlg,\ell}_{A,V}(\rn{p}_A)$ with
  probability $1$. Going one step further, it suffices to show that for every $P\in\cL$ and every
  $\beta\colon[k(P)]\to A$ injective, we have
  \begin{align}\label{eq:disteq}
    \Bigl(
    M^\cH_A\bigl(\iota_{A,\NN_+}^*(\rn{z})\bigr)
    \vDash P(\beta)
    \Bigr)
    & \iff
    \bigl(
    M^{\amlg,\ell}_{A,V}(\rn{p}_A)
    \vDash P(\beta)
    \bigr)
  \end{align}
  with probability $1$.

  Recalling that $\overline{\beta}$ denotes the function obtained from $\beta$ by restricting its
  co-domain to $\im(\beta)$ (so $\beta = \iota_{\im(\beta),A}\comp\overline{\beta}$, hence
  $\iota_{A,\NN_+}\comp\beta=\iota_{\im(\beta),\NN_+}\comp\overline{\beta}$), note that
  \begin{align*}
    \Bigl(
    M^\cH_A\bigl(\iota_{A,\NN_+}^*(\rn{z})\bigr)
    \vDash P(\beta)
    \Bigr)
    & \iff
    \beta^*\bigl(\iota_{A,\NN_+}^*(\rn{z})\bigr)
    \in\cH_P
    \\
    & \iff
    \Biggl(
    M^{\amlg,\ell}_{k(P),V}\biggl(f_{k(P)}\Bigl(\beta^*\bigl(\iota_{A,\NN_+}^*(\rn{z})
    \bigr)\Bigr)\biggr)
    \vDash P(1,\ldots,k(P))
    \Biggr)
    \\
    & \iff
    \Biggl(
    M^{\amlg,\ell}_{k(P),V}\biggl(f_{k(P)}\Bigl(\overline{\beta}^*\bigl(
    \iota_{\im(\beta),\NN_+}^*(\rn{z})
    \bigr)\Bigr)\biggr)
    \vDash P(1,\ldots,k(P))
    \Biggr)
    \\
    & \iff
    \Biggl(
    M^{\amlg,\ell}_{k(P),V}\biggl(\overline{\beta}^*\Bigl(f_{\im(\beta)}\bigl(
    \iota_{\im(\beta),\NN_+}^*(\rn{z})
    \bigr)\Bigr)\biggr)
    \vDash P(1,\ldots,k(P))
    \Biggr)
    \\
    & \iff
    \Biggl(
    \overline{\beta}^*\biggl(M^{\amlg,\ell}_{\im(\beta),V}\Bigl(f_{\im(\beta)}\bigl(
    \iota_{\im(\beta),\NN_+}^*(\rn{z})
    \bigr)\Bigr)\biggr)
    \vDash P(1,\ldots,k(P))
    \Biggr)
    \\
    & \iff
    \Bigl(
    \overline{\beta}^*\bigl(M^{\amlg,\ell}_{\im(\beta),V}(\rn{p}_{\im(\beta)})\bigr)
    \vDash P(1,\ldots,k(P))
    \Bigr),
  \end{align*}
  where the fourth equivalence follows from Claim~\ref{clm:fAequiv}, the fifth equivalence follows
  from equivariance of $M^{\amlg,\ell}$ (Lemma~\ref{lem:amlgcomm}) and the sixth equivalence follows
  from~\eqref{eq:pAfA}.

  By Claim~\ref{clm:sampleamlg} and equivariance of $\pp^{\amlg,\ell,\widehat{\cN}}$
  (Lemma~\ref{lem:amlgcomm}), we know that with probability $1$ we have $\rn{p}_{\im(\beta)} =
  \iota_{\im(\beta),A}^*(\rn{p}_A)$, so with probability $1$ we have
  \begin{align*}
    \Bigl(
    \overline{\beta}^*\bigl(M^{\amlg,\ell}_{\im(\beta),V}(\rn{p}_{\im(\beta)})\bigr)
    \vDash P(1,\ldots,k(P))
    \Bigr)
    & \iff
    \biggl(
    \overline{\beta}^*\Bigl(M^{\amlg,\ell}_{\im(\beta),V}\bigl(
    \iota_{\im(\beta),A}^*(\rn{p}_A)
    \bigr)\Bigr)
    \vDash P(1,\ldots,k(P))
    \biggr)
    \\
    & \iff
    \biggl(
    \overline{\beta}^*\Bigl(\iota_{\im(\beta),A}^*\bigl(
    M^{\amlg,\ell}_{A,V}(\rn{p}_A)
    \bigr)\Bigr)
    \vDash P(1,\ldots,k(P))
    \biggr)
    \\
    & \iff
    \Bigl(
    \beta^*\bigl(M^{\amlg,\ell}_{A,V}(\rn{p}_A))\bigr)
    \vDash P(1,\ldots,k(P))
    \Bigr)
    \\
    & \iff
    \bigl(
    M^{\amlg,\ell}_{A,V}(\rn{p}_A)
    \vDash P(\beta)
    \bigr),
  \end{align*}
  where the second equivalence follows from equivariance of $M^{\amlg,\ell}$
  (Lemma~\ref{lem:amlgcomm}).

  Thus~\eqref{eq:disteq} holds, so $\phi_\cH = \phi_{\widehat{\cN}}=\phi_\cN$, concluding the proof
  of the proposition.
\end{proof}

As previously mentioned in Section~\ref{subsec:mainthms}, we can now derive Theorem~\ref{thm:Indep}
from Proposition~\ref{prop:UCouple->1ellindep} and Corollary~\ref{cor:astindep->indep}:
\begin{corollary}[Theorem~\ref{thm:Indep}]\label{cor:Indep}
  $\UCouple[\ell]\implies\Independence[\ell']$ whenever $\ell,\ell'\in\NN$ satisfy $\ell' < \ell$.
\end{corollary}

\begin{proof}
  By Proposition~\ref{prop:UCouple->1ellindep}, we know that every $\UCouple[\ell]$ limit is
  $1$-$\ell$-independent, in particular, $\ast$-$\ell$-independent, so
  Corollary~\ref{cor:astindep->indep} implies that it satisfies $\Independence[\ell']$.
\end{proof}

\section{Simulating order variables with the quasirandom orientation}
\label{sec:sim}

In this section, we prove the
implication~\ref{thm:UCouple:1ellindep}$\implies$\ref{thm:UCouple:kqrOcoupling} of
Theorem~\ref{thm:UCouple}, which essentially amounts to the almost trivial fact that we can use a
quasirandom orientation to simulate order variables. The proof of this implication is mostly about
bookkeeping.

\begin{proposition}[Theorem~\ref{thm:UCouple}\ref{thm:UCouple:1ellindep}$\implies$\ref{thm:UCouple:kqrOcoupling}]\label{prop:1ellindep->kqrOcoupling}
  If $\cN$ is an $\ell$-independent $1$-$T$-on, then there exist a canonical theory
  $T'$, an $\ell$-independent limit $\psi\in\HomT[T']$ and a quantifier-free definition $I\colon
  T\leadsto T'\cup\TkO[(\ell+1)]$ such that $\phi_\cN = I^*(\psi\otimes\phi^{\kqrO[(\ell+1)]})$.
\end{proposition}

\begin{proof}
  Let $\Omega=(X,\mu)$ be the underlying space of $\cN$, let $\cL$ be the language of $T$ and let
  $k(\cL)\df\max\{k(Q)\mid Q\in\cL\}\cup\{1\}$.

  For each predicate symbol $Q\in\cL$ and each ${\lhd}\in\prod_{A\in\binom{[k(Q)]}{\ell+1}}\cO_A$,
  we introduce a predicate symbol $Q_\lhd$ of arity $k(Q_\lhd)\df k(Q)$ and we let $\cL'$ be the
  language consisting of all such predicate symbols.

  For each $n\in\NN_+$, define the function
  \begin{align*}
    \pi_n\colon\cE_n^{(1)}\times\prod_{A\in\binom{[n]}{\ell+1}}\cO_A\to\cE_n^{(1)}
  \end{align*}
  by letting
  \begin{align*}
    \pi_n(x,{\ll},{\lhd}) & \df (x,{\ll^\lhd})
    \qquad
    \left(
    x\in\cE_n, {\ll}\in\prod_{A\in r(n)}\cO_A, {\lhd}\in\prod_{A\in\binom{[n]}{\ell+1}}\cO_A
    \right),
  \end{align*}
  where
  \begin{align*}
    {\ll^\lhd_A}
    & \df
    \begin{dcases*}
      \ll_A, & if $\lvert A\rvert\neq\ell+1$,\\
      \lhd_A, & if $\lvert A\rvert=\ell+1$,
    \end{dcases*}
  \end{align*}
  that is, $\pi_n$ simply replaces the order variables of $(x,\ll)$ corresponding to sets of size
  $\ell+1$ by those in $\lhd$.

  We now define a $1$-Euclidean structure $\cH$ in $\cL'$ over $\Omega$ by
  \begin{align*}
    \cH_{Q_\lhd}
    & \df
    \{(x,{\ll})\in\cE_{k(Q)}^{(1)} \mid \pi_{k(Q)}(x,{\ll},{\lhd})\in\cN_Q\}
    \qquad \left(
    Q\in\cL, {\lhd}\in\prod_{A\in\binom{[k(Q)]}{\ell+1}}\cO_A
    \right).
  \end{align*}

  We also let $T'$ be the (canonical) universal theory whose finite models are precisely those
  (canonical) $\cL'$-structures $M$ such that $\tind(M,\cH) > 0$ (such a universal theory exists as
  this set of structures is closed under substructures). In particular, $\cH$ is a $1$-$T'$-on.

  Note that since $\cN$ is $\ell$-independent and the functions $\pi_n$ replace the order
  coordinates of order $\ell+1$, it follows that $\cH$ is $\ell$-independent and does not depend on
  coordinates in $\cO_A$ with $\lvert A\rvert=\ell+1$, so by
  Lemma~\ref{lem:widehatEuclidean}\ref{lem:widehatEuclidean:dreal2}, we know that $\psi\df\phi_\cH$
  satisfies $\Independence[\ell]$.

  Let now $P$ be the unique predicate symbol in the language of $\TkO[(\ell+1)]$ and recall that the
  standard representation of the quasirandom $(\ell+1)$-orientation $\cN^{\kqrO[(\ell+1)]}$ over
  $\Omega$ is given by
  \begin{align*}
    \cN^{\kqrO[(\ell+1)]}_P
    & \df
    \{(x,{\lhd})\in\cE_{\ell+1}^{(1)} \mid {\lhd}_{[\ell+1]} = {\leq}\},
  \end{align*}
  where ${\leq}\in\cO_{\ell+1}$ is the usual order on $[\ell+1]$.

  Define the $1$-Euclidean structure $\cG$ in $\cL'\cup\{P\}$ over $\Omega$ by
  \begin{align*}
    \cG_P & \df \cN^{\kqrO[(\ell+1)]}_P, &
    \cG_{Q_\lhd} & \df \cH_{Q_\lhd}
    \qquad\left(
    Q\in\cL, {\lhd}\in\prod_{A\in\binom{[k(Q)]}{\ell+1}}\cO_A
    \right).
  \end{align*}

  Note that since $\cH$ does not depend on coordinates in $\cO_A$ with $\lvert A\rvert=\ell+1$ and
  $\cN^{\kqrO[(\ell+1)]}$ is completely determined by such coordinates, it follows that $\phi_\cG =
  \phi_{\cN^{\kqrO[(\ell+1)]}}\otimes\phi_\cH = \phi^{\kqrO[(\ell+1)]}\otimes\psi$.

  We now define the quantifier-free definition $I\colon T\leadsto T'\cup\TkO[(\ell+1)]$ by
  \begin{multline*}
    I(Q)(x_1,\ldots,x_{k(Q)})
    \\
    \df
    \bigvee_{{\lhd}\in\prod_{A\in\binom{[k(Q)]}{\ell+1}}\cO_A}
    \left(
    Q_\lhd(x_1,\ldots,x_{k(Q)})
    \land\bigwedge_{A\in\binom{[k(Q)]}{\ell+1}}
    P(x_{\alpha_{\lhd_A}(1)},\ldots,x_{\alpha_{\lhd_A}(k(Q))})
    \right),
  \end{multline*}
  where $\alpha_{\lhd_A}\colon[\ell+1]\to A$ is the unique bijection such that
  \begin{align*}
    \alpha_{\lhd_A}^*({\lhd_A}) & = {\leq}.
  \end{align*}
  Note that when $k(Q) < \ell+1$, then there is a unique element in the empty product
  $\prod_{A\in\binom{[k(Q)]}{\ell+1}}\cO_A$ and the conjunction
  $\bigwedge_{A\in\binom{[k(Q)]}{\ell+1}}$ is empty (which is to be interpreted as always true).

  We have to check that $I^*(\phi^{\kqrO[(\ell+1)]}\otimes\psi) = \phi_\cN$ and that $I$ is indeed a
  quantifier-free definition (i.e., it maps axioms of $T$ to theorems of $T'\cup\TkO[(\ell+1)]$).

  To check the former, it suffices to show that $I^*(\cG)=\cN$. In turn, it suffices to check that
  for every $Q\in\cL$, we have $T(I(Q),\cG) = \cN_Q$.

  When $k(Q) < \ell+1$, this is obvious as there is a unique element $\lhd$ in the empty product
  $\prod_{A\in\binom{[k(Q)]}{\ell+1}}\cO_A$ and we have $T(I(Q),\cG) = \cG_{Q_\lhd} = \cH_{Q_\lhd} =
  \cN_Q$.

  Suppose then that $k(Q)\geq\ell+1$ and note that
  \begin{align*}
    T\left(
    \bigwedge_{A\in\binom{[k(Q)]}{\ell+1}}
    P(x_{\alpha_{\lhd_A}(1)},\ldots,x_{\alpha_{\lhd_A}(k(Q))}),
    \cG
    \right)
    & =
    \bigcap_{A\in\binom{[k(Q)]}{\ell+1}}
    \bigl((\iota_{A,[k(Q)]}\comp\alpha_{\lhd_A})^*\bigr)^{-1}(\cN^{\kqrO[(\ell+1)]}_P)
    \\
    & =
    \left\{(x,{\ll})\in\cE_{k(Q)}^{(1)} \;\middle\vert\;
    \forall A\in\binom{k(Q)}{\ell+1},
    \alpha_{\lhd_A}^*({\ll}_A) = {\leq}
    \right\}
    \\
    & =
    \left\{(x,{\ll})\in\cE_{k(Q)}^{(1)} \;\middle\vert\;
    \forall A\in\binom{[k(Q)]}{\ell+1},
    {\ll}_A = {\lhd_A}
    \right\},
  \end{align*}
  from which we conclude that
  \begin{align*}
    & \!\!\!\!\!\!
    T(I(Q),\cG)
    \\
    & =
    \left\{(x,{\ll})\in\cE_{\ell+1}^{(1)} \;\middle\vert\;
    \exists{\lhd}\in\prod_{A\in\binom{[k(Q)]}{\ell+1}}\cO_A,
    \left(
    \pi_{k(Q)}(x,{\ll},{\lhd})\in\cN_Q
    \land\forall A\in\binom{[k(Q)]}{\ell+1},
    {\ll}_A = {\lhd_A}
    \right)
    \right\}
    \\
    & =
    \cN_Q,
  \end{align*}
  as desired.

  We now check that $I\colon T\leadsto T'\cup\TkO[(\ell+1)]$ is indeed a quantifier-free
  definition. We certainly know that $I$ is a quantifier-free definition of the form $T_\cL\leadsto
  T'\cup\TkO[(\ell+1)]$ as $T_\cL$ is the pure canonical theory. To show that it is also a
  quantifier-free definition of the form $T\leadsto T'\cup\TkO[(\ell+1)]$, it suffices to show that
  for every $n\in\NN$ and every $M\in\cK_n[T'\cup\TkO[(\ell+1)]]$, we have $I^*_n(M)\in\cK_n[T]$.

  By definition of $T'$, we know that $\tind(M\rest_{\cL'},\cH) > 0$. Since
  $\tind(M\rest_{\{P\}},\cN^{\kqrO[(\ell+1)]}) > 0$ (as every $(\ell+1)$-orientation has positive
  density in $\phi^{\kqrO[(\ell+1)]}$) and $\phi_\cG = \phi_{\cN^{\kqrO[(\ell+1)]}}\otimes\phi_\cH$,
  we get
  \begin{align*}
    \tind(M,\cG) & = \tind(M\rest_{\{P\}},\cN^{\kqrO[(\ell+1)]})\cdot\tind(M\rest_{\cL'},\cH) > 0,
  \end{align*}
  so $\phi_\cG(M) > 0$.

  On the other hand, since $I^*(\cG)=\cN$, we get
  \begin{align*}
    \phi_\cN\bigl(I^*_n(M)\bigr)
    & =
    \phi_{I^*(\cG)}\bigl(I^*_n(M)\bigr)
    =
    I^*(\phi_\cG)\bigl(I^*_n(M)\bigr)
    \geq
    \phi_\cG(M)
    >
    0,
  \end{align*}
  that is, $I^*_n(M)$ has positive density in $\cN$, so it must be a model of $T$ as $\cN$ is a $1$-$T$-on.
\end{proof}

\section{Concluding remarks and open problems}
\label{sec:conc}

In this work, we have closed the final implications of the poset of natural quasirandomness
introduced in~\cite{CR23}, namely, we proved in Theorem~\ref{thm:UCouple} that
$\UCouple[\ell]\implies\Independence[\ell']$ whenever $\ell > \ell'$. In fact, we proved a stronger
result in Theorem~\ref{thm:UCouple} that gives further alternative characterizations of
$\UCouple[\ell]$ in terms of $\ast$-$\ell$-independence and shows that the ``gap'' between
$\Independence[\ell]$ and $\UCouple[\ell]$ is precisely a quasirandom $(\ell+1)$-orientation
(item~\ref{thm:UCouple:kqrOcoupling}).

\medskip

Our new characterizations of $\UCouple[\ell]$ also make progress in some of the other questions
asked in~\cite{CR23}, on which we elaborate now.

In~\cite[Theorem~3.10]{CR23}, it was shown that $\phi\in\UCouple[\ell]$ is also equivalent to
$\phi\otimes\psilin\in\UInduce[\ell]$ where $\psilin\in\HomT[\TLinOrder]$ is the unique limit of
linear orders and it was asked whether a different kind of converse of this result holds. Namely, it
was asked whether every $\phi\in\UInduce[\ell]$ is of the form $I^*(\psi\otimes\psilin)$ for some
$\psi\in\UCouple[\ell]$ and some quantifier-free definition $I\colon T\leadsto T'\cup\TLinOrder$, in
other words, similarly to Theorem~\ref{thm:UCouple}\ref{thm:UCouple:kqrOcoupling}, the ``gap''
between $\UCouple[\ell]$ and $\UInduce[\ell]$ is precisely a linear order. Since the quasirandom
$(\ell+1)$-orientation $\phi^{\kqrO[(\ell+1)]}$ can itself be written as $J^*(\psi\otimes\psilin)$ for
some $\psi\in\Independence[\ell]$ and a quantifier-free interpretation
$J\colon\TkO[(\ell+1)]\leadsto T\cup\TLinOrder$, by our
Theorem~\ref{thm:UCouple}\ref{thm:UCouple:kqrOcoupling}, the aforementioned question of~\cite{CR23}
is equivalent to asking whether every $\phi\in\UInduce[\ell]$ is in fact of the form
$I^*(\psi\otimes\psilin)$ for some $\psi\in\Independence[\ell]$ and some quantifier-free definition
$I\colon T\leadsto T'\cup\TLinOrder$ (i.e., the ``larger gap'' between $\UInduce[\ell]$ and
$\Independence[\ell]$ is precisely a linear order).

In~\cite[Section~10]{CR23}, the notion of \emph{$B$-compatibility} was introduced as a common
generalization of rank and $\Independence[\ell]$: namely, for $B\subseteq\NN_+$, we say that $\phi$
is \emph{$B$-compatible} if it has a representation $\phi=\phi_\cN$ where $\cN$ is a theon that
depends only on coordinates indexed by sets $A$ with $\lvert A\rvert\in B$. Thus, rank at most
$\ell$ amounts to $[\ell]$-compatibility and $\Independence[\ell]$ amounts to
$(\NN_+\setminus[\ell])$-compatibility. It was already observed in~\cite{CR23} that if
$B\subseteq\NN_+$, then every $B$-compatible limit is uniquely coupleable with every
$(\NN_+\setminus B)$-compatible limit (see~\cite[Theorem~4.8.2]{Cor21} for a proof). In turn, since
$\Independence[\ell]\neq\UCouple[\ell]$, we also know that the set of limits that are uniquely
coupleable with all $B$-compatible limits can be strictly larger than the set of all
$(\NN_+\setminus B)$-compatible limits.

On the other hand, since Theorem~\ref{thm:UCouple} says that $\UCouple[\ell]$ is equivalent to
$\ast$-$\ell$-independence, it is natural to consider the version of compatibility that allows for
order variables. Namely, let us say that $\phi$ is \emph{$\ast$-$B$-compatible} ($B\subseteq\NN_+$)
if it has a representation $\phi=\phi_\cN$ where $\cN$ is a $d$-theon (for some $d\in\NN$) that
depends only on coordinates indexed by sets $A$ with $\lvert A\rvert\in B$. Thus, rank at most
$\ell$ amounts to $\ast$-$[\ell]$-compatibility and $\ast$-$\ell$-independence (i.e.,
$\UCouple[\ell]$) amounts to $\ast$-$(\NN_+\setminus[\ell])$-compatibility. Is it then true that
$\phi$ is uniquely coupleable with all $\ast$-$B$-compatible limits if and only if $\phi$ is
$\ast$-$(\NN_+\setminus B)$-compatible? The backward direction is a straightforward generalization
of~\cite[Theorem~4.8.2]{Cor21} and Theorem~\ref{thm:UCouple} says that the forward direction holds
when $B=[\ell]$ for some $\ell\in\NN$.

\bibliographystyle{alpha}
\bibliography{refs}

\end{document}